\documentclass[12pt,letterpaper]{article}

\usepackage{tikz-cd}
\usepackage{tikz}
\usepackage{amsmath}
\usepackage{amssymb}
\usepackage{geometry}
\usepackage{setspace}
\usepackage[OT2,T1]{fontenc}
\usepackage{mathtools,leftindex,tensor,mhchem}
\usepackage{mathrsfs}
\usepackage{indentfirst}
\usepackage{xypic}
\usepackage{amsthm}
\usepackage{graphicx}

\usepackage{hyperref}
\hypersetup{colorlinks=true,linkcolor=blue,anchorcolor=blue,citecolor=blue}
\usepackage{cleveref}
\usepackage{etoolbox}
\usepackage{color}
\usepackage{relsize}
\usepackage{amsfonts}
\usepackage{quiver}
\input pdfmsym
\pdfmsymsetscalefactor{10}

\numberwithin{equation}{subsection}

\theoremstyle{plain}
\newtheorem{theorem}[equation]{Theorem}

\newtheorem{proposition}[equation]{Proposition}
\newtheorem{lemma}[equation]{Lemma}

\newtheorem{corollary}[equation]{Corollary}

\newtheorem{construction/proposition}[equation]{Construction/Proposition}

\theoremstyle{definition}
\newtheorem{definition}[equation]{Definition}

\newtheorem{example}[equation]{Example}

\newtheorem{remark}[equation]{Remark}

\AddToHook{env/theorem/begin}{\crefalias{equation}{theorem}}
\AddToHook{env/thm/begin}{\crefalias{equation}{theorem}}
\AddToHook{env/conjecture/begin}{\crefalias{equation}{conjecture}}
\AddToHook{env/proposition/begin}{\crefalias{equation}{proposition}}
\AddToHook{env/lemma/begin}{\crefalias{equation}{lemma}}
\AddToHook{env/claim/begin}{\crefalias{equation}{claim}}
\AddToHook{env/corollary/begin}{\crefalias{equation}{corollary}}
\AddToHook{env/cor/begin}{\crefalias{equation}{corollary}}
\AddToHook{env/construction/proposition/begin}{\crefalias{equation}{construction/proposition}}

\AddToHook{env/definition/begin}{\crefalias{equation}{definition}} 
\AddToHook{env/defn/begin}{\crefalias{equation}{definition}}
\AddToHook{env/notation/begin}{\crefalias{equation}{notation}}
\AddToHook{env/convention/begin}{\crefalias{equation}{convention}}
\AddToHook{env/construction/begin}{\crefalias{equation}{construction}}
\AddToHook{env/variant/begin}{\crefalias{equation}{variant}}
\AddToHook{env/question/begin}{\crefalias{equation}{question}}
\AddToHook{env/hypothesis/begin}{\crefalias{equation}{hypothesis}}
\AddToHook{env/example/begin}{\crefalias{equation}{example}}
\AddToHook{env/digression/begin}{\crefalias{equation}{digression}}
\AddToHook{env/examples/begin}{\crefalias{equation}{example}} 
\AddToHook{env/situation/begin}{\crefalias{equation}{situation}}
\AddToHook{env/remark/begin}{\crefalias{equation}{remark}}
\AddToHook{env/rmk/begin}{\crefalias{equation}{remark}}
\AddToHook{env/exercise/begin}{\crefalias{equation}{exercise}}
\AddToHook{env/warning/begin}{\crefalias{equation}{warning}}
\AddToHook{env/setup/begin}{\crefalias{equation}{setup}}

\newcommand{\Z}{\mathbb{Z}}
\newcommand{\N}{\mathbb{N}}
\newcommand{\F}{\mathbb{F}}

\newcommand{\cris}{\mathrm{cris}}

\newcommand{\Br}{\mathrm{Br}}

\newcommand{\HH}{\mathrm{H}}
\newcommand{\Hom}{\mathrm{Hom}}

\newcommand{\Fil}{\mathrm{Fil}}

\newcommand{\Tot}{\mathrm{Tot}}  
\newcommand{\Ch}{\mathrm{Ch}}
\newcommand{\dt}{\widetilde{\tau}}
\newcommand{\Dt}{\widetilde{D}}
\newcommand{\tH}{\widetilde{\HH}}
\newcommand{\tF}{\widetilde{F}}
\newcommand{\tV}{\widetilde{V}}

\newcommand{\FGauge}{{F\text{-}\mathsf{Gauge}}}
\newcommand{\FG}{F\mathsf{G}}
\newcommand{\td}{\text{-}}
\newcommand{\I}{\mathrm{I}}
\newcommand{\rmt}{\mathrm{t}}
\newcommand{\Ext}{\mathrm{Ext}}
\newcommand{\coeur}{\text{\rm {c\oe ur}}}
\newcommand{\fg}{f\mathsf{g}}
\newcommand{\RC}{R\mathsf{C}}
\newcommand{\rmL}{\mathrm{L}}
\newcommand{\Nyg}{\mathrm{Nyg}}
\newcommand{\HW}{\mathrm{HW}}
\newcommand{\MO}{\mathrm{MO}}
\newcommand{\Hod}{\mathrm{Hod}}

\newcommand{\rmT}{\mathrm{T}}

\newcommand{\dc}{%
  \mathrel{\mathchoice
    {\dcl{0.16}}
    {\dcl{0.16}}
    {\dcl{0.11}}
    {\dcl{0.08}}
  }%
}
\newcommand{\dcl}[1]{\,%
  \mathrel{\raisebox{-0.15ex}{\hbox{%
  \begin{tikzpicture}[scale=#1, line width=0.6pt]
    \draw (0,0) -- (2,0) -- (1.2,1.75) -- (0.8,1.75) -- cycle;
    \draw (1.6,0)-- (0.8,1.75);
  \end{tikzpicture}}}%
}}

\usepackage{pict2e}

\makeatletter
\newcommand{\olddc}{\mathord{\mathpalette\rDelta@\relax}}
\newcommand{\rDelta@}[2]{%
  \begingroup
  \edef\rDelta@thick{%
    1.5\fontdimen8
      \ifx\displaystyle#1\textfont\else
      \ifx\textstyle#1\textfont\else
      \ifx\scriptstyle#1\scriptfont\else
      \scriptscriptfont\fi\fi\fi 3
    }%
  \edef\rDelta@slant{%
      \ifx\displaystyle#10.71\else
      \ifx\textstyle#10.71\else
      \ifx\scriptstyle#10.71\else
      0.69\fi\fi\fi
  }%
  \sbox\z@{$\m@th#1\Delta$}%
  \begin{picture}(\wd\z@,\ht\z@)
    \linethickness{\rDelta@thick}
    \put(0,0){\usebox{\z@}}
    \Line(\rDelta@slant\wd\z@,0.2)(0.4\wd\z@,0.8\ht\z@)
  \end{picture}%
  \endgroup
}

\title{Ekedahl's Equivalence Explained}  
\author{Yuan Yang}

\begin{document}
\maketitle

\begin{abstract}
Ekedahl's theory relates diagonal complexes to \(F\)-gauges, but
turning this equivalence into explicit calculations requires careful
control of its constructions. These notes explain the equivalence
and clarify the generators-and-relations definition of \(\rmt\).
At each finite level, we construct bounded projective resolutions
of bounded coherent gauge complexes by finite sums of the standard
generators. Morita equivalence expresses \(\rmt\) as a tensor product
functor. Comparing the constructions across levels yields an exact
fully faithful embedding of the diagonal heart into \(\Ch(R)\).
Finally, we compute three distinct domino types arising from abelian
threefolds that are all supergeneral.
\end{abstract}

\tableofcontents

\section*{Introduction}
\label{sec:introduction}

The de Rham--Witt complex carries several structures at once:
Frobenius, Verschiebung, the de Rham differential, and the filtrations
relating crystalline and Hodge cohomology. Each is familiar on its
own; their interaction is much harder to see. Ekedahl's theory of
diagonal complexes and \(F\)-gauges \cite{Ek3} offers a way to bring
them into a single picture. These notes grew out of a desire to
understand that picture well enough to compute with it.

Let \(k\) be a perfect field of characteristic \(p\), and let \(R\)
be its Cartier--Dieudonn\'e--Raynaud ring. An \(R\)-complex has two
degrees: the internal grading, on which the Raynaud differential
acts, and the cohomological grading. The diagonal \(t\)-structure
on \(D_c^b(R)\) relates these two gradings. Its heart \(\dc\)
contains both Dieudonn\'e modules and dominos. The latter record
the nonzero differentials in Hodge--Witt cohomology and account for
much of the theory's distinctive behavior.

At a fixed finite level \(\I\), Ekedahl's functor \(\FG^{\Tot}\)
gives an equivalence between \(D_c^b(R\td\I)\) and
\(D_c^b(\FGauge\td\I)\). The forward functor assembles Nygaard
modifications into a complex of gauges. The inverse,
\(\widehat{(\rmL\rmt)^{\Tot}}\), starts with a
generators-and-relations construction \(\rmt\), then derives,
totalizes, and completes it. The equivalence is a powerful statement,
but a concrete example soon raises more questions: what does
\(\widehat{(\rmL\rmt)^{\Tot}}(M)\) look like, how can we compute it,
and does it depend on the interval \(\I\)?

We work through these questions explicitly. Several formulas in
\cite{Ek3} need clarification or correction, especially where grading
shifts and Frobenius twists meet. We give a precise definition of
\(\rmt\), check its adjunction, and compute its values on the
projective generators \(G_r\). We also prove that bounded coherent
gauge complexes at a fixed level admit bounded projective
resolutions by finite sums of the \(G_r\). This gives a practical
starting point for computing the derived inverse.

Morita equivalence \cite[Theorem 2.5]{Schwede} makes the construction
easier to recognize.
Gauges of level \(\I\) are modules over the opposite endomorphism
ring \(A\) of \(\oplus_{r\in\I}G_r\), and
\(\rmt(M)\cong\mathcal T\otimes_A M\), where
\(\mathcal T=\rmt(A)\) has commuting left \(R\)- and right
\(A\)-actions. The generators and relations thus become a tensor
product. Resolutions and Tor groups can then be used to study
completion and change of level.

We retain Ekedahl's notation where possible; the table below records
the main differences. Our notation also keeps two constructions
apart. The functor \(\widehat{(\rmL\rmt)^{\Tot}}\) is Ekedahl's
completed derived inverse, whereas
\((\rmL\widehat{\rmt})^{\Tot}\) derives the completed underived
functor. Understanding their relation is one of the threads running
through these notes.

A further question emerges from these calculations. The projective
generators depend on the chosen interval \(\I\). How much of the
resulting complex really depends on that choice? We construct
compatible comparison isomorphisms for coherent gauges as the
interval grows. They lead to an exact fully faithful functor
\[
    \mathcal S:\dc\longrightarrow\Ch(R),
    \qquad q\mathcal S\simeq\mathrm{id}_{\dc},
\]
where \(q:\Ch(R)\to D(R)\) is the localization; see
\Cref{embedding-diagonal-complexes}.
Thus a diagonal complex admits a functorial representative supported
in bidegrees \((i,-i)\) and \((i,1-i)\). Its morphisms are represented
by actual chain maps, with no nonzero chain homotopies between these
representatives. An object initially given only in the derived
category therefore acquires a particularly concrete form.

The final examples return to the question that makes such
constructions useful: what can they tell us about Hodge--Witt
cohomology? We consider three supergeneral abelian threefolds.
For each one, we start from a cyclic presentation of
\(\HH^1_{\cris}(A/W)\), take its exterior square, and compute the
associated Nygaard gauge. We then describe the resulting domino
through extensions of elementary dominos. All three examples are
supergeneral, yet they give three distinct domino types. Even in
this small collection, the integral Frobenius relations leave
different traces in Hodge--Witt cohomology.

The organization largely follows \cite{Ek3}.
\Cref{sec:diagonal-t-structure}, corresponding to Chapter~0,
constructs the diagonal \(t\)-structure from radical filtrations
and Ext vanishing.
\Cref{sec:basic-facts}, corresponding to Chapter~I, studies duality,
Hodge--Witt objects, and Mazur--Ogus objects.
\Cref{sec:f-gauge-structures} develops Chapter~II, including the
gauge functors, completion, and coherent projective resolutions.
\Cref{sec:independence-of-interval} compares levels and constructs
\(\mathcal S\).
We return to Chapter~III in
\Cref{sec:diagonal-complexes-revisited}, before carrying out the
three examples in \Cref{sec:application-to-supergeneral-h2}.

For ease of comparison, the main correspondences are collected below.
\begin{center}
\begingroup
\renewcommand{\arraystretch}{1.4}
\begin{tabular}{lll}
\hline
Notation in \cite{Ek3} & Notation here & Meaning \\
\hline
\(\mathrm{S}\) & \(\FG\) & Gauge construction \\
\(\underline{\underline{\mathrm{S}}}\) & \(\FG^{\Tot}\) & Total gauge functor \\
\(\rmL\widehat{\rmt}\) & \(\widehat{(\rmL\rmt)^{\Tot}}\) & Completed derived inverse \\
\(\underline{\underline{\mathrm{s}}}\) & \(\Tot\) & Totalization \\
\(\rmt'\) & \(\widehat{\rmt}\) & Completed underived functor \\
\(\mathrm{Hodge}(M)\) & \(\Nyg(M)\) & Nygaard \(F\)-gauge \\
Type \(-\frac12\) diagonal dominos & Type \(\frac12\) diagonal dominos & Same class \\
\hline
\end{tabular}
\endgroup
\end{center}

\subsection*{Acknowledgements}
The author thanks Yuanning Zhang for pointing out an error in
Ekedahl's treatment of the classification discussed in
\Cref{weakly simple type 1/2 diagonal dominos}, and for sharing his
ideas on the Harder--Narasimhan filtration on diagonal dominos,
now developed in \cite{ZhangHN}.
The author also acknowledges Zhang's classification of dominos of
fixed type, including all dominos of type \(\{2,3,4\}\), his description
of their extension data, and his computation for cyclic supergeneral
threefolds \cite[Section~3, especially Proposition~3.21, and
Theorem~5.13]{ZhangDom}, as well as his work on
extensions between stable diagonal dominos
\cite[Section~4.3]{ZhangHN}.
The author used AI tools to assist in the preparation of these notes
and takes responsibility for their mathematical content and any
remaining errors.

\subsection*{Notations}
\label{subsec:notations}

Fix a perfect field \(k\) of characteristic \(p\), let \(W=W(k)\), and
write \(\sigma\) for the Frobenius on both \(k\) and \(W\). The
Cartier--Dieudonn\'e--Raynaud ring of \(k\) \cite{IR} is the graded ring
\[
    R=R^0\oplus R^1
\]
generated by \(F,V\) in degree \(0\) and \(d\) in degree \(1\), subject to
\[
Fa=a^\sigma F,\quad Va=a^{\sigma^{-1}}V,\quad
FV=VF=p,\quad da=ad,\quad FdV=d
\qquad(a\in W).
\]

Unless stated otherwise, an \(R\)-module is a graded left \(R\)-module.
It may be viewed as a complex of \(W\)-modules with operators \(F,V\)
satisfying \(FdV=d\). Accordingly, an \(R\)-complex \(M\) is a double
complex: \(M^{i,j}\) denotes its component of internal \(R\)-degree
\(i\) and cohomological degree \(j\). For fixed \(j\),
\(M^{\bullet,j}\) is the \(j\)-th graded \(R\)-module of the complex,
whereas for fixed \(i\), \(M^{i,\bullet}\) is a complex of
\(W\)-modules. The two differentials are
\[
    d:M^{i,j}\longrightarrow M^{i+1,j},
    \qquad
    \delta:M^{i,j}\longrightarrow M^{i,j+1}.
\]
Here \(d\) is the horizontal Raynaud differential, while \(\delta\)
is the cohomological differential; the latter is \(R\)-linear and
hence commutes with \(F,V,d\).
\[
\begin{tikzcd}[
  row sep=large, column sep=large,
  execute at end picture={
    \foreach \r in {2,3,4} {
      \draw[blue!45!gray, thin, rounded corners=2pt]
        ([xshift=-5pt,yshift=14pt]\tikzcdmatrixname-\r-1.center)
        rectangle
        ([xshift=5pt,yshift=-10pt]\tikzcdmatrixname-\r-5.center);
      \node[anchor=west, text=blue!45!gray, font=\small]
        at ([xshift=10pt]\tikzcdmatrixname-\r-5.center)
        {\(R\)-module};
    }
  }
]
& \vdots & \vdots & \vdots & \\
\cdots \arrow[r, "d"]
& M^{i-1,j+1} \arrow[r, "d"] \arrow[u, "\delta"]
& M^{i,j+1} \arrow[r, "d"] \arrow[u, "\delta"]
& M^{i+1,j+1} \arrow[r, "d"] \arrow[u, "\delta"]
& \cdots
\\
\cdots \arrow[r, "d"]
& M^{i-1,j} \arrow[r, "d"] \arrow[u, "\delta"]
& M^{i,j} \arrow[r, "d"] \arrow[u, "\delta"]
& M^{i+1,j} \arrow[r, "d"] \arrow[u, "\delta"]
& \cdots
\\
\cdots \arrow[r, "d"]
& M^{i-1,j-1} \arrow[r, "d"] \arrow[u, "\delta"]
& M^{i,j-1} \arrow[r, "d"] \arrow[u, "\delta"]
& M^{i+1,j-1} \arrow[r, "d"] \arrow[u, "\delta"]
& \cdots
\\
& \vdots \arrow[u, "\delta"]
& \vdots \arrow[u, "\delta"]
& \vdots \arrow[u, "\delta"] &
\end{tikzcd}
\]

In the above diagram, each row is a graded \(R\)-module, with \(F\) and \(V\) acting on
each entry; the upward maps \(\delta\) form the \(R\)-complex.
The squares commute, and \(d^2=\delta^2=0\).

\begin{definition}[Shifts of \(R\)-complexes]\label{shifts of R-complexes}
We follow Ekedahl's convention \cite[Section~0, p.~188]{Ek1}; see
also \cite[Chapter~0]{Ek3}. For \(i,j\in\Z\), write \(M(i)[j]\) for
the shift by \(i\) in the internal \(R\)-grading and by \(j\) in
cohomological degree. If \(M=(M^{k,\ell},d,\delta)\), then
\[
    M(i)[j]^{k,\ell}=M^{k+i,\ell+j},\qquad
    d_{M(i)[j]}=(-1)^i d_M,\qquad
    \delta_{M(i)[j]}=(-1)^j\delta_M.
\]
In particular,
\(\HH^q(M(i)[j])^k\simeq\HH^{q+j}(M)^{k+i}\).
\end{definition}

\begin{definition}[Standard filtration and reduction by \(R_n\)]
\label{reduction by R_n}
For a graded \(R\)-module \(M\), its standard filtration is
\[
    \operatorname{Fil}^nM^i:=V^nM^i+dV^nM^{i-1}
    \qquad(n\geqslant0).
\]
The operators \(W,d,V\) preserve this filtration, whereas
\(F(\operatorname{Fil}^{n+1}M)\subseteq\operatorname{Fil}^nM\).
For \(n\geqslant1\), put \(W_n=W/p^nW\) and define the graded right
\(R\)-module
\[
    R_n:=R/(V^nR+dV^nR).
\]
The denominator is a right ideal, and \(R_n\) is understood as a
quotient of right \(R\)-modules. Left multiplication by \(W\) and \(d\)
descends to \(R_n\), and \(p^nR_n=0\). Thus \(R_n\) is a graded
\((W_n[d]/(d^2),R)\)-bimodule, with \(d\) in degree \(1\).
Quotienting by the standard filtration is tensoring with \(R_n\):
\[
    (R_n\otimes_R M)^i\cong M^i/\operatorname{Fil}^nM^i.
\]
For \(M\in D(R)\), define its derived reduction by
\[
    R_n\otimes_R^{\rmL}M
    \in D\bigl(W_n[d]/(d^2)\bigr).
\]
The target is the derived category of graded modules: internal and
cohomological degrees remain separate, with no totalization. For
\(n=1\), the target is \(D(k[d]/(d^2))\).

For a \(K\)-flat resolution \(P^\bullet\to M\), the derived reduction
is represented by \(P^\bullet/\operatorname{Fil}^nP^\bullet\).
\end{definition}

See \cite[Chapter~0, Section~3]{Ek3}.

\begin{definition}[Ekedahl completion]
\label{Ekedahl completion}
The quotient maps \(R_{n+1}\to R_n\) form a tower. Following Ekedahl,
define
\[
    \widehat M
    :=R\varprojlim_{n\geqslant1}
       \bigl(R_n\otimes_R^{\rmL}M\bigr).
\]
The limit is taken before totalization in complexes of graded
\(W[d]/(d^2)\)-modules. It has a natural \(R\)-action and a natural
map \(M\to\widehat M\). We call \(M\) complete when this map is a
quasi-isomorphism.
\end{definition}

See \cite[Chapter~0, Section~4]{Ek3}.

To describe the \(R\)-action, choose a \(K\)-projective resolution
\(P^\bullet\to M\) by projective graded \(R\)-modules. Its
quotient tower is degreewise surjective, and \(\widehat M\) is represented by
\[
    \varprojlim_n
    P^\bullet/\operatorname{Fil}^nP^\bullet.
\]
The operators \(W,d,V\) act at each level, while \(F\) maps level
\(n+1\) to level \(n\) by the filtration property above. Together
they define the \(R\)-action on the limit, with the Raynaud relations
inherited from \(P^\bullet\).

For \(R\)-complexes, a hat always denotes this completion, not the
derived \(p\)-adic completion
\(R\varprojlim_n(W_n\otimes_W^{\rmL}M)\) used for F-gauges.

Completeness allows us to detect vanishing at the first reduction:
under a mild boundedness hypothesis, a complete complex is zero as
soon as its derived reduction by \(R_1\) is zero.

\begin{proposition}\label{R1 detects vanishing of complete complexes}
Let \(M\in D(R)\) be complete and bounded from either of the four directions
in the sense of \Cref{derived category of R-modules of level I}.
If \(R_1\otimes_R^{\rmL}M=0\), then \(M=0\).
\end{proposition}

\begin{proof}
By \Cref{Ek2I.1.1}, applied with \(A=0\) and varying \(r,s\),
we have \(R_n\otimes_R^{\rmL}M=0\) for every \(n\geqslant1\).
Thus completeness gives
\(M\simeq R\varprojlim_n(R_n\otimes_R^{\rmL}M)=0\).
\end{proof}

\begin{definition}[Nygaard modification]
\label{Nygaard modification}
Let \(\varphi:\Z\to\Z\) have finite support, and let \(M\) be a
graded left \(R\)-module. Its Nygaard modification
\(M(\varphi)\) has graded pieces
\[
    M(\varphi)^i=\sigma_*^{s_i}M^i,
    \qquad
    s_i=-\sum_{j\geqslant i+1}\varphi(j),
\]
with the induced operators \(F,V\) and differential
\[
    d(\varphi):M(\varphi)^{i-1}\longrightarrow M(\varphi)^i,
    \qquad
    d(\varphi)=
    \begin{cases}
        F^{\varphi(i)}d,&\varphi(i)\geqslant0,\\
        dV^{-\varphi(i)},&\varphi(i)<0.
    \end{cases}
\]
The scalar twists make \(d(\varphi)\) \(W\)-linear, and the relation
\(FdV=d\) makes \(M(\varphi)\) an \(R\)-module. We use
\(F^{\varphi(i)}d=dV^{-\varphi(i)}\) when \(\varphi(i)<0\); this
does not mean that \(F\) is inverted.

The construction is functorial and exact. Moreover, for
finitely supported functions \(\varphi,\psi:\Z\to\Z\),
\[
    M(0)=M,\qquad
    (M(\varphi))(\psi)=M(\varphi+\psi).
\]
Thus \((\mathrm{Fun}_c(\Z,\Z),+)\) acts on graded \(R\)-modules by
exact autoequivalences, with \((-)(-\varphi)\) inverse to
\((-)(\varphi)\). Applied termwise, the same construction defines
exact autoequivalences of \(D(R)\).
\end{definition}

Compare \cite[Chapter~0, Section~2]{Ek3}.

\begin{lemma}\label{completion commutes with Nygaard modification}
Let \(M\in D(R)\), and let \(\varphi:\Z\to\Z\) have finite support.
Then completion commutes with Nygaard modification:
\[
\widehat M(\varphi)\simeq\widehat{M(\varphi)}.
\]
\end{lemma}

\begin{proof}
As in \cite[Proposition~0.2.1]{Ek3}, the identity
\[
(-)(\varphi)(\psi)=(-)(\varphi+\psi)
\]
reduces the claim to \(\varphi=\pm\delta_i\), where \(\delta_i\) is
the characteristic function of \(\{i\}\). By d\'evissage and free
resolutions, it then suffices to treat \(R(1-i)\), the only shifted
free module whose differential is changed. This is the explicit
free-module calculation there.
\end{proof}

\begin{definition}[Elementary and coherent \(R\)-modules]
\label{coherent R-modules}
Following Ekedahl, an elementary graded \(R\)-module is one
of the following:
\begin{enumerate}
    \item A finite free \(W\)-module in degree \(0\), with \(F\)
    bijective.
    \item A finite free \(W\)-module in degree \(0\), with both
    \(F\) and \(V\) topologically nilpotent for the \(p\)-adic topology.
    \item A finitely generated \(W\)-module in degree \(0\), killed
    by \(p\), with \(F\) bijective.
    \item The module \(k\) in degree \(0\), with \(F=V=d=0\).
    \item An elementary domino \(U_i\), for \(i\in\Z\), concentrated
    in degrees \(0\) and \(1\):
    \[
        U_i=\left(
        \prod_{n\geqslant0}kV^n
        \xrightarrow{\ d\ }
        \prod_{n\geqslant0}k\,dV^{n+i}
        \right).
    \]
    Here \(dV^{-j}\) means \(F^jd\) for \(j\geqslant0\).
    In degree \(0\), \(F=0\) and \(V\) is the usual semilinear
    shift; in degree \(1\), \(V=0\) and \(F\) sends \(dV^t\)
    to \(dV^{t-1}\), with the term below the lower bound \(i\)
    set equal to zero. The differential sends \(V^n\) to \(dV^n\)
    if \(n\geqslant i\), and to zero otherwise.
\end{enumerate}
A graded \(R\)-module is coherent if it has a finite filtration by
graded \(R\)-submodules whose successive quotients are grading shifts
of elementary modules. This is not the usual notion of finite
presentation over \(R\); elementary dominos are included by definition.
\end{definition}

\begin{definition}[Coherent derived categories]
\label{coherent derived categories of R}
Let \(D_c(R)\) be the full subcategory of \(D(R)\) consisting of
complexes \(M\) such that every \(\HH^q(M)\) is a coherent graded
\(R\)-module, and set
\[
    D_c^\ast(R):=D_c(R)\cap D^\ast(R),
    \qquad \ast\in\{-,+,b\}.
\]
The superscripts \(-\), \(+\), and \(b\) mean that the cohomology
vanishes in sufficiently large degrees, in sufficiently small degrees,
and outside a finite range, respectively.
Each coherent cohomology module is supported in finitely many internal
degrees, but objects of \(D_c(R)\) need not be cohomologically bounded.
\end{definition}

Ekedahl's coherence theorem \cite[Section 4]{Ek3} gives the following
bounded-above criterion: \(M\in D^-(R)\) belongs to \(D_c^-(R)\)
if and only if \(M\) is complete in the sense of
\Cref{Ekedahl completion} and, for every \(n\geqslant1\),
\(R_n\otimes_R^{\rmL}M\) has bounded cohomology, each cohomology
module having finite total length as a graded \(W_n\)-module.
Equivalently, each reduction is a bounded coherent \(W_n\)-complex,
with its internal grading retained. For a graded \(R\)-module in
cohomological degree \(0\), this recovers the preceding definition.

\begin{definition}\label{stable cycles boundaries and dominos}
Let \(M\) be a graded \(R\)-module. Following \cite{IR}, put
\[
F^\infty BM^i
 :=\bigcup_{j\geqslant0}F^j\bigl(dM^{i-1}\bigr),
\qquad
V^{-\infty}ZM^i
 :=\bigcap_{j\geqslant0}
 \ker\bigl(dV^j:M^i\longrightarrow M^{i+1}\bigr).
\]
Thus \(F^\infty BM^i\) is the \(F\)-saturation of the boundaries,
and \(V^{-\infty}ZM^i\) consists of those \(x\in M^i\) for which
\(dV^j x=0\) for every \(j\geqslant0\).
Define the diagonal truncations of \(M\) by
\[
\dt_{\leqslant i}M
 :=
 \bigl(\cdots\xrightarrow{d}M^{i-1}\xrightarrow{d}M^i
 \xrightarrow{d}F^\infty BM^{i+1}\longrightarrow0\bigr),
\qquad
\dt_{\geqslant i+1}M:=M/\dt_{\leqslant i}M.
\]
The displayed terms retain their original module degrees.
Finally, define
\[
\mathrm{dom}^i(M)
 :=
 \bigl(M^i/V^{-\infty}ZM^i
       \xrightarrow{\ \bar d\ }F^\infty BM^{i+1}\bigr),
\]
with the two terms in module degrees \(i\) and \(i+1\), respectively,
and with \(\bar d\) induced by \(d\).
\end{definition}

\begin{definition}\label{R-modules of level I}
Let \(\I=[m,n]\). An \(R\)-module of level \(\I\) is a graded
left \(R\)-module of the form
\[
M^m\xrightarrow{d}M^{m+1}\xrightarrow{d}\cdots\xrightarrow{d}M^n
\]
with \(F\) bijective on \(M^n\). Equivalently, the top term carries
an action of \(R^{\mathrm{top}}=W_{\sigma}[F,F^{-1}]\) extending its
\(R^0\)-action. Morphisms are grading-preserving \(R\)-linear maps.
We denote this abelian category by \(R\td\mathrm{Mod}\td\I\).
\end{definition}

This is the fixed-level condition of \cite[Chapter~0, Section~5]{Ek3}.

For more of the basic theory of \(R\)-modules, see \cite{IR,Ek3}.

\section{Diagonal t-structure on \(D_c^b(R)\)}
\label{sec:diagonal-t-structure}

Ekedahl found a striking feature of \(D_c^b(R)\): it carries two
commuting \(t\)-structures, the ordinary one and the diagonal one.
They organize the same complexes in different ways, yet their
truncations are compatible. What makes this observation especially
fruitful is that it points beyond \(R\)-complexes. Ekedahl extracted
from it a general principle for commuting \(t\)-structures on a
triangulated category: their interaction is encoded by a filtration
on a heart, from which the second \(t\)-structure can be recovered.

We begin with this principle, following the preliminary chapter of
\cite{Ek3}. The key notion is a radical filtration. It turns the
problem of constructing a second \(t\)-structure into concrete
vanishing conditions. We then return to coherent \(R\)-modules and
verify these conditions, recovering the diagonal \(t\)-structure
on \(D_c^b(R)\). This passage from an abstract filtration to an
explicit construction is one of the appealing features of Ekedahl's
approach.

\subsection{Radical filtration}
\label{subsec:radical-filtration}

Let a triangulated category carry two commuting \(t\)-structures:
each truncation functor of the first commutes with each truncation
functor of the second. On the heart of the first, the second induces
a filtration by functorial subobjects. What is less obvious, and
central to Ekedahl's approach, is that this filtration can recover
the second \(t\)-structure under suitable orthogonality conditions.
This leads to the following definition.

\begin{definition}[{\cite[Definition~0.1.2]{Ek3}}]\label{radical filtration}
Let \(D\) be a triangulated category endowed with a \(t\)-structure
whose heart is \(D^0\). A radical filtration on \((D,D^0)\) is an
increasing sequence
\[
\cdots\subseteq T^i\subseteq T^{i+1}\subseteq\cdots
\subseteq\mathrm{id}_{D^0}
\]
of subfunctors \(T^i:D^0\to D^0\) of the identity functor, satisfying the following conditions.

{\rm (1)} For every \(M\in D^0\), one has \(T^i(M)=0\) for
\(i\ll0\) and \(T^i(M)=M\) for \(i\gg0\).

{\rm (2)} For all \(i,j\in\Z\) and \(M,N\in D^0\),
\[
\Hom_D^{j-i}\bigl(T^i(M),N/T^j(N)\bigr)=0.
\]

{\rm (3)} For all \(i<j\) and \(M,N\in D^0\),
\[
\Hom_D^{j-i+1}\bigl(T^i(M),N/T^j(N)\bigr)=0.
\]

Here
\(\Hom_D^r(M,N):=\Hom_D(M,N[r])\).
\end{definition}

\begin{remark}
The orthogonality conditions can also be read on the successive
layers of the filtration. Set \(S^i(M)=T^i(M)/T^{i-1}(M)\). By d\'evissage,
conditions (2) and (3) are together equivalent to the layerwise
vanishing
\[
\Hom_D^r\bigl(S^i(M),S^j(N)\bigr)=0
\]
whenever \(r\leqslant j-i-1\) or
\(2\leqslant r\leqslant j-i\).
\end{remark}

Before stating the reconstruction theorem, we describe the filtration
that motivates it. For coherent \(R\)-modules, its terms are explicit:
they retain the lower internal degrees and add the Frobenius-saturated
boundaries in the next degree. The main work will be to check the
required Ext group vanishing; see
\Cref{Ekedahl vanishing over any perfect field}.

\begin{example}[The filtration for coherent \(R\)-modules]\label{example:radical-filtration}
The motivating example is the ordinary heart of \(D_c^b(R)\). For a
coherent \(R\)-module \(M\), set
\[
\dt_{\leqslant i}M
 =\bigl(\cdots\xrightarrow dM^{i-1}\xrightarrow dM^i
   \xrightarrow dF^\infty BM^{i+1}\longrightarrow0\bigr)
\]
and
\[
\dt_{\geqslant i+1}M:=M/\dt_{\leqslant i}M
 \simeq
 \bigl(0\longrightarrow M^i/F^\infty BM^i
   \xrightarrow dM^{i+1}\xrightarrow d\cdots\bigr).
\]
The Ext calculations in
\Cref{Ekedahl vanishing over any perfect field} show that the
subfunctors \(\dt_{\leqslant i}\) form a radical filtration. The
associated \(t\)-structure is the diagonal \(t\)-structure.
\end{example}

We now return to the abstract setting. The theorem below shows that a
radical filtration is exactly the data needed to construct a bounded
\(t\)-structure commuting with the original one. Its formulas give a
useful test: membership in either half of the new structure can be
checked on the original cohomology objects, one degree at a time.

\begin{theorem}[{\cite[Proposition~0.1.4]{Ek3}}]
    \label{diagonal t-structure}
Let \((D^{\leqslant0},D^{\geqslant0})\) be a \(t\)-structure on a
triangulated category \(D\), with heart \(D^0\), and suppose that it
is bounded: every object has finite cohomological amplitude.
Write \(\HH^i\) for its cohomology functors.

{\rm (1)} A radical filtration \(T^\bullet\) on \((D,D^0)\) defines
a bounded \(t\)-structure by
\[
\begin{aligned}
\Dt_{\leqslant0}
 &=\bigl\{X\in D\mid
      T^{-i}(\HH^i(X))=\HH^i(X)\text{ for every }i\bigr\},\\
\Dt_{\geqslant1}
 &=\bigl\{X\in D\mid
      (\mathrm{id}/T^{-i})(\HH^i(X))=\HH^i(X)
      \text{ for every }i\bigr\}.
\end{aligned}
\]
Thus the lower half consists of objects whose \(i\)-th cohomology is
entirely contained in filtration step \(-i\); the upper half consists
of those for which this step is zero. This \(t\)-structure commutes
with the original one: each of its truncation functors commutes, up
to natural isomorphism, with every truncation functor of the original
\(t\)-structure.

{\rm (2)} Conversely, let \(\dt\) denote the truncation functors of
any bounded \(t\)-structure commuting with the original one.
Restricting these functors to the heart \(D^0\) gives a radical
filtration
\[
\cdots\subseteq\dt_{\leqslant i}
 \subseteq\dt_{\leqslant i+1}
 \subseteq\cdots\subseteq\mathrm{id}_{D^0}.
\]

{\rm (3)} These constructions are inverse to one another. In particular,
a bounded commuting \(t\)-structure is determined by the filtration
that its truncations induce on the original heart.
\end{theorem}

\begin{proof}
Throughout the proof,
\(D^{[m,n]}\) and \(\HH^q\) refer to the original \(t\)-structure,
and \((\mathrm{id}/T^a)(A)=A/T^a(A)\).

We first prove orthogonality. Let \(X\in\Dt_{\leqslant0}\) and
\(Y\in\Dt_{\geqslant1}\).
By definition, for every \(r,s\in\Z\),
\[
    \HH^r(X)=T^{-r}(\HH^r(X)),\qquad
    \HH^s(Y)\simeq(\mathrm{id}/T^{-s})(\HH^s(Y)).
\]
Condition (2) of \Cref{radical filtration}, with indices
\(-r\) and \(-s\), therefore gives
\[
\begin{aligned}
&\Hom_D\bigl(\HH^r(X)[-r],\HH^s(Y)[-s]\bigr)\\
&\qquad=
\Hom_D^{r-s}\bigl(
T^{-r}(\HH^r(X)),(\mathrm{id}/T^{-s})(\HH^s(Y))
\bigr)=0.
\end{aligned}
\]
The original \(t\)-structure is bounded, so \(X\) and \(Y\) are
finite successive extensions of their ordinary cohomology
objects with the indicated shifts. Applying the long exact
Hom sequences to these truncation filtrations, first in \(Y\)
and then in \(X\), proves
\[
    \Hom_D(X,Y)=0.
\]
We next construct the truncation triangle. Since the filtration
\(T^\bullet\) is increasing, the defining
conditions immediately imply
\[
    \Dt_{\leqslant0}[1]\subseteq\Dt_{\leqslant0},
    \qquad
    \Dt_{\geqslant1}[-1]\subseteq\Dt_{\geqslant1}.
\]
We also need the identities
\[
    T^a(T^a(A))=T^a(A),\qquad
    T^a\bigl((\mathrm{id}/T^a)(A)\bigr)=0
    \qquad(A\in D^0).
\]
To see this, condition (2) with equal indices says that
\(\Hom_D(T^a(M),(\mathrm{id}/T^a)(N))=0\) for all \(M,N\in D^0\).
Taking \(M=A\), \(N=T^a(A)\), the quotient map
\[
    T^a(A)\twoheadrightarrow
    (\mathrm{id}/T^a)(T^a(A))
\]
is zero; being an epimorphism in \(D^0\), it has zero target.
Taking \(M=(\mathrm{id}/T^a)(A)\), \(N=A\), the inclusion
\[
    T^a\bigl((\mathrm{id}/T^a)(A)\bigr)
    \hookrightarrow(\mathrm{id}/T^a)(A)
\]
is zero; being a monomorphism, it has zero source.
These identities ensure that the subobject and quotient at a
single filtration step belong to the required new halves.

For \(X\in D^{[m,n]}\), we construct a distinguished triangle
\[
Y\longrightarrow X\longrightarrow Z\longrightarrow Y[1],
\qquad Y\in\Dt_{\leqslant0},\quad Z\in\Dt_{\geqslant1},
\]
with both \(Y,Z\in D^{[m,n]}\).
We induct on \(n-m\). If \(m=n\), set \(A=\HH^m(X)\);
the required triangle is the shift by \([-m]\) of
\[
0\longrightarrow T^{-m}(A)\longrightarrow A
 \longrightarrow (\mathrm{id}/T^{-m})(A)\longrightarrow0.
\]

Otherwise choose \(m\leqslant c<n\), and put
\(X'=\tau_{\leqslant c}X\), \(X''=\tau_{\geqslant c+1}X\).
Induction gives triangles
\[
Y'\longrightarrow X'\longrightarrow Z'\longrightarrow Y'[1],
\qquad
Y''\longrightarrow X''\longrightarrow Z''\longrightarrow Y''[1],
\]
where \(Y',Z'\in D^{[m,c]}\), \(Y'',Z''\in D^{[c+1,n]}\),
the \(Y\)-terms lie in \(\Dt_{\leqslant0}\), and the \(Z\)-terms
lie in \(\Dt_{\geqslant1}\).

The obstruction to lifting the composite
\(Y''[-1]\to X''[-1]\to X'\) through \(Y'\to X'\)
lies in \(\Hom_D(Y''[-1],Z')\). This group vanishes.
Indeed, by ordinary cohomological d\'evissage its relevant groups
are
\[
\begin{aligned}
&\Hom_D\bigl(T^{-a}(A)[-a-1],(\mathrm{id}/T^{-b})(B)[-b]\bigr)\\
&\hspace{2em}=
\Hom_D^{a-b+1}\bigl(T^{-a}(A),(\mathrm{id}/T^{-b})(B)\bigr)=0,
\qquad a\geqslant c+1,\quad b\leqslant c.
\end{aligned}
\]
Here \(-a<-b\), so the last equality is exactly condition (3)
of \Cref{radical filtration}.
Choose a lift \(Y''[-1]\to Y'\).
The octahedral axiom, in its \(3\times3\) form, completes the
resulting square to a diagram of distinguished triangles
commuting up to the usual signs:
\[
\begin{tikzcd}[column sep=large,row sep=large]
Y' \arrow[r] \arrow[d] &
Y \arrow[r] \arrow[d] &
Y'' \arrow[r,"+1"] \arrow[d] & {}\\
X' \arrow[r] \arrow[d] &
X \arrow[r] \arrow[d] &
X'' \arrow[r,"+1"] \arrow[d] & {}\\
Z' \arrow[r] \arrow[d,"+1"'] &
Z \arrow[r] \arrow[d,"+1"'] &
Z'' \arrow[r,"+1"] \arrow[d,"+1"'] & {}\\
{} & {} & {} & {}
\end{tikzcd}
\]
The separated ordinary amplitudes give
\[
\HH^q(Y)\simeq
\begin{cases}
\HH^q(Y'),&q\leqslant c,\\
\HH^q(Y''),&q\geqslant c+1,
\end{cases}
\]
and the analogous formula for \(Z\).
Thus \(Y\in\Dt_{\leqslant0}\), \(Z\in\Dt_{\geqslant1}\), and
both retain amplitude \([m,n]\). This proves the truncation axiom.

By orthogonality, the constructed triangle is the new truncation triangle.
The same diagram and the amplitude bounds identify
\[
\begin{aligned}
\tau_{\leqslant c}\dt_{\leqslant0}X
 &\simeq Y'\simeq\dt_{\leqslant0}\tau_{\leqslant c}X,\\
\tau_{\geqslant c+1}\dt_{\leqslant0}X
 &\simeq Y''\simeq\dt_{\leqslant0}\tau_{\geqslant c+1}X,
\end{aligned}
\]
and similarly for \(\dt_{\geqslant1}\), using \(Z',Z''\).
By shifting, all ordinary and new truncations commute.
Finally, every \(X\) has finitely many ordinary cohomology objects,
and the filtration on each is eventually zero and eventually
the whole object. The defining formulas therefore place \(X\)
in a finite interval for the new \(t\)-structure as well.

Conversely, suppose that a bounded \(t\)-structure
\((\Dt_{\leqslant0},\Dt_{\geqslant1})\) commutes with the original one.
Its truncations preserve \(D^0\), so for \(A\in D^0\) they give
short exact sequences in \(D^0\)
\[
0\longrightarrow T^a(A):=\dt_{\leqslant a}A
 \longrightarrow A\longrightarrow
 \dt_{\geqslant a+1}A\longrightarrow0.
\]
Thus \(\dt_{\geqslant a+1}A\simeq(\mathrm{id}/T^a)(A)\).
The subfunctors \(T^a\) form an increasing filtration.
Boundedness gives condition (1) of \Cref{radical filtration}.
Moreover, \(T^a(A)[a]\in\Dt_{\leqslant0}\) and
\((\mathrm{id}/T^b)(B)[b]\in\Dt_{\geqslant1}\); orthogonality gives condition (2).

For condition (3), we first show that
\[
\Hom_D(X[-1],Y)=0
\quad\text{if}\quad
X\in D^{\geqslant c+1}\cap\Dt_{\leqslant0},\qquad
Y\in D^{\leqslant c}\cap\Dt_{\geqslant1}.
\]
Given \(f:X[-1]\to Y\), complete it to a triangle
\[
Y\longrightarrow E\longrightarrow X\xrightarrow{\,f[1]\,}Y[1].
\]
The ordinary truncations of \(E\) are \(Y\) and \(X\).
Commutation of the two truncation systems gives
\[
\tau_{\leqslant c}\dt_{\leqslant0}E
 \simeq\dt_{\leqslant0}Y=0,\qquad
\tau_{\geqslant c+1}\dt_{\leqslant0}E
 \simeq\dt_{\leqslant0}X=X.
\]
Consequently the natural composite
\(\dt_{\leqslant0}E\to E\to X\) is an isomorphism.
It supplies a section of \(E\to X\), so the triangle splits and
\(f=0\).
Apply this with
\[
X=T^a(A)[a],\qquad Y=(\mathrm{id}/T^b)(B)[b],\qquad c=-b,\qquad a<b.
\]
The ordinary degrees of \(X,Y\) are \(-a>-b\) and \(-b\),
respectively. The resulting vanishing is
\[
\Hom_D^{b-a+1}(T^a(A),(\mathrm{id}/T^b)(B))=0,
\]
which is condition (3).

It remains to show that the two constructions are inverse. Commutation
implies that each \(\HH^q(X)[-q]\) belongs to
\(\Dt_{\leqslant0}\), respectively \(\Dt_{\geqslant1}\), whenever
\(X\) does. Conversely, the finite ordinary truncation filtration
and extension closure recover membership of \(X\) from these pieces.
Thus
\[
\begin{aligned}
X\in\Dt_{\leqslant0}
 &\Longleftrightarrow
 T^{-q}(\HH^q(X))=\HH^q(X)\quad\text{for every }q,\\
X\in\Dt_{\geqslant1}
 &\Longleftrightarrow
 T^{-q}(\HH^q(X))=0\quad\text{for every }q.
\end{aligned}
\]
These are precisely the formulas in (1).
In the other direction, the single-cohomology construction of the
truncation triangle recovers \(T^a(A)\) as
\(\dt_{\leqslant a}A\) for \(A\in D^0\).
Thus the two constructions are inverse, proving (3).
\end{proof}

For coherent \(R\)-complexes, the task is therefore to establish the
orthogonality conditions for the filtration in
\Cref{example:radical-filtration}. Once these are known, the theorem
provides both the diagonal truncations and their compatibility with
ordinary truncation. We turn to the Ext calculations that make this
construction possible.

\subsection{Diagonal truncation}
\label{subsec:diagonal-truncation}

To construct diagonal truncation on \(D_c^b(R)\), we verify the orthogonality in \Cref{radical filtration} for the filtration of \Cref{example:radical-filtration}. By d\'evissage, this reduces to Ext calculations between Dieudonn\'e modules and elementary dominos.

\begin{definition}[Graded Hom]
\label{graded Hom}
For graded left \(R\)-modules \(M,N\), define the graded Hom group by
\[
    \Hom_R(M,N):=\bigoplus_{i\in\Z}\Hom_R(M,N)^i,
    \qquad
    \Hom_R(M,N)^i
       :=\Hom_{R\text{-}\mathrm{Mod}_{\mathrm{gr}}}(M,N(i)).
\]
Thus \(\Hom_R(M,N)^i\) consists of the grading-preserving
\(R\)-linear maps \(M\to N(i)\).

For \(M,N\in D(R)\), similarly set
\[
    \Hom_{D(R)}(M,N)
       :=\bigoplus_{i\in\Z}\Hom_{D(R)}(M,N)^i,
    \qquad
    \Hom_{D(R)}(M,N)^i
       :=\Hom_{D(R)}(M,N(i)),
\]
where the right-hand side is computed in the derived category of
graded \(R\)-modules. This internal grading is independent of the
cohomological shift \([j]\); its degree-zero part is the usual
morphism group.
\end{definition}

\begin{definition}
For graded left \(R\)-modules \(M,N\), define
\[
\Ext_R^j(M,N)^i:=\Hom_{D(R)}(M,N[j])^i.
\]
Equivalently,
\(\Ext_R^j(M,N)^i\simeq R^j\Hom_R(M,N)^i\).
\end{definition}

We compute these Ext groups using short resolutions by copies of \(\widehat R\).
All free generators in the next three resolutions have internal
degree \(0\). The maps between copies of
\(\widehat R\) are given by right multiplication, so they
are homomorphisms of graded left \(R\)-modules. For a complete target \(N\), the completion comparison gives
\[
    R\Hom_R(\widehat R,N)^i\simeq N^i.
\]
Indeed, derived completion is a localization, left adjoint to the
inclusion of complete objects, so maps from \(\widehat R\) to a
complete target agree with maps from \(R\). Thus these resolutions
compute derived Hom into \(N\), although we are not asserting that
\(\widehat R\) is projective in the category of all graded
\(R\)-modules.

Let
\[
 P=a_sF^s+a_{s-1}F^{s-1}+\cdots+a_0
     +b_1V+\cdots+b_tV^t\in R^0,
 \qquad s,t\geqslant1,
\]
where \(a_s,b_t\in W^\times\), while
\[
 a_i\in pW\quad(0\leqslant i<s),
 \qquad
 b_j\in pW\quad(1\leqslant j<t).
\]
We also allow \(P\) to be a polynomial in \(F\) alone,
\[
 P(F)=a_sF^s+\cdots+a_0,
 \qquad a_s,a_0\in W^\times .
\]
Put
\[
 E_P:=R^0/R^0P,
\]
regarded as an \(R\)-module concentrated in degree \(0\). It is a
finite free cyclic Dieudonn\'e module. If \(P=F^a+V^b\), then \(E_P\)
is the module denoted \(E_{b/(a+b)}\) by Ekedahl. \(F\)-\(V\) relations
arising from an isoclinic factor satisfy stronger Newton-polygon
inequalities; the divisibility displayed above is all that is needed
for the resolution
\begin{equation}\label{completed resolution of a cyclic Dieudonne module}
 0\longrightarrow\widehat R
 \xrightarrow{\;x\mapsto xP\;}
 \widehat R\longrightarrow E_P\longrightarrow0 .
\end{equation}

The augmentation sends \(1\) to the cyclic generator. In internal
degree \(0\), exactness follows from the unique \(F\)-\(V\) expansions
and division by the two extreme unit coefficients. On
\(\widehat R^1\), right multiplication by \(V\) is invertible and
\(F=pV^{-1}\). Modulo \(p\), multiplication by an \(F\)-\(V\) relation \(P\) therefore
reduces to multiplication by \(\bar b_tV^t\), while for an
\(F\)-only relation it reduces to the unit \(\bar a_0\). Both are invertible,
and \(p\)-adic completeness lifts their inverses.

The torsion factors require two further resolutions. First, write
\(k_0\) for \(k\) with \(F=V=0\), in degree \(0\). It has the
exact sequence
\begin{equation}\label{completed resolution of local torsion Dieudonne module}
    0\longrightarrow\widehat R
    \xrightarrow{\;x\mapsto(xV,-xF)\;}
    \widehat R\oplus\widehat R
    \xrightarrow{\;(x,y)\mapsto xF+yV\;}
    \widehat R\longrightarrow k_0\longrightarrow0.
\end{equation}
The composite is zero because \(VF=FV\).
In degree \(0\), the unique \(F\)-\(V\) expansions identify the kernel of the second
map with the pairs \((xV,-xF)\), and its cokernel with
\(R^0/(R^0F+R^0V)=k\). In degree \(1\), exactness follows
from the bijectivity of right multiplication by \(V\).

The \'etale torsion case follows from the cyclic resolution by
reduction modulo \(p\). Let
\[
 Q(F)=c_sF^s+\cdots+c_0\in k_\sigma[F],
 \qquad c_s,c_0\in k^\times ,
\]
and choose a lift
\(\widetilde Q(F)\in W_\sigma[F]\) whose leading and constant
coefficients are units. Define the cyclic \'etale torsion module
\[
 N_Q:=E_{\widetilde Q}/pE_{\widetilde Q}.
\]
If \(Q\) is irreducible, \(N_Q\) is simple. Taking the mapping cone of
multiplication by \(p\) on the resolution of \(E_{\widetilde Q}\)
gives
\begin{equation}\label{completed resolution of a cyclic etale torsion module}
 0\longrightarrow\widehat R
 \xrightarrow{\;z\mapsto(-z\widetilde Q,zp)\;}
 \widehat R\oplus\widehat R
 \xrightarrow{\;(x,y)\mapsto xp+y\widetilde Q\;}
 \widehat R\longrightarrow N_Q\longrightarrow0 .
\end{equation}
The two maps compose to zero because \(p\) is central. The case
\(Q=F-1\) recovers the usual module \(k_{\mathrm{et}}\).

For a complete Dieudonn\'e module \(N\) in degree \(0\), the first
two resolutions give
\[
\begin{aligned}
 R\Hom_R(E_P,N)^0
   &\simeq [\,N\xrightarrow{\,P\,}N\,],\\
 R\Hom_R(k_0,N)^0
   &\simeq
   [\,N\xrightarrow{\,n\mapsto(Fn,Vn)\,}N\oplus N
       \xrightarrow{\,(x,y)\mapsto Vx-Fy\,}N\,].
\end{aligned}
\]
These complexes of abelian groups begin in cohomological degree \(0\);
their maps need not be \(W\)-linear. All other internal degrees
vanish. The corresponding bound for \(N_Q\) follows from
\eqref{completed resolution of a cyclic etale torsion module}.

\begin{lemma}\label{completed resolution of elementary domino}
For \(t\in\Z\), define homogeneous elements of degree \(1\) by
\[
\begin{array}{c|cc}
 &t\leqslant0&t\geqslant1\\ \hline
 w_t&F^{1-t}d&dV^{t-1}\\
 h_t&F^{-t}dV&dV^{t+1}.
\end{array}
\]
Then the elementary domino \(U_t\) has the exact sequence
\begin{equation}\label{elementary domino completed resolution sequence}
0\longrightarrow\widehat R(-1)
 \xrightarrow{\;\alpha_t\;}
 \widehat R\oplus\widehat R(-1)
 \xrightarrow{\;\beta_t\;}
 \widehat R
 \xrightarrow{\;\epsilon_t\;}U_t
 \longrightarrow0,
\end{equation}
where
\[
 \alpha_t(z)=(zh_t,-zV),\qquad
 \beta_t(x,y)=xF+yw_t.
\]
The augmentation \(\epsilon_t\) sends \(1\) to \(1\in U_t^0=k[[V]]\).
\end{lemma}

\begin{proof}
All products in the displayed maps are on the right.
The Raynaud relations give
\(h_tF=Vw_t\) and \(w_tV=Fh_t,\)
thus \(\beta_t\alpha_t=0\).

We first identify the cokernel of \(\beta_t\). In degree \(0\), it is
\(\widehat R^0/\widehat R^0F=k[[V]].\)
In degree \(1\), right multiplication by \(F\) has image
\(p\widehat R^1\), since right multiplication by \(V\) is
bijective. Modulo \(p\), elements of \(\widehat R^1\) have unique
expansions in the symbols \(dV^j\), bounded below in \(j\);
negative exponents mean \(dV^{-j}=F^jd\).
The element \(w_t\) is the symbol \(dV^{t-1}\), so its
left \(F\)-multiples remove precisely the symbols with index
less than \(t\). Consequently
\[
    \operatorname{coker}(\beta_t)^1
       \simeq\prod_{j\geqslant t}k\,dV^j.
\]
The induced differential and the actions of \(F,V\) are those
of \(U_t\). This identifies the cokernel and the augmentation
in \eqref{elementary domino completed resolution sequence}.

It remains to check exactness at the two terms on the left.
The unique \(F\)-\(V\) expansions give the divisibility criterion
\[
    y\,d\in p\widehat R^1
       \quad\Longleftrightarrow\quad
    y\in\widehat R^0V
    \qquad(y\in\widehat R^0).
\]
Indeed, \(V^nd=p^ndV^n\), whereas the terms \(F^jd\) carry no \(p\)-factor. 
The same criterion holds with \(d\) replaced by \(w_tV\):
for \(t\geqslant1\), this is \(dV^t\), and for \(t\leqslant0\),
it is \(dV^t=F^{-t}d\). Right multiplication by any integral
power of \(V\) is an automorphism of \(\widehat R^1\) preserving
\(p\widehat R^1\).

In internal degree \(0\), the middle kernel is zero because
right multiplication by \(F\) on \(\widehat R^0\) is injective.
In internal degree \(1\), suppose
\[
    xF+yw_t=0,\qquad x\in\widehat R^1,\quad y\in\widehat R^0.
\]
Multiplying on the right by \(V\) gives
\(px+yw_tV=0\). The preceding criterion implies \(y=zV\)
for some \(z\in\widehat R^0\). Since \(Vw_tV=ph_t\) and
\(\widehat R^1\) is \(p\)-torsion-free, it follows that
\(x=-zh_t\). Thus \((x,y)=\alpha_t(-z)\).
In degree \(2\), the middle term is \(\widehat R^1\), its
map to \(\widehat R\) is zero, and \(\alpha_t\) maps onto it
by right multiplication by \(-V\).
Finally, \(\alpha_t\) is injective because right multiplication
by \(V\) is injective on \(\widehat R^0\) and bijective on
\(\widehat R^1\). This proves exactness in every internal degree.
\end{proof}

For a complete graded \(R\)-module \(N\), this resolution gives
\[
 R\Hom_R(U_t,N)^i\simeq
 \left[
 N^i\xrightarrow{\;n\mapsto(Fn,w_tn)\;}
 N^i\oplus N^{i+1}
 \xrightarrow{\;(x,y)\mapsto h_tx-Vy\;}
 N^{i+1}
 \right],
\]
in cohomological degrees \(0,1,2\).
Here \(F\) and \(V\) act as on \(N\), while every occurrence of \(d\)
in \(w_t\) and \(h_t\) acts as \((-1)^id_N\), since internal degree
\(i\) is computed using the shifted module \(N(i)\). In particular, if \(N\) is a Dieudonn\'e module
concentrated in degree \(0\), the only possible internal
degrees are \(i=0,-1\). The two complexes are respectively
\[
    [\,N\xrightarrow{F}N\,]\quad\text{in degrees }0,1,
    \qquad
    [\,N\xrightarrow{-V}N\,]\quad\text{in degrees }1,2.
\]
\begin{remark}[Internal degrees]
In the statements below we use the convention of \Cref{graded Hom}:
\(\Ext_R^j(P,Q)^i=\Hom_{D(R)}(P,Q(i)[j])\).
In particular,
\[
 \Hom_{D(R)}(P(-a),Q(-b)[j])=\Ext_R^j(P,Q)^{a-b}.
\]

\end{remark}

To pass to finite-type Dieudonn\'e modules, we use integral filtrations by these cyclic modules.

\begin{lemma}[Cyclic d\'evissage]\label{cyclic devissage of Dieudonne modules}
Let \(k\) be perfect. Every finite-type \(V\)-complete Dieudonn\'e module
admits a finite filtration whose successive quotients are of the form
\[
 E_P,\qquad k_0,\qquad N_Q.
\]
\end{lemma}

\begin{proof}
First separate the finite \(p\)-power torsion submodule. Its
\(p\)-adic filtration reduces us to a finite-dimensional Dieudonn\'e
module killed by \(p\). The simple complete factors are \(k_0\) and
the simple \'etale modules: a nonzero vector generates any simple
\'etale factor, and its annihilator is an irreducible skew polynomial
\(Q(F)\) with nonzero leading and constant coefficients. These
give the modules \(N_Q\). Completeness excludes the
\(V\)-bijective simple factor.

It remains to treat a \(W\)-torsion-free module \(M\). Choose a simple
constituent of \(M[1/p]\) and a cyclic vector. The Newton-polygon
inequalities normalize its minimal skew relation to an \(F\)-\(V\)
relation \(P\) as above when the slope is positive, and to an
\(F\)-only relation \(P(F)\) with unit leading and constant
coefficients when the slope is zero. After multiplying the vector by a power of \(p\),
the corresponding cyclic lattice \(E_P\) lies in \(M\).
The quotient has smaller \(W\)-rank and finite torsion, so induction
on the rank completes the filtration.
\end{proof}

The long exact Ext sequences now give the required vanishing bounds.

\begin{proposition}
\label{Ekedahl vanishing over any perfect field}
Let \(k\) be a perfect field. All Dieudonn\'e modules below are of finite
type and concentrated in internal degree \(0\).

{\rm (1)} For Dieudonn\'e modules \(M,N\),
\[
 \Ext_R^j(M,N)^i=0
 \quad\text{unless}\quad i=0,\quad 0\leqslant j\leqslant2.
\]

{\rm (2)} For a Dieudonn\'e module \(M\) and \(t\in\Z\),
\[
 \Ext_R^j(U_t,M)^i=0
 \quad\text{unless}\quad
 \begin{cases}
 i=-1,&1\leqslant j\leqslant2,\\
 i=0,&0\leqslant j\leqslant1,
 \end{cases}
\]
whereas
\[
 \Ext_R^j(M,U_t)^i=0
 \quad\text{unless}\quad i\in\{0,1\},\quad
 0\leqslant j\leqslant2.
\]

{\rm (3)} For \(t,s\in\Z\),
\[
 \Ext_R^j(U_t,U_s)^i=0
 \quad\text{unless}\quad
 \begin{cases}
 i=-1,&j=2,\\
 i=0,&0\leqslant j\leqslant2,\\
 i=1,&j=0.
 \end{cases}
\]
\end{proposition}

This is a corrected version of \cite[Proposition~0.4.1]{Ek3}, valid
over every perfect field.

\begin{proof}
Let \(M\) be a Dieudonn\'e module. By
\Cref{cyclic devissage of Dieudonne modules}, it has a finite
filtration with successive quotients \(E_P\), \(k_0\), or
\(N_Q\). Their completed \(\widehat R\)-resolutions have lengths at
most two and compute \(R\Hom_R\) against every complete target.

If the target is a Dieudonn\'e module \(N\), concentrated in internal
degree \(0\), the resulting Hom complexes occur only in internal
degree \(0\) and cohomological degrees \(0,1,2\). If the target is
\(U_t\), they occur only in internal degrees \(0,1\), again in
cohomological degrees \(0,1,2\). The long exact Ext sequences along
the filtration prove (1) and the second assertion of (2).

The assertions with domino source follow directly from
\Cref{completed resolution of elementary domino}. For target \(N\),
its two relevant components are
\[
 [\,N\xrightarrow{F}N\,]\quad (i=0),\qquad
 [\,N\xrightarrow{-V}N\,]\quad (i=-1),
\]
in cohomological degrees \(0,1\) and \(1,2\), respectively. For target
\(U_s\), the injectivity of \(V\) on \(U_s^0\) and the surjectivity of
\(F\) on \(U_s^1\) give exactly the ranges in (3).
\end{proof}




The structure theorem for coherent \(R\)-modules reduces the layers of
\(\dt_{\leqslant i}\) to Dieudonn\'e modules and elementary dominos.
The preceding Ext bounds, applied by d\'evissage, give conditions (2)
and (3) of \Cref{radical filtration}; finite graded support gives
condition (1). We may therefore apply \Cref{diagonal t-structure}.

The resulting \(t\)-structure aligns ordinary cohomological degree
with the internal grading of the Raynaud ring. Its heart is supported
on the two diagonals \((a,-a)\) and \((a,1-a)\), which is the source of
the name ``diagonal.''

\begin{definition}\label{diagonal t-structure on coherent R-complexes}
The diagonal \(t\)-structure on \(D_c^b(R)\) is the
\(t\)-structure associated, as in \Cref{diagonal t-structure},
to the radical filtration \(\dt_{\leqslant i}\) on coherent
\(R\)-modules. Explicitly, its two halves are
\[
\begin{aligned}
\Dt_{\leqslant0}
 &=\bigl\{M\in D_c^b(R)\mid
   \dt_{\leqslant-j}\HH^j(M)=\HH^j(M)
   \text{ for every }j\bigr\},\\
\Dt_{\geqslant0}
 &=\bigl\{M\in D_c^b(R)\mid
   \dt_{\leqslant-j-1}\HH^j(M)=0
   \text{ for every }j\bigr\}.
\end{aligned}
\]
Here \(\HH^j\) denotes ordinary cohomology, and the truncations
applied to \(\HH^j(M)\) are those of graded \(R\)-modules.
We write
\[
\dc:=\Dt_{\leqslant0}\cap\Dt_{\geqslant0}
\]
for its heart, the abelian category of diagonal complexes.
We denote the associated truncation functors on \(D_c^b(R)\) by
\(\dt_{\leqslant n}\) and \(\dt_{\geqslant n}\), and its cohomology
functors by
\[
\tH^n(M):=
 \bigl(\dt_{\geqslant n}\dt_{\leqslant n}M\bigr)[n]\in\dc.
\]
Thus ordinary truncations \(\tau_{\leqslant n},\tau_{\geqslant n}\)
and cohomology \(\HH^n\) are distinguished from their diagonal
counterparts \(\dt_{\leqslant n},\dt_{\geqslant n}\) and \(\tH^n\).
\end{definition}

By construction, the diagonal \(t\)-structure is bounded and commutes
with the ordinary one. Its heart admits the concrete description
\[
M\in\dc
\quad\Longleftrightarrow\quad
\begin{cases}
\HH^j(M)^a=0&\text{for }a\notin\{-j,1-j\},\\
F^\infty B\HH^j(M)^{1-j}=\HH^j(M)^{1-j}
 &\text{for every }j.
\end{cases}
\]
In particular, a coherent \(R\)-module \(N\) in cohomological degree
zero belongs to \(\dc\) precisely when
\[
N^a=0\quad(a\neq0,1),\qquad F^\infty BN^1=N^1.
\]
Such modules are precisely the extensions of Dieudonn\'e modules by
dominos.

For \(r\in\Z\), the functor \([r](-r)\) preserves both halves of the
diagonal \(t\)-structure, and therefore its heart; we call these
functors twists. Since the two \(t\)-structures commute, the ordinary
truncations of \(M\in\dc\) form a finite filtration in \(\dc\),
\[
\cdots\subseteq\tau_{\leqslant j-1}M
 \subseteq\tau_{\leqslant j}M\subseteq\cdots\subseteq M,
\]
whose successive quotients are \(\HH^j(M)[-j]\).
Writing \(N_j=\HH^j(M)(-j)\), we have
\(N_j\in\dc\!\cap R\td\mathrm{Mod}\) and
\[
\HH^j(M)[-j]=N_j(j)[-j].
\]
Consequently, every diagonal complex is a finite iterated extension of
twists of objects in \(\dc\!\cap R\td\mathrm{Mod}\).

\section{Basic facts}
\label{sec:basic-facts}

With the diagonal heart in place, we can ask what it tells us about
coherent \(R\)-complexes. We begin with reduction and totalization,
then study subobjects, duality, and torsion. The aim is to keep track
of what survives when the two gradings are combined. Hodge--Witt and
Mazur--Ogus objects will describe two useful ways in which this
interaction becomes simpler.

We study how conditions on the cohomology of the reduction by \(R_1\)
propagate to reductions by \(R_n\).

Define the graded right \(R\)-modules
\[
Z_1=\ker(d:R_1\longrightarrow R_1(-1)),\qquad
B_1=\operatorname{im}(d:R_1(1)\longrightarrow R_1),
\]
where \(d\) acts by left multiplication. Both have a left \(k\)-action
and trivial left \(d\)-action, making them
\((W_1[d]/(d^2),R)\)-bimodules. Left multiplication by \(F\) induces
the Cartier isomorphism
\[
C^{-1}:R_1\xrightarrow{\;F\cdot\;}\sigma_*(Z_1/B_1)
\]
of graded \((W_1,R)\)-bimodules. The twist concerns the left scalar
action; this is not an isomorphism of graded
\((W_1[d]/(d^2),R)\)-bimodules, since it does not preserve the left
\(d\)-action; see also \cite{IR}.

Iterating Cartier defines the bimodules
\[
B_1\subset B_2\subset\cdots\subset Z_2\subset Z_1\subset R_1
\]
characterized by isomorphisms
\[
C^n:\sigma_*^n(Z_n/B_n)\xrightarrow{\sim}R_1
\]
of graded right \(R\)-modules.

\subsection{Basic facts}
\label{subsec:basic-facts}

Diagonal support gives the following bounds on totalization and reduction.

\begin{proposition}\label{Basic Vanishing}
Let \(M\in\dc\).

{\rm (1)} One has \(\HH^i(\Tot(M))=0\) for \(i\notin\{0,1\}\), and
\(\HH^1(\Tot(M))\) is a finite-length torsion \(W\)-module.

{\rm (2)} If \(D\) is one of \(R_n,B_n,Z_n,R_1/B_n\), then
\[
\HH^j(D\otimes_R M)^i=0
\qquad\text{for }i+j\notin\{-1,0,1\}.
\]
\end{proposition}

This is the vanishing range of \cite[Proposition~I.1.4]{Ek3},
with the finite-length assertion in (1) made explicit.

\begin{proof}
For (1), use the finite filtration of \(M\) by its ordinary
truncations. Its successive quotients are twists \(N(j)[-j]\), with
\(N\in\dc\!\cap R\td\mathrm{Mod}\), and twists do not change
totalization. We may therefore reduce to such an \(N\).

The module \(N\) is an extension of a finite-type Dieudonn\'e module
by a coherent domino. The former contributes only to
\(\HH^0(\Tot(N))\), while the latter has a finite filtration by
elementary dominos, whose total cohomology has finite length and is
concentrated in degrees \(0\) and \(1\). The long exact sequences
associated with the two finite filtrations prove
(1).

For (2), d\'evissage, together with homotopy limits and colimits,
reduces to an elementary coherent \(R\)-module. Use the resolutions

    \[\begin{aligned}
        &0\rightarrow R(-1)\xrightarrow{(F,-Fd)\cdot} R(-1)\oplus R\xrightarrow{(dV,V)\cdot} R\rightarrow R_1\rightarrow 0,\\
        &0\rightarrow R(-1) \xrightarrow{(V,-F)\cdot } R(-1)\oplus R(-1) \xrightarrow{(F,V)\cdot} R(-1)\rightarrow B_1 \rightarrow 0,\\
        &0\rightarrow R(-1) \xrightarrow{(F,-d)\cdot } R(-1)\oplus R \xrightarrow{(dV^2,V)\cdot} R\xrightarrow{C^{-1}} \sigma_*(Z_1) \rightarrow 0,\\
        &0\rightarrow R(-1) \xrightarrow{(V^n,-d)\cdot} R(-1)\oplus R \xrightarrow{(F^{n-1}d,V)\cdot} R\rightarrow R_1/B_n \rightarrow 0.
    \end{aligned}\]

For \(B_1\), the claimed range is immediate for a finite-type
Dieudonn\'e module; for a domino it follows from the surjectivity of
left multiplication by \(F\) in degree \(1\). The remaining cases
follow inductively from

    \[\begin{aligned}
          &0\rightarrow B_1\rightarrow B_{n+1}\xrightarrow{C}\sigma_*^{-1}B_n\rightarrow 0,\\
        &0\rightarrow B_1\rightarrow Z_{n+1}\xrightarrow{C}\sigma^{-1}_*Z_n\rightarrow 0,
    \end{aligned}\]
and the corresponding long exact cohomology sequences. Frobenius
twists do not alter the stated vanishing range.
\end{proof}

Only two adjacent diagonal cohomology objects contribute to each degree of total cohomology.

\begin{corollary}[{\cite[Corollary~I.1.4.1]{Ek3}}]\label{Totsubquo}
Let \(M\in D_c^b(R)\).

{\rm (1)} For every \(i\), there is a natural short exact sequence
\[
0\longrightarrow \HH^1(\Tot(\tH^{i-1}(M)))
\longrightarrow \HH^i(\Tot(M))
\longrightarrow \HH^0(\Tot(\tH^i(M)))
\longrightarrow0.
\]

{\rm (2)} If \(D\) is one of \(R_n,B_n,Z_n,R_1/B_n\), then
\(\HH^i(D\otimes_R^{\rmL}\tH^j(M))^{-i}\) is a subquotient of
\(\HH^{i+j}(D\otimes_R^{\rmL}M)^{-i}\).
\end{corollary}

\begin{proof}
The diagonal Postnikov filtration of \(M\), whose graded pieces are
\(\tH^j(M)[-j]\), gives after totalization the spectral sequence
\[
E_2^{i,j}=\HH^i\bigl(\Tot(\tH^j(M))\bigr)
 \Longrightarrow \HH^{i+j}(\Tot(M)).
\]
By \Cref{Basic Vanishing}(1), it has only two adjacent columns, which
gives (1). Applying \(D\otimes_R^{\rmL}-\) to the same filtration
gives
\[
E_2^{i,j}=\HH^i\bigl(D\otimes_R^{\rmL}\tH^j(M)\bigr)
 \Longrightarrow \HH^{i+j}(D\otimes_R^{\rmL}M).
\]
In internal degree \(-i\), \Cref{Basic Vanishing}(2) leaves no
incoming or outgoing differential. Thus the indicated \(E_2\)-term
survives to \(E_\infty\), proving (2).
\end{proof}

For diagonal modules, monomorphisms in \(\dc\) can be tested in internal degree zero.

\begin{lemma}[{\cite[Lemma~I.1.5.1]{Ek3}}]\label{subobjects of diagonal modules}
Let \(M\in\dc\!\cap R\td\mathrm{Mod}\).

{\rm (1)} If \(N\hookrightarrow M\) is a subobject in \(\dc\), then
\(N\in\dc\!\cap R\td\mathrm{Mod}\), and the map
\(N^0\hookrightarrow M^0\) is injective as a map of \(R^0\)-modules.

{\rm (2)} Let \(N\in\dc\!\cap R\td\mathrm{Mod}\). If a morphism
\(f:N\to M\) in \(\dc\) induces an injection
\(f^0:N^0\to\HH^0(M)^0\), then \(f\) is a monomorphism in \(\dc\).
\end{lemma}

\begin{proof}
For (1), let \(Q\) be the cokernel of \(N\hookrightarrow M\) in
\(\dc\). The resulting distinguished triangle gives
\[
\HH^{q-1}(Q)^{-q}\longrightarrow
\HH^q(N)^{-q}\longrightarrow\HH^q(M)^{-q}.
\]
The first term vanishes because \(Q\) is diagonal. Thus the second map
is injective. For \(q\neq0\), its target is zero; since
\(\HH^q(N)\) is generated by its component of internal degree \(-q\),
we obtain \(\HH^q(N)=0\). Hence \(N\) is an actual \(R\)-module, and
the case \(q=0\) gives the asserted injection \(N^0\hookrightarrow
M^0\).

For (2), let \(K=\ker_{\dc}(f)\). By (1), \(K\) is an actual diagonal
\(R\)-module and \(K^0\hookrightarrow N^0\). Since \(N^0\to\HH^0(M)^0\) is injective,
we have \(K^0=0\). Every diagonal \(R\)-module is generated by its
component of internal degree zero, so \(K=0\).
\end{proof}

The preceding lemma also shows that kernels and cokernels in \(\dc\)
need not agree with those in \(R\td\mathrm{Mod}\), even when both
source and target lie in \(\dc\!\cap R\td\mathrm{Mod}\). For example,
the morphism \(f:U_0\to U_1\) defined by \(1\mapsto1\) is surjective
as a morphism of \(R\)-modules, with kernel \(k(-1)\). Its internal
degree-zero component is an isomorphism, so
\Cref{subobjects of diagonal modules}(2) makes \(f\) a monomorphism in
\(\dc\), with cokernel \(k(-1)[1]\). Thus
\(\ker_R(f)=k(-1)\) and \(\operatorname{coker}_R(f)=0\), whereas
\(\ker_{\dc}(f)=0\) and
\(\operatorname{coker}_{\dc}(f)=k(-1)[1]\).

The following calculation will control ascending chains of domino subobjects.

\begin{lemma}[{\cite[Lemma~0.4.2]{Ek3}}]\label{homomorphism between U_i}
Let \(U_a,U_b\) be the elementary dominos of types \(a,b\in\Z\).
Then
\[
\Hom_R(U_a,U_b)^0=0\qquad\text{if }a>b,
\]
and every nonzero grading-preserving \(R\)-endomorphism of \(U_a\) is
an isomorphism. The superscript \(0\) denotes internal grading degree;
see \Cref{graded Hom}.
\end{lemma}

\begin{proof}
Let \(\varphi:U_a\to U_b\) be a nonzero grading-preserving
\(R\)-homomorphism. Since \(U_a\) is generated in degree \(0\),
\(\varphi^0:U_a^0\to U_b^0\) is nonzero. Both degree-zero parts are
\(k[[V]]\), and every nonzero map compatible with scalars and \(V\)
is injective. Hence \(\varphi^0\) is injective.

Apply the Nygaard modification \((-)(\psi)\), where
\(\psi(n)=-b\delta_{1n}\).
It sends \(U_a,U_b\) to \(U_{a-b},U_0\), respectively, without
changing Hom groups. We may therefore assume \(b=0\). The induced map
\[
\ker(d:U_a^0\to U_a^1)
 \longrightarrow
\ker(d:U_0^0\to U_0^1),
\]
is injective. For \(a>0\), however, its source is nonzero and its
target is zero. This proves the vanishing.

For endomorphisms, the same modification reduces to
\(\varphi:U_0\to U_0\). Its restriction to
\[
\ker(Fd:U_0^0\longrightarrow U_0^1)
\]
is bijective because this kernel is one-dimensional over \(k\). It
generates \(U_0\) as an \(R\)-module, so \(\varphi\) is surjective and
\(\varphi^0\) is an isomorphism. Since \(d:U_0^0\to U_0^1\) is an
isomorphism and \(\varphi^1d=d\varphi^0\), so is \(\varphi^1\).
\end{proof}

\begin{proposition}[{\cite[Proposition~I.1.5]{Ek3}}]\label{diagonal category is Noetherian}
The abelian category \(\dc\) is Noetherian.
\end{proposition}

\begin{proof}
The ordinary truncation filtration of \Cref{diagonal t-structure on coherent R-complexes},
together with the structure theorem for coherent \(R\)-modules, reduces
the assertion by twists and d\'evissage to a finite-type Dieudonn\'e
module or an elementary domino \(U_j\).

For a finite-type Dieudonn\'e module, \Cref{subobjects of diagonal modules}
identifies every ascending chain of subobjects in \(\dc\) with an
ascending chain of \(W\)-submodules. It stabilizes because the module is
finitely generated over the Noetherian ring \(W\).

Now let
\[
N_0\subseteq N_1\subseteq\cdots\subseteq U_j
\]
be an ascending chain. Every nonzero \(N_a\) is an actual
one-dimensional domino, hence \(N_a\simeq U_{r_a}\) for some
\(r_a\in\Z\). By \Cref{homomorphism between U_i}, the inclusions give
\[
r_a\leqslant r_{a+1}\leqslant j.
\]
The sequence \((r_a)\) therefore stabilizes. Once
\(r_a=r_{a+1}\), the inclusion \(N_a\hookrightarrow N_{a+1}\) is an
isomorphism, again by \Cref{homomorphism between U_i}. Thus the chain
stabilizes.
\end{proof}

\begin{lemma}\label{dcDegen}
{\rm (1)} The canonical exact sequence
\(0\to\coeur(M)\to M\to\operatorname{dom}(M)\to0\)
for \(M\in\dc\!\cap R\td\mathrm{Mod}\) induces
\[
\begin{gathered}
0\longrightarrow\coeur(M)\longrightarrow\HH^0(\Tot(M))
 \longrightarrow\HH^0(\Tot(\operatorname{dom}(M)))
 \longrightarrow0,\\
\HH^1(\Tot(M))
 \xrightarrow{\sim}
 \HH^1(\Tot(\operatorname{dom}(M))).
\end{gathered}
\]

{\rm (2)} Let \(M\in\dc\). If \(i\) and \(j\) are respectively
the smallest and largest integers for which \(\HH^i(M)\) and
\(\HH^j(M)\) are nonzero, then
\[
\begin{gathered}
\HH^0\bigl(\Tot(\HH^i(M)(-i))\bigr)
 \hookrightarrow\HH^0(\Tot(M)),\\
\HH^1(\Tot(M))
 \twoheadrightarrow
\HH^1\bigl(\Tot(\HH^j(M)(-j))\bigr).
\end{gathered}
\]
\end{lemma}

\begin{proof}
Part (1) is the long exact sequence of total cohomology. For (2), use
the \(E_2\)-page of the slope spectral sequence
\[
E_2^{a,b}=\HH^a(\Tot(\HH^b(M)))
 \Longrightarrow\HH^{a+b}(\Tot(M)).
\]
At the two extremal nonzero cohomological degrees, the relevant edge
terms admit no possible incoming or outgoing differential. They
therefore survive to \(E_\infty\), giving the asserted subobject and
quotient.
\end{proof}

Although totalization alone can miss part of a diagonal complex, its Nygaard modifications together detect the diagonal amplitude.

\begin{proposition}[{\cite[Proposition~I.1.6]{Ek3}}]\label{dcamplitude}
{\rm (1)} The subcategories
\(\widetilde D_c^b(R)_{\leqslant0}\) and
\(\widetilde D_c^b(R)_{\geqslant0}\) are stable under every Nygaard
modification \((-)(\varphi)\), with
\(\varphi\in\mathrm{Fun}_c(\Z,\Z)\).

{\rm (2)} Let \(M\in D_c^b(R)\), and choose
\(\I_1=[m_1,n_1]\) such that
\(M\in D_c^b(R\td\I_1)\) (see \Cref{derived category of R-modules of level I}). For another interval
\(\I=[m,n]\), the complex \(M\) has diagonal amplitude \([m,n]\) if
and only if there exists \(N\in\Z\) such that
\(\Tot(M(\varphi))\) has ordinary amplitude \([m,n]\) whenever
\[
\varphi(k)\geqslant N
\qquad(m_1+1\leqslant k\leqslant n_1).
\]
\end{proposition}

\begin{proof}
For a coherent \(R\)-module \(A\), one has
\[
\dt_{\leqslant i}(A(\varphi))
 =\bigl(\dt_{\leqslant i}A\bigr)(\varphi),
\]
which proves (1).

Suppose first that \(M\) has diagonal amplitude \([m,n]\). The domino
parts of the finitely many modules
\(\HH^i(\tH^j(M))\), with \(m\leqslant j\leqslant n\) and
\(-n_1+1\leqslant i\leqslant-m_1\), are iterated extensions of
finitely many \(U_t\). Choose \(N\) so large that
\(t+N>0\) for every such constituent. If \(\varphi\geqslant N\) on
\([m_1+1,n_1]\), Nygaard modification replaces these constituents by
\(U_{t+\varphi(-i+1)}\). Their indices are positive, so their
totalizations have the expected single degree. The spectral sequence
of \Cref{Totsubquo} now gives ordinary amplitude \([m,n]\) for
\(\Tot(M(\varphi))\).

Conversely, assume the stated uniform amplitude condition and suppose
that \(\tH^j(M)\neq0\) for some maximal \(j>n\). By
\Cref{Totsubquo}(1), both total cohomology groups of
\(\tH^j(M(\varphi))\) must vanish. Let \(j'\) be its smallest ordinary
cohomological degree. The heart part would survive by
\Cref{dcDegen}(2), so this extremal term must be a domino and
\(j'+1\in[m_1+1,n_1]\). Taking \(\varphi(j'+1)\) sufficiently large
makes
\[
\HH^{-j'}\!\left(\Tot\bigl(\HH^{j'}(\tH^j(M(\varphi)))\bigr)\right)
\neq0.
\]
Again by \Cref{dcDegen}(2), this injects into the vanishing group
\(\HH^0(\Tot(\tH^j(M(\varphi))))\), a contradiction. The argument at
the lower endpoint is dual.
\end{proof}

The amplitude criterion also makes duality accessible in the
diagonal heart. We use its compatibility with Nygaard modification:
\[
D(M(\varphi))=D(M)(\td\varphi\td),
\qquad \td\varphi\td(n)=-\varphi(-n).
\]
Together with \Cref{dcamplitude} and
\(\Tot(D(M))=R\Hom_W(\Tot(M),W)\), this shows that \(D\) has diagonal
amplitude \([0,1]\). Put
\(\widetilde D^i=\tH^i\circ D\). For \(M\in\dc\), the biduality
spectral sequence is
\[
E_2^{i,j}=\widetilde D^i(\widetilde D^j(M))\Longrightarrow M.
\]
Since \(D\) has diagonal amplitude \([0,1]\), this spectral sequence
has only four possible terms. The two terms of nonzero total degree
vanish:
\[
\widetilde D^0\widetilde D^1=0,
\qquad
\widetilde D^1\widetilde D^0=0.
\]
The remaining terms give a canonical short exact sequence
\[
0\longrightarrow\widetilde D^1\widetilde D^1(M)
 \longrightarrow M\longrightarrow
 \widetilde D^0\widetilde D^0(M)\longrightarrow0.
\]
We write
\[
t^1:=\widetilde D^1\widetilde D^1,
\qquad t^0:=\widetilde D^0\widetilde D^0.
\]
These functors are unrelated to \(\rmt\) from \Cref{rmt}.
The exact sequence separates two parts of a diagonal object.
We call the first its finite-torsion part; the next results explain
this terminology and the maximality of the construction.

\begin{definition} \label{finite torsion}
Let \(M\in\dc\).

{\rm (1)} The object \(M\) is finite torsion if \(t^1(M)=M\).

{\rm (2)} It is without finite torsion if \(t^1(M)=0\).
\end{definition}

\begin{proposition}[{\cite[Proposition~I.1.11]{Ek3}}]\label{without finite torsion implies artinian}
Let \(M\in\dc\).

{\rm (1)} The object \(\widetilde D^1(M)\) is finite torsion, whereas
\(\widetilde D^0(M)\) is without finite torsion. Moreover, \(t^1(M)\)
is the maximal finite-torsion subobject of \(M\), and \(t^0(M)\) is its
maximal quotient without finite torsion.

{\rm (2)} Every descending chain of subobjects
\[
M\supseteq M_1\supseteq M_2\supseteq\cdots
\]
for which each \(M/M_i\) is without finite torsion eventually
stabilizes.
\end{proposition}

The condition on the quotients is essential. If \(M\) is a finite free
Dieudonn\'e module, the chain \((p^iM)\) need not stabilize, but every
quotient \(M/p^iM\) is finite torsion. Part (2) asserts stabilization
precisely after this finite-torsion phenomenon has been excluded.

\begin{proof}
Apply the biduality sequence to \(\widetilde D^1(M)\). Its \(t^0\)-term
vanishes because \(\widetilde D^0\widetilde D^1=0\); hence
\[
t^1(\widetilde D^1(M))\simeq\widetilde D^1(M).
\]
Thus \(\widetilde D^1(M)\) is finite torsion. Similarly,
\[
t^1(\widetilde D^0(M))
 =\widetilde D^1\widetilde D^1\widetilde D^0(M)=0,
\]
so \(\widetilde D^0(M)\) is without finite torsion. Applying these
statements once more shows that \(t^1(M)\) is finite torsion and
\(t^0(M)\) is without finite torsion.

If \(A\) is finite torsion and \(B\) is without finite torsion, then
every morphism \(A\to B\) vanishes. Indeed, this follows from the
naturality of \(t^1(X)\to X\), since \(t^1(A)=A\) and \(t^1(B)=0\).
Consequently, every finite-torsion subobject \(A\hookrightarrow M\)
maps trivially to \(t^0(M)\) and therefore lies in \(t^1(M)\).
Likewise, every quotient \(M\twoheadrightarrow B\) without finite
torsion kills \(t^1(M)\) and hence factors through \(t^0(M)\). This
proves (1).

For (2), apply the contravariant functor \(\widetilde D^0\) to the
quotient tower. It gives an ascending chain
\[
\cdots\hookrightarrow\widetilde D^0(M/M_i)
 \hookrightarrow\widetilde D^0(M/M_{i+1})
 \hookrightarrow\cdots\hookrightarrow\widetilde D^0(M).
\]
By \Cref{diagonal category is Noetherian}, this chain stabilizes. Biduality on
objects without finite torsion then shows that the quotients
\(M/M_i\), and hence the subobjects \(M_i\), stabilize as well.
\end{proof}

The following agrees with Ekedahl's duality calculation \cite[Theorem 3.3]{Ek1}; we prove it using total duality and Nygaard modification.

\begin{lemma}\label{elementary finite-torsion objects}
Let \(N\) be a finite-type Dieudonn\'e module.

{\rm (1)} If \(N\) has finite \(W\)-length, then \(N\) is finite
torsion.

{\rm (2)} If \(N\) is finite free over \(W\), then \(N\) is without
finite torsion.

{\rm (3)} Every elementary domino \(U_i\) is without finite torsion.
\end{lemma}

\begin{proof}
A Dieudonn\'e module is unchanged by Nygaard modification. If \(N\)
has finite \(W\)-length, then
\[
R\Hom_W(N,W)\simeq\Hom_W(N,K/W)[-1]
\]
is concentrated in cohomological degree \(1\). Hence total duality
shows that \(\Tot(D(N)(\varphi))\) has amplitude \([1,1]\) for every
Nygaard modification \(\varphi\). By \Cref{dcamplitude}, \(D(N)\) has
diagonal amplitude \([1,1]\). Thus \(\widetilde D^0(N)=0\), and the
biduality sequence reduces to \(t^1(N)\simeq N\). This proves (1).

If \(N\) is finite free over \(W\), then
\(R\Hom_W(N,W)=\Hom_W(N,W)\) is concentrated in degree \(0\). The
same argument shows that \(D(N)\) has diagonal amplitude \([0,0]\).
Consequently, \(\widetilde D^1(N)=0\) and \(t^1(N)=0\), proving (2).

We first show that no elementary domino \(U_r\) is finite torsion.
Suppose otherwise. Since
\(U_r=t^1(U_r)=\widetilde D^1\widetilde D^1(U_r)\), the cross-vanishing
\(\widetilde D^0\widetilde D^1=0\) gives
\(\widetilde D^0(U_r)=0\). Hence
\[
D(U_r)\simeq\widetilde D^1(U_r)[-1].
\]
Choose a Nygaard modification \(\varphi\) such that
\(U_r(\varphi)\simeq U_s\) with \(s<0\). Then
\(\HH^1(\Tot(U_r(\varphi)))\neq0\). This group has finite
\(W\)-length, so total duality gives
\[
\HH^0\!\left(\Tot(D(U_r(\varphi)))\right)\neq0.
\]
On the other hand, compatibility of duality with Nygaard modification
gives
\[
D(U_r(\varphi))
 \simeq\widetilde D^1(U_r)(\td\varphi\td)[-1].
\]
The object \(\widetilde D^1(U_r)(\td\varphi\td)\) lies in \(\dc\), so
its total complex has amplitude \([0,1]\). After shifting by \([-1]\),
it has no cohomology in degree \(0\), a contradiction.

Finally, suppose that \(t^1(U_i)\neq0\). It is a finite-torsion
subobject of \(U_i\). By \Cref{subobjects of diagonal modules}, it is
an actual diagonal \(R\)-module; the classification of subobjects of a
one-dimensional domino identifies it with some \(U_r\). This
contradicts the preceding paragraph. Thus \(t^1(U_i)=0\).
\end{proof}

\begin{theorem}[{\cite[Theorem~I.1.12]{Ek3}}, in our notation]\label{description of finite torsion object}
Let \(M\in\dc\).

{\rm (1)} The object \(M\) is finite torsion if and only if every
\(\HH^i(M)(-i)\) is a finite-length Dieudonn\'e module. In this case,
\[
M\cong\bigoplus_i\HH^i(M)[-i].
\]

{\rm (2)} If \(M\in\dc\!\cap R\td\mathrm{Mod}\), then \(M\) is
without finite torsion if and only if \(\coeur(M^0)\) is torsion free.

{\rm (3)} The object \(M\) is without finite torsion if and only if
every \(\HH^i(M)(-i)\) is without finite torsion.
\end{theorem}

\begin{proof}
By \Cref{without finite torsion implies artinian}(1), finite-torsion
objects are closed under quotients, extensions, and direct summands,
whereas objects without finite torsion are closed under subobjects and
extensions.

We first prove (1). Suppose that \(M\) is finite torsion, and let \(r\)
be the largest integer for which \(\HH^r(M)\neq0\). Since the ordinary
and diagonal truncations commute, the canonical map
\(M\twoheadrightarrow\HH^r(M)[-r]\) is an epimorphism in \(\dc\).
Hence
\[
N_r:=\HH^r(M)(-r)\in\dc\!\cap R\td\mathrm{Mod}
\]
is finite torsion.

The diagonal module \(N_r\) is an extension of a finite-type
Dieudonn\'e module by a coherent domino. If its domino quotient were
nonzero, it would have an elementary domino \(U_i\) as a quotient.
That quotient would be finite torsion, contradicting
\Cref{elementary finite-torsion objects}(3). Thus \(N_r\) is a
Dieudonn\'e module.

If \(N_r\) were not of finite \(W\)-length, its quotient by its
\(W\)-torsion would be a nonzero finite free Dieudonn\'e module. This
would again be a finite-torsion quotient of \(N_r\), contradicting
\Cref{elementary finite-torsion objects}(2). Consequently, \(N_r\) is
a finite-length Dieudonn\'e module.

The cohomology of \(M\) is now cut at the boundary between internal
degrees \(-r\) and \(1-r\). The two grading blocks in the cut category
therefore separate the ordinary truncation triangle, giving
\[
M\simeq\tau_{\leqslant r-1}M
 \oplus\HH^r(M)[-r].
\]
The first summand is again finite torsion. Induction on the ordinary
cohomological amplitude gives
\[
M\simeq\bigoplus_i\HH^i(M)[-i],
\]
and every \(\HH^i(M)(-i)\) is a finite-length Dieudonn\'e module.

Conversely, suppose that every \(\HH^i(M)(-i)\) is a finite-length
Dieudonn\'e module. Repeating the same cut argument splits the ordinary
cohomology filtration and gives the displayed decomposition. Each
summand is finite torsion by
\Cref{elementary finite-torsion objects}(1), and hence so is \(M\).
This proves (1).

We next prove (3). If every \(\HH^i(M)(-i)\) is without finite torsion,
then so is \(M\), because such objects are closed under extensions.

Conversely, suppose that \(M\) is without finite torsion, and let \(r\)
be its largest nonzero ordinary cohomological degree. The subobject
\(\tau_{\leqslant r-1}M\subseteq M\) is without finite torsion, so
induction treats all lower cohomology objects. After a twist, we may
assume \(r=0\) and write
\[
0\longrightarrow\tau_{\leqslant-1}M\longrightarrow M
\longrightarrow\HH^0(M)\longrightarrow0.
\]
For an object \(X\) without finite torsion, one has
\(\widetilde D^1(X)=0\): indeed, \(\widetilde D^1(X)\) is finite
torsion and
\(t^1(\widetilde D^1(X))=\widetilde D^1(t^1(X))=0\). Thus the long
exact sequence for the dual contains
\[
\widetilde D^0(M)\xrightarrow{\phi}
\widetilde D^0(\tau_{\leqslant-1}M)\longrightarrow\widetilde D^1(\HH^0(M))
\longrightarrow0.
\]
Put \(T=\operatorname{im}(\phi)\) and
\(N=\widetilde D^0(T)\). Naturality and biduality give a monomorphism
\(N\hookrightarrow\widetilde D^0\widetilde D^0(M)\simeq M\) and an
exact sequence
\[
0\longrightarrow\tau_{\leqslant-1}M\longrightarrow N
 \longrightarrow t^1(\HH^0(M))\longrightarrow0.
\]
In particular, \(N\) is without finite torsion.

The object \(t^1(\HH^0(M))\) is a subobject of the actual diagonal module
\(\HH^0(M)\). By \Cref{subobjects of diagonal modules}, it is itself an actual
diagonal module, and by (1) it is a finite-length Dieudonn\'e module.
It follows that \(N\) is cut at internal degree \(0\), so the last exact
sequence splits. Thus \(t^1(\HH^0(M))\) is a direct summand of the object \(N\)
without finite torsion. Since it is itself finite torsion, it must
vanish. Hence \(\HH^0(M)\) is without finite torsion, completing the induction
and proving (3).

Finally, let \(M\in\dc\!\cap R\td\mathrm{Mod}\). By
\Cref{subobjects of diagonal modules}, every finite-torsion subobject
of \(M\) is an actual diagonal \(R\)-module. Part (1) shows that these
are precisely the finite-length Dieudonn\'e submodules of \(M\). They
all lie in
\[
V^{-\infty}ZM^0=\coeur(M^0),
\]
and every \(W\)-torsion element of \(\coeur(M^0)\) generates a
finite-length Dieudonn\'e submodule. Therefore \(M\) contains no
nonzero finite-torsion subobject if and only if \(\coeur(M^0)\) is
torsion free. This proves (2).
\end{proof}

\subsection{\(R\)-modules cut at \(\I\)}
\label{subsec:r-modules-cut-at-i}

For bounded coherent complexes, a cut in cohomology can be realized termwise.

\begin{definition}\label{R-modules cut at I}
Let \(\I\subseteq\Z\). A graded \(R\)-module \(M\) is said to be
cut at \(\I\) if
\[
d:M^i\longrightarrow M^{i+1}
\]
is zero for every \(i\in\I\). We denote the abelian category of such
modules, with grading-preserving morphisms, by
\(R\td\mathrm{Cut}\td\I\).
A complex in \(D(R\td\mathrm{Cut}\td\I)\) is called coherent if its
image in \(D(R)\) is coherent. We write
\(D_c^b(R\td\mathrm{Cut}\td\I)\) for the bounded coherent subcategory.
\end{definition}

Here \(D(R\td\mathrm{Cut}\td\I)\) is the derived category of the
abelian category \(R\td\mathrm{Cut}\td\I\). Its exact inclusion into
graded \(R\)-modules induces a functor
\[
    D(R\td\mathrm{Cut}\td\I)\longrightarrow
    \mathcal C_\I
    :=\{M\in D(R)\mid \HH^q(M)\text{ is cut at }\I
                         \text{ for every }q\},
\]
where \(\mathcal C_\I\) is a full subcategory of \(D(R)\). This functor
need not be fully faithful, even on bounded complexes.

For example, take \(k=\F_p\), \(W=\Z_p\), and
\(\I=\{0\}\). Let \(A\) and \(B\) be graded \(R\)-modules with
zero differential, specified by
\[
\begin{array}{c|c|c|c}
 &\text{nonzero graded piece}&F&V\\ \hline
 A&A^0=W&p&1\\
 B&B^1=W&1&p .
\end{array}
\]
Both are cut at \(\{0\}\). With these operators on its two pieces, the
graded module \(E=(W\xrightarrow{\,d=1\,}W)\), in internal degrees
\(0,1\), is an \(R\)-module because \(F_1dV_0=d\). It gives a nonsplit
exact sequence
\[
    0\longrightarrow B\longrightarrow E\longrightarrow A
    \longrightarrow0.
\]
A splitting cannot commute with \(d\), whereas every extension of
\(A\) by \(B\) in the cut category splits by internal degree. Hence
\[
    \Hom_{D(R\td\mathrm{Cut}\td\{0\})}(A,B[1])=0,
    \qquad
    \Hom_{D(R)}(A,B[1])\neq0.
\]
Thus the functor is not full. Since \(V=1\) on \(A\), this example
lies outside the coherent category. For bounded coherent complexes,
the distinction disappears: a cut in cohomology can be realized
termwise, and the cut category has the same derived morphisms.

\begin{proposition}[{\cite[Proposition~0.5.1]{Ek3}}]\label{D_c^b(R-Cut-I)}
The natural functor
\[
D_c^b(R\td\mathrm{Cut}\td\I)\longrightarrow D_c^b(R)
\]
is fully faithful. Its essential image consists of those complexes
\(M\) for which every cohomology module \(\HH^q(M)\) is cut at \(\I\).
\end{proposition}

\begin{proof}
By d\'evissage along
cohomological truncations, full faithfulness reduces to comparing the
grading-preserving derived Hom complexes of coherent modules \(M,N\)
cut at \(\I\).

Partition the grading indices into maximal consecutive blocks
\(\mathrm J_\alpha\) by cutting between \(i\) and \(i+1\) for every
\(i\in\I\). Since the differentials across these cuts vanish, we have
decompositions into \(R\)-submodules
\[
M=\bigoplus_\alpha M_\alpha,\qquad
N=\bigoplus_\alpha N_\alpha,
\]
where \(M_\alpha\) and \(N_\alpha\) are supported in
\(\mathrm J_\alpha\). Only finitely many summands are nonzero. Within
each block, imposing the cuts does not change the degree-preserving
derived Hom complex:
\[
R\operatorname{Hom}_{R\td\mathrm{Cut}\td\I}
 (M_\alpha,N_\alpha)^0
 \simeq
R\operatorname{Hom}_R(M_\alpha,N_\alpha)^0.
\]
Between distinct blocks, derived Hom vanishes in the cut category, and
\[
R\operatorname{Hom}_R(M_\alpha,N_\beta)^0=0
\qquad(\alpha\neq\beta)
\]
follows by d\'evissage from Ekedahl's calculation \cite[Proposition~0.4.1]{Ek3},
as in the proof of \cite[Proposition~0.5.1]{Ek3}. Here \(0\) denotes
internal, not cohomological, degree. Summing over the blocks proves
full faithfulness.

Cohomology preserves the cut condition, so every object in the image
has the stated property. Conversely, let \(M\in D_c^b(R)\) have
cohomology cut at \(\I\), and argue by induction on its cohomological
amplitude \([a,b]\). The case \(a=b\) follows from
\(M\simeq\HH^a(M)[-a]\). Otherwise, choose \(a\leqslant q<b\) and
consider
\[
\tau_{\leqslant q}M\longrightarrow M\longrightarrow
\tau_{\geqslant q+1}M
\xrightarrow{\ \varphi\ }(\tau_{\leqslant q}M)[1].
\]
By induction, the two truncations come from objects \(M_1,M_2\) of
\(D_c^b(R\td\mathrm{Cut}\td\I)\). Full faithfulness lifts \(\varphi\)
to \(\widetilde\varphi:M_2\to M_1[1]\), and
\[
M'=\operatorname{Cone}(\widetilde\varphi)[-1]
\]
maps to an object isomorphic to \(M\). This proves the assertion about
the essential image.
\end{proof}

\begin{remark}
We henceforth identify \(D_c^b(R\td\mathrm{Cut}\td\I)\) with its
essential image in \(D_c^b(R)\). Thus a coherent complex is cut at
\(\I\) when its cohomology modules are, or equivalently when it has a
representative cut at \(\I\) termwise.
\end{remark}

\subsection{Hodge--Witt objects}
\label{subsec:hodge-witt-objects}

The simplest case is a cut at every internal degree.
An object \(M\in D_c^b(R)\) is Hodge--Witt if it is cut at \(\Z\), or
equivalently if every horizontal differential
\(d:\HH^j(M)^i\to\HH^j(M)^{i+1}\) vanishes. Within the diagonal heart,
this condition singles out the objects closest to classical
Dieudonn\'e modules. We write \(\dc_{\HW}\) for their full subcategory.
We first check that this class is preserved by the basic operations
in the heart.

\begin{proposition}[{\cite[Proposition~I.2.1]{Ek3}}]\label{Hodge--Witt objects form a Serre subcategory}
The category \(\dc_{\HW}\) is closed under subobjects, quotients, and
extensions in \(\dc\).
\end{proposition}

\begin{proof}
Closure under extensions follows from \Cref{D_c^b(R-Cut-I)}, since
complexes cut at \(\Z\) form a triangulated subcategory. For an actual diagonal module \(A\), the structure
theorem shows that \(A\) is Hodge--Witt precisely when \(A^0\) is
finitely generated over \(W\): a nonzero domino quotient contains a
copy of \(k[[V]]\) in degree zero.

Let \(M'\hookrightarrow M\) be a subobject in \(\dc\), with \(M\)
Hodge--Witt. As in \Cref{subobjects of diagonal modules}, ordinary
cohomology gives injections
\(\HH^i(M')^{-i}\hookrightarrow\HH^i(M)^{-i}\) for every \(i\). The
targets are finitely generated over \(W\), so the preceding criterion
shows that \(M'\) is Hodge--Witt. For \(M''=M/M'\), the same long exact
sequence gives surjections
\(\HH^i(M)^{-i}\twoheadrightarrow\HH^i(M'')^{-i}\) for every \(i\), and
hence \(M''\) is Hodge--Witt as well.
\end{proof}

Hodge--Witt objects can be detected by their \(p\)-torsion or by the total cohomology of their Nygaard modifications.

\begin{proposition}[{\cite[Proposition~I.2.2]{Ek3}}]\label{Hodge--Witt criterion}
Let \(M\in\dc\).

{\rm (1)} The object \(M\) is Hodge--Witt if and only if its
\(p\)-torsion subobject
\({}_pM:=\ker_{\dc}(p:M\to M)\) is Hodge--Witt.

{\rm (2)} The object \(M\) is Hodge--Witt if and only if the lengths of
\(\HH^1(\Tot(M(\varphi)))\) are uniformly bounded as \(\varphi\) varies.
\end{proposition}

\begin{proof}
For (1), closure under subobjects gives the forward implication.
For the converse, first suppose \(N\in\dc\!\cap R\td\mathrm{Mod}\). Here
\(p:N\to N\) is the scalar map \(p\operatorname{id}_N\) in \(D(R)\),
so it induces multiplication by \(p\) on ordinary cohomology. Write
\(P=\operatorname{im}_{\dc}(p)\). By
\Cref{subobjects of diagonal modules}, both \({}_pN\) and \(P\) are actual diagonal
modules, and \(P^0\hookrightarrow N^0\). Ordinary cohomology of
\(0\to{}_pN\to N\to P\to0\), together with the factorization
\(p:N\twoheadrightarrow P\hookrightarrow N\), therefore gives
\(({}_pN)^0=\ker(N^0\to P^0)=\ker(p:N^0\to N^0)\).

Set \(C=V^{-\infty}ZN^0\). If \({}_pN\) is Hodge--Witt, then \(({}_pN)^0\)
has finite \(W\)-length. The finite-type Dieudonn\'e module \(C\) also
has \(C/pC\) of finite length, so the snake lemma applied to
\(0\to C\to N^0\to N^0/C\to0\) makes
\(\ker(p:N^0/C\to N^0/C)\) finite. If \(N^0/C\neq0\), its domino
filtration places a copy of \(k[[V]]\) in this kernel, a contradiction.
Thus \(C=N^0\), and \(N\) is Hodge--Witt.

For general \(M\), induct on its ordinary cohomological amplitude and
let \(r\) be maximal with \(\HH^r(M)\neq0\). Since
\({}_p(\tau_{\leqslant r-1}M)\) is a subobject of \({}_pM\), induction
shows that \(\tau_{\leqslant r-1}M\) is Hodge--Witt. The snake lemma
for \(0\to\tau_{\leqslant r-1}M\to M\to\HH^r(M)[-r]\to0\), together
with the preceding proposition, shows that
\({}_p(\HH^r(M)[-r])\) is Hodge--Witt. The actual-module case, after shifting both gradings, shows that
\(\HH^r(M)[-r]\) is Hodge--Witt. Closure under extensions completes
the induction.

For (2), if \(M\) is Hodge--Witt, choose a representative cut at
\(\Z\) using \Cref{D_c^b(R-Cut-I)}. Nygaard modification only
Frobenius-twists its columns, and diagonal support gives
\(\HH^1(\Tot(M(\varphi)))=0\) for every \(\varphi\).

Conversely, first take \(N\in\dc\!\cap R\td\mathrm{Mod}\). If \(N\)
is not Hodge--Witt, its nonzero domino quotient has an elementary
quotient \(U_t\). Exactness of Nygaard modification, the long exact
sequence of total cohomology, and \Cref{Basic Vanishing}(1) give a
surjection \(\HH^1(\Tot(N(\varphi)))\twoheadrightarrow
\HH^1(\Tot(U_t(\varphi)))\). Taking \(\varphi_n=-n\delta_1\) gives
\(U_t(\varphi_n)\simeq U_{t-n}\), whose degree-one total
cohomology has length \(\max\{0,n-t\}\). Thus the lengths for
\(N\) are unbounded.

For general \(M\), let \(r\) again be its largest nonzero ordinary
cohomological degree. Applied to \(M(\varphi)\), \Cref{dcDegen}(2)
passes the uniform bound to the modifications of \(\HH^r(M)(-r)\), up
to reindexing \(\varphi\) under the twist. The actual-module case makes
\(\HH^r(M)[-r]\) Hodge--Witt. The cut between internal degrees \(-r\)
and \(1-r\) then gives
\(M\simeq\tau_{\leqslant r-1}M\oplus\HH^r(M)[-r]\). The lower summand
inherits the bound, and induction completes the proof.
\end{proof}

\begin{corollary}[{\cite[Corollary~I.2.2.1]{Ek3}}]\label{derived Hodge--Witt criterion}
Let \(M\in D_c^b(R)\). Then \(M\) is Hodge--Witt if and only if the
torsion lengths of \(\HH^*(\Tot(M(\varphi)))\) are uniformly bounded in
\(\varphi\).
\end{corollary}

\begin{proof}
If \(M\) is Hodge--Witt, choose a representative cut at \(\Z\) using
\Cref{D_c^b(R-Cut-I)}. Nygaard modification only Frobenius-twists its
horizontal columns, so the torsion lengths of its total cohomology are
unchanged.

Conversely, \Cref{dcamplitude}(1) identifies
\(\tH^j(M(\varphi))\) with \(\tH^j(M)(\varphi)\), while
\Cref{Totsubquo}(1) gives injections
\(\HH^1(\Tot(\tH^j(M)(\varphi)))\hookrightarrow
\HH^{j+1}(\Tot(M(\varphi)))\). The sources are torsion by
\Cref{Basic Vanishing}(1), so their lengths are uniformly bounded.
By \Cref{Hodge--Witt criterion}(2), every \(\tH^j(M)\) is therefore
Hodge--Witt, and so is \(M\) by its diagonal cohomology filtration.
\end{proof}

\begin{definition}\label{definition of HW}
For \(M\in\dc\), let \({}_{\HW}(M)\) denote its maximal
Hodge--Witt subobject. Its existence and functoriality follow from
\Cref{diagonal category is Noetherian,Hodge--Witt objects form a Serre subcategory}.
The subfunctors \(0\subseteq{}_{\HW}\subseteq\mathrm{id}_{\dc}\)
form a radical filtration on the diagonal heart.
\end{definition}

Compare \cite[Chapter~I, Section~2, after Corollary~2.2.1]{Ek3}.

The Hodge--Witt part commutes with ordinary cohomology.

\begin{proposition}[{\cite[Proposition~I.2.2.2]{Ek3}}]\label{HW of cohomology equals cohomology of HW}
For every \(M\in\dc\), there is a canonical isomorphism
\[
\HH^i({}_{\HW}(M))(-i)
 \cong{}_{\HW}\bigl(\HH^i(M)(-i)\bigr).
\]
After undoing the internal twist, we abbreviate this as
\(\HH^i({}_{\HW}(M))\cong{}_{\HW}(\HH^i(M))\).
\end{proposition}

\begin{proof}
Set \(A={}_{\HW}(M)\). Since \(A\) is Hodge--Witt, each
\(\HH^j(A)\) is concentrated in internal degree \(-j\). Its ordinary
cohomology filtration therefore splits by internal degree, so
\(\HH^i(A)[-i]\) is a direct summand of \(A\). Let \(Q\) be the
cokernel of the resulting monomorphism
\(\HH^i(A)[-i]\hookrightarrow M\) in \(\dc\). The connecting map
\(\HH^{i-1}(Q)\to\HH^i(A)\) vanishes: its source is generated in
internal degree \(1-i\), whereas its target is concentrated in degree
\(-i\). Hence \(\HH^i(A)\to\HH^i(M)\) is injective, with Hodge--Witt
image.

We claim that \({}_{\HW}(X)=0\) implies
\({}_{\HW}(\HH^i(X)(-i))=0\) for every \(i\). Put
\(S={}_{\HW}(\HH^i(X)(-i))\). Through the canonical monomorphism
\(\HH^i(X)[-i]\hookrightarrow\tau_{\geqslant i}X\), regard
\(S(i)[-i]\) as a subobject of \(\tau_{\geqslant i}X\), and let
\(N\subseteq X\) be its inverse image. There is an exact sequence
\(0\to\tau_{\leqslant i-1}X\to N\to S(i)[-i]\to0\) in \(\dc\).
The left term is supported in internal degrees at least \(1-i\), while
the right term is supported in degree \(-i\). Thus \(N\) is cut at \(\{-i\}\), and
\Cref{D_c^b(R-Cut-I)} splits the sequence by internal degree.
Hence \(S(i)[-i]\) is a Hodge--Witt subobject of \(X\), forcing
\(S=0\).

Now apply the claim to \(X=M/A\), for which \({}_{\HW}(X)=0\). The
image of any Hodge--Witt subobject of \(\HH^i(M)(-i)\) in
\(\HH^i(X)(-i)\) is again Hodge--Witt and hence zero. It is therefore
contained in \(\HH^i(A)(-i)\); the reverse inclusion was proved above.
This gives the asserted canonical isomorphism.
\end{proof}

The next proposition describes which cycles survive in the slope spectral sequence and shows that their torsion classes admit torsion lifts.

\begin{proposition}[{\cite[Proposition~I.2.3]{Ek3}}]\label{survie du coeur}
Let \(M\in D_c^b(R)\), and consider the slope spectral sequence
\[
E_1^{i,j}=\HH^j(M)^i\Longrightarrow\HH^{i+j}(\Tot(M)).
\]

{\rm (1)} One has
\[
B_\infty^{i,j}\subseteq F^\infty B\HH^j(M)^i
\subseteq V^{-\infty}Z\HH^j(M)^i
\subseteq Z_\infty^{i,j}.
\]

{\rm (2)} Under the canonical identification
\[
\Fil^j\HH^{i+j}(\Tot(M))/\Fil^{j-1}\HH^{i+j}(\Tot(M))
 \cong Z_\infty^{i,j}/B_\infty^{i,j},
\]
every torsion class in
\(V^{-\infty}Z\HH^j(M)^i/B_\infty^{i,j}\) lifts to a torsion class in
\(\Fil^j\HH^{i+j}(\Tot(M))\).
\end{proposition}

We will compare successive pages using the following elementary
observation.

\begin{lemma}
Let \(f:X^\bullet\to Y^\bullet\) be a morphism of complexes in an
abelian category. If \(f^i\) is surjective and \(f^{i+1}\) is
injective, then \(\HH^i(f)\) is surjective and
\(\HH^{i+1}(f)\) is injective.
\end{lemma}

\begin{proof}
Every cycle in \(Y^i\) lifts to \(X^i\), and the lift is a cycle
because \(f^{i+1}\) is injective. This proves surjectivity on \(\HH^i\).
If a cycle in \(X^{i+1}\) maps to a boundary, lift a primitive through
the surjection \(f^i\); injectivity of \(f^{i+1}\) then shows that the
original cycle is a boundary.
\end{proof}

\begin{proof}
Put \(n=i+j\) and \(X=\dt_{\leqslant n-1}M\). Since ordinary and
diagonal truncations commute, \(X\to M\) induces the inclusion
\[
E_1^{i,j}(X)=F^\infty B\HH^j(M)^i
 \hookrightarrow E_1^{i,j}(M)=\HH^j(M)^i,
\]
and is an isomorphism in total degrees below \(n\). Repeated application
of the preceding lemma gives the same comparison on every page, with
injectivity in total degree \(n\). Every boundary entering \((i,j)\)
therefore lifts from \(X\), proving
\(B_\infty^{i,j}\subseteq F^\infty B\HH^j(M)^i\).
The inclusion \(F^\infty B\subseteq V^{-\infty}Z\) follows from
\(dV^rF^a d=0\) for \(a,r\geqslant0\).

For the last inclusion, consider the radical filtration on \(\dc\)
\[
T^a=0\quad(a<-1),\qquad T^{-1}={}_{\HW},\qquad
T^a=\mathrm{id}\quad(a\geqslant0).
\]
Its only nonformal orthogonality is
\(\Hom_{\dc}({}_{\HW}(A),B/{}_{\HW}(B))=0\): the image would be a
Hodge--Witt subobject of \(B/{}_{\HW}(B)\). Write
\(\widehat\tau_{\leqslant a}\) for the associated truncation. By construction,
\[
\tH^q(\widehat\tau_{\leqslant a}M)\cong
\begin{cases}
\tH^q(M),&q\leqslant a,\\
{}_{\HW}(\tH^{a+1}(M)),&q=a+1,\\
0,&q>a+1.
\end{cases}
\]
For \(M'=\widehat\tau_{\leqslant n-1}M\), this formula and
\Cref{HW of cohomology equals cohomology of HW} give
\[
E_1^{i,j}(M')=V^{-\infty}Z\HH^j(M)^i
 \hookrightarrow E_1^{i,j}(M),\qquad
E_1^{\ell,k}(M')=0\quad(k+\ell>n).
\]
All outgoing differentials at \((i,j)\) for \(M'\) have zero target,
so functoriality gives \(V^{-\infty}Z\HH^j(M)^i\subseteq Z_\infty^{i,j}\),
completing (1).

For (2), first note that every class in
\(F^\infty B\HH^j(M)^i/B_\infty^{i,j}\) has a torsion lift in
\(\Fil^j\HH^n(\Tot M)\). Indeed, this quotient is the image of
\(E_\infty^{i,j}(X)\), since \(X\) has no \(E_1\)-terms in total
degrees above \(n\). The group \(\HH^n(\Tot X)\) is torsion by
\Cref{Totsubquo}(1) and \Cref{Basic Vanishing}(1).

Now choose \(x\in V^{-\infty}Z\HH^j(M)^i\) representing the given
torsion class \(\bar x\). Put
\[
A=\dt_{\leqslant n-1}(\tau_{\leqslant j}M),\qquad
M'=\dt_{\leqslant n}(\tau_{\leqslant j}M),\qquad
D=\tH^n(\tau_{\leqslant j}M),\quad C=D[-n].
\]
The triangle \(A\to M'\to C\to A[1]\) induces, in bidegree \((i,j)\),
the quotient map
\[
\HH^j(M')^i=\HH^j(M)^i\longrightarrow
\HH^j(C)^i=\HH^j(M)^i/F^\infty B\HH^j(M)^i.
\]
Let \(x_C\) be the image of \(x\). It is torsion because
\(p^a\bar x=0\) implies \(p^ax\in B_\infty^{i,j}\subseteq F^\infty B\).
Moreover, \(B_\infty^{i,j}(C)=0\), since the \(E_1\)-page of \(C\)
is supported in total degrees \(n\) and \(n+1\).

Apply the finite-torsion sequence
\(0\to t^1(D)\to D\to t^0(D)\to0\).
By \Cref{description of finite torsion object}(2),(3), the stable cycles
of \(t^0(D)[-n]\) in total degree \(n\) are torsion free. Thus \(x_C\)
maps to zero there and lifts to \(\HH^j(t^1(D)[-n])^i\).
By part (1) of the same theorem, \(t^1(D)\) splits into shifts of
finite-length Dieudonn\'e modules. Its slope spectral sequence
therefore degenerates at \(E_1\), giving a torsion class
\(u\in\Fil^j\HH^n(\Tot C)\) representing \(x_C\).

Since \(\HH^{n+1}(\Tot A)=0\), lift \(u\) to
\(v\in\HH^n(\Tot M')\). The lift is torsion since \(u\) and
\(\HH^n(\Tot A)\) are torsion, the latter by
\Cref{Totsubquo}(1) and \Cref{Basic Vanishing}(1).
As \(M'\to M\) factors through \(\tau_{\leqslant j}M\), its image lies
in \(\Fil^j\HH^n(\Tot M)\) and represents \(\bar x\) modulo
\(F^\infty B\HH^j(M)^i/B_\infty^{i,j}\).
Correcting it by a torsion lift constructed above gives a torsion
lift of \(\bar x\).
\end{proof}

\begin{definition}[{\cite[Definition~I.2.4]{Ek3}}]\label{Filtration of torsion}
Let \(M\in D^{b}_{c}(R)\).

{\rm (1)} The exotic \(1\)-torsion of \(M\) in degree \(i\) is the
image of
\[
\HH^1\bigl(\Tot(\widetilde\HH^{i-1}(M))\bigr)
 \longrightarrow \HH^i(\Tot(M)).
\]

{\rm (2)} The finite total torsion of \(M\) in degree \(i\) is
\[
\HH^0\bigl(\Tot(t^0(\widetilde\HH^i(M)))\bigr).
\]

{\rm (3)} The exotic \(0\)-torsion of \(M\) in degree \(i\) is the
torsion subgroup of
\[
\HH^0\bigl(\Tot(t^1(\widetilde\HH^i(M)))\bigr).
\]
\end{definition}

These three pieces form a canonical three-step filtration of the
torsion in \(\HH^i(\Tot(M))\), with successive quotients given, in
order, by the exotic \(1\)-torsion, the finite total torsion, and the
exotic \(0\)-torsion.

If \(N=D(M)\), then the formula
\(
\Tot(D(-))
=
R\!\operatorname{Hom}_{W}
\bigl(\Tot(-),W\bigr)
\)
identifies the torsion in degree \(i\) of \(\Tot(M)\) with the dual of
the torsion in degree \(-i+1\) of \(\Tot(N)\). Under this identification,
the two three-step filtrations are dual.

For a smooth proper variety \(X\), the divisorial and exotic torsion of
\(
\HH^{2}_{\cris}(X/W)
\)
in the sense of \cite{Il1} agree, respectively, with the finite total
torsion and the exotic \(0\)-torsion defined above.

\begin{theorem}[{\cite[Theorem~I.2.5]{Ek3}}]\label{When one of the differential is zero}
Let \(M\in D^b_c(R)\), and suppose that
\(d:\HH^j(M)^i\to\HH^j(M)^{i+1}\) is zero.

{\rm (1)} The quotient of \(\HH^{i+j}(\Tot(M))\) by its exotic
\(1\)-torsion admits a canonical decomposition
\(\HH^{\leqslant i}\oplus\HH^{>i}\).

{\rm (2)} In the slope spectral sequence, the differentials
\[
E_r^{k,l}\longrightarrow E_r^{k+r,l-r+1}
\]
vanish whenever \(k+l=i+j\) and \(i-r<k\leqslant i\).
\end{theorem}

\begin{proof}
Put \(n=i+j\) and \(C=\tH^n(M)[-n]\). Since ordinary and diagonal
truncations commute,
\[
\HH^l(C)=
\frac{\dt_{\leqslant n-l}\HH^l(M)}
     {\dt_{\leqslant n-l-1}\HH^l(M)}.
\]
This module is supported in internal degrees \(n-l\) and \(n-l+1\).
The differential from internal degree \(i\) to \(i+1\) can be
nonzero only for \(l=j\), when it vanishes by hypothesis. Thus \(C\) is cut at \(\{i\}\). By
\Cref{D_c^b(R-Cut-I)}, it splits canonically as
\(C=C_{\leqslant i}\oplus C_{>i}\), with the two summands supported
in the indicated internal degrees.

By \Cref{Totsubquo}(1), the quotient of \(\HH^n(\Tot M)\) by its
exotic \(1\)-torsion is canonically \(\HH^n(\Tot C)\). The splitting
of \(C\) therefore proves (1), with
\(\HH^{\leqslant i}=\HH^n(\Tot C_{\leqslant i})\) and
\(\HH^{>i}=\HH^n(\Tot C_{>i})\).

For (2), set \(X=\dt_{\leqslant n}M\) and
\(A=\dt_{\leqslant n-1}M\), so that
\(A\to X\to C\to A[1]\) is a distinguished triangle.
On the \(E_1\)-page, both maps \(X\to M\) and \(X\to C\) are
surjective in total degree \(n\) and injective in total degree
\(n+1\): for the first this follows from the truncation formula,
and for the second from \(\HH^l(A)^k=0\) when \(k+l>n\).
The comparison lemma preceding the proof of \Cref{survie du coeur}
propagates these properties to every page.

If \(i-r<k\leqslant i\), the differential \(d_r\) crosses the cut
and hence vanishes for \(C\). Given a class in \(E_r^{k,l}(M)\), lift it to
\(E_r^{k,l}(X)\). Its differential maps to zero for \(C\), hence
is zero by injectivity in total degree \(n+1\). Its image for
\(M\) is therefore zero as well, proving (2).
\end{proof}

\begin{definition}\label{Hodge Witt in degree r}
An object \(M\in D_c^b(R)\) is Hodge--Witt in degree \(r\) if
\(d:\HH^j(M)^i\to\HH^j(M)^{i+1}\) vanishes whenever \(i+j=r\).
\end{definition}

\begin{definition}[{\cite[Definition~I.4.1]{Ek3}}]\label{Tot torsion definition}
Let \(M\in\dc\).

{\rm (1)} The object \(M\) is without \(\Tot\)-torsion if
\(\HH^*(\Tot(M))\) is torsion free.

{\rm (2)} It is \(\Tot\)-torsion if \(\HH^*(\Tot(M))\) is torsion.

{\rm (3)} It is \(\Tot\)-\(0\)-torsion if it is \(\Tot\)-torsion and
\(\HH^1(\Tot(M))=0\); it is \(\Tot\)-\(1\)-torsion if it is
\(\Tot\)-torsion and \(\HH^0(\Tot(M))=0\).

{\rm (4)} It is \(\Tot\)-acyclic if \(\Tot(M)=0\).
\end{definition}

\begin{lemma}[{\cite[Lemma~I.4.2]{Ek3}}]\label{sub and quotient object of Tot-torsion objects}
The following assertions hold in \(\dc\).

{\rm (1)} The \(\Tot\)-torsion objects are closed under subobjects,
quotients, and extensions.

{\rm (2)} The \(\Tot\)-\(1\)-torsion
(\(\Tot\)-\(0\)-torsion) objects are closed under extensions and
subobjects (respectively, quotients).

{\rm (3)} \(\Tot\)-\(1\)-torsion objects are without finite torsion.

{\rm (4)} For \(M\in D_c^b(R)\), one has \(\Tot(M)=0\) if and only if
every diagonal cohomology object \(\widetilde\HH^i(M)\) is
\(\Tot\)-acyclic.
\end{lemma}

\begin{proof}
An exact sequence
\(0\to M_1\to M_2\to M_3\to0\) in \(\dc\) gives
\[\begin{aligned}
    0 \longrightarrow \HH^{0}(\Tot(M_{1}))
\longrightarrow \HH^{0}(\Tot(M_{2}))
\longrightarrow \HH^{0}(\Tot(M_{3}))\\
\longrightarrow \HH^{1}(\Tot(M_{1}))
\longrightarrow \HH^{1}(\Tot(M_{2}))
\longrightarrow \HH^{1}(\Tot(M_{3}))
\longrightarrow 0.
\end{aligned}
\]
Since \(\HH^1(\Tot(M))\) is torsion for every \(M\in\dc\), this
sequence gives (1) and (2). If \(M\) is
\(\Tot\)-\(1\)-torsion, then so is its finite-torsion subobject
\(t^1(M)\). By \Cref{description of finite torsion object}, this
subobject is Hodge--Witt; hence
\(\HH^0(\Tot(t^1(M)))=0\) forces \(t^1(M)=0\). This proves (3), while
(4) is precisely \Cref{Totsubquo}.
\end{proof}

\subsection{Mazur--Ogus objects}
\label{subsec:mazur-ogus-objects-in-diagonal-heart}

Totalization does not detect every object of the diagonal heart:
a nonzero object may be \(\Tot\)-acyclic. To isolate objects free
from this ambiguity, we impose orthogonality to the acyclic objects
in both directions. This is the defining condition for a
Mazur--Ogus object.

\begin{definition}[{\cite[Definition~I.4.3]{Ek3}}]\label{Mazur-Ogus objects}
A diagonal complex \(M\) is Mazur--Ogus if
\[\operatorname{Hom}_{\dc}(N,M)=\operatorname{Hom}_{\dc}(M,N)=0\]
for every \(\Tot\)-acyclic object \(N\). We denote the resulting full
subcategory by \(\dc_{\MO}\).
\end{definition}

This condition has a concrete cohomological consequence. Reduction
by \(R_1\) is concentrated in total degree zero, and total cohomology
has no torsion.

\begin{lemma}[{\cite[Lemma~I.4.4]{Ek3}}]\label{Mazur-Ogus vanishing}
Suppose \(M\in \dc_{\MO}\). Then
\(
\HH^{i}\!\left(R_{1}\otimes^{\rmL}_{R}M\right)^{j}=0
\)
if \(i+j\neq 0\), and \(M\) is without
\(\Tot\)-torsion.
\end{lemma}

\begin{proof}
Each \(U_0^i\) is \(\Tot\)-acyclic. The defining orthogonality of \(M\),
together with \Cref{Hodge numbers in degree plus minus 1}, gives the
vanishing when \(i+j=\pm1\); \Cref{Basic Vanishing}(2) gives it in all
remaining nonzero total degrees.

The spectral sequence 
\(
\HH^{i}\!\left(R_{1}\otimes^{\rmL}_{R}M\right)^{j}
\Longrightarrow
\HH^{i+j}\!\left(k\otimes^{\rmL}_{W}\Tot(M)\right)
\)
therefore concentrates \(k\otimes_W^{\rmL}\Tot(M)\) in degree \(0\).
The triangle
\(
\Tot(M)
\xrightarrow{\;p\;}
\Tot(M)
\longrightarrow
k\otimes^{\rmL}_{W}\Tot(M)
\longrightarrow
\)
then shows that \(\Tot(M)\) is concentrated in degree \(0\) and that
\(\HH^0(\Tot(M))\) is torsion free.
\end{proof}

The link between this orthogonality and Hodge numbers comes from
the elementary domino \(U_0\). Maps to and from its shifts recover
the cohomology of the reduction; in total degrees \(1\) and \(-1\),
these maps lie in the diagonal heart itself.

\begin{proposition}\label{Hodge numbers in degree plus minus 1}
Let \(M\in D_c^b(R)\), and write \((-)^\vee=\Hom_k(-,k)\).
For all \(i,j\in\Z\), there are canonical isomorphisms
\[
\begin{aligned}
\HH^j(R_1\otimes_R^{\rmL}M)^i
 &\cong \Hom_{D(R)}(M,U_0[-j](1-i))^\vee\\
 &\cong \Hom_{D(R)}(U_0[-j-2](1-i),M).
\end{aligned}
\]
In particular, for \(M\in\dc\), the Hodge numbers with \(i+j=\pm1\) are
\[
\begin{aligned}
h^{1-j,j}(M)&=\dim_k\Hom_{\dc}(M,U_0[-j](j)),\\
h^{-1-j,j}(M)&=\dim_k\Hom_{\dc}(U_0[-j-2](j+2),M).
\end{aligned}
\]
\end{proposition}

These formulas are \cite[Equation~0.4.3]{Ek3} in our notation.

\begin{proof}
Regard \(R_1\) as a \((W,R)\)-bimodule, with \(W\) acting on the left.
Its graded dual \(L=\Hom_W(R_1,K/W)\), where
\(L^a=\Hom_W(R_1^{-a},K/W)\), is a graded left \(R\)-module by
\[
(r\lambda)(x)=\lambda(xr).
\]
Thus \(F,V,d\) on \(L\) are induced by right multiplication by
\(F,V,d\) on \(R_1\). The same formula defines the left \(W\)-action
on \(L\); the Hom is taken with respect to the left \(W\)-action on
\(R_1\). With this structure, \(L\cong U_0(1)\) as graded left
\(R\)-modules. The identification incorporates the sign
\(d_{U_0(1)}=-d_{U_0}\) from our grading-shift convention.
Tensor--Hom adjunction now gives
\[
\bigl(\HH^j(R_1\otimes_R^{\rmL}M)^i\bigr)^\vee
 \cong\Hom_{D(R)}(M,U_0[-j](1-i)).
\]
All these spaces are finite-dimensional over \(k\), so dualizing
proves the first isomorphism.

Compatibility of duality with reduction gives
\(\HH^{-j}(R_1\otimes_R^{\rmL}D(M))^{-i}
\cong(\HH^j(R_1\otimes_R^{\rmL}M)^i)^\vee\).
Apply the first isomorphism to \(D(M)\), then use biduality and
\(D(U_0)\cong U_0(2)[-2]\) to obtain
\[
\begin{aligned}
\HH^j(R_1\otimes_R^{\rmL}M)^i
 &\cong\Hom_{D(R)}(D(M),U_0[j](i+1))\\
 &\cong\Hom_{D(R)}(U_0[-j-2](1-i),M).
\end{aligned}
\]
Finally, when \(i=1-j\) or \(i=-1-j\), the corresponding shift of
\(U_0\) lies in \(\dc\). Taking dimensions gives the stated formulas.
\end{proof}

\subsection{Two fundamental filtrations}
\label{subsec:two-fundamental-filtrations}

The preceding results suggest two ways to separate the
\(\Tot\)-torsion of a diagonal complex. We can first extract the
maximal \(\Tot\)-\(0\)-torsion subobject, or first remove the maximal
\(\Tot\)-\(1\)-torsion quotient. Does the order matter? The two
filtrations differ only by a \(\Tot\)-acyclic term and have the same
Mazur--Ogus middle quotient. The following proposition makes the
comparison precise.

\begin{proposition}[{\cite[Theorem~I.4.5]{Ek3}}]\label{two fundamental filtrations}
For every \(M\in\dc\), there are two canonical functorial filtrations
\[
\begin{aligned}
0\subseteq F^1\subseteq F^2\subseteq F^3=M,\\
0\subseteq G^1\subseteq G^2\subseteq G^3=M,
\end{aligned}
\]
characterized as follows.

{\rm (i)} \(F^1\) is the largest \(\Tot\)-\(0\)-torsion subobject of
\(M\), and \(M/F^2\) is the largest \(\Tot\)-\(1\)-torsion quotient
of \(M/F^1\).

{\rm (ii)} \(M/G^2\) is the largest \(\Tot\)-\(1\)-torsion quotient
of \(M\), and \(G^1\) is the largest \(\Tot\)-\(0\)-torsion subobject
of \(G^2\).

Here a largest quotient means a quotient through which every quotient
of the specified kind factors. The filtrations satisfy:

{\rm (1)} \(G^1=F^1\cap G^2\) and \(F^2=F^1+G^2\).

{\rm (2)} There are canonical isomorphisms
\[
F^2/F^1\simeq G^2/G^1,\qquad F^1/G^1\simeq F^2/G^2.
\]
The first object is Mazur--Ogus, and the second is \(\Tot\)-acyclic.

{\rm (3)} There is a canonical decomposition
\[
F^2/G^1\simeq (F^2/F^1)\oplus(F^1/G^1).
\]
\end{proposition}

\begin{proof}
By \Cref{sub and quotient object of Tot-torsion objects}, the
\(\Tot\)-\(0\)-torsion objects are closed under quotients and extensions,
so their subobjects in \(M\) are closed under finite sums.
Noetherianity therefore gives a largest such subobject \(F^1\).
Similarly, the kernels of \(\Tot\)-\(1\)-torsion quotients are closed
under finite intersections, since this class is closed under subobjects
and extensions. These quotients are without finite torsion, so
\Cref{without finite torsion implies artinian}(2) gives a smallest
such kernel \(G^2\). Thus \(M/G^2\) is the largest
\(\Tot\)-\(1\)-torsion quotient of \(M\).
Applying the same constructions to \(M/F^1\) and \(G^2\) gives
\(F^2\) and \(G^1\). Functoriality follows from the defining
properties: morphisms preserve \(F^1\) because its images are
\(\Tot\)-\(0\)-torsion, and preserve \(G^2\) because the image in
any \(\Tot\)-\(1\)-torsion quotient is again \(\Tot\)-\(1\)-torsion.

Put \(C=F^1/(F^1\cap G^2)\). As a quotient of \(F^1\) and a
subobject of \(M/G^2\), it is both \(\Tot\)-\(0\)-torsion and
\(\Tot\)-\(1\)-torsion, hence \(\Tot\)-acyclic.
Thus \(\Tot(F^1\cap G^2)\simeq\Tot(F^1)\), so
\(F^1\cap G^2\subseteq G^1\). The reverse inclusion follows from
\(G^1\subseteq F^1\), giving \(G^1=F^1\cap G^2\).
Likewise, quotienting \(M/G^2\) by \(C\) does not change its
totalization, so \(M/(F^1+G^2)\) is \(\Tot\)-\(1\)-torsion.
The defining property of \(M/F^2\) gives \(F^2\subseteq F^1+G^2\).
Conversely, \(M/F^2\) is a \(\Tot\)-\(1\)-torsion quotient of
\(M\), so \(G^2\subseteq F^2\), proving \(F^2=F^1+G^2\).
This proves (1), the isomorphisms in (2), and the acyclicity of
\(F^1/G^1\).

By maximality of \(F^1\), the quotient \(M/F^1\) has no nonzero
\(\Tot\)-\(0\)-torsion subobject: its inverse image would be a
larger such subobject of \(M\). Similarly, \(G^2\) has no nonzero
\(\Tot\)-\(1\)-torsion quotient, since any such quotient
\(G^2/K\) would make \(M/K\) a larger \(\Tot\)-\(1\)-torsion
quotient of \(M\). Here both assertions use closure under extensions.
Consequently, for every \(\Tot\)-acyclic object \(N\), any map
\(N\to F^2/F^1\) or \(G^2/G^1\to N\) has zero image.
By \Cref{Mazur-Ogus objects}, the common quotient
\(F^2/F^1\simeq G^2/G^1\) is Mazur--Ogus, completing (2).

Finally, (1) shows that \(G^2/G^1\) and \(F^1/G^1\) have zero
intersection and sum to \(F^2/G^1\). Their inclusions give the
canonical decomposition in (3), using
\(G^2/G^1\simeq F^2/F^1\).
\end{proof}

After totalization, the difference between the two filtrations
disappears. Their successive pieces recover the torsion in degree
zero, the quotient by that torsion, and the degree-one cohomology.

\begin{corollary}[{\cite[Corollary~I.4.5.3]{Ek3}}]\label{total cohomology of the two fundamental filtrations}
For the filtrations of \Cref{two fundamental filtrations}, there are
canonical isomorphisms
\[
\begin{aligned}
\HH^0(\Tot(F^1))
  &\simeq \operatorname{tors}\HH^0(\Tot(M))
   \simeq \HH^0(\Tot(G^1)),\\
\HH^0(\Tot(F^2/F^1))
  &\simeq \HH^0(\Tot(M))/\operatorname{tors}
   \simeq \HH^0(\Tot(G^2/G^1)),\\
\HH^1(\Tot(F^3/F^2))
  &\simeq \HH^1(\Tot(M))
   \simeq \HH^1(\Tot(G^3/G^2)).
\end{aligned}
\]
Here \(\operatorname{tors}\) denotes the \(W\)-torsion submodule.
\end{corollary}

\section{\(F\)-Gauge Structures}
\label{sec:f-gauge-structures}

We now turn to Ekedahl's equivalence between coherent \(R\)-complexes
and \(F\)-gauges. We construct \(\FG^{\Tot}\) and its completed
derived inverse \(\widehat{(\rmL\rmt)^{\Tot}}\), keeping the maps,
relations, and completion explicit. The projective generators
\(G_r\) make the equivalence computable: we describe \(\rmt(G_r)\)
and construct projective resolutions of coherent gauges at a fixed
level. Finally, Morita equivalence expresses gauges as modules over
a ring and identifies \(\rmt\) with a tensor product functor.

\subsection{The t-structure \(\widetilde D_g\)}
\label{subsec:g-t-structure}

To relate totalization to a heart, we must account for its two
possible cohomological degrees. For \(M\in\dc\), the complex
\(\Tot(M)\) is concentrated in degrees \(0\) and \(1\); the
degree-one term prevents totalization from landing directly in the
ordinary heart. The second fundamental filtration provides the
corresponding adjustment of the t-structure. Define
\[
\cdots\subseteq M^i\subseteq M^{i+1}\subseteq\cdots\subseteq M,
\qquad
M^i=
\begin{cases}
0,&i<0,\\
G^2,&i=0,\\
M,&i>0.
\end{cases}
\]
This is an idempotent radical filtration on \(\dc\).

\begin{definition}\label{The t-structure: G}
Let \((\widetilde D_g^{\leqslant0},\widetilde D_g^{\geqslant0})\) be the t-structure on \(D_c^b(R)\) associated
with this radical filtration, and let
\[
\dc_g:=\widetilde D_g^{\leqslant0}\cap\widetilde D_g^{\geqslant0}
\]
be its heart, and put \(\widetilde D_g^{[m,n]}:=\widetilde D_g^{\geqslant m}\cap\widetilde D_g^{\leqslant n}\). We write \(\tau^g_{\leqslant i}\) and
\(\tau^g_{\geqslant i}\) for its truncation functors and set
\[
\HH_g^i(-)=
\tau^g_{\leqslant i}\tau^g_{\geqslant i}(-)[i].
\]
\end{definition}

This is the \(t\)-structure of \cite[Definition~II.1.1]{Ek3}, in our notation.

Under Ekedahl's equivalence at a fixed level \(\I\), this
\(t\)-structure corresponds to the standard \(t\)-structure on
\(D_c^b(\FGauge\td\I)\); see
\Cref{Ekedahl equivalence and the G t-structure}. Thus its heart
\(\dc_g\cap D_c^b(R\td\I)\) corresponds to coherent \(F\)-gauges
of level \(\I\), placed in cohomological degree zero.

\begin{proposition}[{\cite[Proposition~II.1.2]{Ek3}}, in our notation]\label{criteria for G-amplitude [m n]}
Let \(M\in D_c^b(R)\). Then \(M\in\widetilde D_g^{[m,n]}\) if
and only if \(\Tot(M)\) has amplitude \([m,n]\) and
\[
h^{i,j}:=\dim_k\HH^j(R_1\otimes_R^{\rmL}M)^i=0
\]
unless \(m-2\leqslant i+j\leqslant n\).
\end{proposition}

\begin{proof}
By definition, \(M\in\widetilde D_g^{[m,n]}\) precisely when
\(\tH^q(M)=0\) for \(q\notin[m-1,n]\),
\(\tH^{m-1}(M)\) is \(\Tot\)-\(1\)-torsion, and
\(\tH^n(M)\) has no nonzero \(\Tot\)-\(1\)-torsion quotient.
We compare these three conditions with the two bounds in the statement.

Suppose these conditions hold. The last condition and
\Cref{total cohomology of the two fundamental filtrations} give
\(\HH^1(\Tot(\tH^n(M)))=0\), while
\(\HH^0(\Tot(\tH^{m-1}(M)))=0\). Thus \Cref{Totsubquo}(1) places
\(\Tot(M)\) in degrees \([m,n]\).
Every morphism from \(\tH^n(M)\) to \(U_0[t](-t)\) is zero: its image
would be a \(\Tot\)-\(1\)-torsion quotient of \(\tH^n(M)\).
Hence \Cref{Hodge numbers in degree plus minus 1} gives
\(h^{i,j}(\tH^n(M))=0\) for \(i+j=1\).
By \Cref{Basic Vanishing}(2), the reduction of each
\(\tH^q(M)[-q]\) has Hodge numbers only in total degrees
\([q-1,q+1]\), with the upper degree absent for \(q=n\).
The diagonal cohomology filtration therefore gives
\(h^{i,j}(M)=0\) outside \([m-2,n]\).

Conversely, assume the stated bounds on total cohomology and Hodge
numbers. By \Cref{Totsubquo}(1), \(\tH^q(M)\) is \(\Tot\)-acyclic for
\(q<m-1\) or \(q>n\), and
\(\HH^0(\Tot(\tH^{m-1}(M)))=\HH^1(\Tot(\tH^n(M)))=0\).
The first step is to show that the acyclic terms outside \([m-1,n]\)
actually vanish.
If \(M=0\), there is nothing to prove; otherwise let \(a\) and \(b\)
be its lowest and highest nonzero diagonal cohomological degrees.

If \(b>n\), choose a simple quotient
\(\tH^b(M)\twoheadrightarrow U_0[t](-t)\) using
\Cref{simple Tot acyclic objects are U_0^i}.
The composite \(M\to \tH^b(M)[-b]\to U_0[t-b](-t)\) is nonzero,
as its \(b\)-th diagonal cohomology is the chosen quotient.
By \Cref{Hodge numbers in degree plus minus 1}, this gives
\(h^{t+1,b-t}(M)\neq0\), contradicting the hypothesis since \(b+1>n\).
Similarly, if \(a<m-1\), a simple subobject
\(U_0[t](-t)\hookrightarrow \tH^a(M)\) gives a nonzero composite
\(U_0[t-a](-t)\to \tH^a(M)[-a]\to M\).
The same formula gives \(h^{t+1,a-t-2}(M)\neq0\), contradicting the hypothesis since \(a-1<m-2\). Thus \(\tH^q(M)=0\) outside \([m-1,n]\), and
\(\tH^{m-1}(M)\) is \(\Tot\)-\(1\)-torsion.

It remains to exclude a nonzero \(\Tot\)-\(1\)-torsion quotient
\(Q\) of \(\tH^n(M)\). Since \(\HH^1(\Tot(\tH^n(M)))=0\), the long exact
sequence gives \(\HH^1(\Tot(Q))=0\); hence \(Q\) is
\(\Tot\)-acyclic. A simple quotient of \(Q\) would give a nonzero
map \(M\to \tH^n(M)[-n]\to U_0[t-n](-t)\), and therefore
\(h^{t+1,n-t}(M)\neq0\), again a contradiction.
All three defining conditions are satisfied.
\end{proof}

\begin{lemma}\label{simple Tot acyclic objects are U_0^i}
The \(\Tot\)-acyclic diagonal dominos form an abelian category in which every object has finite length.
Its simple objects are the \(U_0[i](-i)\).
\end{lemma}

Compare the proof of \cite[Lemma~II.1.2.1]{Ek3}.

\begin{proof}
Each \(U_0[i](-i)\) is \(\Tot\)-acyclic and has no proper nonzero
acyclic subobject, so it is simple. Closure under kernels and cokernels
follows from the long exact sequence of total cohomology. The category
is Noetherian because \(\dc\) is Noetherian, and Artinian by
\Cref{without finite torsion implies artinian}(2); thus it is a length
category.

Let \(M\) be simple. If every morphism between \(M\) and every shifted
\(U_0\) vanished, \Cref{Hodge numbers in degree plus minus 1} would
give \(h^{i,j}(M)=0\) whenever \(i+j=\pm1\). The Hodge--de Rham
spectral sequence and \(\Tot(M)=0\) would then force all Hodge groups,
and hence \(M\), to vanish. Therefore \(M\) has a nonzero morphism to
or from some \(U_0[i](-i)\); simplicity makes it an isomorphism.
\end{proof}

\subsection{\(F\)-gauge structures}
\label{subsec:definition-of-f-gauge-structures}

The gauge description keeps track of a graded family together with
maps in both directions. A further semilinear identification relates
its two limiting ends. We begin with this data before constructing
gauges from Raynaud modules.

\begin{definition}[{\cite[Definition~II.2.1]{Ek3}}, in our notation]\label{definition of F-gauge structure}
An \(F\)-gauge is a graded \(W\)-module
\(M=\bigoplus_{i\in\Z}M^i\), equipped with \(W\)-linear operators
\(\tF\) and \(\tV\) of degrees \(1\) and \(-1\), respectively, such
that \(\tF\tV=\tV\tF=p\), together with a \(\sigma\)-linear
isomorphism
\[
\tau:M^\infty\xrightarrow{\sim}M^{-\infty}.
\]
For an interval \(\I=[m,n]\), the gauge \(M\) has level \(\I\) if
\(\tF:M^i\to M^{i+1}\) is an isomorphism for \(i\geq n\) and
\(\tV:M^i\to M^{i-1}\) is an isomorphism for \(i\leq m\).

We denote the category of \(F\)-gauges over \(k\) by \(\FGauge\), and
its full subcategory of gauges of level \(\I\) by \(\FGauge\td\I\).
\end{definition}



We write a complex of \(F\)-gauges as \(M^{\bullet,\bullet}\), where
\(M^{i,j}\) has gauge degree \(i\) and cohomological degree \(j\).
Thus the differential raises the second degree, whereas \(\tF\) and
\(\tV\) shift the first:
\[
d:M^{i,j}\to M^{i,j+1},\qquad
\tF:M^{i,j}\to M^{i+1,j},\qquad
\tV:M^{i,j}\to M^{i-1,j}.
\]
\[
\begin{tikzcd}[
  row sep=large, column sep=large,
  execute at end picture={
    \foreach \r in {2,3,4} {
      \draw[blue!45!gray, thin, rounded corners=2pt]
        ([xshift=-5pt,yshift=14pt]\tikzcdmatrixname-\r-1.center)
        rectangle
        ([xshift=5pt,yshift=-14pt]\tikzcdmatrixname-\r-5.center);
      \node[anchor=west, text=blue!45!gray, font=\small]
        at ([xshift=10pt]\tikzcdmatrixname-\r-5.center)
        {\(F\)-gauge};
    }
  }
]
& \vdots & \vdots & \vdots & \\
\cdots
  \arrow[r, "\tF", shift left=0.6ex]
& M^{i-1,j+1}
  \arrow[u, "d"]
  \arrow[l, "\tV", shift left=0.6ex]
  \arrow[r, "\tF", shift left=0.6ex]
& M^{i,j+1}
  \arrow[u, "d"]
  \arrow[l, "\tV", shift left=0.6ex]
  \arrow[r, "\tF", shift left=0.6ex]
& M^{i+1,j+1}
  \arrow[u, "d"]
  \arrow[l, "\tV", shift left=0.6ex]
  \arrow[r, "\tF", shift left=0.6ex]
& \cdots \arrow[l, "\tV", shift left=0.6ex]
\\
\cdots
  \arrow[r, "\tF", shift left=0.6ex]
& M^{i-1,j}
  \arrow[u, "d"]
  \arrow[l, "\tV", shift left=0.6ex]
  \arrow[r, "\tF", shift left=0.6ex]
& M^{i,j}
  \arrow[u, "d"]
  \arrow[l, "\tV", shift left=0.6ex]
  \arrow[r, "\tF", shift left=0.6ex]
& M^{i+1,j}
  \arrow[u, "d"]
  \arrow[l, "\tV", shift left=0.6ex]
  \arrow[r, "\tF", shift left=0.6ex]
& \cdots \arrow[l, "\tV", shift left=0.6ex]
\\
\cdots
  \arrow[r, "\tF", shift left=0.6ex]
& M^{i-1,j-1}
  \arrow[u, "d"]
  \arrow[l, "\tV", shift left=0.6ex]
  \arrow[r, "\tF", shift left=0.6ex]
& M^{i,j-1}
  \arrow[u, "d"]
  \arrow[l, "\tV", shift left=0.6ex]
  \arrow[r, "\tF", shift left=0.6ex]
& M^{i+1,j-1}
  \arrow[u, "d"]
  \arrow[l, "\tV", shift left=0.6ex]
  \arrow[r, "\tF", shift left=0.6ex]
& \cdots \arrow[l, "\tV", shift left=0.6ex]
\\
& \vdots \arrow[u, "d"]
& \vdots \arrow[u, "d"]
& \vdots \arrow[u, "d"] &
\end{tikzcd}
\]
Each row is an \(F\)-gauge, and each column is a complex; the
vertical maps commute with the gauge operators.

\begin{definition}[The projective object \(G_r\)]\label{definition of G_r}
Let \(\I=[m,n]\). The evaluation functor
\[
T^r:\FGauge\td\I\longrightarrow\mathrm{Ab},\qquad T^r(M)=M^r,
\]
is represented by a projective \(F\)-gauge \(G_r\).
The object \(G_r\) is generated by an element \(\alpha_r\in G_r^r\), together with all of its iterated images under \(\widetilde{F}\), \(\widetilde{V}\), \(\tau\), and \(\tau^{-1}\). More explicitly,
\[
G_r^r=\bigoplus_{l=0}^{\infty}W\left(\tF^{\,r-m}\tau\tF^{\,n-r}\right)^l\alpha_r+\bigoplus_{l=1}^{\infty}W\left(\tV^{\,n-r}\tau^{-1}\tV^{\,r-m}\right)^l\alpha_r.
\]

For \(m\leqslant i<r\),
\[
G_r^i=\bigoplus_{l=1}^{\infty}W\tF^{\,i-m}\tau\tF^{\,n-r}\left(\tF^{\,r-m}\tau\tF^{\,n-r}\right)^{l-1}\alpha_r+\bigoplus_{l=0}^{\infty}W\tV^{\,r-i}\left(\tV^{\,n-r}\tau^{-1}\tV^{\,r-m}\right)^l\alpha_r.
\]

For \(r<i\leqslant n\),
\[
G_r^i=\bigoplus_{l=0}^{\infty}W\tF^{\,i-r}\left(\tF^{\,r-m}\tau\tF^{\,n-r}\right)^l\alpha_r+\bigoplus_{l=1}^{\infty}W\tV^{\,n-i}\tau^{-1}\tV^{\,r-m}\left(\tV^{\,n-r}\tau^{-1}\tV^{\,r-m}\right)^{l-1}\alpha_r.
\]
The operators \(\tF\) and \(\tV\) act on these generators as indicated
by their expressions. The identification
\(\tau:G_r^\infty\overset{\sim}{\longrightarrow}G_r^{-\infty}\)
completes the gauge structure. By construction, for every
\(M\in \FGauge\td{\I}\), \(\Hom_{\FGauge\td{\I}}(G_r,M)\cong M^r\):
a morphism \(G_r\to M\) is determined uniquely by the image of
\(\alpha_r\).
\end{definition}

This is the projective generator of
\cite[Chapter~II, Section~2(v)]{Ek3}.

For \(\I=[0,2]\), the projective \(G_1\) is generated by
\(\alpha_1\in G_1^1\), and the preceding formulas become
\[
G_1^1=\sum_{l=0}^{\infty}W\left(\tF\tau\tF\right)^l\alpha_1+\sum_{l=1}^{\infty}W\left(\tV\tau^{-1}\tV\right)^l\alpha_1.
\]

Moreover,
\[
G_1^0=\sum_{l=1}^{\infty}W\tau\tF\left(\tF\tau\tF\right)^{l-1}\alpha_1+\sum_{l=0}^{\infty}W\tV\left(\tV\tau^{-1}\tV\right)^l\alpha_1,
\]
and
\[
G_1^2=\sum_{l=0}^{\infty}W\tF\left(\tF\tau\tF\right)^l\alpha_1+\sum_{l=1}^{\infty}W\tau^{-1}\tV\left(\tV\tau^{-1}\tV\right)^{l-1}\alpha_1.
\]
\subsection{The functor \(\FG\)}
\label{subsec:functor-fg}

The functor \(\FG\) converts the horizontal grading of an \(R\)-module
into gauge degree. Its \(i\)-th gauge component is the totalization of
the elementary Nygaard modification at \(i\). The diagram below
shows how these modifications fit together: the change in the
differential occurs where the gauge degree crosses the module degree.

\begin{definition}[{\cite[Definition~II.3.1(i)]{Ek3}}, displayed with both gradings]\label{definition of FG}
For a graded left \(R\)-module \(M\), let \(\FG(M)\) be the following
complex of \(F\)-gauges:
\[
\begin{tikzcd}[
  row sep=large, column sep=large,
  execute at end picture={
    \foreach \r/\degree in {2/{i+1},3/i,4/{i-1}} {
      \draw[blue!55!gray, thin, rounded corners=2pt]
        ([xshift=-6pt,yshift=15pt]\tikzcdmatrixname-\r-1.center)
        rectangle
        ([xshift=6pt,yshift=-15pt]\tikzcdmatrixname-\r-5.center);
      \node[anchor=west, text=blue!55!gray, font=\small]
        at ([xshift=12pt]\tikzcdmatrixname-\r-5.center)
        {complex degree \(\degree\)};
    }
    \foreach \c/\degree in {2/{i-1},3/i,4/{i+1}} {
      \draw[orange!75!black, thin, rounded corners=2pt]
        ([xshift=-27pt,yshift=10pt]\tikzcdmatrixname-1-\c.center)
        rectangle
        ([xshift=27pt,yshift=-10pt]\tikzcdmatrixname-5-\c.center);
      \node[anchor=north, text=orange!75!black, font=\small,
            align=center]
        at ([yshift=-16pt]\tikzcdmatrixname-5-\c.center)
        {\(F\)-gauge\\degree \(\degree\)};
    }
  }
]
& \vdots & \vdots & \vdots & \\
\cdots \arrow[r,"p",shift left=0.6ex]
& M^{i+1} \arrow[u,"d"] \arrow[l,"\mathrm{id}",shift left=0.6ex] \arrow[r,"p",shift left=0.6ex]
& M^{i+1} \arrow[u,"d"] \arrow[l,"\mathrm{id}",shift left=0.6ex] \arrow[r,"p",shift left=0.6ex]
& M^{i+1} \arrow[u,"d"] \arrow[l,"\mathrm{id}",shift left=0.6ex] \arrow[r,"F",shift left=0.6ex]
& \cdots \arrow[l,"V",shift left=0.6ex] \\
\cdots \arrow[r,"p",shift left=0.6ex]
& M^i \arrow[u,"d"] \arrow[l,"\mathrm{id}",shift left=0.6ex] \arrow[r,"p",shift left=0.6ex]
& M^i \arrow[u,"d"] \arrow[l,"\mathrm{id}",shift left=0.6ex] \arrow[r,"F",shift left=0.6ex]
& \sigma_*M^i \arrow[u,"dV"] \arrow[l,"V",shift left=0.6ex] \arrow[r,"\mathrm{id}",shift left=0.6ex]
& \cdots \arrow[l,"p",shift left=0.6ex] \\
\cdots \arrow[r,"p",shift left=0.6ex]
& M^{i-1} \arrow[u,"d"] \arrow[l,"\mathrm{id}",shift left=0.6ex] \arrow[r,"F",shift left=0.6ex]
& \sigma_*M^{i-1} \arrow[u,"dV"] \arrow[l,"V",shift left=0.6ex] \arrow[r,"\mathrm{id}",shift left=0.6ex]
& \sigma_*M^{i-1} \arrow[u,"d"] \arrow[l,"p",shift left=0.6ex] \arrow[r,"\mathrm{id}",shift left=0.6ex]
& \cdots \arrow[l,"p",shift left=0.6ex] \\
& \vdots \arrow[u,"dV"] & \vdots \arrow[u,"d"] & \vdots \arrow[u,"d"] &
\end{tikzcd}
\]
Here the rightward arrows are \(\tF\), the leftward arrows are
\(\tV\), and the upward arrows are the cohomological differential.
More explicitly:

\smallskip
\noindent{\rm (1)} In gauge degree \(i\) and cohomological degree \(j\), set
\[
\FG(M)^{i,j}=
\begin{cases}
M^j,&i\leqslant j,\\
\sigma_*M^j,&i>j.
\end{cases}
\]

\noindent{\rm (2)} In each row \(j\), the maps
\(\tF:\FG(M)^{i,j}\to\FG(M)^{i+1,j}\) and
\(\tV:\FG(M)^{i+1,j}\to\FG(M)^{i,j}\) are
\[
(\tF,\tV)=
\begin{cases}
(p,\mathrm{id}),&i<j,\\
(F,V),&i=j,\\
(\mathrm{id},p),&i>j.
\end{cases}
\]

\noindent{\rm (3)} In each column \(i\), the complex \(\FG(M)^i\) is
\[
\cdots\longrightarrow\sigma_*M^{i-2}
\xrightarrow{d}\sigma_*M^{i-1}
\xrightarrow{dV}M^i\xrightarrow{d}M^{i+1}
\longrightarrow\cdots.
\]
Thus \(\FG(M)^i=\Tot(M(\varphi^i))\), where
\(\varphi^i(a)=-\delta_{i,a}\), as in \Cref{Nygaard modification}.
The vertical maps commute with \(\tF\) and \(\tV\) by the Raynaud relations.

\smallskip
\noindent{\rm (4)} The limiting complexes are
\(\FG(M)^\infty=\sigma_*\Tot(M)\) and
\(\FG(M)^{-\infty}=\Tot(M)\).
The operator \(\tau:\FG(M)^\infty\to\FG(M)^{-\infty}\) is the
canonical semilinear identification, given by the identity on the
underlying complexes.
\end{definition}

This construction defines the functor denoted \(\mathrm S\) in
\cite{Ek3}:
\[
\FG:R\text{-}\mathrm{Mod}\longrightarrow\Ch(\FGauge).
\]
Applying it termwise and deriving gives
\[
\FG:\Ch(R)\longrightarrow\Ch\Ch(\FGauge),
\qquad
\FG:D(R)\longrightarrow D(\Ch(\FGauge)).
\]


Passing from double complexes to the derived category requires some
care with totalization. For a general abelian category \(A\),
totalization \(\Tot:\Ch\Ch(A)\to\Ch(A)\) need not preserve
quasi-isomorphisms and therefore does not define a functor
\(D(\Ch(A))\to D(A)\). Fix a finite interval \(\I=[m,n]\), and let
\(\Ch_{\I}(A)\) denote the category of complexes in \(A\) concentrated
in degrees belonging to \(\I\). For a double complex in
\(\Ch(\Ch_{\I}(A))\), the inner complex is thus uniformly bounded.
Totalization preserves quasi-isomorphisms in the outer direction
and therefore induces a functor
\[
\Tot:D(\Ch_{\I}(A))\longrightarrow D(A).
\]

\begin{definition}[{\cite[Definition~II.3.1(ii)]{Ek3}}]\label{definition of FG^Tot}
For a fixed interval \(\I=[m,n]\), the functor \(\FG\) restricts to
\[
\FG:R\td\mathrm{Mod}\td\I
\longrightarrow\Ch_{\I}(\FGauge\td\I).
\]
Its totalization therefore induces
\[
\FG^{\Tot}:D(R\td\mathrm{Mod}\td\I)
\longrightarrow D(\FGauge\td\I).
\]
\end{definition}



The gauge operators retain the Hodge filtration data. The following
diagram makes this relation precise by identifying their cones with
the two Hodge truncations and their common graded contribution.

\begin{proposition}[{\cite[Proposition~II.3.4]{Ek3}}, with the comparison maps made explicit]\label{Nygaard filtration and Hodge filtration}
Let \(M\in D(R\td\mathrm{Mod}\td\I)\).
There is a natural diagram in \(D(W)\), commuting up to signs, whose rows and columns are distinguished triangles:
\[
\begin{tikzcd}[column sep=small,row sep=large]
{\FG^{\Tot}(M)^{i,\bullet}}
 \arrow[r,"\tF"] \arrow[d,"\tV"'] &
{\FG^{\Tot}(M)^{i+1,\bullet}}
 \arrow[r] \arrow[d,"\tV"'] &
{(t_{\geqslant i+1}\Hod(M))^\bullet}
 \arrow[r,"+1"] \arrow[d] & {}\\
{\FG^{\Tot}(M)^{i-1,\bullet}}
 \arrow[r,"\tF"] \arrow[d] &
{\FG^{\Tot}(M)^{i,\bullet}}
 \arrow[r] \arrow[d] &
{(t_{\geqslant i}\Hod(M))^\bullet}
 \arrow[r,"+1"] \arrow[d] & {}\\
{\sigma_*(\tau_{\leqslant i-1}\Hod(M))^\bullet}
 \arrow[r] \arrow[d,"+1"'] &
{\sigma_*(\tau_{\leqslant i}\Hod(M))^\bullet}
 \arrow[r] \arrow[d,"+1"'] &
{(R_1\otimes_R^{\rmL}M)^{i,\bullet}[-i]}
 \arrow[r,"+1"] \arrow[d,"+1"'] & {}\\
{} & {} & {} & {}
\end{tikzcd}
\]
The rightmost triangle is the first triangle of
\eqref{Hodge filtration triangles}; the bottom triangle is
\(\sigma_*\) applied to its second triangle.
\end{proposition}

The \(k\)-complexes in this diagram are viewed as \(W\)-complexes
by restriction of scalars. The bottom-right term
\((R_1\otimes_R^{\rmL}M)^{i,\bullet}[-i]\) records the contribution
of the \(i\)-th column to totalization. In the \(i\)-th column,
complex degree \(j\) becomes degree \(i+j\) after totalization;
this accounts for the shift \([-i]\).

\begin{proof}
We make the comparison maps of \cite[Proposition~II.3.4]{Ek3} explicit.
Choose a common semi-free resolution \(P\to M\) in complexes of
graded \(R\)-modules. Its terms are \(p\)-torsion-free and flat over
\(R\), so both \(R_1\otimes_RP\) and \(Z_1\otimes_RP\) compute
derived tensor products. All maps below are defined before
totalization and commute with the differential of \(P\).
It therefore suffices to describe them on a graded free
\(R\)-module \(L\), a term of \(P\).

Since \(FV=VF=p\), the maps \(F,V\) on \(L\) are injective.
Consequently the two gauge maps are termwise injective, and their
ordinary quotient complexes represent their cones. Put
\[
Q_{F,i}(L):=\operatorname{coker}
 \bigl(\tF:\FG(L)^{i-1,\bullet}\to \FG(L)^{i,\bullet}\bigr),\qquad
Q_{V,i}(L):=\operatorname{coker}
 \bigl(\tV:\FG(L)^{i+1,\bullet}\to \FG(L)^{i,\bullet}\bigr).
\]
The definition of the gauge maps gives
\[
\begin{aligned}
Q_{F,i}(L)&=
 \bigl(0\longrightarrow \sigma_*L^{i-1}/F L^{i-1}
 \xrightarrow{dV} L^i/pL^i\xrightarrow{d}L^{i+1}/pL^{i+1}
 \longrightarrow\cdots\bigr),\\
Q_{V,i}(L)&=
 \bigl(\cdots\xrightarrow{d}\sigma_*(L^{i-1}/pL^{i-1})
 \xrightarrow{dV}L^i/V L^i\longrightarrow0\bigr).
\end{aligned}
\]
Here the images of \(F:L^{i-1}\to\sigma_*L^{i-1}\) and
\(V:\sigma_*L^i\to L^i\) are understood with their indicated
scalar actions. The terms retain their original horizontal degrees.

Set \(A_L=R_1\otimes_RL\) and \(Z_L=Z_1\otimes_RL\).
The first comparison sends the degree \(i-1\) term to zero and is
the quotient map in every degree \(a\geqslant i\):
\[
\begin{tikzcd}[column sep=large]
0 \arrow[r] &
{\sigma_*L^{i-1}/F L^{i-1}} \arrow[r,"dV"] \arrow[d] &
{L^i/pL^i} \arrow[r,"d"] \arrow[d] &
{L^{i+1}/pL^{i+1}} \arrow[r] \arrow[d] & \cdots\\
0 \arrow[r] & 0 \arrow[r] &
{A_L^i} \arrow[r,"d"] &
{A_L^{i+1}} \arrow[r] & \cdots .
\end{tikzcd}
\]
Indeed, \(A_L^a=L^a/(VL^a+dVL^{a-1})\), so the quotient maps
are well defined and kill the displayed \(dV\).
After totalization this is a map
\[
Q_{F,i}(L)\longrightarrow t_{\geqslant i}\Hod(L).
\]

For the second comparison, recall that
\(Z_L^i=\ker(d:A_L^i\to A_L^{i+1})\), since \(L\) is flat.
The presentation of \(\sigma_*Z_1\) identifies
\(\sigma_*Z_L^i\cong L^i/(VL^i+dV^2L^{i-1})\), with \(F\otimes x\) corresponding to the class of \(x\). Consider the following comparison:
\[
\begin{tikzcd}[column sep=large]
\cdots \arrow[r,"d"] &
{\sigma_*(L^{i-1}/pL^{i-1})}
 \arrow[r,"dV"] \arrow[d] &
{L^i/VL^i} \arrow[r] \arrow[d,"F"] & 0\\
\cdots \arrow[r,"d"] &
{\sigma_*A_L^{i-1}} \arrow[r,"\sigma_*\partial"] &
{\sigma_*Z_L^i} \arrow[r] & 0.
\end{tikzcd}
\]
Under this identification, \(\sigma_*\partial\) sends \([y]\)
to \([dVy]\), since \(d\otimes y=F\otimes dVy\).
The vertical maps below degree \(i\) are the natural quotient maps.
The right vertical map sends \([x]\) to \(F\otimes x\);
it is well defined because \(F\in Z_1\) and \(FV=0\) in \(R_1\).
Equivalently, it is the quotient map
\(L^i/VL^i\to L^i/(VL^i+dV^2L^{i-1})\).
The last square therefore commutes.
This gives
\[
Q_{V,i}(L)\longrightarrow\sigma_*\tau_{\leqslant i}\Hod(L).
\]

To check that these maps are quasi-isomorphisms, it suffices to take
\(L=R\), since the constructions commute with direct sums and grading
shifts. The four cases with nonzero kernels are shown below;
Frobenius twists are suppressed, and \(\pi\) denotes the natural
quotient map. The two nonzero positions have complex degrees \(0,1\).
\begin{center}
\begin{tabular}{@{}c@{\hspace{24pt}}c@{}}
\(\tF,\quad i\leqslant0\) & \(\tF,\quad i=1\)\\[7pt]
\begin{tikzcd}[column sep=10pt,row sep=15pt,ampersand replacement=\&]
0\arrow[r]\& {R^0/p}\arrow[r,"d"]\arrow[d,"\pi"']\& {R^1/p}\arrow[r]\arrow[d,"\pi"]\&0\\
0\arrow[r]\& R_1^0\arrow[r,"d"]\& R_1^1\arrow[r]\&0
\end{tikzcd}
&
\begin{tikzcd}[column sep=10pt,row sep=15pt,ampersand replacement=\&]
0\arrow[r]\& {R^0/FR^0}\arrow[r,"dV"]\arrow[d]\& {R^1/p}\arrow[r]\arrow[d,"\pi"]\&0\\
0\arrow[r]\&0\arrow[r]\&R_1^1\arrow[r]\&0
\end{tikzcd}
\\[32pt]
\(\tV,\quad i=1\)&\(\tV,\quad i\geqslant2\)\\[7pt]
\begin{tikzcd}[column sep=10pt,row sep=15pt,ampersand replacement=\&]
0\arrow[r]\& {R^0/p}\arrow[r,"dV"]\arrow[d,"\pi"']\& {R^1/p}\arrow[r]\arrow[d,"F"]\&0\\
0\arrow[r]\&R_1^0\arrow[r,"d"]\&R_1^1\arrow[r]\&0
\end{tikzcd}
&
\begin{tikzcd}[column sep=10pt,row sep=15pt,ampersand replacement=\&]
0\arrow[r]\&{R^0/p}\arrow[r,"d"]\arrow[d,"\pi"']\&{R^1/p}\arrow[r]\arrow[d,"\pi"]\&0\\
0\arrow[r]\&R_1^0\arrow[r,"d"]\&R_1^1\arrow[r]\&0
\end{tikzcd}
\end{tabular}
\end{center}
Here \(R^a/p=R^a/pR^a\), and the arrow labelled \(F\) sends
\([x]\) to \([Fx]\); in the third diagram we use
\(VR^1=pR^1\) and \(Z_1^1=R_1^1\).
Computing the kernels in these four diagrams shows that the
comparison maps are quasi-isomorphisms.
The remaining cases are isomorphisms, including
\(k[F]\xrightarrow{F}\sigma_*(Fk[F])\) for the
\(\tV\)-comparison at \(i=0\).
These explicit maps therefore identify the two cones with the Hodge
truncations. Applying the calculation to the terms of \(P\) gives
the comparisons for bounded double complexes. For a semi-free
resolution the general case follows by its filtration by free
cells and filtered colimits; all displayed constructions use
direct sums and are compatible with that filtration.
We obtain the natural distinguished triangles
\[
\begin{aligned}
\FG^{\Tot}(M)^{i-1}\xrightarrow{\tF}\FG^{\Tot}(M)^i
 &\longrightarrow t_{\geqslant i}\Hod(M)
 \longrightarrow\FG^{\Tot}(M)^{i-1}[1],\\
\FG^{\Tot}(M)^{i+1}\xrightarrow{\tV}\FG^{\Tot}(M)^i
 &\longrightarrow\sigma_*\tau_{\leqslant i}\Hod(M)
 \longrightarrow\FG^{\Tot}(M)^{i+1}[1].
\end{aligned}
\]

Finally, the upper left square commutes because
\(\tF\tV=\tV\tF=p\).
The quotient maps just constructed are compatible with this square:
on the \(F\)-quotients the induced map is the inclusion of the
horizontal tail starting at \(i+1\) into the tail starting at \(i\);
on the \(V\)-quotients it is the map between consecutive truncated
complexes induced by \(Z_1\hookrightarrow R_1\).
Thus their cones are the rightmost and bottom triangles of
\eqref{Hodge filtration triangles}, the latter after applying
\(\sigma_*\). Cartier identifies their common final object with
\((R_1\otimes_R^{\rmL}M)^{i,\bullet}[-i]\).
The resulting \(3\times3\) diagram is the one in the statement;
the usual cone convention accounts for the possible signs.
\end{proof}


\subsection{The functor \(\rmt\).}
\label{subsec:functor-rmt}

Fix an interval \(\I=[m,n]\).

To construct an inverse to \(\FG^{\Tot}\), we begin with an adjunction
on the abelian level and then pass to the derived category.
Ekedahl's functor \(\rmt\) is the left adjoint of
\[
Z^0\circ\FG^{\Tot}:
\Ch(R\td\mathrm{Mod}\td\I)\longrightarrow\FGauge\td\I.
\]
A presentation by generators and relations will make its left derived
functor \(\rmL\rmt\) accessible to calculation.

Here \(\Ch(R\td\mathrm{Mod}\td\I)\) is the abelian category of
complexes of \(R\)-modules of level \(\I\). The functor above sends
\(X\) to the \(F\)-gauge
\[
\ker\bigl(\FG^{\Tot}(X)^0\xrightarrow{D}\FG^{\Tot}(X)^1\bigr).
\]
It commutes with inverse limits and therefore admits a left adjoint
\[
\rmt:\FGauge\td\I\longrightarrow\Ch(R\td\mathrm{Mod}\td\I).
\]
We use \(\delta\) for the cohomological differential of an
\(R\)-complex. The relations below express the requirement that a
gauge map into totalization land in degree-zero cycles.

\begin{definition}\label{rmt}
Let \(M\in\FGauge\td\I\). The \(R\)-complex \(\rmt(M)\) is generated
by symbols \(\beta_i\otimes x\), with \(i\in\I\) and \(x\in M\), where
\(\beta_i\otimes x\) has internal degree \(i\) and cohomological degree
\(-i\):
    {\small \[\begin{tikzcd}
	{R^0(\beta_m\otimes m)} & {R^1(\beta_m\otimes m)} \\
	& {R^0(\beta_{m+1}\otimes m)} & {...} \\
	&& {...} & {...} \\
	&&& {R^0(\beta_{n-1}\otimes m)} & {R^1(\beta_{n-1}\otimes m)} \\
	&&&& {R^0(\beta_{n}\otimes m).} \\
	&&&& {}
	\arrow["d", from=1-1, to=1-2]
	\arrow["\delta", from=2-2, to=1-2]
	\arrow["d", from=2-2, to=2-3]
	\arrow["\delta", from=3-3, to=2-3]
	\arrow["d", from=3-3, to=3-4]
	\arrow["\delta", from=4-4, to=3-4]
	\arrow["d", from=4-4, to=4-5]
	\arrow["\delta", from=5-5, to=4-5]
\end{tikzcd}\]}
In cohomological degree \(-n\), the complex is generated over \(R^0\);
there is no component of bidegree \((n+1,-n)\).

The generators are subject to the following relations.

    {\rm (1)} For any \(\lambda\in W\), \(m_1,m_2\in M^j\), 
    \[\begin{aligned}
        \beta_i \otimes (\lambda m_1+m_2) =\lambda \beta_i\otimes m_1+ \beta_i \otimes m_2, \quad \text{if}\; j\leqslant i, \\
        \beta_i \otimes (\lambda m_1+m_2) =\lambda^\sigma \beta_i\otimes m_1+ \beta_i \otimes m_2, \quad \text{if}\; j\geqslant i+1.
    \end{aligned}\]

    {\rm (2)} If \(m_j\in M^j\) then 
    \[\begin{aligned}
        &\delta(\beta_{i}\otimes m_j)=0,\quad  \text{if}\;i=m;\\
        &(-1)^{i}\delta(\beta_{i}\otimes m_j)+d(\beta_{i-1}\otimes m_j)=0,\quad  \text{if}\; i\in [m+1,n],i\neq j,\\
        &(-1)^{j}\delta(\beta_{j}\otimes m_j)+dV(\beta_{j-1}\otimes m_j)=0, \quad  \text{if}\; j=i.
    \end{aligned}\] 

    {\rm (3)} If \(m_j\in M^j\) then 
    \[\beta_i\otimes \tF m_j =\begin{cases}
            \beta_i\otimes m_j,\;&\text{if}\;j>i\\
            F(\beta_i\otimes m_j),\;&\text{if}\;j=i\\
            p(\beta_i\otimes m_j),\;&\text{if}\;j<i.
        \end{cases}\]
    \[\beta_i\otimes \tV m_j =\begin{cases}
            \beta_i\otimes m_j,\;&\text{if}\;j<i+1\\
            V(\beta_i\otimes m_j),\;&\text{if}\;j=i+1\\
            p(\beta_i\otimes m_j),\;&\text{if}\;j>i+1.
        \end{cases}\]

{\rm (4)} If \(\tau(m_{\infty})=m_{-\infty}\), impose
    \[\beta_i\otimes m_{\infty}=\beta_i\otimes m_{-\infty}. \]
\end{definition}

This makes the construction of \cite[Definition~II.4.1]{Ek3}
explicit, with our corrected twist convention.

Relation (3) makes the limiting symbols in (4) well defined. Namely,
\(\beta_i\otimes m_\infty\) is represented by
\(\beta_i\otimes m_j\) for any \(j>i\) and any lift
\(m_j\in M^j\) of \(m_\infty\in M^\infty\); relation (3) makes this
independent of both choices. Similarly,
\(\beta_i\otimes m_{-\infty}\) is represented by
\(\beta_i\otimes m_j\) for \(j<i\) and a lift of
\(m_{-\infty}\in M^{-\infty}\).

\begin{lemma}\label{degree (n -n) component of t(M)}
For \(M\in\FGauge\td\I\), the operator \(F\) is bijective on the
component of \(\rmt(M)\) in bidegree \((n,-n)\).
\end{lemma}

\begin{proof}
By relation \Cref{rmt}(3), it suffices to consider a generator
\(\beta_n\otimes x\) with \(x\in M^j\) and \(m\leq j\leq n\). Let
\(\bar x\in M^{-\infty}\) be its image. Choose
\(a_\infty\in M^\infty\) with \(\tau(a_\infty)=\bar x\), and represent
\(a_\infty\) by \(a_n\in M^n\). Relations (3) and (4) give
\[
F(\beta_n\otimes a_n)
=\beta_n\otimes\tF a_n
=\beta_n\otimes a_\infty
=\beta_n\otimes\bar x
=\beta_n\otimes x.
\]
Thus \(F\) is surjective. Repeating the construction with
\(\tau^{-1}\) gives its inverse, so \(F\) is bijective.
\end{proof}
\begin{lemma}\label{degree (n -n+1) component of t(M)}
For \(M\in\FGauge\td\I\), the operator \(F\) is bijective on
\(\rmt(M)^{n,-n+1}\).
\end{lemma}

\begin{proof}
If \(m=n\), this component is zero. Suppose \(m<n\).
Left multiplication by \(F\) is bijective on \(R^1\): in the unique
\(F\)-\(V\) expansions, it shifts the string
\(\ldots,dV^2,dV,d,Fd,F^2d,\ldots\) and applies \(\sigma\) to
coefficients. Its inverse \(S\) commutes with right multiplication
by \(R^0\), so \(S(uq)=S(u)q\) for \(u\in R^1\), \(q\in R^0\).
Define
\[
S\bigl(u(\beta_{n-1}\otimes x)\bigr)
:=S(u)(\beta_{n-1}\otimes x).
\]
Right \(R^0\)-linearity shows that this respects the scalar and gauge
relations \Cref{rmt}(1),(3): a relation with coefficients \(u\) and
\(uq\) is sent to the same relation with coefficients \(S(u)\) and
\(S(u)q\). It also respects the endpoint identifications in (4).
Relation (2) specifies the differential, compatibly with these
relations, rather than imposing further relations within this
component. Thus \(S\) descends to \(\rmt(M)^{n,-n+1}\), where
\(FS=SF=1\).
\end{proof}

Together with \Cref{degree (n -n) component of t(M)}, this proves
that \(\rmt(M)\in\Ch(R\td\mathrm{Mod}\td\I)\).


\begin{proposition}[{Adjunction}]\label{adjunction between FG and t}
For \(M\in\Ch(R\td\mathrm{Mod}\td\I)\) and
\(N\in\FGauge\td\I\), there are natural isomorphisms
\[
\Hom_{\FGauge}(N,Z^0\FG^{\Tot}(M))
\cong
\Hom_{\Ch(\FGauge)}(N,\FG^{\Tot}(M))
\cong
\Hom_{\Ch(R)}(\rmt(N),M).
\]
\end{proposition}

Compare \cite[Chapter~II, Section~4]{Ek3}.

\begin{proof}
The first isomorphism regards \(N\) as a complex concentrated in degree
zero. For the second, suppose that \(x\in N^j\) maps to a cycle in
\(\FG^{\Tot}(M)^j\), represented by a compatible sequence
\((m_i)_i\), where
\[
m_i\in\sigma_*M^{i,-i}\quad(i<j),
\qquad
m_i\in M^{i,-i}\quad(i\geq j).
\]
Compatibility with the gauge operators means
\[\begin{aligned}
x&\longmapsto(\ldots,m_i,\ldots),\\
\tF x&\longmapsto(\ldots,m_{j-1},F(m_j),pm_{j+1},\ldots),\\
\tV x&\longmapsto(\ldots,pm_{j-2},Vm_{j-1},m_j,\ldots).
\end{aligned}\]

The cycle condition is
\[\begin{cases}
\delta(m_i)=0,&i=m,\\
dm_{i-1}+(-1)^i\delta(m_i)=0,
    &i\in[m+1,n],\ i\neq j,\\
dVm_{j-1}+(-1)^j\delta(m_j)=0,&i=j\neq m.
\end{cases}\]

Thus sending \(\beta_i\otimes x\) to \(m_i\) respects the relations
in \Cref{rmt} and defines a morphism
\(\rmt(N)\to M\). The construction is reversible and natural, which
proves the second isomorphism. Under it, the adjunction unit is
\[
x\longmapsto\sum_{i\in\I}\beta_i\otimes x.
\]
\end{proof}

\subsection{The functors \(\rmt^{\Tot}\) and \((\rmL\rmt)^{\Tot}\)}
\label{subsec:total-and-derived-rmt}


The adjunction extends to complexes by applying \(\rmt\) termwise.
For a complex of \(F\)-gauges, this produces a double complex whose
internal cohomological direction is supported in \(-\I\). Since
this direction is bounded, totalization is unambiguous and defines
\[
\rmt^{\Tot}:=\Tot\circ\rmt:
\Ch(\FGauge\td\I)\longrightarrow\Ch(R\td\mathrm{Mod}\td\I).
\]

\begin{proposition}\label{adjunction between FG^Tot and t^Tot}
The functor \(\rmt^{\Tot}\) is left adjoint to
\[
\FG^{\Tot}:\Ch(R\td\mathrm{Mod}\td\I)
\longrightarrow\Ch(\FGauge\td\I).
\]
\end{proposition}

Compare the adjunction in \cite[Equation~II.4.3]{Ek3}.

\begin{proof}
Both categories are locally presentable. Since \(\I\) is finite,
\(\FG^{\Tot}\) is computed componentwise by finite direct sums and
scalar twists, so it preserves small limits and colimits. The adjoint
functor theorem therefore gives a left adjoint \(\rmT\).

Totalizing the units \(N^q\to\FG^{\Tot}(\rmt(N^q))\) of
\Cref{adjunction between FG and t} gives a natural map
\[
N^\bullet\longrightarrow
\Tot\bigl(\FG^{\Tot}(\rmt(N^\bullet))\bigr)
\cong\FG^{\Tot}(\rmt^{\Tot}(N^\bullet)).
\]
The identification reorders the finite sums, with the signs required
when interchanging the complex directions. By adjunction, we obtain
\(\alpha:\rmT\to\rmt^{\Tot}\). The preceding adjunction and
compatibility with shifts show that \(\alpha\) is an isomorphism
on complexes concentrated in one degree.

For bounded complexes, we proceed by induction on the length.
Choose \(a\) just above the lowest nonzero degree.
The stupid truncation sequence
\(0\to t_{\geqslant a}N^\bullet\to N^\bullet\to
 t_{<a}N^\bullet\to0\) is split in every degree, so
\(\rmt^{\Tot}\) preserves its exactness. Right exactness of
\(\rmT\) gives a commutative diagram with exact rows:
\[
\begin{tikzcd}[column sep=small]
{} & \rmT(t_{\geqslant a}N^\bullet)
 \arrow[r] \arrow[d,"\alpha"'] &
\rmT(N^\bullet) \arrow[r] \arrow[d,"\alpha"] &
\rmT(t_{<a}N^\bullet) \arrow[r] \arrow[d,"\alpha"] & 0\\
0 \arrow[r] & \rmt^{\Tot}(t_{\geqslant a}N^\bullet)
 \arrow[r] &
\rmt^{\Tot}(N^\bullet) \arrow[r] &
\rmt^{\Tot}(t_{<a}N^\bullet) \arrow[r] & 0.
\end{tikzcd}
\]
The outer vertical maps are isomorphisms by induction.
Since the bottom-left map is injective, so is the top-left map;
the short five lemma then shows that the middle vertical map is an
isomorphism.

Every complex is the filtered colimit of its bounded subcomplexes
\(t_{\geqslant a}\tau_{\leqslant b}N^\bullet\) as \(a\) decreases
and \(b\) increases, with \(\tau_{\leqslant b}\) denoting ordinary
truncation. Both functors preserve colimits: \(\rmT\) is a left
adjoint, while \(\rmt^{\Tot}\) is the totalization of the left
adjoint \(\rmt\) applied termwise. Thus \(\alpha\) is an isomorphism
for every complex, proving the adjunction.
\end{proof}


The next step is to derive this adjunction. Since \(\rmt\) is a left
adjoint, it is right exact. K-projective resolutions give its left
derived functor
\[
\rmL\rmt:D(\FGauge\td\I)\longrightarrow
D\bigl(\Ch(R\td\mathrm{Mod}\td\I)\bigr).
\]
The target retains two cohomological degrees. It is the derived
category of the abelian category of \(R\)-complexes, so the resolution
degree is distinct from the complex degree within each value of
\(\rmt\). The latter is supported in \(-\I\), so totalization on
bounded-above objects defines
\[
(\rmL\rmt)^{\Tot}:=\Tot\circ\rmL\rmt:
D^-(\FGauge\td\I)\longrightarrow D^-(R\td\mathrm{Mod}\td\I).
\]
If \(P^\bullet\to M^\bullet\) is a projective resolution, then
\((\rmL\rmt)^{\Tot}(M^\bullet)\) is represented by
\(\rmt^{\Tot}(P^\bullet)\).

\begin{definition}\label{definition of (Lt)^Tot}
For a fixed interval \(\I\), we use the notation
\[
\FG^{\Tot}:=\Tot\circ\FG:
D(R\td\mathrm{Mod}\td\I)\longrightarrow D(\FGauge\td\I)
\]
and
\[
(\rmL\rmt)^{\Tot}:=\Tot\circ\rmL\rmt:
D^-(\FGauge\td\I)\longrightarrow D^-(R\td\mathrm{Mod}\td\I).
\]
\end{definition}

\begin{corollary}\label{derived adjunction between FG^Tot and t^Tot}
The functor \((\rmL\rmt)^{\Tot}\) is left adjoint to
\(\FG^{\Tot}\) on the bounded-above derived categories.
\end{corollary}

This is the derived form of the adjunction in
\cite[Chapter~II, Section~4]{Ek3}.

\begin{proof}
Resolve \(N^\bullet\in D^-(\FGauge\td\I)\) by a bounded-above complex
\(P^\bullet\) of projective generators \(G_r\). The adjunction of
\Cref{adjunction between FG^Tot and t^Tot} descends to homotopy
categories and gives
\[
\begin{aligned}
\Hom_{D^-(\FGauge\td\I)}(N^\bullet,\FG^{\Tot}(M^\bullet))
&\cong
\Hom_{K^-(\FGauge\td\I)}(P^\bullet,\FG^{\Tot}(M^\bullet))\\
&\cong
\Hom_{K^-(R\td\mathrm{Mod}\td\I)}
  (\rmt^{\Tot}(P^\bullet),M^\bullet)\\
&\cong
\Hom_{D^-(R\td\mathrm{Mod}\td\I)}
  ((\rmL\rmt)^{\Tot}(N^\bullet),M^\bullet).
\end{aligned}
\]
\end{proof}

For \(N\in D^-(\FGauge\td\I)\) and
\(M\in D^-(R\td\mathrm{Mod}\td\I)\), write the adjunction unit and
counit as
\[
\eta_N:N\longrightarrow\FG^{\Tot}((\rmL\rmt)^{\Tot}(N)),
\qquad
\epsilon_M:(\rmL\rmt)^{\Tot}(\FG^{\Tot}(M))\longrightarrow M.
\]





\subsection{Description of \(\rmt(G_r)\)}
\label{subsec:rmt-of-projective-generators}

To determine when the adjunction gives an equivalence, we first
compute it on the projective generators. Their images under \(\rmt\)
have an explicit description in terms of free modules and right
multiplication. This reduces the comparison to a calculation of
kernels and cokernels.

\begin{proposition}\label{description of t(G_r)}
Let \(r\in\I=[m,n]\).
    
{\rm (1)} The complex \(\rmt(G_r)\), supported in cohomological degrees
\([-n,-m]\), is
    \[\rmt(G_r)=\left(W_{\sigma}[F,F^{-1}](-n)\xrightarrow{\cdot d}R(1\!\!-\!\!n) \xrightarrow{\cdot d} ... \xrightarrow{\cdot d}R(-\!m)\right),\]
where every differential is right multiplication by \(d\), except for
\(R(-r)\xrightarrow{\cdot dV}R(-r-1)\).
Here \(R(i)\) denotes the complex
\(R^0(i)\xrightarrow{(-1)^id\cdot}R^1(i-1)\), as in
\Cref{shifts of R-complexes}.

{\rm (2)} The adjunction unit
\[
G_r\longrightarrow\FG^{\Tot}(\rmt(G_r))
\]
is a canonical quasi-isomorphism in \(\Ch(\FGauge\td\I)\).
\end{proposition}

The calculation in (1) is \cite[Lemma~II.4.2]{Ek3}; for (2), compare
the proof of \cite[Theorem~II.4.5]{Ek3}.

For \(\I=[0,2]\) and \(r=1\), this complex is
    \[\begin{tikzcd}
	{R^0 (\beta_0\otimes \alpha)} & {R^1(\beta_0\otimes \alpha)} \\
	& {R^0(\beta_1\otimes \alpha)} & {R^1(\beta_1\otimes \alpha)} \\
	&& {R^1(\beta_2\otimes \alpha)}
	\arrow["{d\cdot}", from=1-1, to=1-2]
	\arrow["{\cdot dV}"', from=2-2, to=1-2]
	\arrow["{-d\cdot}", from=2-2, to=2-3]
	\arrow["{\cdot d}"', from=3-3, to=2-3]
\end{tikzcd}\]

\begin{proof}
Fix \(i\in[m,n-1]\). We first show that the component of
\(\rmt(G_r)\) in cohomological degree \(-i\) is the free
\(R\)-module generated by \(\beta_i\otimes\alpha_r\).

By \Cref{definition of G_r}, every basis element of every \(G_r^j\) is
obtained from \(\alpha_r\) by iterating \(\tF,\tV,\tau\), and
\(\tau^{-1}\). Relations \Cref{rmt}(3)--(4) reduce every generator
\(\beta_i\otimes x\) to an \(R\)-multiple of
\(\beta_i\otimes\alpha_r\): each occurrence of \(\tF\) or \(\tV\)
is replaced by the corresponding operator \(F,V,p\), or \(1\)
specified in (3), while (4) removes occurrences of
\(\tau^{\pm1}\) by identifying the endpoint generators. To see that no further relations occur, send
\(\beta_i\otimes\alpha_r\) to the generator \(e\) of \(R(-i)\)
and evaluate the remaining generators by (3)--(4).
This respects the defining relations of \(G_r\):
\(\tF\tV=\tV\tF=p\) becomes \(FV=VF=p\) or \(p\cdot1=p\),
the endpoint identifications are inverse, and (1) accounts for
the scalar twists. Compatibility with (2) follows from \(FdV=d\).
The resulting map is inverse to
\(R(-i)\to\rmt(G_r)^{-i}\), proving freeness.

For the component in bidegree \((n,-n)\), put
\(b:=\beta_n\otimes\alpha_r\). Here we must also use gauge degree
\(n+1\): for \(x\in G_r^n\), relation \Cref{rmt}(3) gives
\(\beta_n\otimes\tF x=F(\beta_n\otimes x)\).
Since \(\tF x\) represents the positive-tail image of \(x\),
relation (4) identifies \(\beta_n\otimes\tF x\) with
\(\beta_n\otimes\tau x\). Thus passing through \(\tau\) acts
by \(F\), and passing through \(\tau^{-1}\) acts by \(F^{-1}\).
Within gauge degrees \([m,n]\), each \(\tF\) contributes \(p\)
and each \(\tV\) acts as the identity. Consequently, for \(l\geqslant0\),
\[
\begin{aligned}
\beta_n\otimes
\left(\tF^{\,r-m}\tau\tF^{\,n-r}\right)^l\alpha_r
 &= (p^{\,n-m}F)^l b,\\
\beta_n\otimes
\left(\tV^{\,n-r}\tau^{-1}\tV^{\,r-m}\right)^l\alpha_r
 &= F^{-l}b.
\end{aligned}
\]
Conversely, evaluating the gauge generators by these rules in
\(W_\sigma[F,F^{-1}]b\) respects (1)--(4) and gives an inverse
to the natural map onto the top component, so no further relations occur.
Thus the top component is
freely generated by \(b\) over \(W_\sigma[F,F^{-1}]\). Finally, \Cref{rmt}(2), applied to \(\alpha_r\), gives
\(\delta(\beta_i\otimes\alpha_r)=(-1)^{i+1}d(\beta_{i-1}\otimes\alpha_r)\)
for \(i>m\), with \(d\) replaced by \(dV\) when \(i=r\),
and \(\delta(\beta_m\otimes\alpha_r)=0\). After the standard sign
adjustment of the generators, these are exactly the displayed
differentials, proving (1).

For (2), the identity of \(\rmt(G_r)\) corresponds under adjunction to
the canonical morphism
\(
\eta:G_r\longrightarrow \FG^{\Tot}(\rmt(G_r)).
\)
By \Cref{adjunction between FG and t}, this map is given, for \(x\in G_r^j\), by
\[
x\longmapsto \sum_{i\in\I}\beta_i\otimes x.
\]
Since \(G_r\) is concentrated in cohomological degree \(0\), it suffices,
in each gauge degree \(j\), to identify \(G_r^j\) with the kernel of
\[
\FG^{\Tot}(\rmt(G_r))^{j,0}\longrightarrow \FG^{\Tot}(\rmt(G_r))^{j,1},
\]
and to show that the displayed differential is surjective.

Surjectivity in degree \(1\) follows directly from (1): each copy of
\(R^1(\beta_i\otimes\alpha_r)\) is reached either from
\(dV\cdot R^0(\beta_i\otimes\alpha_r)\) or from
\(\delta(R^0(\beta_{i+1}\otimes\alpha_r))\).

The map \(x\mapsto\sum_i\beta_i\otimes x\) lands in the kernel. Indeed,
the total cycle condition for \((x_i)_{i\in\I}\) is
\[
\delta(x_m)=0,
\]
\[
d(x_{i-1})+(-1)^i\delta(x_i)=0,\qquad i\in[m+1,n],\ i\neq j,
\]
and at the wall,
\[
dV(x_{j-1})+(-1)^j\delta(x_j)=0.
\]
These are precisely the relations in \Cref{rmt}(2).

To identify the entire kernel, fix \(j\). For \(\lambda\in W\), let
\(\rho_\lambda:G_r\to G_r\) send
\(\alpha_r\) to \(\lambda\alpha_r\). Functoriality makes the adjunction
unit equivariant for the induced actions.

Grouping the canonical basis elements by their \(\sigma^f\)-weights
for the action of \(\rho_\lambda\), we obtain
\[
G_r^j=\bigoplus_{f\in\Z}Y_f^j,\qquad
\FG^{\Tot}(\rmt(G_r))^{j,\bullet}=\bigoplus_{f\in\Z}X_f^{j,\bullet},
\]
where \(\rho_\lambda\) acts on \(Y_f^j\) and \(X_f^{j,\bullet}\) by
multiplication by \(\lambda^{\sigma^f}\). If \(k\) is infinite,
these characters are distinct. If \(k\) is finite, they may coincide,
but the same basis decomposition remains direct and is preserved by
the differential and the unit: this can be checked after the faithfully
flat base change \(W(k)\to W(\bar k)\), where the characters are
distinct. It therefore suffices to prove that \(Y_f^j[0]\to X_f^{j,\bullet}\) is a
quasi-isomorphism for every \(f\).

The module \(Y_f^j\) is free of rank one over \(W\). To specify its
generator, define the two cycling operators
\[
\Phi:=\tF^{\,r-m}\tau\tF^{\,n-r},
\qquad
\Psi:=\tV^{\,n-r}\tau^{-1}\tV^{\,r-m}.
\]
A basis of \(Y_f^j\) over \(W\) is given by the following canonical
vector.
\[
y_{f,j}:=
\begin{cases}
\tF^{\,j-r}\Phi^f\alpha_r,
  & f\geqslant0,\ j\geqslant r,\\[3pt]
\left(\tF^{\,j-m}\tau\tF^{\,n-r}\right)\Phi^{f-1}\alpha_r,
  & f\geqslant1,\ j<r,\\[3pt]
\tV^{\,r-j}\Psi^{-f}\alpha_r,
  & f\leqslant0,\ j\leqslant r,\\[3pt]
\left(\tV^{\,n-j}\tau^{-1}\tV^{\,r-m}\right)\Psi^{-f-1}\alpha_r,
  & f\leqslant-1,\ j>r.
\end{cases}
\]

The complex \(X_f^{j,\bullet}\) is the \(\sigma^f\)-weight subcomplex of
\(\FG^{\Tot}(\rmt(G_r))^{j,\bullet}\). Let \(b_i\) denote the element of
\(\FG^{\Tot}(\rmt(G_r))^{j,0}\) corresponding to
\(\beta_i\otimes\alpha_r\). Then \(b_i\in X_{w_i}^{j,0}\), where
\[
w_i:=\mathbf{1}_{i<r}-\mathbf{1}_{i<j}.
\]
Here \(\mathbf{1}_P\) denotes the indicator of a condition \(P\).
Indeed, \Cref{rmt}(1) contributes \(\mathbf{1}_{i<r}\) to the weight,
while the \(\sigma_*\)-twist in \Cref{definition of FG}(1) subtracts
\(\mathbf{1}_{i<j}\). Put \(h_i:=f-w_i\).
Since multiplication by \(F\) raises the weight exponent by \(1\)
and multiplication by \(V\) lowers it by \(1\), a \(W\)-basis of
\(X_f^{j,0}\) is given by
\[
e_i=
\begin{cases}
F^{h_i}b_i,&h_i\geqslant0,\\
V^{-h_i}b_i,&h_i<0,
\end{cases}
\quad(m\leqslant i<n),
\qquad e_n=F^fb_n.
\]
At the top component, both indicators vanish and \(F\) is invertible,
so the last formula applies to every \(f\in\Z\).

Similarly, \(db_i\) has internal degree \(i+1\), so the twist
\(\sigma_*\) occurs when \(i+1<j\). Since \(d\) commutes with
scalars, set
\[
w_i':=\mathbf{1}_{i<r}-\mathbf{1}_{i+1<j},\qquad h_i':=f-w_i'.
\]
A \(W\)-basis of \(X_f^{j,1}\) is therefore
\[
u_i=
\begin{cases}
F^{h_i'}db_i,&h_i'\geqslant0,\\
dV^{-h_i'}b_i,&h_i'<0,
\end{cases}
\qquad m\leqslant i<n.
\]
In the negative branch, the order in \(dV^ab_i\) matters: using
\(V^adb_i\) would introduce the factor in \(V^ad=p^adV^a\).
Also, \(h_i'=h_i\) except at
\(i=j-1\), where \(h_i'=h_i-1\).

In these bases, the adjunction unit sends
\[
\eta(y_{f,j})=(a_i e_i)_{i=m}^n,\qquad a_i=p^{c_i}.
\]
The exponents count the powers of \(p\) contributed by the complete
cycles and the remaining path from \(r\) to \(j\).
Explicitly, for \(m\leqslant i<n\), we have
\[
c_i=
\begin{cases}
f(i-m)+\min(i,j)-\min(i,r),
 & f>0\text{ or }(f=0,\ j\geqslant r),\\[3pt]
-f(n-i-1)+\max(i+1,r)-\max(i+1,j),
 & f<0\text{ or }(f=0,\ j\leqslant r).
\end{cases}
\]
At \(i=n\), we have \(c_n=f(n-m)+j-r\) in the first case
and \(c_n=0\) in the second. Both formulas give \(c_i=0\)
when \(f=0\) and \(j=r\).

For each \(j\) and each \(f\), at least one coefficient \(a_i\) is a
unit. Thus the image of \(Y_f^j\) in \(X_f^{j,0}\) is a saturated
\(W\)-submodule. This will identify the image once the kernel of the
differential is known.

Let \(D:X_f^{j,0}\to X_f^{j,1}\) be the differential.
The definitions in \Cref{definition of FG}(3) and \Cref{rmt}(2), with
the totalization sign, give the following formulas, where
\(u_{m-1}=0\).

For \(m\leqslant i<n\),
\[
D(e_i)=
\begin{cases}
p^{h_i}u_i-u_{i-1},
 & f>0\text{ or }(f=0,\ j\geqslant r),\\[3pt]
u_i-p^{-h_i}u_{i-1},
 & f<0\text{ or }(f=0,\ j\leqslant r).
\end{cases}
\]
In both cases, \(D(e_n)=-u_{n-1}\).

The coefficients \(-1\) in the first case and \(1\) in the second
allow us to solve successively for every \(u_i\), so \(D\) is surjective. Since \(X_f^{j,0}\) and \(X_f^{j,1}\) have ranks \(n-m+1\)
and \(n-m\), its kernel is free of rank one. The nonzero saturated
image of \(Y_f^j\) is therefore exactly this kernel. If \(m=n\),
\(X_f^{j,1}=0\) and the same conclusion is immediate.

Thus \(Y_f^j[0]\to X_f^{j,\bullet}\) is a quasi-isomorphism for every
\(f\). Summing over the weights \(f\) shows that
\(G_r^j\to \FG^{\Tot}(\rmt(G_r))^{j,\bullet}\) is a quasi-isomorphism
for every \(j\in\I\). Hence
\(G_r\to \FG^{\Tot}(\rmt(G_r))\) is a quasi-isomorphism in
\(\Ch(\FGauge\td\I)\).
\end{proof}

Projective resolutions extend this calculation from the generators
to bounded-above complexes.

\begin{corollary}\label{M simeq FGLtM}
For every \(M\in D^-(\FGauge\td\I)\), the adjunction unit is a natural
quasi-isomorphism
\[
M\xrightarrow{\sim}\FG^{\Tot}((\rmL\rmt)^{\Tot}(M)).
\]
\end{corollary}

Compare the generator calculation in the proof of
\cite[Theorem~II.4.5]{Ek3}.

\begin{proof}
The \(G_r\) form a family of projective generators of
\(\FGauge\td\I\). Choose a bounded-above projective resolution
\(P^\bullet\to M\) whose terms are direct sums of the \(G_r\).

By definition of the left derived functor,
\(
(\rmL\rmt)^\Tot(M)\simeq \rmt^{\Tot}(P^\bullet).
\)
Applying \(\FG^{\Tot}\), we get
\(
\FG^{\Tot}((\rmL\rmt)^{\Tot}(M))\simeq \FG^{\Tot}(\rmt^{\Tot}(P^\bullet)).
\)
For each projective generator \(G_r\), the preceding proposition shows
that the adjunction unit
\(
G_r\longrightarrow \FG^{\Tot}(\rmt(G_r))
\)
is a quasi-isomorphism. Since both \(\rmt^{\Tot}\) and \(\FG^{\Tot}\) commute with direct sums and with bounded totalizations, the same holds for every bounded complex \(P^\bullet\) whose terms are direct sums of the \(G_r\). Thus the natural map
\(
P^\bullet\longrightarrow \FG^{\Tot}(\rmt^{\Tot}(P^\bullet))
\)
is a quasi-isomorphism. 
Composing with the quasi-isomorphism \(P^\bullet\to M\), we obtain a natural quasi-isomorphism
\[
M\simeq P^\bullet\xrightarrow{\sim}\FG^{\Tot}(\rmt^{\Tot}(P^\bullet))
\simeq \FG^{\Tot}((\rmL\rmt)^{\Tot}(M)).
\]
\end{proof} 

The unit comparison does not yet settle the role of completion. The
following example allows us to examine the derived inverse directly,
starting from a gauge with finite free components.

\begin{example}
\label{The first example}
Let \(\I=[0,2]\), and consider the \(F\)-gauge \(M\) defined by
\[
M^0=We_1\oplus We_2,\qquad
M^1=pWe_1\oplus We_2,\qquad
M^2=p^2We_1\oplus We_2.
\]
For \(i=0,1\), the map \(\tV:M^{i+1}\hookrightarrow M^i\) is the
inclusion and \(\tF:M^i\to M^{i+1}\) is multiplication by \(p\). The
Frobenius-semilinear isomorphism \(\tau:M^2\to M^0\) is determined by
\[
\tau(p^2e_1)=e_2,\qquad \tau(e_2)=e_1.
\]
As usual, set \(\Phi=\tau\tF^2\) and \(\Psi=\tV^2\tau^{-1}\).
The projective gauge \(G_0\) surjects onto \(M\) by
\(\alpha_0\mapsto e_1\), since both \(\Phi(\alpha_0)\) and
\(\Psi(\alpha_0)\) map to \(e_2\). The resulting relation generates
the kernel, giving an exact sequence
\[
0\longrightarrow G_0\xrightarrow{\ \Phi-\Psi\ }G_0
\longrightarrow M\longrightarrow0.
\]

Applying \((\rmL\rmt)^{\Tot}\) gives
\((\rmL\rmt)^{\Tot}(M)\simeq \Tot\left[\rmt(G_0)\xrightarrow{\ \rmt(\Phi-\Psi)\ }\rmt(G_0)\right],\)
with the left term in degree \(-1\). For \(\I=[0,2]\), the object
\(\rmt(G_0)\) is represented by the following \(R\)-complex:
\[
\begin{tikzcd}
{R^0(\beta_0\otimes\alpha_0)} & {R^1(\beta_0\otimes\alpha_0)} \\
& {R^0(\beta_1\otimes\alpha_0)} & {R^1(\beta_1\otimes\alpha_0)} \\
&& {W_\sigma[F,F^{-1}](\beta_2\otimes\alpha_0)}
\arrow["{d\cdot}", from=1-1, to=1-2]
\arrow["{\cdot d}"', from=2-2, to=1-2]
\arrow["{d\cdot}", from=2-2, to=2-3]
\arrow["{\cdot(-d)}"', from=3-3, to=2-3]
\end{tikzcd}
\]
Equivalently, Ekedahl's shifting notation identifies this complex with
\[
\rmt(G_0)=\left[W_\sigma[F,F^{-1}](-2)\xrightarrow{\cdot d}R(-1)\xrightarrow{\cdot d}R\right],
\]
where the left term lies in degree \(-2\); the shift \(R(-1)\) absorbs
the sign in the final differential.

The endomorphism \(\Phi-\Psi\) of \(G_0\) induces the following
endomorphism of \(\rmt(G_0)\):
\[
\begin{tikzcd}[column sep=large,row sep=large]
R \arrow[r,"\cdot(F-pV)"] & R \\
R(-1) \arrow[r,"\cdot(pF-V)"] \arrow[u,"\cdot d"] & R(-1) \arrow[u,"\cdot d"] \\
W_\sigma[F,F^{-1}](-2) \arrow[r,"\cdot(p^2F-F^{-1})"] \arrow[u,"\cdot d"] & W_\sigma[F,F^{-1}](-2) \arrow[u,"\cdot d"]
\end{tikzcd}
\]
Its total complex represents \((\rmL\rmt)^{\Tot}(M)\). Indeed, if
\(b_i=\beta_i\otimes\alpha_0\), the defining relations of \(\rmt\) give
\[
\begin{aligned}
\beta_0\otimes(\Phi-\Psi)\alpha_0&=(F-pV)b_0,\\
\beta_1\otimes(\Phi-\Psi)\alpha_0&=(pF-V)b_1,\\
\beta_2\otimes(\Phi-\Psi)\alpha_0&=(p^2F-F^{-1})b_2.
\end{aligned}
\]

In module degree \(0\), one obtains
\(
\HH^0((\rmL\rmt)^{\Tot}(M))^0\cong 
 R^0/R^0(F-pV) \coloneqq Q. 
\)
Inside \(Q\), we have \(F=pV\). Using \(FV=p\), we obtain
\(
p(V^2-1)=0.
\)
Multiplying by \(V^{i-2}\) gives
\(
 p(V^i-V^{i-2})=0\qquad(i\geqslant 2).
\)
Hence every \(V^i\) for \(i\geqslant 2\) is \(p\)-torsion modulo lower terms. Therefore, as a \(W\)-module,
\(
Q\simeq W\oplus WV\oplus \bigoplus_{i\geqslant 2} k\cdot V^i.
\)
In particular, \(Q\) is the direct sum of a free \(W\)-module of rank \(2\), generated by \(1\) and \(V\), and a \(p\)-torsion module generated by \(V^i\) for \(i\ge2\). 
After \(V\)-adic completion,
\(
\widehat Q\simeq W[[V]]/(p(V^2-1)).
\)
Since \(V^2-1\) is a unit in \(W[[V]]\), we obtain
\(
\widehat Q\simeq W[[V]]/(p)\simeq k[[V]].
\)
Thus \(Q\), and hence \((\rmL\rmt)^{\Tot}(M)\), is not complete.

A direct calculation in module degree \(1\) gives
\[\HH^0((\rmL\rmt)^{\Tot}(M))^1\cong R^1/\left(R^1 d+R^1(F-pV)\right)\cong k[dV]_{>0}.\]
Indeed, consider \(R^1/(R^0d+R^1(F-pV))\).
Modulo \(R^0d\), the quotient is generated by
\(dV,\ dV^2,\ dV^3,\ldots .\)
For \(i\geqslant 1\), we have
\(dV^i(F-pV)=p\,dV^{i-1}-p\,dV^{i+1}.\)
Since \(d=dV^0\) vanishes modulo \(R^0d\), the case \(i=1\) gives
\(p\,dV^2=0,
\)
and for \(i\geqslant 2\),
\(
p(dV^{i+1}-dV^{i-1})=0.
\)
Hence
\(
R^1/(R^0d+R^1(F-pV))
\cong k[dV]_{>0}
\)
After \(V\)-adic completion,
\(
R^1/(R^0d+\widehat{R^1}(F-pV))
\simeq k[[dV]]_{>0}.
\)

A direct check shows that all other graded cohomology components vanish.
\end{example}

\subsection{Compatibility with completion, and The Completed Functor \(\widehat{(\rmL\rmt)^{\Tot}}\)}
\label{subsec:compatibility-with-completion}

The preceding example explains why completion must enter the
correspondence. On the gauge side, the relevant notion is derived
completion in the coefficient ring. The operators of an \(F\)-gauge
descend modulo \(p^\nu\), so derived reduction by \(W/p^\nu\) is
defined within \(D(\FGauge\td\I)\). Since \(\FGauge\td\I\) is
Grothendieck, its derived category admits the inverse limits needed
below.

\begin{definition}[{\cite[Definition~II.3.2]{Ek3}}]\label{completion of F-gauge structure}
For \(M\in D(\FGauge\td\I)\), define its derived \(p\)-completion by
\[
\widehat M:=
\mathbf R\varprojlim_\nu
\bigl((W/p^\nu)\otimes_W^{\rmL}M\bigr).
\]
The object \(M\) is complete if the canonical map
\(M\to\widehat M\) is an isomorphism.
\end{definition}

\begin{proposition}[{\cite[Proposition~II.3.3]{Ek3}}]\label{FG commutes with completion}
For a fixed finite interval \(\I\), the functor
\[
\FG^{\Tot}:D(R\td\mathrm{Mod}\td\I)
 \longrightarrow D(\FGauge\td\I)
\]
commutes with completion: there is a natural isomorphism
\[
\FG^{\Tot}(\widehat M)\simeq\widehat{\FG^{\Tot}(M)}.
\]
Here completion on the \(R\)-side is in the sense of Ekedahl
(\Cref{Ekedahl completion}), and completion on the gauge side is
derived \(p\)-completion (\Cref{completion of F-gauge structure}).
In particular, \(\FG^{\Tot}\) takes complete complexes to complete
complexes.
\end{proposition}

\begin{proof}
Recall the compatibility
\(\Tot(\widehat N)\simeq\widehat{\Tot(N)}\) of totalization
with these completions. By \Cref{definition of FG}(3), gauge degree
\(i\) of \(\FG^{\Tot}(M)\) is the total complex of the Nygaard
modification \(M(\varphi^i)\), where
\(\varphi^i(j)=-\delta_{i,j}\).
Using \Cref{completion commutes with Nygaard modification}, we obtain
\[
\begin{aligned}
\FG^{\Tot}(\widehat M)^i
 &=\Tot\bigl(\widehat M(\varphi^i)\bigr)\\
 &\simeq\Tot\bigl(\widehat{M(\varphi^i)}\bigr)\\
 &\simeq\widehat{\Tot(M(\varphi^i))}
 =\widehat{\FG^{\Tot}(M)^i}.
\end{aligned}
\]
These natural identifications respect \(\tF,\tV\) and \(\tau\),
so they give the asserted isomorphism of complexes of \(F\)-gauge
structures. The assertion about complete complexes follows.
\end{proof}

The projective generator \(G_r\) is not complete: each \(G_r^i\) is an
infinite direct sum of copies of \(W\). Its completion
\(\widehat{G}_r^i\) consists of formal sums whose coefficients converge
to zero \(p\)-adically.


The compatibility just proved for \(\FG^{\Tot}\) leaves a separate
question for \((\rmL\rmt)^{\Tot}\): its preservation of completeness
is not immediate, even when the input gauge is complete (see
\Cref{The first example}). In particular, this requires attention for
\(\widehat{G}_r\). The complex \(\rmt(G_r)\) itself is not complete,
since \(\FG^{\Tot}(\rmt(G_r))\) is not. We therefore include completion
in the definition of the derived inverse.

\begin{definition}\label{definition of completed derived rmt}
Set
\[
\widehat{(\rmL\rmt)^{\Tot}}
:=\widehat{(-)}\circ(\rmL\rmt)^{\Tot}.
\]
\end{definition}

By \Cref{FG commutes with completion}, completion gives maps
\[
\widehat\eta_N:\widehat N\longrightarrow
\FG^{\Tot}(\widehat{(\rmL\rmt)^{\Tot}}(N)),
\qquad
\widehat\epsilon_M:
\widehat{(\rmL\rmt)^{\Tot}}(\FG^{\Tot}(M))\longrightarrow\widehat M.
\]

To compare the completed unit and counit, we return to the gauge
operators and their cones. The distinction between shifts in gauge
degree and shifts in cohomological degree will matter in the resulting
triangles.

\begin{definition}\label{shifts and derived quotients of F-gauges}
For an \(F\)-gauge \(M\), define \(M(i)\) to be the \(F\)-gauge with
\(M(i)^j=M^{i+j}\) and the same operators \(\tF,\tV,\tau\).
The operators \(\tF\) and \(\tV\) define morphisms of \(F\)-gauges
\(M\to M(1)\) and \(M\to M(-1)\), respectively. 
For \(M\in D(\FGauge\td\I)\), define the derived quotients of its
underlying gauge by
\[
\begin{aligned}
M/\tF&:=\operatorname{Cone}\bigl(\tF:M(-1)\to M\bigr),\\
M/\tV&:=\operatorname{Cone}\bigl(\tV:M(1)\to M\bigr),\\
M/(\tF,\tV)&:=(M/\tF)/\tV.
\end{aligned}
\]
Thus these symbols do not denote ordinary cokernels unless the
relevant higher derived terms vanish.
\end{definition}

\begin{proposition}\label{distinguished triangle of M(F V)}
Let \(M\in D(\FGauge\td\I)\).

{\rm (1)} In \(D(\FGauge)\)
there is a natural diagram, commuting up to the usual signs, whose
rows and columns are distinguished triangles:
\[
\begin{tikzcd}[column sep=large,row sep=large]
M \arrow[r,"\tF"] \arrow[d,"\tV"'] &
M(1) \arrow[r] \arrow[d,"\tV"'] &
{(M/\tF)(1)} \arrow[r,"+1"] \arrow[d,"\tV"'] & {}\\
M(-1) \arrow[r,"\tF"] \arrow[d] &
M \arrow[r] \arrow[d] &
{M/\tF} \arrow[r,"+1"] \arrow[d] & {}\\
{(M/\tV)(-1)} \arrow[r,"\tF"] \arrow[d,"+1"'] &
{M/\tV} \arrow[r] \arrow[d,"+1"'] &
{M/(\tF,\tV)} \arrow[r,"+1"] \arrow[d,"+1"'] & {}\\
{} & {} & {} & {}
\end{tikzcd}
\]

{\rm (2)} There is a natural distinguished triangle
\begin{equation}\label{derived mod p triangle for gauges}
(M/\tF)(1)\xrightarrow{\ \tV\ }
(W/p)\otimes_W^{\rmL}M
\longrightarrow M/\tV
\longrightarrow (M/\tF)(1)[1].
\end{equation}
Here \((1)\) is a gauge-grading shift, whereas \([1]\) is a
cohomological shift.

{\rm (3)} If \(M=\FG^{\Tot}(N)\) for
\(N\in D(R\td\mathrm{Mod}\td\I)\), then the gauge-degree \(i\)
component of the diagram in (1) is naturally isomorphic to the
diagram of \Cref{Nygaard filtration and Hodge filtration}.
In particular,
\[
\begin{aligned}
(M/\tF)^i&\simeq t_{\geqslant i}\Hod(N),\\
(M/\tV)^i&\simeq\sigma_*\tau_{\leqslant i}\Hod(N),\\
\bigl(M/(\tF,\tV)\bigr)^i
 &\simeq(R_1\otimes_R^{\rmL}N)^i[-i].
\end{aligned}
\]
\end{proposition}

These triangles appear in \cite[Definition~II.4.6]{Ek3}.

\begin{proof}
Choose a complex representing \(M\) and use the usual mapping cones.
The upper-left square commutes because \(\tF\tV=\tV\tF=p\).
Taking its row and column cones gives (1): the bottom-right term is
\(\operatorname{Cone}(\tV:(M/\tF)(1)\to M/\tF)\), which is
\(M/(\tF,\tV)\) by definition. Taking the cones in the opposite
order gives the same total complex, with the usual sign change.
This construction is functorial on complexes and preserves
quasi-isomorphisms, so it descends to the derived category.

For (2), apply the triangle of cones to
\(M\xrightarrow{\tF}M(1)\xrightarrow{\tV}M\), whose composite
is multiplication by \(p\). Its three terms are
\((M/\tF)(1)\), \(\operatorname{Cone}(p:M\to M)\), and
\(M/\tV\). The middle term is naturally
\((W/p)\otimes_W^{\rmL}M\), since
\([W\xrightarrow{p}W]\) resolves \(W/p\).
The first map is induced by \(\tV\) on targets and the identity
on sources, giving the displayed triangle.

For (3), \Cref{Nygaard filtration and Hodge filtration} identifies
the row and column cones with the stated Hodge truncations and the
iterated cone with \((R_1\otimes_R^{\rmL}N)^i[-i]\).
The comparison maps there are compatible with the gauge operators,
so these identifications give the entire diagram.
\end{proof}

\begin{proposition}\label{FG is conservative on complete objects}
For a fixed finite interval \(\I\), the functor
\(\FG^{\Tot}:D(R\td\mathrm{Mod}\td\I)\to D(\FGauge\td\I)\)
is conservative on complete objects.
\end{proposition}

Compare \cite[Corollary~II.3.4.1]{Ek3}.

\begin{proof}
Since complete objects are closed under cones, it suffices to show
that \(\FG^{\Tot}(M)=0\) implies \(M=0\) for complete \(M\).
The derived quotient triangles of \Cref{distinguished triangle of M(F V)},
together with the comparisons in \Cref{Nygaard filtration and Hodge filtration},
then give \((R_1\otimes_R^{\rmL}M)^{i,\bullet}=0\) for every \(i\).
As \(M\) has internal degrees in the finite interval \(\I\),
\Cref{R1 detects vanishing of complete complexes} gives \(M=0\).
\end{proof}

The cone comparison supplies the link between reduction on the
Raynaud side and the derived quotients of a gauge. Together with
completion, it gives the following comparison of the adjunction maps.

\begin{proposition}[{\cite[Theorem~II.4.5]{Ek3}}]\label{completed adjunction unit and counit are isomorphisms}
For \(N\in D^-(\FGauge\td\I)\) and
\(M\in D^-(R\td\mathrm{Mod}\td\I)\), both
\(\widehat\eta_N\) and \(\widehat\epsilon_M\) are isomorphisms.
\end{proposition}

\begin{proof}
Resolve by the projective generators \(G_r\). For each generator, the
completed counit is an isomorphism by
\Cref{description of t(G_r)}(2), hence so is
\(\widehat\epsilon_M\) for every bounded-above complex. Since
\(\FG^{\Tot}\) is conservative on complete objects by
\Cref{Ek2I.1.1,distinguished triangle of M(F V)}, the triangle
identity for the adjunction then implies that \(\widehat\eta_N\) is an
isomorphism as well.
\end{proof}

\begin{proposition}[{\cite[Proposition~II.4.7]{Ek3}}]\label{Ekedahl's equivalence theorem}
The completed derived functor has finite amplitude and yields the
following equivalence.

{\rm (1)} It has finite amplitude on \(D^-(\FGauge\td\I)\) and
therefore extends to \(D(\FGauge\td\I)\).

{\rm (2)} The functors \(\FG^{\Tot}\) and
\(\widehat{(\rmL\rmt)^{\Tot}}\) are inverse equivalences between
complete \(R\)-complexes of level \(\I\) and complete complexes of
\(F\)-gauges of level \(\I\). They preserve bounded objects.
\end{proposition}

\begin{proof}
The finite-amplitude assertion follows from
\Cref{distinguished triangle of M(F V)} and the conservativity of
\(R_1\otimes_R^{\rmL}-\) on complete objects. The equivalence then
follows from \Cref{completed adjunction unit and counit are isomorphisms}.
\end{proof}

\subsection{Coherent \(F\)-gauge structures}
\label{sec:coherent-f-gauge-structures}
\label{subsec:completed-functor}

Having compared complete objects, we can ask which finiteness
conditions the equivalence preserves. We impose coherence on
\(F\)-gauge structures at a fixed finite interval \(\I=[m,n]\).
For a \(k\)-complex, coherence will mean that its cohomology in each
degree is finite-dimensional; boundedness, when required, is specified
separately.

\begin{definition}\label{coherence of F-gauge complexes}
An object \(M\in D(\FGauge\td\I)\) is coherent if it is
derived \(p\)-complete and, for every \(i\in\I\), the complex
\[
\bigl((W/p)\otimes_W^{\rmL}M\bigr)^i
\]
has finite-dimensional \(k\)-cohomology in every degree.
Here derived \(p\)-completeness means that
\[
M\xrightarrow{\;\sim\;}
R\varprojlim_{\nu}\bigl((W/p^\nu)\otimes_W^{\rmL}M\bigr);
\]
see \Cref{completion of F-gauge structure}.
An \(F\)-gauge structure is coherent if it is coherent when placed
in cohomological degree zero.
We denote the full subcategory of coherent complexes by
\(D_c(\FGauge\td\I)\), and its bounded and bounded-above parts by
\(D_c^b(\FGauge\td\I)\) and \(D_c^-(\FGauge\td\I)\).
\end{definition}

\begin{proposition}[{\cite[Proposition~II.5.2]{Ek3}}]\label{criteria for coherent F-gauge complexes}
{\rm (1)} An object \(M\in D^-(\FGauge\td\I)\) is coherent if and only if
\(\HH^q(M)^i\) is a finitely generated \(W\)-module for every
\(q\in\Z\) and \(i\in\I\).
In particular, an \(F\)-gauge structure \(M\) is coherent if and
only if every \(M^i\), \(i\in\I\), is finitely generated over \(W\).

{\rm (2)} A derived \(p\)-complete object \(M\in D(\FGauge\td\I)\)
is coherent if and only if every gauge-degree component of
\[
M/(\tF,\tV)
\]
has finite-dimensional \(k\)-cohomology in every degree.
The quotient here is the derived quotient of
\Cref{distinguished triangle of M(F V)}.
\end{proposition}

\begin{proof}
For (1), evaluation at any gauge degree commutes with derived
reduction and derived completion, so we work with a bounded-above
\(W\)-complex \(X=M^i\).

If all \(\HH^q(X)\) are finitely generated over \(W\), then \(X\)
is derived \(p\)-complete. This follows from completeness of finite
\(W\)-modules, applied to the truncations of \(X\); equivalently,
one can use the derived completion spectral sequence over the
complete discrete valuation ring \(W\).
Moreover,
\[
0\longrightarrow\HH^q(X)/p
 \longrightarrow\HH^q\bigl((W/p)\otimes_W^{\rmL}X\bigr)
 \longrightarrow\HH^{q+1}(X)[p]\longrightarrow0
\]
shows that its derived reduction has finite-dimensional cohomology.

Conversely, suppose \(X\) is derived complete and its reduction
has finite-dimensional cohomology. If \(r\) is the highest possible
cohomological degree, the bounded-above derived Nakayama argument
shows that \(\HH^r(X)\) is finitely generated:
lift finitely many generators of
\(\HH^r(X)/p=\HH^r((W/p)\otimes_W^{\rmL}X)\) to obtain a map
\(W^{\oplus d}[-r]\to X\). Its cone is derived complete, and its
reduction has no cohomology in degrees at least \(r\).
Derived Nakayama gives the same upper bound for the cone, so the
map is surjective on \(\HH^r\).
The finite module \(\HH^r(X)\) is derived complete; hence
\(\tau_{\leqslant r-1}X\) is derived complete and still has coherent
reduction. Descending induction proves finite generation in
every degree.

For (2), we compare reduction modulo \(p\) with the
two gauge quotients. Put \(C=(W/p)\otimes_W^{\rmL}M\),
\(Q_F=M/\tF\), \(Q_V=M/\tV\), and \(K=M/(\tF,\tV)\).
The cone triangles of \Cref{distinguished triangle of M(F V)} give
\[
Q_F^{i+1}\xrightarrow{\tV}Q_F^i\longrightarrow K^i
 \longrightarrow Q_F^{i+1}[1],
\qquad
Q_V^{i-1}\xrightarrow{\tF}Q_V^i\longrightarrow K^i
 \longrightarrow Q_V^{i-1}[1].
\]
The level condition gives \(Q_F^i=0\) for \(i>n\) and
\(Q_V^i=0\) for \(i<m\), since the maps defining these cones are
isomorphisms. If every \(K^i\) has finite-dimensional cohomology,
descending induction for \(Q_F\) and ascending induction for \(Q_V\)
give the same finiteness for both quotients. The triangle
\(Q_F^{i+1}\to C^i\to Q_V^i\to Q_F^{i+1}[1]\) then gives
finite-dimensional cohomology for \(C^i\). Together with the assumed
completeness, this proves that \(M\) is coherent.

Conversely, suppose \(M\) is coherent. Then \(C^i\) has
finite-dimensional cohomology in every gauge degree, including the
tails. Derived reduction commutes with mapping cones, so
\[
(W/p)\otimes_W^{\rmL}K\simeq (C/\tF)/\tV.
\]
The two successive cones on the right involve only shifts and finite
sums of \(C^{i-1},C^i,C^{i+1}\) in gauge degree \(i\), and hence
have finite-dimensional cohomology. On \(M/\tF\), the operator
\(\tF\) is null-homotopic by the standard cone homotopy. This
homotopy commutes with \(\tV\), so it extends to \(K\); consequently,
\(p=\tV\tF\) acts trivially on \(\HH^q(K^i)\).
The reduction triangle therefore gives an injection
\(\HH^q(K^i)\hookrightarrow
\HH^q((W/p)\otimes_W^{\rmL}K^i)\).
Thus \(K^i\) has finite-dimensional cohomology, as required.
\end{proof}

\begin{theorem}[{\cite[Theorem~II.5.3]{Ek3}}]\label{Ekedahl equivalence for coherent complexes}
Fix a finite interval \(\I\).
The functors \(\FG^{\Tot}\) and
\(\widehat{(\rmL\rmt)^{\Tot}}\) induce inverse equivalences
\[
D_c^-(R\td\I)\simeq D_c^-(\FGauge\td\I),
\qquad
D_c^b(R\td\I)\simeq D_c^b(\FGauge\td\I).
\]
\end{theorem}

\begin{proof}
By \Cref{Ekedahl's equivalence theorem}, the functors are inverse
equivalences on complete objects. Their finite amplitude at fixed
level ensures that they preserve bounded-above and bounded objects.
It remains to check that the equivalence preserves coherence.
For \(N\in D^-(R\td\I)\) complete, put \(M=\FG^{\Tot}(N)\).
By \Cref{distinguished triangle of M(F V)}(3),
\[
\bigl(M/(\tF,\tV)\bigr)^i
 \simeq (R_1\otimes_R^{\rmL}N)^i[-i].
\]
By Ekedahl's bounded-above coherence criterion, \(N\) is coherent
if and only if its derived \(R_1\)-reduction has finite-dimensional
cohomology; see \cite[Proposition 2.32]{LY}.
By \Cref{criteria for coherent F-gauge complexes}(2), coherence of
\(M\) is tested by the same condition on the left-hand side.
This proves the first equivalence, and restriction to bounded objects
gives the second; compare \cite[Theorem~II.5.3]{Ek3}.
\end{proof}

\begin{remark}\label{coherent equivalence and boundedness}
For unbounded complexes, the same argument identifies coherent gauge
complexes with complete \(R\)-complexes whose derived
\(R_1\)-reduction has finite-dimensional cohomology in every
bidegree. We do not identify this condition with coherence of all
cohomology modules without an additional boundedness hypothesis.
For bounded-below complexes, the coherence criterion is provided by
Li--Yang \cite[Corollary 3.21]{LY}; the same argument therefore gives
\(D_c^+(R\td\I)\simeq D_c^+(\FGauge\td\I)\).
\end{remark}

To identify the abelian categories within these equivalences, we
must also compare their \(t\)-structures. On the \(R\)-side, the
\(t\)-structure corresponding to the standard one for gauges is the
\(G\)-\(t\)-structure of \Cref{The t-structure: G}; the two standard
structures do not agree in general.
Write \(G^{\leqslant0}\td\I\) and \(G^{\geqslant0}\td\I\)
for its restriction to \(D_c^b(R\td\I)\).

\begin{proposition}[{\cite[Proposition~II.5.4]{Ek3}}]\label{Ekedahl equivalence and the G t-structure}
Under the equivalence of
\Cref{Ekedahl equivalence for coherent complexes},
\(\FG^{\Tot}\) takes the \(t\)-structure \(\widetilde D_g\) on
\(D_c^b(R\td\I)\) to the standard \(t\)-structure on
\(D_c^b(\FGauge\td\I)\).
In particular, it induces an equivalence between
\(\dc_g\cap D_c^b(R\td\I)\) and the category of coherent
\(F\)-gauge structures of level \(\I\).
\end{proposition}

\begin{proof}
Let \(N\in D_c^b(R\td\I)\) and suppose that
\(M=\FG^{\Tot}(N)\) is concentrated in cohomological degree zero.
We show that \(N\in\dc_g\).
The negative tail of the gauge construction identifies
\[
\Tot(N)\simeq M^{-\infty},
\]
so \(\Tot(N)\) is concentrated in degree zero.
Moreover, the three-term free Koszul resolution of \(k\) over
\(\widetilde R\), used in
\Cref{distinguished triangle of M(F V)}, shows that
\(M/(\tF,\tV)\) has cohomological amplitude contained in \([-2,0]\).
Its identification with the shifted columns of
\(R_1\otimes_R^{\rmL}N\) therefore gives
\[
\HH^j(R_1\otimes_R^{\rmL}N)^i=0
\qquad\text{unless }-2\leqslant i+j\leqslant0.
\]
Together with the assertion about \(\Tot(N)\),
\Cref{criteria for G-amplitude [m n]} with \(m=n=0\) yields \(N\in\dc_g\).

Thus \(\widehat{(\rmL\rmt)^{\Tot}}\) takes the standard heart into the
heart \(\dc_g\). Both \(t\)-structures are bounded, so finite
cohomological truncation filtrations imply that \(\widehat{(\rmL\rmt)^{\Tot}}\) is \(t\)-exact. Since it is an equivalence, it also
reflects the two halves: its compatibility with cohomology and
conservativity detect their vanishing conditions.
Consequently its inverse \(\FG^{\Tot}\) is \(t\)-exact as well,
which proves the assertion.
\end{proof}

\subsection{Resolution of coherent F-gauge structures}
\label{subsec:resolution-of-coherent-f-gauge-structures}

The equivalence gives a description of coherent objects, but explicit
derived calculations also require finite resolutions. Coherent
\(F\)-gauges admit projective resolutions built from finitely many
copies of the \(G_r\). The proof reduces the rank by passing to a
quotient by a cyclic subgauge, then treats the finite-length case.

\begin{proposition}\label{resolution of coherent complexes}
Let \(k\) be perfect and \(\I=[m,n]\). Every \(M\in D_c^b(\FGauge\td\I)\) admits a
bounded projective resolution whose terms are finite direct sums
of the \(G_r\), \(r\in\I\).
\end{proposition}

The proof of \cite[Theorem~II.5.6]{Ek3} uses finite sums of completed
generators; the bounded resolution asserted here is stronger.

\begin{proof}
By \Cref{two-out-of-three finite sums of G_r}, it suffices to treat a
single coherent \(F\)-gauge. The statement for bounded complexes then
follows by d\'evissage along the standard truncation triangles.

Let \(\mathcal C\) be the full subcategory of coherent \(F\)-gauges of
level \(\I\) admitting such a resolution. The horseshoe lemma and
mapping cones make \(\mathcal C\) stable under the two-out-of-three
operations in short exact sequences. We will construct the objects
of \(\mathcal C\) needed for d\'evissage, treating positive rank and
finite length separately.

\begin{lemma}\label{two-out-of-three finite sums of G_r}
Let
\[0\to A^\bullet\to B^\bullet\to C^\bullet\to0\]
be exact in \(\Ch^-(\FGauge\td\I)\). If two of its terms admit finite
resolutions by finite direct sums of the \(G_r\), then so does the
third.
\end{lemma}

\begin{proof}
The horseshoe lemma treats the case of \(A^\bullet\) and \(C^\bullet\).
If \(B^\bullet\) and \(C^\bullet\) have the required resolutions, lift
\(B^\bullet\to C^\bullet\) to a map between them; its shifted cone
resolves \(A^\bullet\). Similarly, the cone of a lifted map resolving
\(A^\bullet\to B^\bullet\) resolves \(C^\bullet\). All terms remain
finite direct sums of the \(G_r\).
\end{proof}

\begin{lemma}\label{lowering the rank}
Let \(M\) be a coherent \(F\)-gauge of level \(\I=[m,n]\) and positive
rank. Then \(M\) contains a torsion-free positive-rank subgauge
\(L\) admitting a finite resolution by finite direct sums of the
\(G_r\), \(r\in\I\).
\end{lemma}

\begin{proof}
We first reduce to the torsion-free case. Coherence and the level
condition imply that the torsion subgauge
\(M[p^\infty]\) is killed by some \(p^N\). Multiplication by \(p^N\)
then identifies \(M/M[p^\infty]\) with the subgauge \(p^NM\subseteq M\).
We henceforth assume that \(M\) is torsion free and put \(K=W[1/p]\).

Shifting the grading, we may also assume that \(\I=[0,n]\).

    For \(0\leqslant i\leqslant n\), write
    \[
        \Phi_i=\tF^i\tau\tF^{n-i},\qquad
        \Psi_i=\tV^{n-i}\tau^{-1}\tV^i.
    \]
    These operators satisfy \(\Phi_i\Psi_i=\Psi_i\Phi_i=p^n\).
    We have \(\tF\Phi_i=\Phi_{i+1}\tF\) and \(\tV\Phi_{i+1}=\Phi_i\tV\).
    Since \(\tF\tV=\tV\tF=p\), these maps identify the isocrystals
    \((M^i[1/p],\Phi_i)\). Their common slopes lie in \([0,n]\), since
    \(\Phi_i\Psi_i=p^n\) and both operators preserve \(M^i\).

    Choose a simple subisocrystal \(D\) of their common isocrystal,
    and write its slope as
    \[
        \lambda=l+\frac{r}{r+s}=\frac{e}{r+s},
    \]
    where \(l=\lfloor\lambda\rfloor\), \(r,s\) are integers with
    \(r\geqslant0\), \(s>0\), and \(r+s=\dim_KD\), and
    \(e=l(r+s)+r\). The integers \(r,s\) need not be coprime.

    If \(l<n\), choose \(j\in\{0,\ldots,n-1\}\) and \(t\in\Z\)
    uniquely by \(j+nt=l-e\). If \(l=n\), put \(j=n\) and
    \(t=-(r+s)\); the same equality holds.
    Choose \(0\neq x\in D\cap M^l\), and let \(L\subseteq M\)
    be the subgauge of level \(\I=[0,n]\) generated by \(x\), equivalently,
    the image of the morphism \(G_l\to M\) sending \(\alpha_l\) to \(x\). It is coherent and torsion free.
    Simplicity of \(D\) makes \(x\) cyclic, so \(L\) has rank \(r+s>0\).
    We claim that the surjection \(\pi:G_l\to L\), \(\alpha_l\mapsto x\),
    fits into an exact sequence
    \[
        0\longrightarrow G_j\xrightarrow{\ u_P\ }G_l
        \xrightarrow{\ \pi\ }L\longrightarrow0.
    \]

    To construct \(u_P\) and compute its image, set
    \[
        \mathcal A:=W_\sigma[\Phi,\Phi^{-1}],\qquad
        \mathcal B:=\mathcal A[1/p],
    \]
    where \(\Phi a=\sigma(a)\Phi\), and let \(\Phi\) act on
    \(D\) as \(\Phi_l\).
    The cyclic vector \(x\) has a monic relation
    \[
        P(\Phi)=\Phi^{r+s}+\sum_{a=0}^{r+s-1}c_a\Phi^a.
    \]
    We have \(P(\Phi)x=0\) and \(D\cong\mathcal B/\mathcal B P\).
    The Newton polygon theorem for cyclic difference modules gives
    \[
        v_p(c_a)\geqslant (r+s-a)\lambda,\qquad
        v_p(c_0)=(r+s)\lambda=e.
    \]
    Here and below \(v_p(0)=+\infty\).
    This applies over any perfect \(k\): simplicity makes \(x\) cyclic, without assuming that \(k\) is algebraically closed.\footnote{For the cyclic-vector and Newton-polygon statements, see K. S. Kedlaya,
    \href{https://kskedlaya.org/18.787/difference.pdf}{\emph{A little difference algebra}}, Lemma 1 and Proposition 18.}

    To compute the kernel integrally, we describe the graded pieces of
    the projective generators as lattices in \(\mathcal B\).
    For \(B=\sum_q b_q\Phi^q\in\mathcal B\), define
    \[
        v_0(B)=\min_q v_p(b_q),\qquad
        v_n(B)=\min_q\{v_p(b_q)+nq\},
    \]
    with both valuations equal to \(+\infty\) at zero. Both are multiplicative.
    For \(v_0\), this follows by reduction modulo \(p\), since \(k_\sigma[\Phi,\Phi^{-1}]\) has no zero divisors;
    the assertion for \(v_n\) follows by substituting \(p^n\Phi\) for \(\Phi\).
    For every integer \(b\), put
    \[
        \Lambda_b=\{B\in\mathcal B:v_0(B)\geqslant0,\ v_n(B)\geqslant-b\}
        =\bigoplus_{q\in\Z}W p^{\max\{0,-nq-b\}}\Phi^q.
    \]

    The lattices \(\Lambda_i\) give coordinates for the projective
    generators. For \(G_a\), consider the map of right
    \(W_\sigma[\Phi,\Psi]\)-modules
    \[
        \tF^{a-i}:G_a^i\longrightarrow G_a^a[1/p]\cong\mathcal B
        \qquad(0\leqslant i,a\leqslant n).
    \]
    Here the latter identification sends \(\Phi_a^q\alpha_a\) to \(\Phi^q\)
    for every \(q\in\Z\). This map identifies \(G_a^i\) with the
    lattice \(\Lambda_{i-a}\):
    \[
        \tF^{a-i}:G_a^i\xrightarrow{\sim}\Lambda_{i-a}\subseteq\mathcal B.
    \]
    To check this, if \(i\geqslant a\), the two strings of
    \Cref{definition of G_r}
    become \(\Phi^q\) for \(q\geqslant0\) and
    \(p^{-nq-(i-a)}\Phi^q\) for \(q<0\); if \(i\leqslant a\),
    they become \(\Phi^q\) for \(q>0\) and
    \(p^{-nq-(i-a)}\Phi^q\) for \(q\leqslant0\).
    Here \(\Psi=p^n\Phi^{-1}\). In these coordinates, \(\tF\) is inclusion and \(\tV\) is
    multiplication by \(p\).
    Thus the gauge \(G_a\) of level \(\I\) has the following lattice
    description:
    \[
    \begin{tikzcd}[column sep=large]
        \Lambda_{-a}
          \arrow[r,hook,"\mathrm{incl.}",shift left=0.6ex]
        & \cdots
          \arrow[l,"p",shift left=0.6ex]
          \arrow[r,hook,"\mathrm{incl.}",shift left=0.6ex]
        & \Lambda_{i-a}
          \arrow[l,"p",shift left=0.6ex]
          \arrow[r,hook,"\mathrm{incl.}",shift left=0.6ex]
        & \Lambda_{i+1-a}
          \arrow[l,"p",shift left=0.6ex]
          \arrow[r,hook,"\mathrm{incl.}",shift left=0.6ex]
        & \cdots
          \arrow[l,"p",shift left=0.6ex]
          \arrow[r,hook,"\mathrm{incl.}",shift left=0.6ex]
        & \Lambda_{n-a}
          \arrow[l,"p",shift left=0.6ex].
    \end{tikzcd}
    \]
    The endpoint identification is
    \(\tau:\Lambda_{n-a}\xrightarrow{\sim}\Lambda_{-a}\),
    \(Q\mapsto\Phi Q\).

    In these coordinates we can compute the kernel of \(\pi:G_l\to L\).
    Under the preceding identification, the kernel in gauge degree \(i\)
    of \(\pi^i:G_l^i\to L^i\) is \(\Lambda_{i-l}\cap\mathcal B P\).
    Indeed, if \(z\in G_l^i\) corresponds to \(Q\in\Lambda_{i-l}\),
    then \(\tF^{l-i}\pi^i(z)=Q(\Phi_l)x\) after inverting \(p\).
    Since \(L^i\) is torsion free and \(D\cong\mathcal B/\mathcal B P\),
    this vanishes precisely when \(Q\in\mathcal B P\).
    The coefficient inequalities for \(P\) imply \(v_0(P)=0\) and \(v_n(P)=e\).
    Indeed, the first minimum is attained by \(\Phi^{r+s}\), and
    \(v_p(c_a)+na\geqslant e+(n-\lambda)a\), with equality at \(a=0\).
    Multiplicativity therefore gives, for every integer \(b\),
    \begin{equation}\label{lowering rank lattice kernel}
        \Lambda_b\cap\mathcal B P=\Lambda_{b+e}P.
    \end{equation}
    In fact, \(BP\in\Lambda_b\) if and only if \(v_0(B)\geqslant0\) and \(v_n(B)\geqslant-b-e\).
    As a result, the kernel gauge of \(\pi:G_l\to L\) is
    \[
    \begin{tikzcd}[column sep=1.8em]
        \Lambda_{-l+e}P
          \arrow[r,hook,"\mathrm{incl.}",shift left=0.6ex]
        & \cdots
          \arrow[l,"p",shift left=0.6ex]
          \arrow[r,hook,"\mathrm{incl.}",shift left=0.6ex]
        & \Lambda_{i-l+e}P
          \arrow[l,"p",shift left=0.6ex]
          \arrow[r,hook,"\mathrm{incl.}",shift left=0.6ex]
        & \cdots
          \arrow[l,"p",shift left=0.6ex]
          \arrow[r,hook,"\mathrm{incl.}",shift left=0.6ex]
        & \Lambda_{n-l+e}P
          \arrow[l,"p",shift left=0.6ex].
    \end{tikzcd}
    \]
    Its endpoint identification is left multiplication by \(\Phi\).
    This gauge is isomorphic to \(G_j\): since \(j+nt=l-e\),
    right multiplication by \(\Phi^tP\) gives isomorphisms
    \[
        \Lambda_{i-j}\xrightarrow{\sim}\Lambda_{i-l+e}P,
        \qquad Q\longmapsto Q\Phi^tP.
    \]
    Indeed, \(\Lambda_{i-j}\Phi^t=\Lambda_{i-j-nt}=\Lambda_{i-l+e}\),
    and multiplication by the nonzero element \(P\) is injective.
    These maps commute with inclusion, multiplication by \(p\),
    and left multiplication by \(\Phi\), hence respect all gauge operators.

    The resulting morphism of gauges is displayed below; every upward
    map is right multiplication by \(\Phi^tP\).
    \[
    \begin{tikzcd}[column sep=1.8em,row sep=large]
        G_l:\quad
        & \Lambda_{-l}
          \arrow[r,hook,"\mathrm{incl.}",shift left=0.6ex]
        & \cdots
          \arrow[l,"p",shift left=0.6ex]
          \arrow[r,hook,"\mathrm{incl.}",shift left=0.6ex]
        & \Lambda_{i-l}
          \arrow[l,"p",shift left=0.6ex]
          \arrow[r,hook,"\mathrm{incl.}",shift left=0.6ex]
        & \cdots
          \arrow[l,"p",shift left=0.6ex]
          \arrow[r,hook,"\mathrm{incl.}",shift left=0.6ex]
        & \Lambda_{n-l}
          \arrow[l,"p",shift left=0.6ex]
        \\
        G_j:\quad \arrow[u,"u_P"]
        & \Lambda_{-j}
          \arrow[u,"\cdot\Phi^tP"]
          \arrow[r,hook,"\mathrm{incl.}",shift left=0.6ex]
        & \cdots
          \arrow[l,"p",shift left=0.6ex]
          \arrow[r,hook,"\mathrm{incl.}",shift left=0.6ex]
        & \Lambda_{i-j}
          \arrow[u,"\cdot\Phi^tP"]
          \arrow[l,"p",shift left=0.6ex]
          \arrow[r,hook,"\mathrm{incl.}",shift left=0.6ex]
        & \cdots
          \arrow[l,"p",shift left=0.6ex]
          \arrow[r,hook,"\mathrm{incl.}",shift left=0.6ex]
        & \Lambda_{n-j}
          \arrow[u,"\cdot\Phi^tP"]
          \arrow[l,"p",shift left=0.6ex].
    \end{tikzcd}
    \]
    It is injective, since \(\mathcal B\) has no zero divisors.
    Outside \([0,n]\), exactness follows from the level condition. Hence
    \[
        0\longrightarrow G_j\xrightarrow{\ u_P\ }G_l\xrightarrow{\ \pi\ }L\longrightarrow0
    \]
    is the required finite projective resolution.

\end{proof}

The rank reduction leaves coherent \(F\)-gauges of finite length.
We describe the simple objects that occur as composition factors and
give their resolutions before proving the classification.

\begin{itemize}
    \item For the interior simple \(S_r\), with \(m<r<n\), let
\(S_r\) be the gauge with
\[S_r^r=k,\qquad S_r^a=0\quad(a\neq r).\]
The natural map \(G_r\to S_r\) sends \(\alpha_r\mapsto 1\in S_r^r\). Its kernel is generated by the two relations
\(
\tV\alpha_r=0,\; \tF\alpha_r=0,
\)
lying in degrees \(r-1\) and \(r+1\). The only syzygy is
\(
\tF\beta_{r-1}-\tV\beta_{r+1}=0,
\)
since both terms map to \(p\alpha_r\). Hence \(S_r\) has the resolution
\[
0\to G_r\xrightarrow{} G_{r-1}\oplus G_{r+1}\xrightarrow{} G_r\to S_r\to0.
\]
Thus \(S_r\in\mathcal C\).

    \item 
The boundary simple \(S_{\tau}\) is defined by
\[
S_{\tau}^m=k,\qquad S_{\tau}^n=k,\qquad S_{\tau}^a=0\quad(m<a<n),
\]
with \(\tau:S_{\tau}^n\cong S_{\tau}^m\). The map \(G_m\to S_{\tau}\) sends \(\alpha_m\mapsto 1\in S_{\tau}^m\). The kernel is generated by
\(
\tF\alpha_m=0,\;\; \tV\tau^{-1}\alpha_m=0.
\)
The unique overlap relation is
\(
\tV\beta_{m+1}-\tau\tF\beta_{n-1}=0,
\)
because both terms map to \(p\alpha_m\). Therefore \(S_{\tau}\) has a finite resolution
\[
0\to G_m\to G_{m+1}\oplus G_{n-1}\to G_m\to S_{\tau}\to0.
\]
Thus \(S_{\tau}\in\mathcal C\).

   \item 
For the cyclic case, let \(N\) be a \(k\)-vector space equipped with a bijective
semilinear operator \(\Phi\), and define the constant cyclic
\(k\)-gauge \(S_{\Phi,N}\) by
\[
S_{\Phi,N}^a=N\quad\text{for all }a\in[m,n],
\]
with return operator
\[
\tau\tF^{n-m}:S_{\Phi,N}^m\longrightarrow S_{\Phi,N}^m
\]
identified with \(\Phi:N\to N\). If \(N\) is simple as a
\(\Phi\)-module, every nonzero \(x\in N\) is cyclic; let
\(P(\Phi)x=0\) be its minimal relation.
In this case
\(
S_{\Phi,N}\cong G_m/(p,P(\Phi))G_m.
\)
It therefore has the Koszul resolution
\[
0\to G_m\xrightarrow{\ d_2\ }G_m\oplus G_m\xrightarrow{\ d_1\ }G_m\to S_{\Phi,N}\to0,
\]
where
\(
d_1(a,b)=ap+bP(\Phi),
\)
and
\(
d_2(x)=(xP(\Phi),-xp).
\)
Hence \(S_{\Phi,N}\in\mathcal C\).

\item 
Similarly, consider the constant cyclic \(k\)-gauge \(S_{\Psi,N}\).
\end{itemize}

\begin{proposition}\label{finite length F-gauge structures}
Let \(M\) be a coherent \(F\)-gauge whose graded pieces have finite
\(W\)-length. Then \(M\) admits a composition series, and every simple
factor is of one of the following types:
\[
S_r,\qquad S_\tau,\qquad S_{\Phi,N},\qquad S_{\Psi,N},
\]
where \(N\) is simple for the indicated semilinear operator.
\end{proposition}

\begin{proof}
Since \(M\) is supported in finitely many degrees,
\[
    \ell(M):=\sum_{r\in \I} \ell_W(M^r)<\infty .
\]
Every sub-\(F\)-gauge structure \(N\subset M\) satisfies
\[
    \ell(N)=\sum_{r\in\I} \ell_W(N^r)\leq \ell(M),
\]
The bound \(\ell(M)\) therefore makes the category both Artinian and
Noetherian, and induction on \(\ell(M)\) gives a composition series.

It remains to identify the simple objects. Let \(M\) be simple. Since
multiplication by \(p\) is compatible with all the structure maps, \(pM\subset M\)
is a sub-\(F\)-gauge structure. By simplicity, either \(pM=0\) or \(pM=M\).
The second possibility is excluded because \(M\) has finite \(W\)-length.
Thus
\[
    pM=0.
\]
Hence a simple object is a finite-dimensional graded \(k\)-vector
space with induced operators
\[
    \tF,\ \tV,\tau,\tau^{-1}. 
\]
satisfying the reduced relations. Their images and kernels are
subgauges.

Every nonzero homogeneous element generates \(M\), so we can classify
the simple objects by the behavior of the structure maps on such an
element.

If all structure maps out of \(0\neq x\in M^r\) vanish, then
\(m<r<n\), \(\tF x=\tV x=0\), and simplicity forces
\[M=kx
\]
concentrated in degree \(r\). This is the simple object \(S_r\).

If \(M^m=0\), then the support of \(M\) is contained in a strictly smaller
interval. By induction on the length of the interval, \(M\) is one of the
skyscraper simples. Hence
\(
    M\cong S_r
\)
for some \(r\in [m+1,n]\).

Assume now that \(M^m\neq 0\). Consider the two structure maps out of \(M^m\),
namely
\(
    \tau\tF^{n-m}:M^m\longrightarrow M^{m},\;\;
    \tV^{n-m}\tau^{-1}:M^m\longrightarrow M^{m}.
\)
If both maps vanish, simplicity forces \(M\) to be concentrated at the
boundary; thus
\(
    M\cong S_{\tau}.
\)

Suppose now that at least one return map is nonzero. If
\(
    \Phi:M^m\longrightarrow M^{m}
\)
is nonzero, its image generates a nonzero subgauge and therefore all
of \(M\). The iterates of \(\tF\), together with \(\tau\), identify the
graded pieces cyclically:
\[
    M^m \xrightarrow{\tF} M^{m+1}
    \xrightarrow{\tF} \cdots
    \xrightarrow{\tF} M^n
    \xrightarrow{\tau} M^m .
\]
Consequently \(M\) is a constant cyclic \(k\)-gauge. Its monodromy around the
cycle is the semilinear automorphism
\(
    \tau\tF^{\,n-m}:M^m\longrightarrow M^m .
\)
Let
\(
    N:=M^m
\)
equipped with this semilinear operator. If \(N\) had a nontrivial
\(\Phi\)-stable subspace, then propagating it around the cycle would give a
nontrivial sub-\(F\)-gauge structure of \(M\), contradicting simplicity.
Therefore \(N\) is simple as a \(\Phi\)-module. Hence
\(
    M\cong S_{\Phi,N}.
\)
If \(\Phi=0\) but \(\Psi\neq0\), the same argument with \(\Psi\) gives
\(
    M\cong S_{\Psi,N}
\)
for some simple \(\Psi\)-module \(N\).

These cases exhaust the simple objects:
\[
    S_r,\qquad S_{\tau},\qquad S_{\Phi,N},\qquad S_{\Psi,N}.
\]
\end{proof}

Each simple finite-length coherent gauge lies in \(\mathcal C\), and
extension closure therefore puts every finite-length coherent gauge in
\(\mathcal C\). Combined with the positive-rank reduction, this proves
the assertion for coherent \(F\)-gauges.

Resolving the cohomology objects and assembling the resolutions along
the standard truncation triangles proves the statement for bounded
complexes.
\end{proof}

To see how these resolutions depend on the chosen interval, consider
the Dieudonn\'e module \(E_{1/2}\), generated by \(e\) with relation
\(
Fe=Ve=:u.
\)
We record its Nygaard \(F\)-gauge at three successive levels.

For \(\I=[0,1]\), we have
\(
M^0=We\oplus Wu,\; M^1=pWe\oplus Wu.
\)
The map \(G_0\to M\) sending \(\alpha_0\) to \(e\) is surjective.
Writing
\[
\Phi=\tau\tF,\qquad \Psi=\tV\tau^{-1},
\]
the relation \(F=V\) becomes
\(
\Phi\alpha_0=\Psi\alpha_0.
\)
Hence \(M\) admits the resolution
\[
0\to G_0\xrightarrow{\ \Phi-\Psi\ }G_0\longrightarrow  M\to0.
\]

Now let \(\I=[0,2]\). Then
\(
M^0=We\oplus Wu,\; M^1=pWe\oplus Wu,\; M^2=p^2We\oplus pWu,
\)
with \(\tau(p^2e)=u\) and \(\tau(pu)=e\). The map
\(G_1\to M\) sending \(\alpha_1\) to \(u\) is surjective: its translates
\(\tF\alpha_1\) and \(\tau\tF\alpha_1\) map to \(pu\in M^2\) and
\(e\in M^0\), respectively.
Writing
\[
\Phi=\tF\tau\tF,\qquad\Psi=\tV\tau^{-1}\tV,
\]
one computes
\(
(\Phi^2-p)\alpha_1\in\ker(G_1\to M),
\)
and
\(
(p\Phi-\Psi)\alpha_1\in\ker(G_1\to M).
\)
Both arise from the single relation
\(
\tau\tF\Phi\alpha_1-\tV\alpha_1=0.
\)
Equivalently, \(M\) admits the resolution
\[
0\to G_0\xrightarrow{\ \tau\tF\Phi-\tV\ }G_1\longrightarrow  M\to0.
\]

Finally, let \(\I=[-1,2]\). Then
\(
M^{-1}=We\oplus Wu,\; M^0=We\oplus Wu,\; M^1=pWe\oplus Wu,\; M^2=p^2We\oplus pWu,
\)
where \(\tV:M^0\to M^{-1}\) is the identity and \(\tF:M^{-1}\to M^0\) is multiplication by \(p\).
Consider the surjection
\[
G_0\oplus G_1\longrightarrow M,\qquad
\alpha_0\longmapsto e,\quad\alpha_1\longmapsto u.
\]
The kernel is generated by the relations
\(
\tau\tF\alpha_1=\tV\alpha_0\) and \(
\tau\tF^2\alpha_0=\tV^2\alpha_1.
\)
Hence \(M\) admits the resolution
\[
0\to G_0\oplus G_0\to G_0\oplus G_1\to M\to0,
\]
where the first map sends the generators \(\gamma,\delta\in G_0\) to
\(\tau\tF\alpha_1-\tV\alpha_0\) and
\(\tau\tF^2\alpha_0-\tV^2\alpha_1\), respectively.

\subsection{F-gauge structure and \(A\)-module}
\label{subsec:f-gauge-structure-and-a-module}

At a fixed finite level \(\I=[m,n]\), the projective generators \(G_r,r\in \I\) let us express the
gauge category as modules over a single associative ring. The
individual gauge degrees can be recovered by its idempotent elements.

Let \(\I=[m,n]\). Define
\(
A:=\mathrm{End}_{\FGauge\td{\I}}\bigl(\oplus_{r\in\I}G_r\bigr)^{\mathrm{op}}.
\)
Since \(G_r\) represents evaluation in degree \(r\), and evaluation
preserves direct sums, \(\oplus_{r\in\I}G_r\) is a small projective generator.
Morita theory \cite[Theorem 2.5]{Schwede} therefore gives an equivalence
\[
\FGauge\td{\I}\simeq A\text{-}\mathrm{Mod},
\]
where the functor is
\(
M\longmapsto \Hom_{\FGauge\td{\I}}\bigl(\oplus_{r\in\I}G_r,M\bigr).
\)
More explicitly,
\[
\Hom_{\FGauge\td{\I}}\bigl(\oplus_{r\in\I}G_r,M\bigr)
\simeq
\bigoplus_{r\in\I}\Hom_{\FGauge\td{\I}}(G_r,M)
\simeq
\bigoplus_{r\in\I}M^r.
\]
The left \(A\)-module structure comes from composition. If \(e_r\in A\) denotes the idempotent corresponding to the summand \(G_r\subset\oplus_{s\in\I}G_s\), then
\(
e_r\Hom_{\FGauge\td{\I}}\bigl(\oplus_{r\in\I}G_r,M\bigr)\simeq M^r.
\)
Thus the idempotents \(e_r\) recover the graded pieces, while the off-diagonal components of \(A\) encode the maps \(\tF\), \(\tV\), and \(\tau\). Under this equivalence, \(G_r\) corresponds to the standard projective \(Ae_r\). 

Explicitly, as an additive group,
\[
A\cong\oplus_{r,j\in\I}G_r^j,
\qquad e_jAe_r\cong G_r^j.
\]
An element \(a_r^j\in G_r^j\) represents the unique gauge morphism
\(f_{a_r^j}:G_j\to G_r\) sending \(\alpha_j\) to \(a_r^j\).
In particular, \(e_r\) corresponds to \(\alpha_r\in G_r^r\).
For \(a=(a_r^j)\) and \(b=(b_r^j)\), multiplication is given by
\[
(ab)_r^j=\sum_{s\in\I}f_{b_r^s}(a_s^j).
\]
Indeed, the opposite-ring convention makes \(a_s^jb_r^s\) correspond
to \(f_{b_r^s}\circ f_{a_s^j}:G_j\to G_r\).
In terms of gauge operators, if \(a_s^j=P_s^j\alpha_s\) and
\(b_r^s=Q_r^s\alpha_r\), then
\[
f_{b_r^s}(a_s^j)=P_s^j(Q_r^s\alpha_r),
\qquad (ab)_r^j=\sum_{s\in\I}P_s^jQ_r^s\alpha_r.
\]
Thus the operator expressions multiply in the written order, with
\(Q_r^s\) acting first; the reversal concerns composition of maps
between the projective generators.

\begin{remark}[Tensor-product interpretation of \(\rmt\)]\label{rmt-tensor-product}
Using the preceding Morita equivalence, we also write \(M\) for the
left \(A\)-module associated to an \(F\)-gauge \(M\) of level \(\I\), with \(M^j=e_jM\).
The \(F\)-gauge \(\oplus_{r\in\I}G_r\) corresponds to \(A\) itself, regarded as a left \(A\)-module.
Thus \(\rmt(A)\) means that we first regard this module as
the \(F\)-gauge \(\oplus_{r\in\I}G_r\). Set
\[
\mathcal T:=\rmt(A)\cong\bigoplus_{j\in\I}\rmt(G_j).
\]
This is a complex of graded left \(R\)-modules. 

For \(a\in A\), right multiplication \(r_a:z\mapsto za\) is a
left \(A\)-linear endomorphism of \(A\). It then induces a right \(A\)-action
on \(\mathcal T\) by \(z\cdot a:=\rmt(r_a)(z)\). This action
commutes with the left \(R\)-action and the complex differential,
thus \(\mathcal T\) is a complex of graded left \(R\)-modules and
right \(A\)-modules with commuting actions. The tensor product
\(\mathcal T\otimes_A M\) is taken termwise:
its term in complex degree \(q\) is \(\mathcal T^q\otimes_A M\),
with the inherited left \(R\)-action and differential
\(\delta\otimes1\). It is therefore a complex of graded left
\(R\)-modules. As a matter of fact,
there is a natural isomorphism of \(R\)-complexes
\[
\theta_M:\mathcal T\otimes_A M\xrightarrow{\sim}\rmt(M).
\]
Indeed, by the adjunction proved in \Cref{adjunction between FG and t},
\(\rmt\) preserves direct sums and cokernels. Apply it to a free
presentation
\[
A^{(S_1)}\longrightarrow A^{(S_0)}\longrightarrow M\longrightarrow0.
\]
Both \(\rmt(M)\) and \(\mathcal T\otimes_A M\) are the cokernel
of the same map \(\mathcal T^{(S_1)}\to\mathcal T^{(S_0)}\): its
matrix entries act by the right \(A\)-action just defined.

The description on generators explains how tensor balancing recovers
the defining relations of the functor.
Let \(\alpha_j\in G_j^j\) be the canonical generator, corresponding
to \(e_j\in Ae_j\), and write
\[
b_{i,j}:=\beta_i\otimes\alpha_j
\in\rmt(G_j)\cong\mathcal T e_j.
\]
Viewed inside \(\rmt(A)\), this is the symbol
\(\beta_i\otimes e_j\). For \(x\in M^j\), 
\begin{equation}\label{rmt-tensor-generators}
\theta_M(b_{i,j}\otimes_A x)=\beta_i\otimes x.
\end{equation}
More generally, let \(x\in M^s=e_sM\). Tensor balancing gives
\[
b_{i,j}\otimes_A x
=b_{i,j}\otimes_A e_sx
=(b_{i,j}e_s)\otimes_A x.
\]
Since \(\alpha_j\) corresponds to \(e_j\), we have
\(b_{i,j}e_s=\beta_i\otimes(\alpha_je_s)\).
The orthogonality of the idempotents makes this \(b_{i,j}\) if
\(j=s\) and zero otherwise. Hence
\[
\theta_M(b_{i,j}\otimes_A x)=
\begin{cases}
\beta_i\otimes x,&j=s,\\
0,&j\ne s.
\end{cases}
\]

Tensor balancing encodes the gauge relations. For
\(a\in e_\ell A e_j\), right multiplication gives
\(b_{i,\ell}a=\beta_i\otimes a\), and hence, for \(x\in M^j\),
\[
b_{i,\ell}\otimes_A(ax)
=(b_{i,\ell}a)\otimes_A x
=(\beta_i\otimes a)\otimes_A x.
\]
The defining relations of \(\rmt(A)\) express the last symbol in
terms of \(b_{i,j}\). For example, for \(m\leq i<n\), the arrow
\(\tF_i\in e_{i+1}Ae_i\) satisfies
\[
b_{i,i+1}\tF_i=F b_{i,i},
\]
so balancing recovers
\(\beta_i\otimes\tF_i x=F(\beta_i\otimes x)\).
Likewise, the scalar relation is encoded by
\[
b_{i,j}(\lambda e_j)=
\begin{cases}
\lambda b_{i,j},&j\leq i,\\
\sigma(\lambda)b_{i,j},&j>i,
\end{cases}
\qquad \lambda\in W.
\]
The relations involving \(\tV\), \(\tau\), and the differentials
are inherited in the same way from \(\mathcal T\); in particular,
\(\delta(z\otimes_A x)=(\delta z)\otimes_A x\).
Thus the inverse of \(\theta_M\) sends
\(\beta_i\otimes x\) to \(b_{i,j}\otimes_A x\) for \(x\in M^j\).
These tensors are taken in the whole \(\mathcal T\otimes_A M\):
the right \(A\)-action can move between the summands
\(\mathcal T e_j\), which are not generally right \(A\)-submodules.
\end{remark}

\section{The functor \(\widehat{\rmt}\), and the independence of interval theorem}
\label{sec:independence-of-interval}

The finite-level constructions above depend on the chosen interval:
its endpoints enter both the gauge algebra and the definition of the
completed functor. We now examine which parts of this dependence remain
on coherent objects. The derived output is unchanged when the interval
is enlarged, allowing us to realize diagonal complexes by canonical
complexes of \(R\)-modules.

\subsection{The ring \(\widehat A\)}
\label{subsec:completed-ring-a}

Let \(\widehat A:=\varprojlim_n A/p^nA\). Restriction of scalars along \(A\to\widehat A\) gives a faithful functor
\[
\widehat A\text{-}\mathrm{Mod}\longrightarrow A\text{-}\mathrm{Mod}.
\]
Its essential image consists of the \(A\)-modules whose action extends
to \(\widehat A\). For a \(p\)-adically complete and separated
\(A\)-module, such an extension, when it exists continuously, is
unique. Mere completeness of the underlying module is not by itself a
description of the essential image.

If \(M\) is a coherent \(F\)-gauge, then
\(
\Hom\bigl(\oplus_{r\in\I}G_r,M\bigr)\simeq \bigoplus_{r\in\I}M^r
\)
is finitely generated over \(W\), hence \(p\)-adically complete. Therefore the \(A\)-action on \(\Hom\bigl(\oplus_{r\in\I}G_r,M\bigr)\) extends continuously to an action of \(\widehat A\). Thus coherent \(F\)-gauges define a full subcategory of \(\widehat A\text{-}\mathrm{Mod}\).

For each \(r\in\I\), define the \(p\)-adic completion
\(
\widehat{G}_r:=\varprojlim_n G_r/p^nG_r.
\)
Then
\(
\oplus_{r\in\I}\widehat{G}_r
\)
carries a natural left \(\widehat A\)-module structure. Moreover,
\[
\widehat A\simeq \mathrm{End}_{\widehat A}\bigl(\oplus_{r\in\I}\widehat{G}_r\bigr)^{\mathrm{op}}.
\]
The modules \(\widehat{G}_r\) are the standard projective \(\widehat A\)-modules corresponding to the idempotents \(e_r\in\widehat A\):
\(
\widehat{G}_r=\widehat{A e_r}\simeq \widehat A e_r.
\)
In particular, the family \(\{\widehat{G}_r\}_{r\in\I}\) forms a set of projective generators of the category of left \(\widehat A\)-modules.

Let \(M\) be a coherent \(F\)-gauge of level \(\I\), and choose a
projective resolution, with \(a_{i,r}\in\Z_{\geqslant0}\),
\[
\cdots\to \bigoplus_{r\in\I}G_r^{\oplus a_{-2,r}}
\to \bigoplus_{r\in\I}G_r^{\oplus a_{-1,r}}
\to \bigoplus_{r\in\I}G_r^{\oplus a_{0,r}}
\to M\to0.
\]
Since \(M\) is coherent, each \(M^r\) is a finitely generated \(W\)-module, hence \(p\)-adically complete. Thus \(M\) naturally carries a \(\widehat A\)-module structure.

Replacing every \(G_r\) by \(\widehat{G}_r\) extends the differentials
uniquely by \(p\)-adic continuity and gives a complex of
\(\widehat A\)-modules
\[
\cdots\to \bigoplus_{r\in\I}\widehat{G}_r^{\oplus a_{-2,r}}
\to \bigoplus_{r\in\I}\widehat{G}_r^{\oplus a_{-1,r}}
\to \bigoplus_{r\in\I}\widehat{G}_r^{\oplus a_{0,r}}
\to M\to0.
\]
Each term is projective because
\(\widehat{G}_r\simeq\widehat A e_r\). It remains to check exactness:
extending the maps by continuity alone does not establish this.

\begin{lemma}\label{Exactness after completion}
The above termwise \(p\)-adic completion
is a resolution of \(M\) by finite direct sums of the completed projective generators \(\widehat{G}_r\).
\end{lemma}

\begin{proof}
Let \(d:P_a\to P_{a-1}\) be a differential and set \(K=\ker(d)\).
It suffices to prove that the sequence
\(
0\to \widehat K\to \widehat{P_a}\to \widehat{\operatorname{im}(d)}\to0
\)
is exact. 

We use Artin--Rees in the following form. If \(B\) is left Noetherian,
\(p\in Z(B)\), \(P\) is a finitely generated left \(B\)-module, and
\(K\subseteq P\) is any submodule, then for some \(c\geq0\),
\[
K\cap p^{N+c}P=p^N(K\cap p^cP)\subseteq p^NK
\qquad(N\geq0).
\]
Indeed, the Rees ring
\[
\mathrm{Rees}_{(p)}(B)=\bigoplus_{N\geq0}p^NBt^N
\]
is a quotient of the central polynomial ring \(B[T]\), hence left
Noetherian. The Rees module \(\bigoplus_{N\geq0}p^NPt^N\) is generated
in degree zero by generators of \(P\). Its graded submodule
\(\bigoplus_{N\geq0}(K\cap p^NP)t^N\) therefore has finitely many
homogeneous generators \(x_\alpha t^{n_\alpha}\). Let
\(c=\max_\alpha n_\alpha\), taking \(c=0\) if \(K=0\). In degree
\(N+c\), these generators express every element as \(p^N\) times an
element of \(K\cap p^cP\). The reverse inclusion is immediate.

For the cyclic operator ring
\[
B=W_\sigma[\Phi,\Psi]/
(\Phi\Psi=\Psi\Phi=p^{n-m})
\]
\(B\) is left Noetherian and \(p\) is central. Since \(P_a^j\) and
\(K^j\) are finitely generated over \(B\), Artin--Rees gives
\(
p^{N+c}P_a^j\cap K^j\subset p^NK^j.
\)
Thus the \(p\)-adic topology on \(K^j\) agrees with the topology induced from \(P_a^j\). Equivalently,
\(
\widehat{K^j}\longrightarrow \widehat{P_a^j}
\)
is injective with image the kernel of
\(
\widehat{P_a^j}\longrightarrow \widehat{\operatorname{im}(d)^j}.
\)
This agreement of the induced and intrinsic \(p\)-adic topologies
on \(K^j\) proves exactness after completion in every gauge degree.

Finally, each \(M^j\) is already \(p\)-adically complete, and
completion commutes with finite direct sums:
\[
\widehat{P_q}\simeq
\bigoplus_{r\in\I}\widehat{G}_r^{\oplus a_{q,r}}.
\]
Hence \(\widehat P_\bullet\to M\) is the required resolution.
\end{proof}

The same argument applies to right modules. We record its flatness
consequence, which will be used repeatedly.

\begin{lemma}\label{flatness of central completion}
Let \(H\) be a left and right Noetherian ring and let \(p\in Z(H)\).
Then \(\widehat H^{\,p}:=\varprojlim_{a\geq1}H/p^aH\) is left and
right Noetherian, and is flat as a left and as a right \(H\)-module.
\end{lemma}

\begin{proof}
The Artin--Rees estimate above applies to finite modules on both sides.
For a short exact sequence \(0\to K\to P\to Q\to0\) of finite
modules, it shows that \(p^aK\) and \(K\cap p^aP\) are cofinal
filtrations. For each \(a\), the sequence
\[
0\to K/(K\cap p^aP)\to P/p^aP\to Q/p^aQ\to0
\]
is exact, and the transition maps on its left terms are surjective.
Taking inverse limits therefore shows that completion is exact.

Applying this exactness to
\(0\to p^aH\to H\to H/p^aH\to0\) and
\(H\twoheadrightarrow p^aH\) identifies
\(\widehat{p^aH}^{\,p}=p^a\widehat H^{\,p}\), hence
\(\widehat H^{\,p}/p^a\widehat H^{\,p}\cong H/p^aH\).
The Rees ring \(\bigoplus_{a\geq0}p^aHt^a\) is a quotient of
\(H[T]\), hence left and right Noetherian. Thus
\(\operatorname{gr}_p\widehat H^{\,p}\cong\operatorname{gr}_pH\)
is left and right Noetherian. Lifting generators of graded ideals
through the complete \(p\)-adic filtration shows that
\(\widehat H^{\,p}\) is left and right Noetherian.

For a finite right \(H\)-module \(L\), a finite presentation and this
exactness give \(L\otimes_H\widehat H^{\,p}\cong\widehat L^{\,p}\).
If \(J\subseteq H\) is a right ideal, the map
\(J\otimes_H\widehat H^{\,p}\to\widehat H^{\,p}\) is consequently the
injection \(\widehat J^{\,p}\to\widehat H^{\,p}\). The right-ideal
criterion for flatness proves that \(\widehat H^{\,p}\) is flat as a
left \(H\)-module. Applying the same argument to left ideals proves
right flatness.
\end{proof}

In particular, \(\widehat R^{\,1}\) is flat as a right
\(\widehat R^{\,0}\)-module. Unique \(F\)-\(V\) expansions and \(FdV=d\)
identify \(\widehat R^{\,1}\) with the \(p\)-adic completion of
\(\widehat R^{\,0}[V^{-1}]\) as right \(\widehat R^{\,0}\)-modules:
\(dV^j\) corresponds to \(V^j\) and
\(F^jd\) to \(V^{-j}\). Modulo \(p^a\), both sides have the same
skew Laurent expansions; the Raynaud relations verify right
\(\widehat R^{\,0}\)-linearity. The element \(V\) is regular and normal;
\(\widehat R^{\,0}\) is \(V\)-adically complete, and
\(\operatorname{gr}_V\widehat R^{\,0}\cong(k_\sigma[F])[v;\sigma^{-1}]\)
is left and right Noetherian. Hence \(\widehat R^{\,0}\) and its Ore
localization \(\widehat R^{\,0}[V^{-1}]\) are left and right
Noetherian. The latter is right-flat over \(\widehat R^{\,0}\), so
\Cref{flatness of central completion} and transitivity of flatness
give the claim.

\subsection{\(\widehat R\)-modules of level \(\I\)}
\label{subsec:completed-r-modules-of-level-i}

To construct the completed functor, we first specify the coefficient rings at each internal degree. At the top degree, recall from \Cref{R-modules of level I} that
\[R^{\mathrm{top}}
:=R^0[F^{-1}]
\simeq W_\sigma[F,F^{-1}],\]
where \(V\) corresponds to \(pF^{-1}\). Define its \(p\)-adic completion by
\[\widehat R^{\mathrm{top}}
:=\varprojlim_i R^{\mathrm{top}}/p^iR^{\mathrm{top}}
=W_\sigma\{F,F^{-1}\}.\]
Thus an element of \(\widehat R^{\mathrm{top}}\) is a restricted Laurent series
\(\sum_{r\in\Z}a_rF^r\) such that, for every \(i\geqslant 0\), only finitely many coefficients \(a_r\)
do not lie in \(p^iW\). There is a natural continuous homomorphism
\[\iota:\widehat R^{\,0}\longrightarrow \widehat R^{\mathrm{top}},
\qquad V\longmapsto pF^{-1}.\]

\begin{definition}\label{completed R-modules of level I}
Let \(\I=[m,n]\). An \(\widehat R\)-module of level \(\I\) is a graded \(\widehat R\)-module of the form
\[M^m\xrightarrow{d}M^{m+1}\xrightarrow{d}\cdots
\xrightarrow{d}M^n\]
such that \(M^i\) is an algebraic left \(\widehat R^{\,0}\)-module for
\(i<n\), whereas \(M^n\) is an algebraic left
\(\widehat R^{\mathrm{top}}\)-module. The \(\widehat R^{\,0}\)-action on the
top term is induced by \(\iota\). 
Morphisms are componentwise homomorphisms over these coefficient rings that commute with the differential.
We denote this abelian category by
\(\widehat R\td\mathrm{Mod}\td\I\).

\end{definition}

The top term thus carries an action of restricted Laurent series,
not merely of \(R^{\mathrm{top}}\). 
We next show that \(\FG^{\Tot}\) defines a functor from
\(\Ch(\widehat R\td\I)\) to \(\Ch(\widehat A\td\mathrm{Mod})\). 
Recall that \(\FGauge\td\I\) is equivalent to the category of left \(A\)-modules,
where \[A=\mathrm{End}_{\FGauge\td\I}(\oplus_{r\in\I}G_r)^{\mathrm{op}}.\]
Write \(e_r\in A\) for the idempotent corresponding to
\(G_r\). Thus, for a left \(A\)-module \(M\), its gauge-degree
\(r\) part is \(M^r=e_rM\), and an element of \(e_s A e_j\) acts from \(M^j\) to \(M^s\).

\begin{definition}\label{definition of rho}
For \(i,j,s\in\I\), define the coordinate map
\[
\rho_i^{j,s}:e_sAe_j\longrightarrow
\begin{cases}
R^0,&i<n,\\
R^{\mathrm{top}},&i=n,
\end{cases}
\]
characterized as follows. Write \([x]^{j,i}\) for the element of
\(\FG(N)^{j,i}\) corresponding to \(x\in N^i\) 
under the identification with \(N^i\) or \(\sigma_*(N^i)\) in \Cref{definition of FG}(1). Then
\begin{equation}\label{coordinate characterization of rho}
a\cdot [x]^{j,i} 
=[\rho_i^{j,s}(a)\cdot x]^{s,i}
\in \FG(N)^{s,i}
\end{equation}
for every \(R\)-module \(N\) of level \(\I\). Here \(j\) and \(s\)
are the source and target gauge degrees, whereas \(i\) is the complex
degree.
\end{definition}

For the elementary gauge operators, \Cref{definition of FG}(2) gives
\[
\rho_i^{j,j+1}(\tF)=
\begin{cases}
1,&i<j,\\
F,&i=j,\\
p,&i>j,
\end{cases}
\qquad
\rho_i^{j+1,j}(\tV)=
\begin{cases}
p,&i<j,\\
V,&i=j,\\
1,&i>j.
\end{cases}
\]
The operator \(\tau\) is treated using \Cref{definition of FG}(4), where it
is the identity after identifying the two asymptotic total complexes.
Composition and addition determine \(\rho_i^{j,s}(a)\) for every
\(a\in e_sAe_j\). Since these maps come from the \(A\)-action on
\(\FG(N)\), they are independent of the expression of \(a\) as a sum
of paths.

\begin{lemma}\label{relations satisfied by rho}
The coordinate maps \(\rho_i^{j,s}\) are additive and satisfy, for
\(a\in e_sAe_j\), \(b\in e_tAe_s\),
\[
\rho_i^{j,t}(ba)
=\rho_i^{s,t}(b)\rho_i^{j,s}(a).
\]
Moreover, if \(M\) is an \(F\)-gauge and \(m\in M^j\), then in \(\rmt(M)\)
one has
\begin{equation}\label{rho relation in rmt}
\beta_i\otimes(am)
=\rho_i^{j,s}(a)(\beta_i\otimes m).
\end{equation}
\end{lemma}

\begin{proof}
Clear from \Cref{definition of FG} and \Cref{rmt}. 
\end{proof}

We extend the coordinate maps to the completions by continuity. 
For \(i<n\), the map \(\rho_i^{j,s}\) is continuous from the \(p\)-adic
topology on \(e_sAe_j\) to the defining topology on \(R^0\). Indeed,
for every \(l\geqslant0\),
\[\rho_i^{j,s}(p^la)=p^l\rho_i^{j,s}(a),\quad
\text{and} \quad p^lR^0=V^lF^lR^0\subseteq V^lR^0.\]
For \(i=n\), it is continuous for the \(p\)-adic topologies on
\(e_sAe_j\) and \(R^{\mathrm{top}}\). It therefore extends uniquely, in every
degree, to an additive map
\[
\widehat\rho_i^{j,s}:e_s\widehat A e_j\longrightarrow
\begin{cases}
\widehat R^{\,0},&i<n,\\
\widehat R^{\mathrm{top}},&i=n.
\end{cases}
\]
The composition identity extends by continuity: for
\(a\in e_s\widehat A e_j\) and \(b\in e_t\widehat A e_s\),
\[
\widehat\rho_i^{j,t}(ba)
=\widehat\rho_i^{s,t}(b)\widehat\rho_i^{j,s}(a).
\]
Indeed, it holds on the dense subring \(A\).

Let \(N\in\Ch(\widehat R\td\I)\). In total complex degree
\(q\), an element \(x\in\FG^{\Tot}(N)^{j,q}\) is a collection
\((x_i)_{i\in\I}\), with \(x_i\in\sigma_*N^{i,q-i}\) for
\(i<j\) and \(x_i\in N^{i,q-i}\) for \(i\geqslant j\).
For \(a\in e_s\widehat A e_j\), define
\[
(a\cdot x)_i:=\widehat\rho_i^{j,s}(a)x_i.
\]
The composition identity gives a left \(\widehat A\)-action in each
total degree. 
The construction is natural in \(N\) and therefore defines a functor
\[
\FG^{\Tot}:\Ch(\widehat R\td\I)
\longrightarrow\Ch(\widehat A\td\mathrm{Mod}).
\]
All series are evaluated in the coefficient rings, so no completeness
of the underlying modules is required.

\subsection{The functor \(\widehat{\rmt}\)}
\label{subsec:functor-completed-rmt}

Taking degree-zero cycles
also gives \(Z^0\circ\FG^{\Tot}:\Ch(\widehat R\td\I)
\to\widehat A\td\mathrm{Mod}\), as needed for the adjunction below.

\begin{definition}\label{definition of widehat rmt}
Let \(M\) be a left \(\widehat A\)-module. Define
\[
\widehat\rmt(M)\in\Ch(\widehat R\td\I)
\]
by the same generators as in \Cref{rmt}. Thus it is generated by symbols
\[
\beta_i\otimes m,
\qquad i\in\I,\quad m\in M^j=e_jM,
\]
where \(\beta_i\otimes m\) has Raynaud-module degree \(i\) and complex
degree \(-i\), with the same truncation in degree \(-n\) as in \Cref{rmt}.
The coefficient ring in Raynaud-module degree \(i\) is \(\widehat R^{\,0}\)
for \(i<n\) and \(\widehat R^{\mathrm{top}}\) for \(i=n\).
Impose relations \Cref{rmt}(1), \Cref{rmt}(2), and \Cref{rmt}(4),
interpreted in the completed level ringoid, and replace relation
\Cref{rmt}(3) by the following completed relation: for every
\(a\in e_s\widehat A e_j\) and \(m\in M^j\),
\[
\beta_i\otimes(am)
=\widehat\rho_i^{j,s}(a)(\beta_i\otimes m).
\]
Equivalently, if \(a\) is represented by a \(p\)-adically convergent series
of finite paths, take its image under
\(\widehat\rho_i^{j,s}\) in the corresponding coefficient ring, and let
that element act on \(\beta_i\otimes m\). The series is thus evaluated in
the coefficient ring, with no infinite sum taken inside the algebraically
free module.
\end{definition}

\begin{remark}[Tensor-product interpretation of \(\widehat\rmt\)]
\label{widehat-rmt-tensor-product}
The argument of \Cref{rmt-tensor-product} also applies to
\(\widehat\rmt\), using the adjunction proved below. Put
\[
\mathcal T_{\widehat A}:=\widehat\rmt(\widehat A)
\cong\bigoplus_{j\in\I}\widehat\rmt(\widehat G_j),
\]
where the left regular module \(\widehat A\) corresponds to the
\(F\)-gauge \(\bigoplus_{j\in\I}\widehat G_j\). Right multiplication
on \(\widehat A\) induces a right \(\widehat A\)-action on
\(\mathcal T_{\widehat A}\), commuting with its left coefficient
actions and differential. Since \(\widehat\rmt\) preserves direct
sums and cokernels, there is a natural isomorphism
\[
\mathcal T_{\widehat A}\otimes_{\widehat A}M
\xrightarrow{\sim}\widehat\rmt(M)
\]
for every algebraic left \(\widehat A\)-module \(M\). This is the
ordinary algebraic tensor product, not a completed tensor product;
no completeness assumption on \(M\) is imposed.

Writing \(\widehat b_{i,j}:=\beta_i\otimes\alpha_j
\in\widehat\rmt(\widehat G_j)\cong\mathcal T_{\widehat A}e_j\),
the generator correspondence is
\[
\widehat b_{i,j}\otimes_{\widehat A}x
\longmapsto\beta_i\otimes x,
\qquad x\in M^j.
\]
For \(a\in e_\ell\widehat A e_j\), the completed relation reads
\[
\widehat b_{i,\ell}a
=\widehat\rho_i^{j,\ell}(a)\widehat b_{i,j}.
\]
The balancing relation for the tensor product is therefore
\(\beta_i\otimes(ax)=\widehat\rho_i^{j,\ell}(a)
(\beta_i\otimes x)\) in the preceding definition. The bidegrees
and the vanishing when the projective index differs from the gauge
degree are as in \Cref{rmt-tensor-generators} and its explanation.
\end{remark}

\begin{proposition}\label{adjunction between completed FG and t}
The functor \(\widehat\rmt\) is left adjoint to
\(Z^0\circ\FG^{\Tot}\). More precisely, there is a natural isomorphism
\[
\mathrm{Hom}_{\Ch(\widehat R\td\I)}(\widehat\rmt(M),N)
\cong
\mathrm{Hom}_{\widehat A}
\left(M,Z^0\FG^{\Tot}(N)\right).
\]
\end{proposition}

\begin{proof}
Given a morphism \(f:\widehat\rmt(M)\to N\), send \(m\in M^j\) to the
cycle whose \(i\)-th coordinate is \(f(\beta_i\otimes m)\). Relations
\Cref{rmt}(1), \Cref{rmt}(2), and \Cref{rmt}(4) show, exactly as in
\Cref{adjunction between FG and t}, that this is a well-defined morphism of
\(F\)-gauges. The completed relation gives
\[
f(\beta_i\otimes am)
=\widehat\rho_i^{j,s}(a)f(\beta_i\otimes m),
\]
so the resulting morphism is \(\widehat A\)-linear.

Conversely, let
\(g:M\to Z^0\FG^{\Tot}(N)\) be \(\widehat A\)-linear. Sending
\(\beta_i\otimes m\) to the \(i\)-th coordinate of \(g(m)\) extends uniquely
over the corresponding coefficient ring. The usual relations are
preserved by the same argument as for \(\rmt\), and
the completed relation is preserved because
\[
g(am)_i=(a\cdot g(m))_i
=\widehat\rho_i^{j,s}(a)g(m)_i.
\]
The two constructions are inverse and natural in \(M\) and \(N\), proving
the adjunction.
\end{proof}

To compare this construction with Ekedahl completion, first calculate
on the projective generators. This gives a natural isomorphism
\[
\widehat\rmt(\widehat G_r)\cong\widehat{\rmt(G_r)},
\]
where the right-hand side denotes Ekedahl completion. Indeed, the
completed relations give the same explicit complex as in
\Cref{description of t(G_r)}(1), with \(R\) and
\(R^{\mathrm{top}}\) replaced by \(\widehat R\) and
\(\widehat R^{\mathrm{top}}\), and with the same differentials.
Ekedahl completion gives these same replacements on the free terms
and the top term, where \(V=pF^{-1}\), so the two complexes agree.
Since
\(\widehat\rmt\) is a left adjoint, it is right exact and admits a left
derived functor
\[
\rmL\widehat\rmt:D(\widehat A)\longrightarrow
D\bigl(\Ch_{-\I}(\widehat R\td\I)\bigr),
\]
whose totalization takes values in \(D(\widehat R\td\I)\). In particular,
\(\rmL^i\widehat\rmt(M)\) is an actual complex of completed
\(R\)-modules for every \(M\in D(\widehat A)\).

If \(M\) admits a resolution by finite direct sums of the
\(\widehat G_r\), then \Cref{Exactness after completion} shows that
\((\rmL\widehat\rmt)^{\Tot}(M)\) is complete in Ekedahl's sense. This
need not hold for arbitrary algebraic \(\widehat A\)-modules: for
example, \(\widehat\rmt(\bigoplus_{\Z}\widehat G_r)\) need not be
complete.

Let \(i:D_c^b(\FGauge\td\I)\to D(A)\) be the natural functor.
Define
\[
j:=\widehat A\otimes_A^{\rmL}i(-):
D_c^b(\FGauge\td\I)\longrightarrow D(\widehat A),
\]
and write \(\operatorname{Res}:D(\widehat A)\to D(A)\) for
restriction of scalars.

\begin{lemma}\label{coherent gauges embed into completed algebra}
The adjunction unit \(i(M)\to\operatorname{Res}(j(M))\) is an
isomorphism for every \(M\in D_c^b(\FGauge\td\I)\).
Consequently, \(j\) is fully faithful.
\end{lemma}

\begin{proof}
First let \(M\) be a coherent gauge in degree zero, and choose a
projective resolution \(P^\bullet\to M\) by finite sums of the
\(G_r\). Since \(\widehat A\otimes_A G_r\cong\widehat G_r\),
the unit is represented by the natural map
\(P^\bullet\to\widehat P^\bullet\).
By \Cref{Exactness after completion}, the completed augmentation
\(\widehat P^\bullet\to M\) is again a quasi-isomorphism.
Its composite with \(P^\bullet\to\widehat P^\bullet\) is the
original augmentation, so the unit is a quasi-isomorphism as well.

For a bounded coherent complex \(M\), apply this conclusion to each
\(\HH^q(M)\). The unit is a natural transformation between the exact
functors \(i\) and \(\operatorname{Res}\circ j\); induction along
the finitely many cohomological truncation triangles therefore proves
that it is an isomorphism on \(M\). This does not require termwise
completion to commute with the cohomology of an arbitrary complex.

Derived extension of scalars is left adjoint to
\(\operatorname{Res}\), so
\[
\begin{aligned}
\Hom_{D(\widehat A)}(j(M),j(N))
&\cong\Hom_{D(A)}(i(M),\operatorname{Res}(j(N)))\\
&\cong\Hom_{D(A)}(i(M),i(N)).
\end{aligned}
\]
Morita equivalence identifies the last group with
\(\Hom_{D_c^b(\FGauge\td\I)}(M,N)\), proving full faithfulness.
\end{proof}

\begin{proposition}
\label{derived widehat t equals widehat derived t}    
There is a commutative square: 
\[\begin{tikzcd}
	{D_c^b(\FGauge\td\I)} & {D(A)} & {D(R\td\I)} \\
	&{D(\widehat{A})} & {D(\widehat{R}\td \I)}
    \arrow["{i}",from=1-1, to=1-2]
    \arrow["{j}",from=1-1, to=2-2]
	\arrow["{\widehat{(\rmL\rmt)^{\Tot}}}", from=1-2, to=1-3]
	\arrow["{\rmL\widehat{\rmt}^{\Tot}}", from=2-2, to=2-3]
	\arrow[from=2-3, to=1-3]
\end{tikzcd}\]
\end{proposition}

\begin{proof}
Choose a bounded projective resolution \(P^\bullet\to M\) by
finite sums of the \(G_r\), using \Cref{resolution of coherent complexes}.
Its termwise completion \(\widehat P^\bullet\) represents \(j(M)\)
and consists of projective \(\widehat A\)-modules. Hence the lower
route in the diagram is represented by
\[
 \Tot\bigl(\widehat{\rmt}(\widehat P^\bullet)\bigr).
\]
For every generator,
\(\widehat{\rmt}(\widehat G_r)\cong\widehat{\rmt(G_r)}\),
naturally with respect to maps of generators. Completion commutes with
this bounded-above totalization: as in the uniform-width model in
\Cref{embedding-row-completeness}, the \(R_a\)-resolution direction
has degrees \([-2,0]\), the inverse-limit direction has degrees
\([0,1]\), and the inner direction lies in the finite interval
\(-\I\). Thus all the relevant totalizations have finite width in
these auxiliary directions. It follows that
\[
 \widehat{\Tot(\rmt(P^\bullet))}
 \simeq
 \Tot\bigl(\widehat{\rmt}(\widehat P^\bullet)\bigr).
\]
The left-hand side represents the upper route, and the generator
comparison is natural. The resulting identification is therefore
functorial in \(M\), as required.
\end{proof}

\begin{lemma}\label{widehat t of FG(Hodge--Witt) }
    Let \(N\in \dc_{\HW}\) be a Hodge--Witt diagonal complex 
    . Then the counit of the adjunction defining \(\widehat{\rmt}\) is an
    isomorphism of \(R\)-complexes:
    \[
    \widehat\epsilon_N:\widehat{\rmt}(\FG^{\Tot}(N))\xrightarrow{\ \cong\ }\widehat N.
    \]
\end{lemma}

\begin{proof}
    By the diagonal decomposition for Hodge--Witt objects and the additivity
    of both functors, it is enough to treat a single \(V\)-topologically nilpotent Dieudonn\'e module \(D\)
    placed in bidegree \((r,-r)\), where \(r\in\I\). Put
    \(G=\FG^{\Tot}(D[r])\). After identifying the underlying abelian groups,
    its gauge components and structure maps are
    \[
    G^j=\begin{cases}
        D,&j\leqslant r,\\
        \sigma_*D,&j>r,
    \end{cases}
    \qquad
    \widetilde F|_{G^j}=\begin{cases}
        p,&j<r,\\
        F_D,&j=r,\\
        1,&j>r,
    \end{cases}
    \qquad
    \widetilde V|_{G^j}=\begin{cases}
        1,&j\leqslant r,\\
        V_D,&j=r+1,\\
        p,&j>r+1.
    \end{cases}
    \]
    Here and below the identifications between the two tails are made using
    \(\tau\). For \(x\in D\), recall that \([x]^j\in G^j\) denotes
    the element corresponding to \(x\) under the above identification;
    the tail identification \(\tau\) preserves this underlying element.
    Then \(\beta_i\otimes[x]^j\) has module degree \(i\) and complex
    degree \(-i\). By the
    construction of the adjunction in \Cref{definition of widehat rmt}, the
    counit takes a generator to its \(i\)-th coordinate. Consequently,
    \begin{equation}\label{counit on a Hodge--Witt block}
        \widehat\epsilon_{D[r]}\bigl(\beta_i\otimes[x]^j\bigr)=
        \begin{cases}
            x,&i=r,\\
            0,&i\neq r,
        \end{cases}
    \end{equation}
    where in the first line we use the preceding tail identifications. In
    particular the counit is surjective.

    To identify the kernel, we calculate the relations among the generators.
    First fix \(i<r\).
    The first five identities below follow from \Cref{rmt}(3), and the
    last from \Cref{rmt}(4):
    \[
    \begin{aligned}
        \beta_i\otimes[x]^j&=\beta_i\otimes[x]^i &&(j\leqslant i),\\
        \beta_i\otimes[x]^i&=V\beta_i\otimes[x]^{i+1},& &\\
        \beta_i\otimes[x]^j&=p^{r-j}\beta_i\otimes[x]^r
            &&(i<j\leqslant r),\\
        \beta_i\otimes[y]^j&=\beta_i\otimes[y]^{r+1} &&(j\geqslant r+1),\\
        \beta_i\otimes[F_Dx]^{r+1}&=\beta_i\otimes[x]^r,\\
        \beta_i\otimes[y]^{r+1}&=\beta_i\otimes[y]^i.
    \end{aligned}
    \]
    Combining these identities yields
    \begin{equation}\label{lower Hodge--Witt generators}
        \beta_i\otimes[x]^r
        =p^{r-i-1}V\beta_i\otimes[F_Dx]^r.
    \end{equation}
    Choose finitely many \(W\)-generators of \(D\), and let \(\mathbf b_i\)
    be the corresponding column of elements \(\beta_i\otimes[x]^r\). Equation
    \eqref{lower Hodge--Witt generators} has the form
    \(\mathbf b_i=B_i\mathbf b_i\), where the powers of \(B_i\) tend to
    zero in the defining topology of \(\widehat R^{\,0}\): every iteration
    contributes at least one further power of \(V\). Hence
    \(1-B_i\) is invertible, with inverse \(\sum_{q\geqslant0}B_i^q\), and
    \(\mathbf b_i=0\). Thus the diagonal term in module degree \(i<r\)
    vanishes.

    Next let \(i>r\). The same relations now give
    \[
    \begin{aligned}
        \beta_i\otimes[x]^j&=\beta_i\otimes[x]^r &&(j\leqslant r),\\
        \beta_i\otimes[y]^{r+1}&=\beta_i\otimes[V_Dy]^r,& &\\
        \beta_i\otimes[y]^j&=p^{j-r-1}\beta_i\otimes[y]^{r+1}
            &&(r+1\leqslant j\leqslant i),\\
        \beta_i\otimes[y]^{i+1}&=F\beta_i\otimes[y]^i,& &\\
        \beta_i\otimes[y]^j&=\beta_i\otimes[y]^{i+1} &&(j\geqslant i+1),\\
        \beta_i\otimes[y]^{i+1}&=\beta_i\otimes[y]^r.& &
    \end{aligned}
    \]
    Therefore
    \begin{equation}\label{upper Hodge--Witt generators}
        \beta_i\otimes[x]^r
        =p^{i-r-1}F \beta_i\otimes[V_Dx]^r.
    \end{equation}
    This again gives a finite matrix equation
    \(\mathbf b_i=B_i\mathbf b_i\). If \(i\geqslant r+2\), the positive
    power of \(p\) in \eqref{upper Hodge--Witt generators} makes the powers
    of \(B_i\) converge to zero. If \(i=r+1\), the same conclusion follows
    from the assumed topological nilpotence of \(V_D\). This also works when
    \(i=n\): in that case convergence is in the \(p\)-adic topology of
    \(\widehat R^{\mathrm{top}}\). Thus \(1-B_i\) is invertible and all the
    diagonal terms with \(i>r\) vanish.

    We now identify the surviving term in module degree \(r\).
    All the generators reduce to \(\beta_r\otimes[x]^r\), and the two relations at the break are
    \[
        F\beta_r\otimes[x]^r=\beta_r\otimes[F_Dx]^r,
        \qquad
        V\beta_r\otimes[y]^r=\beta_r\otimes[V_Dy]^r.
    \]
    Since \(D\) is finitely generated over \(W\), complete, and \(V_D\) is
    topologically nilpotent, its \(R^0\)-action extends to
    \(\widehat R^{\,0}\). Hence
    \[
        D\longrightarrow
        \widehat{\rmt}(G)^{r,-r},
        \qquad x\longmapsto \beta_r\otimes[x]^r,
    \]
    is a well-defined inverse to the restriction of the counit. If \(r=n\),
    the level condition says that \(F_D\) is bijective. The same argument
    then uses the natural \(\widehat R^{\mathrm{top}}\)-action on \(D\), with
    \(F^{-1}\) acting as \(F_D^{-1}\); restricted Laurent series converge
    because \(D\) is \(p\)-adically complete. This proves the assertion also
    at the top endpoint, including the case \(\I=[n,n]\).

    To finish, we must also check the off-diagonal term in bidegree \((k,-k+1)\).
    It is generated by \(\widehat R^{\,1}(\beta_{k-1}\otimes[x]^j)\) for \(x\in D\), so it vanishes
    whenever \(k-1\neq r\), by what we have just proved. If \(k=r+1\leqslant
    n\), relation \Cref{rmt}(2), applied with \(i=r+1\) and \([x]^r\in G^r\),
    gives
    \[
        (-1)^{r+1}\delta(\beta_{r+1}\otimes[x]^r)
        +d(\beta_r\otimes[x]^r)=0.
    \]
    The first term is zero because the diagonal term in module degree
    \(r+1\) is zero. Thus \(d(\beta_r\otimes[x]^r)=0\). The topological nilpotence of \(V_D\) now shows that the whole
    \(\widehat R^{\,1}(\beta_r\otimes[x]^r)\) is zero. Hence every
    off-diagonal term vanishes as well. Relation \Cref{rmt}(2) also shows
    that \(\delta\) is zero on the surviving copy of \(D\).

    We have proved that \(\widehat{\rmt}(G)\) is \(D\) in bidegree
    \((r,-r)\) and is zero elsewhere, and by
    \eqref{counit on a Hodge--Witt block} the counit is the identity on this
    copy. Taking the finite direct sum over \(r\) proves the lemma.
\end{proof}

\subsection{\(\rmL^i\widehat{\rmt}(M)\) as \(R\)-complexes}
\label{subsec:derived-completed-rmt-as-r-complexes}

For a bounded coherent gauge complex, the completed inverse is
computed from a finite projective resolution \(P^\bullet\):
\(\widehat{(\rmL\rmt)^{\Tot}}(M)\simeq
\Tot(\widehat\rmt(\widehat P^\bullet))\).
The cohomology of the resolution direction gives actual
\(R\)-complexes \(\rmL^i\widehat\rmt(M)\). We now identify each of
them with a diagonal cohomology object.

Ekedahl proved two comparisons for a single coherent \(F\)-gauge
\cite[Theorem II.5.6 and Corollary II.5.6.1]{Ek3}, though some steps
in their proofs need correction. We first treat the diagonal heart,
where the conclusion holds for actual \(R\)-complexes, and then pass
to bounded coherent gauge complexes.

\begin{theorem}\label{completed rmt on diagonal heart}
Let \(N\in\dc\) have level \(\I\). Then
\[
\rmL^i\widehat\rmt_\I(\FG^{\Tot}(N))=0
\qquad(i\ne0)
\]
in \(\Ch(\widehat R\td\I)\). Moreover,
\(\rmL^0\widehat\rmt_\I(\FG^{\Tot}(N))\in
\Ch(\widehat R\td\I)\) represents \(N\) in \(D(R)\).
\end{theorem}

This strengthens \cite[Theorem~II.5.6]{Ek3} by proving vanishing
for the actual \(R\)-complexes in nonzero outer degrees.

\begin{proof}
Choose a bounded projective
resolution \(P^\bullet\to\FG^{\Tot}(N)\) by finite sums of \(G_r\),
using \Cref{resolution of coherent complexes}. Its termwise
completion \(\widehat P^\bullet\) remains a resolution by
\Cref{Exactness after completion}. Take cohomology first in the
resolution degree. If \(p\) denotes the complex degree of
\(\widehat\rmt\) and \(q\) the resolution degree, this gives
\[
E_1^{p,q}
 =\HH^q\bigl(\widehat\rmt(\widehat P^\bullet)^p\bigr)
 =\bigl(\rmL^q\widehat\rmt(\FG^{\Tot}(N))\bigr)^p,
\qquad
E_2^{p,q}=\HH^p(E_1^{\bullet,q})
  \Longrightarrow\HH^{p+q}(N).
\]
The abutment follows from
\Cref{derived widehat t equals widehat derived t,completed adjunction unit and counit are isomorphisms}.
By
\Cref{widehat-rmt-tensor-product,description of t(G_r)},
\(E_1^{-u,q}\) has nonzero components only in \(R\)-module degrees
\(u,u+1\). Thus, for \(r\geqslant2\), the source and target of
\(d_r:E_r^{p,q}\to E_r^{p+r,q-r+1}\) have disjoint \(R\)-module
degrees. The spectral sequence therefore degenerates at \(E_2\).

We record the flatness step that will eliminate the unwanted rows.
Fix \(q\) and suppose \((E_2^{-u,q})^u=0\) for every \(u\in\I\).
The term \(E_1^{-u-1,q}\) has no \(R\)-module component of degree
\(u\), so
\((E_1^{-u,q})^u\to(E_1^{-u+1,q})^u\) is injective.
Its target is zero for \(u=m\). For \(m<u\leqslant n\), the terms of
\(\widehat\rmt(\widehat P^\bullet)^{-u+1}\) are finite sums of
\(\widehat R(-(u-1))\). Taking resolution cohomology after tensoring
is valid because \(\widehat R^{\,1}\) is right flat over
\(\widehat R^{\,0}\), as established after
\Cref{flatness of central completion}. Hence
\[
(E_1^{-u+1,q})^u
 \cong\widehat R^{\,1}\otimes_{\widehat R^{\,0}}
       (E_1^{-u+1,q})^{u-1}.
\]
Starting at \(u=m\), induction makes both sides of each injection
zero. The remaining components are the same degree-\(u+1\)
components at the next step, and degree \(n+1\) is zero by the
level condition. Thus \(E_1^{\bullet,q}=0\) as an actual
\(R\)-complex.

Since \(N\) is diagonal, \(\HH^{q-u}(N)^u=0\) unless \(q=0\) or
\(1\). Each \((E_2^{-u,q})^u\) is a subquotient of this group.
The induction therefore eliminates every row except \(q=0,1\).
For \(q=1\), \((E_2^{-u+1,1})^u\) is a subquotient of
\(\HH^{2-u}(N)^u=0\). Thus \(E_2^{t,1}\) is concentrated in
\(R\)-module degree \(-t\).

The remaining spectral-sequence filtration gives an exact sequence
of graded \(R\)-modules
\[
 0\longrightarrow E_2^{t,0}\longrightarrow\HH^t(N)
 \longrightarrow E_2^{t-1,1}\longrightarrow0.
\]
The quotient is concentrated in degree \(1-t\), so the
\(R\)-submodule \(E_2^{t,0}\) contains \(\HH^t(N)^{-t}\).
It also contains its \(F\)-saturated boundaries, which equal
\(\HH^t(N)^{1-t}\) by the diagonal-heart condition. Hence
\(E_2^{t-1,1}=0\) for every \(t\), and the same flatness
induction eliminates the \(q=1\) row. This proves strict vanishing.
The completed adjunction also gives
\begin{equation}\label{embedding-total-represents-N}
\Tot\bigl(\rmL\widehat\rmt_\I(\FG^{\Tot}(N))\bigr)
\simeq N\qquad\text{in }D(R).
\end{equation}
Since only outer degree zero remains, its cohomology object represents
\(N\).
\end{proof}

The heart calculation extends to bounded complexes by diagonal
truncation.

\begin{theorem}\label{completed rmt and diagonal cohomology}
Let \(M\in D_c^b(\FGauge\td\I)\). Then
\(\rmL^i\widehat\rmt(M)\in\Ch(\widehat R\td\I)\) represents the
diagonal cohomology object
\[
\rmL^i\widehat\rmt(M)\cong
\widetilde{\HH}^i\bigl((\rmL\widehat\rmt)^{\Tot}(M)\bigr).
\]
\end{theorem}

This extends \cite[Corollary~II.5.6.1]{Ek3} to bounded coherent
gauge complexes and their actual \(R\)-complex representatives.

\begin{proof}
Put \(N=\widehat{(\rmL\rmt)^{\Tot}}(M)\). The completed equivalence
identifies \(M\) with \(\FG^{\Tot}(N)\). By
\Cref{completed rmt on diagonal heart}, the functor
\(\rmL\widehat\rmt\circ\FG^{\Tot}\) takes the diagonal heart to
the standard outer heart. Since the diagonal \(t\)-structure is
bounded, this functor is \(t\)-exact. Consequently
\[
\rmL^i\widehat\rmt(M)
 \cong
\rmL^0\widehat\rmt\bigl(
  \FG^{\Tot}(\widetilde{\HH}^i(N))\bigr).
\]
By \Cref{completed rmt on diagonal heart}, the complex on the right
represents \(\widetilde{\HH}^i(N)\).
This proves the theorem.
\end{proof}

For \(N\in\dc\) of level \(\I\), set
\begin{equation}\label{interval-candidate-representative}
 \mathcal S_\I(N)=H^0_{\mathrm{out}}
    \bigl(\rmL\widehat{\rmt}_\I(\FG^{\Tot}(N))\bigr)
    \in\Ch(R).
\end{equation}

\begin{corollary}\label{embedding-natural-representative}
For \(N\in\dc\) of level \(\I\), there is a natural isomorphism
\(q\mathcal S_\I(N)\simeq N\), where \(q:\Ch(R)\to D(R)\) is localization.
Moreover, \(\mathcal S_\I\) is exact on short exact sequences in the
diagonal heart whose objects have level \(\I\).
\end{corollary}

\begin{proof}
The first assertion is \Cref{completed rmt on diagonal heart}.
A short exact sequence in \(\dc\) gives a distinguished triangle
in \(D(R)\). Applying \(\FG^{\Tot}\), then
\(\rmL\widehat{\rmt}_\I\), and taking the long exact sequence of
outer cohomology, \Cref{completed rmt on diagonal heart} reduces that
sequence to a short exact sequence in \(\Ch(R)\).
\end{proof}

For a coherent \(F\)-gauge \(M\) with
\(\widehat{(\rmL\rmt)^{\Tot}}(M)\in\dc\), right exactness gives
\(\rmL^0\widehat\rmt(M)=\widehat\rmt(M)\). Thus
\Cref{completed rmt and diagonal cohomology} says that this actual
\(R\)-complex represents \(\widehat{(\rmL\rmt)^{\Tot}}(M)\).

\begin{example}
We can see the effect of completion in \Cref{The first example}.
After Ekedahl completion, the free
\(R\)-modules in its resolution are replaced by their \(V\)-adic
completions:
\[\rmL\widehat{\rmt}(M)\simeq \mathrm{Tot}\left[\widehat{\rmt(G_0)}\xrightarrow{\widehat{\rmt(\Phi-\Psi)}}\widehat{\rmt(G_0)}\right]. \]
In degree \(0\), the completed cokernel is
\[\HH^0(\rmL\widehat{\rmt}(M))^0\simeq \widehat{R^0}/\widehat{R^0}(F-pV)\simeq k[[V]].\]
Indeed, after imposing \(F=pV\), the relation \(FV=p\) gives \(p(V^2-1)=0\), and \(V^2-1\) is a unit in \(W[[V]]\), so \(p=0\).

Similarly, in module degree \(1\),
\[
\HH^0(\rmL\widehat\rmt(M))^1\simeq k[[dV]]_{>0},
\qquad
\HH^{-1}(\rmL\widehat\rmt(M))^1\simeq W^2.
\]
All other graded cohomology groups vanish.
\end{example}

The theorem identifies the diagonal cohomology at a fixed level
\(\I\). We next show that its representatives are independent of
this choice: enlarging \(\I\) does not change the actual
\(R\)-complexes \(\rmL^i\widehat\rmt_\I(M)\).

\subsection{Changing the interval}
\label{subsec:changing-the-interval}

An enlarged interval gives another presentation of the same
\(F\)-gauge; the effect on the completed inverse is less immediate.
We construct the natural comparison between levels and describe the
coefficient bimodules of \(\widehat\rmt\). This identifies the
acyclicity needed to derive the comparison.

Throughout the following comparison, \(\I=[m,n]\subseteq \mathrm{J}\) are finite intervals.
We write \(G_{\I,r}\), \(A_\I\), and \(\widehat{\rmt}_\I\) when the level
must be specified, and put
\[
 \I^+=[m,n+1],\qquad \I^-=[m-1,n],\qquad
 \widehat G_{\I,r}=\widehat A_{\I}e_r.
\]
The completion of a gauge or a gauge algebra is \(p\)-adic. In contrast,
completeness of a complex of graded \(R\)-modules always means Ekedahl
completeness. All tensor products below are algebraic tensor products,
unless a derived tensor product is explicitly indicated; no completed
tensor product is implicit in the notation.

\paragraph*{Change of level and the comparison map.}
An \(\I\)-gauge can be regarded as a \(\mathrm{J}\)-gauge using its constant tails.
For the elementary right enlargement, the new component is a copy of
\(M^n\), with
\[
 M^{n+1}=M^n,\qquad
 \tF_n=1,\qquad \tV_{n+1}=p,\qquad \tau_{\I^+}=\tau_\I.
\]
For the elementary left enlargement, use
\[
 M^{m-1}=M^m,\qquad
 \tV_m=1,\qquad \tF_{m-1}=p,\qquad \tau_{\I^-}=\tau_\I.
\]
Here the endpoint identifications are understood in the indicated
copies, and the Frobenius semilinearity of \(\tau\) is retained.
Denote this exact change-of-level functor by \(E_{\mathrm{J}\I}\).
The finite matrices of path operators extend to the completed algebras,
so the same construction applies to algebraic \(\widehat A_\I\)-modules.
Equivalently, regard \(\widehat A_\I\) as a left regular module and
apply \(E_{\mathrm J\I}\). The resulting object
\(E_{\mathrm J\I}(\widehat A_\I)\) is a
\((\widehat A_{\mathrm J},\widehat A_\I)\)-bimodule: its left action
comes from its structure as a \(\mathrm J\)-gauge, while its right
action is obtained by applying \(E_{\mathrm J\I}\) to the
left-module homomorphisms \(x\mapsto xa\), for \(a\in\widehat A_\I\).
For every left \(\widehat A_\I\)-module \(M\), there is a natural
isomorphism
\[
E_{\mathrm J\I}(M)
\cong
E_{\mathrm J\I}(\widehat A_\I)
\otimes_{\widehat A_\I}M.
\]
Indeed, each component of \(E_{\mathrm J\I}(\widehat A_\I)\) in
gauge degree \(j\in\mathrm J\) is a copy of some
\(e_r\widehat A_\I\), with endpoint components repeated along the
added tails. Tensoring replaces this component by
\(e_rM=M^r\), with the same gauge operators, precisely as in
\(E_{\mathrm J\I}(M)\). Since \(\mathrm J\) is finite, the
bimodule is a finite direct sum of projective right
\(\widehat A_\I\)-modules. Thus \(E_{\mathrm J\I}\) is exact.

A completed \(R\)-module of level \(\I\) is regarded as one of level
\(\mathrm{J}\) by adjoining zero components. If its old top becomes an interior
degree, restrict its \(\widehat R^{\mathrm{top}}\)-action along
\(\widehat R^0\to\widehat R^{\mathrm{top}}\). We suppress this exact
inclusion in the notation. For \(N\in\Ch(\widehat R\td\I)\), this gives
a natural identification
\(E_{\mathrm J\I}Z^0\FG_\I^{\Tot}(N)
\cong Z^0\FG_{\mathrm J}^{\Tot}(N)\).
Applying \(E_{\mathrm J\I}\) to the adjunction unit at level \(\I\)
therefore gives a map
\(E_{\mathrm J\I}M\to
Z^0\FG_{\mathrm J}^{\Tot}(\widehat\rmt_\I(M))\).
By the adjunction at level \(\mathrm J\)
(\Cref{adjunction between completed FG and t}), this corresponds to
the natural comparison
\begin{equation}\label{interval-underived-comparison}
 \alpha_{\mathrm{J}\I,M}:\widehat{\rmt}_{\mathrm{J}}(E_{\mathrm{J}\I}M)
       \longrightarrow\widehat{\rmt}_\I(M).
\end{equation}
On symbols with gauge index in \(\I\), it sends
\(\beta_i\otimes x\) to the same symbol if \(i\in \I\), and to zero
if \(i\notin \I\). Symbols in the added gauge components are interpreted
using \Cref{rmt}(3). For example, if \(x_n\in M^n\) and
\(x_{n+1}:=\tF x_n\) is its copy in the added component, then
\(\beta_n\otimes x_{n+1}=F(\beta_n\otimes x_n)\).
These rules respect the completed path relations and the complex
differential.
For \(\I\subseteq \mathrm{J}\subseteq K\), these natural maps satisfy
\begin{equation}\label{interval-comparison-cocycle}
 \alpha_{K\I,M}
 =\alpha_{\mathrm{J}\I,M}\circ\alpha_{K\mathrm{J},E_{\mathrm{J}\I}M},
\end{equation}
with the change-of-level identifications understood.

Before deriving this comparison, we distinguish the two complex degrees in
\(\rmL\widehat{\rmt}_\I(M)\in D(\Ch(R))\).
We write \(a\) for the outer resolution degree, \(s\) for the inner
\(R\)-complex degree in the output of \(\widehat{\rmt}_\I\), and
\(u\) for the \(R\)-module grading.
Thus \(H^a_{\mathrm{out}}(\rmL\widehat{\rmt}_\I(M))
=\rmL^a\widehat{\rmt}_\I(M)\) is itself an object of \(\Ch(R)\).
Throughout the comparison argument, the inner \(R\)-complex degree
remains fixed.

\paragraph*{Coefficient bimodules and acyclicity.}
Under the Morita equivalence, the left regular module
\(\widehat A_{\mathrm J}\) corresponds to the \(F\)-gauge
\(\bigoplus_{r\in\mathrm J}\widehat G_{\mathrm J,r}\).
As in \Cref{widehat-rmt-tensor-product}, put
\[
 \mathcal T_{\mathrm{J}}=\widehat{\rmt}_{\mathrm{J}}(\widehat A_{\mathrm{J}})\cong \bigoplus_{r\in \mathrm{J}}\widehat{\rmt}_{\mathrm{J}}(\widehat G_{\mathrm{J},r}).
\]
Then, in each nonzero bidegree \((u,s)\), \(\mathcal T_{\mathrm J}^{u,s}\) is a finite direct sum of copies of \(\widehat R^0\), \(\widehat R^1\), or \(\widehat R^{\mathrm{top}}\), as appropriate.
The right action of \(a\in\widehat A_{\mathrm{J}}\) is obtained by applying
\(\widehat{\rmt}_{\mathrm{J}}\) to right multiplication by \(a\) on its left
regular module. Since \(\widehat{\rmt}_{\mathrm{J}}\) preserves direct sums and
cokernels, a presentation by free modules gives natural identifications
\begin{equation}\label{interval-coefficient-tensor}
 \widehat{\rmt}_{\mathrm{J}}(M)^{u,s}
   =\mathcal T_{\mathrm{J}}^{u,s}\otimes_{\widehat A_{\mathrm{J}}}M,
 \qquad
 (\rmL^{-a}\widehat{\rmt}_{\mathrm{J}}(M))^{u,s}
   =\operatorname{Tor}^{\widehat A_{\mathrm{J}}}_a(\mathcal T_{\mathrm{J}}^{u,s},M)
 \quad(a\geq0).
\end{equation}
These are ordinary module-theoretic Tor groups, since the right action
on \(\mathcal T_{\mathrm{J}}^{u,s}\) forms part of the coefficient bimodule structure.
It follows that
flat left \(\widehat A_{\mathrm{J}}\)-modules are acyclic for
\(\widehat{\rmt}_{\mathrm{J}}\). The idempotent components of \(\mathcal T_{\mathrm{J}}^{u,s}\)
are copies of \(\widehat R^0\), \(\widehat R^1\), or
\(\widehat R^{\mathrm{top}}\), in their respective nonzero bidegrees.
They are \(p\)-torsion-free. The shifts and semilinear scalar conventions
are those of \Cref{rmt,definition of widehat rmt}.

\subsection{Flatness of \(\widehat{G}_{\I,r}\)}
\label{subsec:flatness-of-completed-generators}

The completed generators provide the first test of the comparison.
Their projectivity at level \(\I\) does not automatically control
their behavior at a larger level. We construct flat resolutions for
the two elementary enlargements and prove that they remain
\(\widehat\rmt_{\mathrm J}\)-acyclic. A final example shows why their
images can nevertheless change.

\paragraph*{Flat resolutions for elementary enlargements.}

All completions of gauge algebras and gauge
modules in these resolutions are ordinary $p$-adic completions; the
tensor products remain algebraic.
We retain the interval in the notation \(G_{\I,r}\) for a projective
generator, and write \(\widehat{G}_{\I,r}\) for its completion.

We write \(A_{\mathrm{J}}\) for the full endomorphism algebra of the
projective generator \(\bigoplus_{r\in\mathrm{J}}G_{\mathrm{J},r}\), with
the opposite-ring convention used above, and
\(\widehat{A}_{\mathrm{J}}\) for its \(p\)-adic completion.
For \(\mathrm{J}=[a,b]\) with \(a<b\), put
\(e=\sum_{r=a}^{b-1}e_r\). The reduced algebra is the full corner
\(B_{\mathrm{J}}\simeq eA_{\mathrm{J}}e\): the endpoint isomorphism
\(\tau_{\mathrm{J}}:b\to a\) gives
\(e_b=\tau_{\mathrm{J}}^{-1}e_a\tau_{\mathrm{J}}
\in A_{\mathrm{J}}eA_{\mathrm{J}}\), so \(e\) is full.
Thus \(A_{\mathrm{J}}\) and \(B_{\mathrm{J}}\) are Morita equivalent.
The same holds after \(p\)-adic completion, with
\(\widehat B_{\mathrm{J}}\simeq e\widehat A_{\mathrm{J}}e\).
Under these equivalences, \(G_{\mathrm{J},r}\) corresponds to
\(B_{\mathrm{J}}e_r\), and \(\widehat G_{\mathrm{J},r}\) to
\(\widehat B_{\mathrm{J}}e_r\), for \(a\leqslant r<b\).
This is a corner construction, not a quotient or a direct-product
decomposition of \(A_{\mathrm{J}}\).

The corner has a convenient path presentation: use
\(\tau_{\mathrm{J}}\) to identify the endpoint objects. Its vertices
are \(a,\ldots,b-1\). The two
orientations of the resulting cycle are given by the forward and
backward gauge arrows; opposite arrows compose to $p$ times the relevant
idempotent.  One forward arrow is $\sigma$-semilinear and its opposite is
$\sigma^{-1}$-semilinear.  The remaining arrows are $W$-linear.  In
particular, these identifications do not assert that $\tau_{{\mathrm{J}}}$ is
$W$-linear.  Modules over the completion
\[
 \widehat{B}_{\mathrm{J}}=\varprojlim_{\nu}B_{\mathrm{J}}/p^{\nu}B_{\mathrm{J}}
\]
are identified, by Morita equivalence, with the completed gauge-algebra
modules used above.  Our convention is that an arrow $u:s\to t$ belongs
to $e_t\widehat{B}_{\mathrm{J}}e_s$.  Thus right multiplication by $u$ defines a map
$\widehat{B}_{\mathrm{J}}e_t\to \widehat{B}_{\mathrm{J}}e_s$ of left modules.

The path presentation also proves the Noetherian properties needed for
completion. These algebras are left and right Noetherian: if the cycle has
$h=b-a$ vertices, put
\[
\begin{aligned}
\Phi&=\sum_{l=a}^{b-1}\Phi_l,
&\Phi_l&=\tF^{l-a}\tau_{\mathrm J}\tF^{b-l},\\
\Psi&=\sum_{l=a}^{b-1}\Psi_l,
&\Psi_l&=\tV^{b-l}\tau_{\mathrm J}^{-1}\tV^{l-a}.
\end{aligned}
\]
Each \(\Phi_l\) and \(\Psi_l\) is a loop at \(l\). Together with the diagonal copy
of $W$, they give a unital cyclic operator subring
\[
    W_{\sigma}[\Phi,\Psi]/
    (\Phi\Psi=\Psi\Phi=p^h).
\]
This ring is left and right Noetherian: it is a quotient of an iterated
skew polynomial ring over $W$, with automorphisms $\sigma$ and
$\sigma^{-1}$.  Cancelling a pair of opposite arrows replaces that pair
by $p$.  Consequently every path is a power of a full forward or backward
cycle times a path of length less than $h$.  There are only finitely many
such remaining paths. Hence $B_{\mathrm{J}}$ is finite on both sides over
this cyclic subring, and is left and right Noetherian. The same path
basis shows that $B_{\mathrm{J}}$ is $p$-torsion-free. By
\Cref{flatness of central completion}, its \(p\)-adic completion
\(\widehat B_{\mathrm J}\) remains left and right Noetherian and
\(p\)-torsion-free.
This path description also follows directly from the basis of
\Cref{definition of G_r}: after the scalar coefficients have been moved
to the left, the reduced paths in the two orientations form a $W$-basis.

The next two lemmas provide the flat modules used in the resolutions
and a way to test injectivity after tensoring by comparing coefficients.

\begin{lemma}\label{interval-flat-completion}
 Let \(H\) be either \(\widehat A_{\mathrm J}\) or
 \(\widehat B_{\mathrm J}\), as defined above, and let
 $z$ be a central indeterminate. Define
 \[
 H\langle z\rangle:=\varprojlim_{\nu}(H/p^{\nu}H)[z].
 \]
 Then \(H\langle z\rangle\) is flat as a left and as a right
 \(H\)-module. In particular, for a vertex \(s\) of \(\mathrm J\), define
 \[
 \widehat G_{\mathrm J,s}\langle z\rangle
 \coloneqq\varprojlim_\nu(\widehat G_{\mathrm J,s}/p^\nu\widehat G_{\mathrm J,s})[z],
 \]
 with gauge operators acting coefficientwise. Then
 \(H\langle z\rangle e_s\cong\widehat G_{\mathrm J,s}\langle z\rangle\)
 under the Morita equivalence, so this is a flat left \(H\)-module.
\end{lemma}

\begin{proof}
 Apply \Cref{flatness of central completion} to \(H[z]\), which is left
 and right Noetherian since \(H\) is and \(z\) is central. Its \(p\)-adic
 completion is \(H\langle z\rangle\), which is therefore flat on both
 sides over \(H[z]\). Since \(H[z]\) is free on both sides over \(H\),
 \(H\langle z\rangle\) is flat on both sides over \(H\). Finally,
 \(H\langle z\rangle e_s\) is a direct summand of \(H\langle z\rangle\).
\end{proof}

\begin{lemma}\label{interval-series-injection}
 With \(H\langle z\rangle\) as in \Cref{interval-flat-completion}, every right
 \(H\)-module \(T\) has a natural injective coefficient map
 \[
 T\otimes_H H\langle z\rangle\longrightarrow T[[z]].
 \]
 Here $T[[z]]=\prod_{j\geq0}Tz^j$ denotes the group of formal
 coefficient sequences.  This assertion does not identify the ordinary
 tensor product with a completed tensor product or with all of $T[[z]]$.
\end{lemma}

\begin{proof}
 An element of \(H\langle z\rangle\) is a series $\sum_{j\geq0}a_jz^j$ with
 $a_j\to0$ $p$-adically in $H$.  Taking coefficients defines the stated
 map. To prove injectivity, consider an element of its kernel and choose
 a finitely generated right $H$-submodule $T_0\subseteq T$ containing the finitely many
 $T$-coordinates in a tensor expression for this element.  Flatness of
 \(H\langle z\rangle\) gives an injection
 \[
 T_0\otimes_H H\langle z\rangle\longrightarrow T\otimes_H H\langle z\rangle.
 \]
 The finite right $H$-module $T_0$ is $p$-adically complete and separated,
 since $H$ is Noetherian and $p$-adically complete.  Applied to the finite
 right $H[z]$-module $T_0[z]$, the finite-module completion formula gives
 \[
 T_0\otimes_H H\langle z\rangle
  \cong\varprojlim_{\nu}T_0[z]/p^{\nu}T_0[z].
 \]
 The latter module is the group of series in $T_0[[z]]$ whose
 coefficients tend to zero $p$-adically.  In particular, it injects into
 $T_0[[z]]$, and hence into $T[[z]]$.  The chosen element is therefore
 zero already in \(T_0\otimes_H H\langle z\rangle\).
\end{proof}

\begin{proposition}\label{interval-right-flat-resolution}
Let \(\I=[m,n]\) and \(\mathrm J=[m,n+1]\), and use the reduced
presentation with vertices \(m,\ldots,n\). View the completed generators
as \(\widehat B_{\mathrm J}\)-modules via the Morita equivalence above.
For every \(r\in\I\), there
is an exact sequence of left \(\widehat B_{\mathrm J}\)-modules
\begin{equation}\label{interval-right-resolution-sequence}
0\longrightarrow
\widehat G_{\mathrm J,m}\oplus
\widehat G_{\mathrm J,m}\langle z\rangle
\xrightarrow{\ d_r\ }
\widehat G_{\mathrm J,r}\oplus
\widehat G_{\mathrm J,n}\langle z\rangle
\xrightarrow{\ \epsilon_r\ }
\widehat G_{\I,r}\longrightarrow0.
\end{equation}
The maps are
\[
\begin{aligned}
\epsilon_r\left(a,\sum_{j\geq0}a_jz^j\right)
&=a\alpha_r+
\sum_{j\geq0}a_j
(\tau_\I^{-1}\tV^{n-m})^j
\tau_\I^{-1}\tV^{r-m}\alpha_r,\\
d_r(u,q(z))
&=\bigl(-u\tV^{r-m},\,
u\tau_{\mathrm J}\tF+
q(z)(\tV^{n-m}-\tau_{\mathrm J}\tF z)\bigr).
\end{aligned}
\]
Here \(\tF:n\to n+1\) is the added boundary arrow. We take
\(u\in\widehat G_{\mathrm J,m}\),
\(q(z)\in\widehat G_{\mathrm J,m}\langle z\rangle\),
\(a\in\widehat G_{\mathrm J,r}\), and
\(a_j\in\widehat G_{\mathrm J,n}\) with \(a_j\to0\) \(p\)-adically.
The inverse \(\tau_\I^{-1}\) acts on \(\widehat G_{\I,r}\), and an empty
power of \(\tV\) is an identity arrow. Both modules to the left of
\(\widehat G_{\I,r}\) are flat over \(\widehat B_{\mathrm J}\).
\end{proposition}

\begin{proof}
 We first prove the uncompleted sequence of \(B_{\mathrm J}\)-modules
 \[
 0\longrightarrow
 G_{\mathrm J,m}\oplus G_{\mathrm J,m}[z]
 \xrightarrow{\ d_r\ }
 G_{\mathrm J,r}\oplus G_{\mathrm J,n}[z]
 \xrightarrow{\ \epsilon_r\ }
 G_{\I,r}\longrightarrow0
 \]
 is exact. Here \(G_{\I,r}\) is viewed at level \(\mathrm J\), and
 the maps use the formulas above with polynomials in \(z\).
 On a gauge of level \(\I\), the added arrow \(\tF\) is invertible and
 \(\tau_{\mathrm J}\tF=\tau_\I\).

 Write \(\beta_r\) for the generator of the \(G_{\mathrm J,r}\)
 summand and \(\beta_n\) for that of the \(G_{\mathrm J,n}[z]\)
 summand, whose \(B_{\mathrm J}\)-generators are \(z^j\beta_n\)
 for \(j\geq0\). When \(r=n\), the two occurrences of \(\beta_n\)
 belong to different summands.
 The augmentation to \(G_{\I,r}\) sends
 \[
 \beta_r\longmapsto\alpha_r,\qquad
 z^j\beta_n\longmapsto
 (\tau_\I^{-1}\tV^{n-m})^j\tau_\I^{-1}\tV^{r-m}\alpha_r.
 \]
 The image of \(d_r\) imposes exactly the relations
 \begin{equation}\label{interval-right-resolution-relations}
 (\tau_{\mathrm J}\tF)\beta_n=\tV^{r-m}\beta_r,\qquad
 (\tau_{\mathrm J}\tF)z^{j+1}\beta_n
 =\tV^{n-m}z^j\beta_n\quad(j\geq0).
 \end{equation}
 Here \(z\) is an auxiliary central variable, not a gauge operator.
 The relations show that this is a \(B_{\mathrm J}\)-linear augmentation
 and that \(\epsilon_r d_r=0\).

 To identify the cokernel, we reduce paths using the defining relations.
 For \(m\leq i\leq n\), the cokernel at vertex \(i\) is spanned by:
 \begin{enumerate}
 \item all pure forward paths from $r$ to $i$, applied to \(\beta_r\);
 \item the short internal backward path from $r$ to $i$, applied to
 \(\beta_r\), if $i<r$;
 \item \(\tV^{n-i}z^j\beta_n\) for \(j\geq0\).
 \end{enumerate}
 The trivial forward path is included when $i=r$.  To verify that this
 collection spans, first cancel opposite arrows in an arbitrary path.
 A backward path on \(\beta_r\) that crosses the new boundary arrow
 \(\tV\tau_{\mathrm J}^{-1}\) can be
 reduced using
 \[
 (\tV\tau_{\mathrm J}^{-1})\tV^{r-m}\beta_r=p\beta_n.
 \]
 A full backward cycle on \(z^j\beta_n\) can be reduced using
 \[
 (\tV\tau_{\mathrm J}^{-1})\tV^{n-m}z^j\beta_n
 =pz^{j+1}\beta_n.
 \]
 Finally, every nonempty forward path on \(z^j\beta_n\) starts with
 \(\tau_{\mathrm J}\tF\);
 \eqref{interval-right-resolution-relations} reduces it either to a
 path on \(\beta_r\), or, when \(j>0\), to a path on
 \(z^{j-1}\beta_n\), after cancellation of opposite internal arrows.
 This proves the spanning assertion.

 After the boundary arrow \(\tau_{\mathrm J}\tF\) has been made
 invertible, the three
 collections above are exactly the distinct reduced forward and
 backward paths that give the standard $W$-basis of $G_{{\I},r}^i$.
 Indeed, the second collection is the backward path that has not yet
 crossed the old endpoint identification, and the third comprises those
 that have crossed it at least once.  Thus their images under the
 augmentation are linearly independent.  The augmentation identifies
 the cokernel with $G_{{\I},r}$, and in particular this cokernel is
 $p$-torsion-free.  Moving coefficients through the paths applies powers
 of $\sigma$ or $\sigma^{-1}$ and leaves this basis calculation unchanged.

 The map on the left is also injective. If \(d_r(u,q(z))=0\), then
 \(u\tV^{r-m}=0\). Composing \(\tV^{r-m}\) with the opposite forward
 path gives \(p^{r-m}e_m\), so \(p\)-torsion-freeness gives \(u=0\). Write
 \(q(z)=\sum_{j\geq0}q_jz^j\), with finite support. Comparing the constant
 term in \(q(z)(\tV^{n-m}-\tau_{\mathrm J}\tF z)=0\) gives
 \(q_0\tV^{n-m}=0\). The opposite forward path shows that right
 multiplication by \(\tV^{n-m}\) is injective on \(G_{\mathrm J,m}\). Hence
 $q_0=0$.  The coefficient equations
 \[
 q_j\tV^{n-m}=q_{j-1}\tau_{\mathrm J}\tF\quad(j\geq1)
 \]
 give $q_j=0$ successively.

 We may now complete this short exact sequence.  Since its cokernel
 $G_{{\I},r}$ is $p$-torsion-free, reduction modulo $p^{\nu}$ remains
 short exact for every $\nu$.  The transition maps on the reduced left
 terms are surjective.  Taking inverse limits therefore preserves
 exactness.  Moreover
 \[
 \varprojlim_{\nu}(G_{\mathrm J,s}/p^{\nu}G_{\mathrm J,s})[z]
  =\widehat G_{\mathrm J,s}\langle z\rangle
  \qquad(s=m,n).
 \]
 This gives \eqref{interval-right-resolution-sequence}. Exactness here
 follows from the reduction and inverse-limit argument; it requires
 neither projectivity of a completed infinite direct sum nor an
 Artin--Rees statement for an infinite module.
 Flatness of its terms follows from
 \Cref{interval-flat-completion}.
\end{proof}

\begin{proposition}\label{interval-left-flat-resolution}
Let \(\I=[m,n]\) and \(\mathrm J=[m-1,n]\), and use the reduced
presentation with vertices \(m,\ldots,n\). View the completed generators
as \(\widehat B_{\mathrm J}\)-modules via the Morita equivalence above.
For every \(r\in\I\), there is an exact sequence of left
\(\widehat B_{\mathrm J}\)-modules
\begin{equation}\label{interval-left-resolution-sequence}
0\longrightarrow
\widehat G_{\mathrm J,n}\oplus
\widehat G_{\mathrm J,n}\langle z\rangle
\xrightarrow{\ d_r\ }
\widehat G_{\mathrm J,r}\oplus
\widehat G_{\mathrm J,m}\langle z\rangle
\xrightarrow{\ \epsilon_r\ }
\widehat G_{\I,r}\longrightarrow0.
\end{equation}
The maps are
\[
\begin{aligned}
\epsilon_r\left(a,\sum_{j\geq0}a_jz^j\right)
&=a\alpha_r+
\sum_{j\geq0}a_j
(\tau_\I\tF^{n-m})^j
\tau_\I\tF^{n-r}\alpha_r,\\
d_r(u,q(z))
&=\bigl(-u\tF^{n-r},\,
u\tau_{\mathrm J}^{-1}\tV+
q(z)(\tF^{n-m}-\tau_{\mathrm J}^{-1}\tV z)\bigr).
\end{aligned}
\]
Here \(\tV:m\to m-1\) is the added boundary arrow. We take
\(u\in\widehat G_{\mathrm J,n}\),
\(q(z)\in\widehat G_{\mathrm J,n}\langle z\rangle\),
\(a\in\widehat G_{\mathrm J,r}\), and
\(a_j\in\widehat G_{\mathrm J,m}\) with \(a_j\to0\) \(p\)-adically.
The map \(\tau_\I\) acts on \(\widehat G_{\I,r}\), and an empty power
of \(\tF\) is an identity arrow. Both modules to the left of
\(\widehat G_{\I,r}\) are flat over \(\widehat B_{\mathrm J}\).
\end{proposition}

\begin{proof}
We first prove the uncompleted sequence of \(B_{\mathrm J}\)-modules
\[
0\longrightarrow
G_{\mathrm J,n}\oplus G_{\mathrm J,n}[z]
\xrightarrow{\ d_r\ }
G_{\mathrm J,r}\oplus G_{\mathrm J,m}[z]
\xrightarrow{\ \epsilon_r\ }
G_{\I,r}\longrightarrow0
\]
is exact. Here \(G_{\I,r}\) is viewed at level \(\mathrm J\), and
the maps use the formulas above with polynomials in \(z\).
On a gauge of level \(\I\), the added arrow \(\tV\) is invertible and
\(\tau_{\mathrm J}^{-1}\tV=\tau_\I^{-1}\).
Write \(\beta_r\) for the generator of the \(G_{\mathrm J,r}\)
summand and \(\beta_m\) for that of the \(G_{\mathrm J,m}[z]\)
summand, whose \(B_{\mathrm J}\)-generators are \(z^j\beta_m\)
for \(j\geq0\). When \(r=m\), the two occurrences of \(\beta_m\)
belong to different summands.
The augmentation sends \(\beta_r\) to \(\alpha_r\) and \(z^j\beta_m\) to
\((\tau_\I\tF^{n-m})^j\tau_\I\tF^{n-r}\alpha_r\).
The image of \(d_r\) imposes exactly the relations
\begin{equation}\label{interval-left-resolution-relations}
(\tau_{\mathrm J}^{-1}\tV)\beta_m=\tF^{n-r}\beta_r,\qquad
(\tau_{\mathrm J}^{-1}\tV)z^{j+1}\beta_m
=\tF^{n-m}z^j\beta_m\quad(j\geq0).
\end{equation}
These relations make it \(B_{\mathrm J}\)-linear and give
\(\epsilon_r d_r=0\).

At vertex \(i\), the cokernel has the reduced-path basis consisting
of pure backward paths on \(\beta_r\), the short internal forward path on
\(\beta_r\) when \(i>r\), and \(\tF^{i-m}z^j\beta_m\) for \(j\geq0\).
Indeed, forward paths crossing the new boundary reduce by
\((\tF\tau_{\mathrm J})\tF^{n-r}\beta_r=p\beta_m\) or
\((\tF\tau_{\mathrm J})\tF^{n-m}z^j\beta_m=pz^{j+1}\beta_m\);
backward paths on \(z^j\beta_m\) reduce by
\eqref{interval-left-resolution-relations}. Their images are the
distinct reduced paths of \(G_{\I,r}\). Thus the cokernel is
\(G_{\I,r}\) and is \(p\)-torsion-free.

If \(d_r(u,q(z))=0\), then \(u\tF^{n-r}=0\). Composing with the opposite
backward path gives \(p^{n-r}u=0\), hence \(u=0\). Writing
\(q(z)=\sum q_jz^j\), the remaining equations are
\(q_0\tF^{n-m}=0\) and
\(q_j\tF^{n-m}=q_{j-1}\tau_{\mathrm J}^{-1}\tV\) for \(j\geq1\).
The opposite backward path makes right multiplication by
\(\tF^{n-m}\) injective, so all \(q_j\) vanish.
Since the cokernel is \(p\)-torsion-free, reduction modulo \(p^\nu\)
and passage to inverse limits preserve exactness, as in
\Cref{interval-right-flat-resolution}. Flatness follows from
\Cref{interval-flat-completion}.
\end{proof}

When \(m=n\), the powers of \(\tF\) and \(\tV\) in both resolutions
are identity arrows. Let \(\theta=\tau_{\mathrm J}\tF\) for the right
enlargement and \(\theta=\tau_{\mathrm J}^{-1}\tV\) for the left.
Cancelling the finite summands gives
\begin{equation}\label{interval-width-zero-resolution}
0\longrightarrow\widehat G_{\mathrm J,m}\langle z\rangle
\xrightarrow{\ \cdot(1-\theta z)\ }
\widehat G_{\mathrm J,m}\langle z\rangle
\longrightarrow\widehat G_{\I,m}\longrightarrow0.
\end{equation}
On the old gauge, \(\theta\) acts as \(\tau_\I\) or \(\tau_\I^{-1}\),
respectively. Before completion the normal forms are
\(\theta^kx\) for \(k\geq0\) and \(y_j=\theta^{-j-1}x\) for \(j\geq0\).
Thus the cokernel is the expected skew Laurent module; after completion
it is its $p$-adic completion.

\begin{corollary}\label{interval-old-generators-acyclic}
 In either elementary enlargement, let $T$ be any $p$-torsion-free
 right $\widehat{B}_{\mathrm{J}}$-module.  Then
 \[
 \operatorname{Tor}^{\widehat{B}_{\mathrm{J}}}_a(T,\widehat{G}_{\I,r})=0
 \qquad(a>0).
 \]
 In particular, every $\widehat{G}_{\I,r}$ is
 $\widehat{\rmt}_{{\mathrm{J}}}$-acyclic:
 \[
 \rmL^q\widehat{\rmt}_{{\mathrm{J}}}(\widehat{G}_{\I,r})=0
 \qquad(q<0).
 \]
\end{corollary}

\begin{proof}
The two-term flat resolutions give vanishing of Tor in degrees at
least two. For degree one, let \(s=m\) for the right enlargement and
\(s=n\) for the left. Since
\(\widehat G_{\mathrm J,s}\cong\widehat B_{\mathrm J}e_s\),
tensoring with \(T\) gives
\(T\otimes_{\widehat B_{\mathrm J}}\widehat G_{\mathrm J,s}
\cong Te_s\) and
\[
Te_s\oplus
\bigl(T\otimes_{\widehat B_{\mathrm J}}
\widehat G_{\mathrm J,s}\langle z\rangle\bigr)
\hookrightarrow Te_s\oplus(Te_s)[[z]],
\]
where the second summand embeds by
\Cref{interval-series-injection}. Thus we may write a kernel element
as \((u,q(z))\), where \(u\in Te_s\) and
\(q(z)=\sum_{j\geq0}q_jz^j\) lies in the image of this injection,
with each \(q_j\in Te_s\). The first component gives
\(u\tV^{r-m}=0\) in the right enlargement or
\(u\tF^{n-r}=0\) in the left. Each path has an opposite whose
composite is a power of \(p\), so \(u=0\).
By \Cref{interval-series-injection}, we may compare the coefficients
\(q_j\) of \(q(z)\). The constant coefficient is
\(q_0\tV^{n-m}\) or \(q_0\tF^{n-m}\), respectively.
The same opposite-path argument gives \(q_0=0\); comparing
successive coefficients then gives \(q_j=0\) for all \(j\).
Thus the differential remains injective, proving the Tor assertion.

 By \Cref{widehat-rmt-tensor-product} and the reduced Morita
 equivalence above,
 \[
 \widehat{\rmt}_{\mathrm J}(X)
 \cong
 \widehat{\rmt}_{\mathrm J}(\widehat B_{\mathrm J})
 \otimes_{\widehat B_{\mathrm J}}X.
 \]
 Each bidegree of \(\widehat{\rmt}_{\mathrm J}(\widehat B_{\mathrm J})\)
 is a finite direct sum of the
 completed coefficient modules $\widehat R^0$, $\widehat R^1$, or the
 completed top module, by \Cref{description of t(G_r)} and the completed
 generator computation.  These modules are $p$-torsion-free.  Applying
 the first assertion to each such right $\widehat{B}_{\mathrm{J}}$-module shows that every
 bidegree of every higher left derived functor vanishes. The vanishing
 is in the resolution direction, with values in
 $\Ch(\widehat R\td{\mathrm{J}})$, and therefore holds before totalizing
 the inner $R$-complex.
\end{proof}

\begin{example}\label{interval-two-vertex-example}
 Let \(\I=[0,1]\) and \(\mathrm J=[0,2]\), and identify \(2\) with \(0\)
 through \(\tau_{\mathrm J}\). The opposite-arrow pairs are
 \((\tF:0\to1,\ \tV:1\to0)\) and
 \((\tau_{\mathrm J}\tF:1\to0,\ \tV\tau_{\mathrm J}^{-1}:0\to1)\);
 their composites are multiplication by \(p\).
 For \(r=0\), the first component of \(d_r\) in
 \eqref{interval-right-resolution-sequence} is \(-u\). Cancelling it gives
 \begin{equation}\label{interval-two-vertex-flat-resolution}
 0\longrightarrow\widehat G_{[0,2],0}\langle z\rangle
 \xrightarrow{\ \cdot(\tV-\tau_{\mathrm J}\tF z)\ }
 \widehat G_{[0,2],1}\langle z\rangle
 \longrightarrow\widehat G_{[0,1],0}\longrightarrow0.
 \end{equation}
 Writing \(\beta_1\) for the generator of the polynomial term, its
 relations are
 \(\tV z^j\beta_1=(\tau_{\mathrm J}\tF)z^{j+1}\beta_1\), and
 \(z^j\beta_1\) maps to
 \((\tau_\I^{-1}\tV)^j\tau_\I^{-1}\alpha_0\) for \(j\geq0\).
 Under endpoint identifications the two
 components of the old generator are copies, with the appropriate
 scalar twists, of
 \[
 \widehat{W_{\sigma}[\Phi,\Psi]/(\Phi\Psi=\Psi\Phi=p)},
 \]
 and these elements correspond to $1,\Psi,\Psi^2,\ldots$ in its
 vertex-$1$ component.  The old endpoint arrow $\tau_{{\I}}$ is
 invertible in the preadditive gauge category, so its two representable
 modules $\widehat{G}_{[0,1],0}$ and $\widehat{G}_{[0,1],1}$ are isomorphic.  The
 semilinearity of $\tau_{{\I}}$ changes scalar identifications, but does
 not obstruct this isomorphism of modules over the gauge algebra.
\end{example}

\begin{example}[The old generator in a larger interval]
\label{interval-uncompleted-tensor-example}
Acyclicity alone does not make the comparison on old completed
generators an isomorphism. To see the remaining obstruction, we
calculate the actual complex in the smallest nontrivial right
enlargement, before applying derived reduction. In this example
only, assume \(k=\F_p\), so that \(W=\Z_p\) and the scalar
twists disappear. Take \(\I=[0,1]\) and \(\mathrm J=[0,2]\).
All tensor products and quotients below are algebraic; in particular, we
do not complete a tensor product or take the closure of a relation
submodule.

Identify the endpoint \(2\) with \(0\) by \(\tau_{\mathrm{J}}\).  The resulting
two-vertex operator algebra has arrows
\[
 \tF_0:0\longrightarrow1,\quad
 \tV_1:1\longrightarrow0,\quad
 \tau_{\mathrm{J}}\tF_1:1\longrightarrow0,\quad
 \tV_2\tau_{\mathrm{J}}^{-1}:0\longrightarrow1,
\]
 with both composites of each opposite pair equal to multiplication
 by \(p\).
Its corner at \(0\) and the old one-vertex operator algebra are
\[
 C=\varprojlim_s\bigl(\Z_p[X,Y]/(XY-p^2)\bigr)/p^s,
 \qquad
 D=\varprojlim_s\bigl(\Z_p[\Phi,\Psi]/(\Phi\Psi-p)\bigr)/p^s.
\]
Here \(X=\tau_{\mathrm J}\tF_1\tF_0\) and
\(Y=\tV_1\tV_2\tau_{\mathrm J}^{-1}\).
Both components of \(\widehat G_{\I,0}\) identify with \(D\);
the four displayed arrows act as \(\Phi,\Psi,1,p\), respectively.
Consequently the corner action on \(\widehat{G}_{\I,0}^0=D\) is
\[
 C\longrightarrow D,\qquad X\longmapsto\Phi,\quad
 Y\longmapsto p\Psi.
\]

The right \(C\)-actions on the coefficient modules are obtained by
right multiplication by the entries in the following table:
\[
\begin{array}{c|cc}
 &X&Y\\\hline
 \rho_0&F&pV\\
 \rho_1&pF&V\\
 \rho_2&p^2F&F^{-1}
\end{array}
\]
The first two rows apply to both \(\widehat R^0\) and
\(\widehat R^1\); the last applies to
\(\widehat R^{\mathrm{top}}\). These actions include the entire completed
corner ring \(C\), not merely its polynomial subring.

For \(\epsilon=0,1\), define
\[
\begin{split}
 K_0^\epsilon
 &=\frac{(\widehat R^{\,\epsilon})_{\rho_0}\otimes_C D}
 {\langle qV\otimes x-q\otimes\Psi x:
             q\in\widehat R^{\,\epsilon},\ x\in D\rangle},\\
 K_1^\epsilon
 &=\frac{(\widehat R^{\,\epsilon})_{\rho_1}\otimes_C D}
 {\langle q\otimes x-qF\otimes\Psi x:
             q\in\widehat R^{\,\epsilon},\ x\in D\rangle},\\
 K_2&=(\widehat R^{\mathrm{top}})_{\rho_2}\otimes_C D.
\end{split}
\]
The angle brackets denote the relation submodules in the displayed
algebraic tensor products.  Then the possibly nonzero components of
\(\widehat{\rmt}_{\mathrm{J}}(\widehat{G}_{\I,0})\), indexed by (Raynaud-module degree, complex
degree), are exactly
\[
\begin{array}{c|ccccc}
 (u,s)&(0,0)&(1,0)&(1,-1)&(2,-1)&(2,-2)\\\hline
 \widehat{\rmt}_{\mathrm{J}}(\widehat{G}_{\I,0})^{u,s}
 &K_0^0&K_0^1&K_1^0&K_1^1&K_2.
\end{array}
\]
The Raynaud action of \(d\) is induced by left multiplication:
\[
 K_i^0\longrightarrow K_i^1,\qquad
 [q\otimes x]\longmapsto[dq\otimes x]\quad(i=0,1).
\]
The differential of the complex, with the signs of \Cref{rmt}, is
\[
 K_2\xrightarrow{\delta_2}K_1^1,
 \quad[q\otimes x]\longmapsto-[qd\otimes x],
 \qquad
 K_1^0\xrightarrow{\delta_1}K_0^1,
 \quad[q\otimes x]\longmapsto[qd\otimes x].
\]
In the first formula, \(qd\) is defined using the action of
\(\widehat R^{\mathrm{top}}\) on \(\widehat R^{\,1}\): left multiplication
by \(F\) is invertible there, and its inverse sends \(d\) to \(dV\).
The identities \(dF=pFd\), \(Vd=pdV\), and \(FdV=d\) show that these
formulas respect the displayed relations.

To verify this presentation, read it directly from the relations in
\Cref{rmt}. At \(\beta_0\), the coefficient of
\(\tau_{\mathrm J}\tF_1\) is \(1\); eliminate the vertex-\(1\) symbols
using this path, and the remaining relation for \(\tV_1\) is
\(qV\otimes x=q\otimes\Psi x\). At \(\beta_1\),
the coefficient of \(\tV_1\) is \(1\), so a vertex-\(1\) symbol with argument
\(x\) becomes the vertex-\(0\) symbol with argument \(\Psi x\).  The
relation for \(\tau_{\mathrm J}\tF_1\) then becomes
\(q\otimes x=qF\otimes\Psi x\).
At \(\beta_2\), the same elimination gives
\(q\otimes x=qpF\otimes\Psi x\), which is already the corner-balancing
relation for \(Y\), since \(\rho_2(Y)=F^{-1}\).
The other arrows give the displayed corner relations or their
consequences.

The new top component \(K_2\) accounts for the failure of the
comparison: it is nonzero, even though multiplication by \(p\) on it
is bijective.  Indeed, \(C,D,\widehat R^{\mathrm{top}}\) are domains,
and the two homomorphisms from \(C\) to \(D\) and
\(\widehat R^{\mathrm{top}}\) are injective.  The injections can be checked in
the Laurent expansions: the coefficients of positive powers and of
negative powers remain distinct, and are only multiplied by nonzero
powers of \(p\).  The domain assertion for \(C,D\) follows, for example,
by assigning an irrational weight strictly between \(0\) and \(2\) to
\(X\), respectively strictly between \(0\) and \(1\) to \(\Phi\);
the complementary weights of \(Y,\Psi\) are then positive, and every
nonzero restricted series has a unique term of least weight.  The
domain assertion for \(\widehat R^{\mathrm{top}}\) follows from its
\(p\)-adic separation and its domain reduction
\(\F_p[F,F^{-1}]\).
Thus, writing \(K=\operatorname{Frac}(C)\), the localization
\[
 K\otimes_C K_2
 \simeq
 (K\otimes_C \widehat R^{\mathrm{top}})
 \otimes_K(K\otimes_C D)
\]
is a tensor product of two nonzero \(K\)-vector spaces, and is nonzero.
On the other hand, corner balancing for \(Y\) gives
\[
 F^{-1}\otimes1=1\otimes p\Psi,
 \qquad\text{hence}\qquad
 1\otimes1=p(F\otimes\Psi)
 \quad\text{in }K_2.
\]
The tensor product \(K_2\) is a commutative algebra, so \(p\) is a unit
in it.  In particular, its derived reduction modulo \(p\) vanishes,
although \(K_2\) itself does not.

The old complex \(\widehat{\rmt}_\I(\widehat{G}_{\I,0})\) has three components
\(\widehat R^0,\widehat R^1,\widehat R^{\mathrm{top}}\)
at \((0,0),(1,0),(1,-1)\), respectively,
and has no component at \((2,-2)\).  Thus the natural comparison on
this old completed generator is not an isomorphism of \(R\)-complexes.
The coherence hypothesis explains why this example falls outside the
interval-independence statement: \(\widehat{G}_{\I,0}\) is not coherent over \(W\).
For coherent inputs, a further completeness argument recovers the
comparison from its derived reduction.
\end{example}

\subsection{Reduction by \(R_1\) is independent of the interval}
\label{subsec:interval-independent-r1-reduction}

Derived reduction by \(R_1\) reveals a simpler comparison: the images
of each completed generator become canonically isomorphic at the two
levels. We make this isomorphism explicit in each fixed complex degree
and check compatibility with the differential. These calculations link
the flat resolutions above to independence for coherent inputs.

\paragraph*{Reduction by \(R_1\).}\label{reduction by R1}
Use the \((k[d]/(d^2),R)\)-bimodule \(R_1\) from
\Cref{reduction by R_n}. Regard the top coefficient ring
\(\widehat R^{\mathrm{top}}\) from
\Cref{subsec:completed-r-modules-of-level-i} as an \(R\)-module in
degree zero with \(d=0\); its reduction modulo \(p\) is
\(k_\sigma[F,F^{-1}]\).

\begin{lemma}\label{interval-r1-completed-coefficients}
The canonical maps induce isomorphisms
\[
 R_1\otimes_R^{\rmL}\widehat R\simeq R_1,
 \qquad
 R_1\otimes_R^{\rmL}\widehat R^{\mathrm{top}}
 \simeq k_\sigma[F,F^{-1}].
\]
Both right-hand sides are concentrated in auxiliary cohomological
degree zero. In particular, neither derived tensor product has higher
\(\operatorname{Tor}\).
\end{lemma}

\begin{proof}
This is clear from the definitions and the free resolution of \(R_1\)
in the proof of \Cref{Basic Vanishing}.
\end{proof}

\paragraph*{Two polynomial calculations.}
After derived reduction, the flat resolutions will give polynomial
matrices. The following calculations determine their kernels and
cokernels.

In these calculations, \(\cdot F\) and \(\cdot V\) denote right
multiplication on \(R_1\). Its basis gives
\[
 R_1^0=k_\sigma[F],\qquad R_1^1=k[F]d,
\]
and
\[
 F^j\cdot V=0,\qquad F^jd\cdot F=0,\qquad
 F^jd\cdot V=
 \begin{cases}F^{j-1}d,&j\geq1,\\0,&j=0.\end{cases}
\]
The formulas are extended additively with the given coefficient
actions. In particular, right multiplication by \(V\) is locally
nilpotent on \(R_1\). The left \(d\)-action sends \(1\) to \(d\) and
\(F^j\) to zero for \(j>0\).

In a polynomial module below, the central variable \(z\) records the
auxiliary right operator used in the change-of-level resolution.
It commutes with the left gauge action. It is not being identified
with a semilinear left gauge operator. All coefficient actions on the
vertex summands remain the ones prescribed by the gauge construction.
The polynomial modules consist of finite sums
\(\sum_{j\geq0}x_jz^j\).

\begin{lemma}\label{interval-r1-polynomial-operators}
There is an exact sequence
\[
 0\longrightarrow R_1[z]
 \xrightarrow{\,\cdot(1-Fz)\,}R_1[z]
 \longrightarrow k_\sigma[F,F^{-1}]\longrightarrow0.
\]
Here the last map is right \(F\)-localization in degree zero and zero
in degree one: it sends \(\sum_jx_jz^j\) to
\(\sum_jx_j^0F^{-j}\), where \(x_j^0\in R_1^0\) is the degree-zero
component of \(x_j\).
Moreover, \(\cdot(1-Vz)\) is an automorphism of \(R_1[z]\).
\end{lemma}

\begin{proof}
If \((\sum_jx_jz^j)\cdot(1-Fz)=0\), its constant coefficient is
\(x_0=0\), and its successive coefficients give
\(x_j=x_{j-1}\cdot F=0\). The cokernel is the direct limit under
right multiplication by \(F\).
On \(R_1^0\) this is \(k_\sigma[F,F^{-1}]\), and on
\(R_1^1\) it is zero because right multiplication by \(F\)
is zero. This also proves the assertion about the displayed last map.
For the second assertion, the inverse is
\(\sum_{a\geq0}(\cdot(Vz))^a\). This expression is a finite sum on every
input polynomial, since there are only finitely many coefficients and
right multiplication by \(V\) is locally nilpotent on each of them.
Thus the inverse is defined on the polynomial module itself; no
formal power-series enlargement is needed.
\end{proof}

\begin{lemma}[Polynomial elimination]\label{interval-r1-polynomial-elimination}
Let \(S\) be an abelian group, and let \(b,b_r\) be commuting
additive endomorphisms of \(S\), written as right operators
\(x\mapsto xb\) and \(x\mapsto xb_r\). Write
\(S[z]=\bigoplus_{j\geq0}Sz^j\). For
\(y(z)=\sum_j y_jz^j\), set \(y(b)=\sum_j y_jb^j\), and extend
\(b\) coefficientwise to \(S[z]\).
Then the following sequence of abelian groups is exact:
\begin{equation}\label{interval-r1-monic-matrix}
0\longrightarrow S\oplus S[z]
\xrightarrow{\ d\ }S\oplus S[z]
\xrightarrow{\ \epsilon\ }S\longrightarrow0.
\end{equation}
Here \(d(u,q)=(-ub_r,u+q(b-z))\) and
\(\epsilon(x,y(z))=x+y(b)b_r\). Under the induced isomorphism
\(\operatorname{coker}(d)\cong S\), the first summand of the middle
term maps identically to \(S\).
\end{lemma}

\begin{proof}
Suppose \(d(u,q)=0\). If \(q\neq0\) has degree \(d_0\) and
leading coefficient \(q_{d_0}\), then the coefficient of
\(z^{d_0+1}\) in \(u+q(b-z)\) is \(-q_{d_0}\neq0\), a
contradiction. Thus \(q=0\), and then \(u=0\), proving injectivity.

Evaluation at \(b\) sends \(q(b-z)\) to zero, so
\(\epsilon\circ d=0\). Moreover, every polynomial \(y\in S[z]\)
admits a unique expression
\[
 y(z)=y(b)+q(z)(b-z).
\]
Indeed, successively eliminating the highest-degree term gives a
constant remainder, which must be \(y(b)\) upon evaluation at \(b\);
uniqueness follows from the same leading-coefficient argument.
If \(\epsilon(x,y)=0\), then \(x=-y(b)b_r\), and this expression
gives
\[
 (x,y)=d\bigl(y(b),q\bigr).
\]
Hence \(\ker(\epsilon)=\operatorname{im}(d)\). Finally,
\(\epsilon(x,0)=x\), proving surjectivity and the assertion about
the first summand.
\end{proof}

In applying these calculations, we retain the prescribed scalar
actions in the vertex coefficient identifications. Polynomial division
uses the coefficient of \(-z\), which is the identity between the
relevant coefficient copies.

\begin{proposition}\label{interval-r1-generator-comparison}
Let \(\I=[m,n]\), let \(r\in \I\), and let
\(\mathrm{J}=[m,n+1]\) or \(\mathrm{J}=[m-1,n]\). Regard
\(\widehat{G}_{\I,r}\) as a gauge of level \(\mathrm{J}\). The natural
comparison from the larger interval to the smaller interval induces,
for every inner \(R\)-complex degree \(s\), an isomorphism
\[
 R_1\otimes_R^{\rmL}
       \bigl(\widehat{\rmt}_{\mathrm{J}}(\widehat{G}_{\I,r})\bigr)^s
 \xrightarrow{\ \sim\ }
 R_1\otimes_R^{\rmL}
       \bigl(\widehat{\rmt}_\I(\widehat{G}_{\I,r})\bigr)^s.
\]
These isomorphisms commute with the inner differential \(\delta\)
and are natural for morphisms between finite sums of the
\(\widehat{G}_{\I,r}\). Both sides have no higher \(R_1\)-Tor. In inner
degree \(s\), their common degree-zero value is
\[
\begin{cases}
R_1(-i),&s=-i,\quad m\leq i<n,\\
k_\sigma[F,F^{-1}](-n),&s=-n,\\
0,&\text{otherwise}.
\end{cases}
\]
\end{proposition}

\begin{proof}
The comparison sends \(\beta_i\otimes x\) to the same symbol for
\(i\in \I\), and sends the symbols at the added endpoint to zero.
At the old top endpoint it uses the coefficient homomorphism
\(\widehat R^0\to\widehat R^{\mathrm{top}}\). These formulas respect the defining
relations of \Cref{definition of widehat rmt}, including those for
\(\delta\), and hence define a natural morphism of actual
\(R\)-complexes.

Let
\[
 0\longrightarrow P^{-1}\xrightarrow{d_r}P^0
 \longrightarrow\widehat G_{\I,r}\longrightarrow0
\]
be the flat resolution of
\Cref{interval-right-flat-resolution,interval-left-flat-resolution}.
By \Cref{interval-old-generators-acyclic}, applying
\(\widehat\rmt_{\mathrm J}\) gives a short exact sequence of actual
\(R\)-complexes. In each fixed inner degree \(s\), the derived
\(R_1\)-reduction of the cokernel is computed by reducing the two-term
complex \([\widehat\rmt_{\mathrm J}(P^{-1})^s
\to\widehat\rmt_{\mathrm J}(P^0)^s]\).

Write \(s=-i\), and put \(S_i=R_1\) below the top of \(\mathrm J\),
\(S_i=k_\sigma[F,F^{-1}]\) at the top, and \(S_i=0\) outside
\(\mathrm J\). By \Cref{description of t(G_r),interval-r1-completed-coefficients},
each generator and polynomial summand, for \(j\in\mathrm J\), reduces as follows:
\[
\begin{aligned}
 R_1\otimes_R^{\rmL}
 \bigl(\widehat\rmt_{\mathrm J}(\widehat G_{\mathrm J,j})\bigr)^{-i}
 &\simeq S_i(-i),\\
 R_1\otimes_R^{\rmL}
 \bigl(\widehat\rmt_{\mathrm J}(\widehat G_{\mathrm J,j}\langle z\rangle)\bigr)^{-i}
 &\simeq S_i[z](-i).
\end{aligned}
\]
Both reductions have no higher \(R_1\)-Tor. For the second formula,
use the coefficient tensor description \eqref{interval-coefficient-tensor}:
\(\widehat G_{\mathrm J,j}\langle z\rangle\) is flat over
\(\widehat A_{\mathrm J}\) by \Cref{interval-flat-completion}, so the
free right \(R\)-resolution of \(R_1\) may first be applied to the
coefficient bimodule. By the first formula at every vertex, its
reduction is concentrated in degree zero and killed by \(p\); moreover,
\(\widehat G_{\mathrm J,j}\langle z\rangle/p
 =(\widehat G_{\mathrm J,j}/p)[z]\).
This gives the polynomial term, with its prescribed scalar action.
The identifications respect the path actions and \(z\), so
\(\widehat\rmt_{\mathrm J}(d_r)\) reduces coefficientwise. After
suppressing the module shift \((-i)\), its source and target are both
\(S_i\oplus S_i[z]\), in auxiliary degrees \(-1,0\). In the formulas
below, each path denotes its coefficient operator \(\rho_i\), acting on
the right of \(S_i\) and coefficientwise on \(S_i[z]\), where \(z\) is
central. Integral coefficients are reduced modulo \(p\).

First suppose that \(\mathrm{J}=[m,n+1]\). By
\Cref{interval-right-flat-resolution}, the reduced differential is
\begin{equation}\label{interval-r1-right-reduced-matrix}
 (u,q)\longmapsto
 \bigl(-u\tV^{r-m},\,
 u\tau_{\mathrm J}\tF_n+q(\tV^{n-m}-\tau_{\mathrm J}\tF_n z)\bigr).
\end{equation}
Here \(\tV^{r-m}:r\to m\), \(\tV^{n-m}:n\to m\), and
\(\tau_{\mathrm J}\tF_n:n\to m\). The table lists their images under
\(\rho_i\), acting on \(S_i\):
\[
\begin{array}{c|c|c|c}
 i&\tV^{r-m}&\tV^{n-m}&\tau_{\mathrm J}\tF_n\\\hline
 m\leq i<n&
 \begin{cases}1,&r\leq i,\\p^{r-i-1}V,&r>i\end{cases}
 &p^{n-i-1}V&1\\
 i=n&1&1&F\\
 i=n+1&1&1&pF.
\end{array}
\]
For \(i<n\), use \eqref{interval-r1-monic-matrix}; it gives zero
kernel and cokernel \(R_1\), with the first generator unchanged.
For \(i=n\), the first component of
\eqref{interval-r1-right-reduced-matrix} is \(-u\). Eliminating it gives
\[
 \operatorname{coker}\bigl(\cdot(1-Fz):R_1[z]\longrightarrow R_1[z]\bigr)
 \simeq k_\sigma[F,F^{-1}],
\]
and zero kernel, by \Cref{interval-r1-polynomial-operators}.
To identify this isomorphism with the natural comparison, we must
keep track of the first summand. Elimination represents
\(x\) in the polynomial summand by \(xF\), so the localization
isomorphism must be followed by right multiplication by \(F^{-1}\).
Equivalently, the comparison on the quotient is
\[
 (x,\sum_jy_jz^j)\longmapsto
 x+\sum_j y_jF^{-(j+1)}\quad\text{in }k_\sigma[F,F^{-1}].
\]
It consequently sends the original generator to that same generator,
as required for the natural comparison. For \(i=n+1\), reduction
kills \(\tau_{\mathrm J}\tF_n=pF\), and the matrix is simply
\((u,q)\mapsto(-u,q)\). Both its kernel and its cokernel are zero.

Now suppose that \(\mathrm{J}=[m-1,n]\). By
\Cref{interval-left-flat-resolution}, the reduced differential is
\begin{equation}\label{interval-r1-left-reduced-matrix}
 (u,q)\longmapsto
 \bigl(-u\tF^{n-r},\,
 u\tau_{\mathrm J}^{-1}\tV_m
 +q(\tF^{n-m}-\tau_{\mathrm J}^{-1}\tV_m z)\bigr).
\end{equation}
Here \(\tF^{n-r}:r\to n\), \(\tF^{n-m}:m\to n\), and
\(\tau_{\mathrm J}^{-1}\tV_m:m\to n\). The table lists their images
under \(\rho_i\), acting on \(S_i\):
\[
\begin{array}{c|c|c|c}
 i&\tF^{n-r}&\tF^{n-m}&\tau_{\mathrm J}^{-1}\tV_m\\\hline
 i=m-1&1&1&V\\
 m\leq i<n&
 \begin{cases}1,&i<r,\\p^{i-r}F,&i\geq r\end{cases}
 &p^{i-m}F&1\\
 i=n&p^{n-r}&p^{n-m}&F^{-1}.
\end{array}
\]
At the added bottom endpoint, elimination of the first component
reduces the matrix to \(\cdot(1-Vz)\) on \(R_1[z]\). It is an
isomorphism by \Cref{interval-r1-polynomial-operators}, so nothing
survives there. For \(m\leq i<n\),
\eqref{interval-r1-monic-matrix} again gives zero kernel and cokernel
\(R_1\), with the first generator unchanged.

At \(i=n\), the coefficient module is \(k_\sigma[F,F^{-1}]\), and
\(\tau_{\mathrm J}^{-1}\tV_m=F^{-1}\) is a unit. Right multiplication
of the second output component by \(F\) changes
\eqref{interval-r1-left-reduced-matrix} into
\[
 (u,q)\longmapsto
       (-u p^{n-r},u+q(p^{n-m}F-z)).
\]
This is \eqref{interval-r1-monic-matrix}, so its kernel is zero and
its cokernel is \(k_\sigma[F,F^{-1}]\), again with the first summand mapping
identically. This calculation includes \(m=n=r\): in that case the
powers \(p^{n-r}\) and \(p^{n-m}\) are both \(1\), not \(0\).
The right-enlargement calculation likewise includes width zero.

In every old inner \(R\)-complex degree, the calculation identifies the isomorphism
induced by the natural comparison. That comparison was defined on
actual complexes, so it is compatible with \(\delta\) and natural in
all gauge morphisms. The result therefore holds separately in every
inner \(R\)-complex degree, before inner totalization, as asserted.
\end{proof}

\subsection{Independence of the interval}
\label{subsec:independence-of-the-interval}

For coherent inputs, the agreement after reduction determines the
comparison itself. Completeness is the key to this step, which yields
independence of the actual \(R\)-complexes
\(\rmL^i\widehat\rmt(M)\). We then use these complexes to construct
functorial representatives of diagonal objects, preparing the passage
from objects to morphisms.

\paragraph*{Independence for coherent inputs.}

\begin{theorem}\label{interval-independence-derived}
Let \(\I\subseteq \mathrm{J}\) be finite intervals and let
\(M\in D_c^b(\FGauge\td \I)\). The natural comparison induces an
isomorphism
\begin{equation}\label{interval-derived-comparison}
 \rmL\widehat{\rmt}_{\mathrm{J}}(E_{\mathrm{J}\I}M)
       \xrightarrow{\ \sim\ }
 \rmL\widehat{\rmt}_\I(M)
 \qquad\text{in }D(\Ch(R)).
\end{equation}
Consequently, for every \(a\in\Z\), there is a natural isomorphism of
actual \(R\)-complexes
\[
 \rmL^a\widehat{\rmt}_{\mathrm{J}}(E_{\mathrm{J}\I}M)
       \xrightarrow{\ \cong\ }
 \rmL^a\widehat{\rmt}_\I(M).
\]
In particular, for a coherent \(F\)-gauge structure, both
\(\rmL^{-1}\widehat{\rmt}\) and
\(\rmL^0\widehat{\rmt}=\widehat{\rmt}\) are independent of the chosen
finite interval. The comparisons are compatible with successive
enlargements.
\end{theorem}

\begin{proof}
First suppose that \(\mathrm{J}=\I^+\) or \(\mathrm{J}=\I^-\). By
\Cref{resolution of coherent complexes,Exactness after completion},
choose a bounded-above resolution \(P^\bullet\to M\), each of whose terms
is a finite direct sum of \(\widehat G_{\I,r}\)'s.
Change of level is exact, and
\Cref{interval-old-generators-acyclic} shows that the same resolution
computes the functor at the larger level. Hence, putting
\[
 X_\I=\widehat{\rmt}_\I(P^\bullet),\qquad
 X_{\mathrm{J}}=\widehat{\rmt}_{\mathrm{J}}(E_{\mathrm{J}\I}P^\bullet),
\]
we obtain representatives of the two objects in
\eqref{interval-derived-comparison}, together with the comparison
\(X_{\mathrm{J}}\to X_\I\).
Fix the inner \(R\)-complex degree \(s\), and take the cone in the outer direction:
\[
 C^s=\operatorname{Cone}
       \bigl(X_{\mathrm{J}}^{\bullet,s}\longrightarrow X_\I^{\bullet,s}\bigr)
       \in D(R).
\]
The generator comparison in \Cref{interval-r1-generator-comparison} applies to every term \(P^a\).
The right-free resolution of \(R_1\) in the proof of
\Cref{Basic Vanishing} has uniformly finite length.
Hence totalizing in the resolution and \(R_1\)-resolution directions
involves only finitely many terms in each total degree and gives
\begin{equation}\label{interval-cone-r1-vanishing}
 R_1\otimes_R^{\rmL} C^s=0
 \qquad\text{for every }s.
\end{equation}
No totalization in the inner direction is taken here.

To deduce vanishing from this reduction, we verify that \(C^s\) is
Ekedahl-complete.
For \(X_\I^{\bullet,s}\), use the chosen bounded-above
\(\I\)-projective resolution: its terms are finite sums of grading
shifts of \(\widehat R\), or of the top module
\(\widehat R^{\mathrm{top}}\). These modules are complete.
For the top module, this can also be checked directly: its differential
\(d\) is zero, \(F\) is invertible, and \(V=pF^{-1}\), so its
Ekedahl completion is its \(p\)-adic completion.
\Cref{embedding-row-completeness} shows that a bounded-above complex
of these complete modules is complete.

For \(X_{\mathrm{J}}^{\bullet,s}\), use a bounded-above resolution of
\(E_{\mathrm{J}\I}M\) by the completed projective generators
\(\widehat G_{\mathrm{J},t}\). The same argument proves that the resulting
outer complex in inner \(R\)-complex degree \(s\) is complete. Since it represents
\(X_{\mathrm{J}}^{\bullet,s}\), both \(X_{\mathrm{J}}^{\bullet,s}\) and \(C^s\) are
complete. This uses a resolution at the larger level; completeness of
\(\widehat{\rmt}_{\mathrm{J}}(\widehat G_{\I,r})^s\) is not required.

Since \(C^s\) is bounded above and complete,
\Cref{R1 detects vanishing of complete complexes} and
\eqref{interval-cone-r1-vanishing} give \(C^s=0\) for every \(s\).
Evaluation in each inner \(R\)-complex degree is exact,
and these evaluations jointly detect outer quasi-isomorphisms.
We therefore obtain \eqref{interval-derived-comparison} in
\(D(\Ch(R))\), not merely after totalization in \(D(R)\).
Taking outer cohomology proves the asserted isomorphisms in \(\Ch(R)\).

An arbitrary finite enlargement is a succession of elementary left
and right enlargements. The cocycle identity
\eqref{interval-comparison-cocycle}, and its derived version, show
that the resulting comparison is independent of the succession chosen
and is natural in \(M\).
\end{proof}

\begin{remark}\label{interval-independence-scope}
The proof uses the bounded-above resolutions with finite direct sums
in each degree supplied by
\Cref{resolution of coherent complexes,Exactness after completion}.
By \Cref{embedding-row-completeness}, these suffice for the completeness
argument, even when the projective resolutions have infinite length.
The intermediate syzygies need not be coherent as \(W\)-modules, and
coherence of \(\rmL^{-1}\widehat{\rmt}(M)\) as a diagonal complex is
not assumed in advance. Likewise, \(p\)-torsion-freeness of a resolution
does not imply \(p\)-torsion-freeness of its cohomology: the derived
functor retains nonzero \(p\)-torsion, including domino contributions.
\end{remark}

\paragraph*{Strict representatives for diagonal complexes.}
The complexes \(\mathcal S_\I(N)\) of
\Cref{embedding-natural-representative} are already strict representatives.
We next compare their chain maps with derived morphisms to obtain a
fully faithful embedding across levels.

\paragraph*{Completeness of row resolutions.}
The resolutions used here need only be bounded above. A uniform bound
on the auxiliary degrees in the completion construction is enough to
preserve completeness.

\begin{lemma}\label{embedding-row-completeness}
Let $P^\bullet$ be a bounded-above complex of graded $R$-modules whose
terms, regarded as objects of $D(R)$, are Ekedahl-complete. Then
$P^\bullet$ is Ekedahl-complete. In particular, a graded $R$-module
represented by a bounded-above complex of finite direct sums of shifts
of $\widehat R$ is Ekedahl-complete.
\end{lemma}

\begin{proof}
For every $a\geq1$, the following sequence for the right $R$-module
$R_a$ of \Cref{reduction by R_n} is exact:
\begin{equation}\label{embedding-rn-resolution}
0\longrightarrow R(-1)
 \xrightarrow{\binom{F^a}{-F^ad}\,\bullet}
 R(-1)\oplus R
 \xrightarrow{(dV^a,V^a)\,\bullet}
 R\longrightarrow R_a\longrightarrow0.
\end{equation}
The entries act by left multiplication on free right modules.

Denote the three-term free complex in
\eqref{embedding-rn-resolution} by $K_a$, placed in degrees $[-2,0]$.
The quotient maps $R_{a+1}\to R_a$ lift to maps $K_{a+1}\to K_a$
given, in degrees $-2,-1,0$, respectively, by
\[
 p,\qquad \operatorname{diag}(V,V),\qquad\mathrm{id}.
\]
Compatibility follows from $VF^{a+1}=pF^a$ and
$dV^aV=dV^{a+1}$.

Compute derived completion with $K_a$ and the usual product-cone
model for $R\varprojlim_a$. The $K_a$-direction has degrees $[-2,0]$,
and the derived inverse-limit direction has degrees $[0,1]$.
The combined auxiliary direction therefore has the uniform finite range
$[-2,1]$. These computations may be checked on the underlying graded
complexes; the completion map retains its natural $R$-linear structure.

Apply this model to $P^\bullet$. For each fixed complex degree $s$,
the corresponding column computes the completion of $P^s$, and the
natural map from $P^s$ to that column is a quasi-isomorphism by
hypothesis. These columnwise quasi-isomorphisms induce a
quasi-isomorphism on total complexes. Indeed, the auxiliary direction
has uniform finite width, so the associated spectral sequence has a
finite filtration in every total degree. Interchanging the
inverse-limit model with totalization uses only products and finite
direct sums in each degree. The resulting total complex is precisely
the model for the derived completion of $P^\bullet$.
Hence $P^\bullet\to\widehat{P^\bullet}$ is a quasi-isomorphism.

The final assertion follows because $\widehat R$ and its grading
shifts are complete and completeness is invariant under
quasi-isomorphism.
\end{proof}

\subsection{Canonical representatives}
\label{subsec:canonical-representatives}

Independence of the interval gives representatives without a preferred
level. It is natural to ask whether their chain maps also recover all
morphisms in the diagonal heart. We show that they do, uniquely, and
obtain an exact fully faithful embedding into \(R\)-complexes, with
no ambiguity from chain homotopies.

\paragraph*{Strict maps between the representatives.}
We first give a criterion comparing chain maps with derived morphisms.

\begin{lemma}\label{embedding-hom-comparison}
Let $C,D\in\Ch(R)$ be bounded complexes such that $C^s$ and $D^s$
are supported in Raynaud-module degrees $-s$ and $-s+1$.
Assume the following:
\begin{enumerate}
 \item For every $s$, the graded module $C^s$ has a bounded-above
 representative in $D(R)$ whose terms are finite direct sums of
 $\widehat R(s)$.
 \item Every graded module $D^s$ is Ekedahl-complete.
\end{enumerate}
Then localization induces a natural bijection
\[
 \operatorname{Hom}_{\Ch(R)}(C,D)
       \xrightarrow{\ \cong\ }
 \operatorname{Hom}_{D(R)}(C,D).
\]
Moreover every chain homotopy between maps $C\to D$ is zero.
\end{lemma}

\begin{proof}
All Hom and Ext groups in this proof preserve the Raynaud-module
grading. We first prove
\begin{equation}\label{embedding-backward-ext-vanishing}
 R\operatorname{Hom}_R(C^s,D^v)=0\qquad(v<s).
\end{equation}
For a complete target $Y$, the adjunction defining derived completion
gives
\[
 R\operatorname{Hom}_R(\widehat R(s),Y)
   \simeq R\operatorname{Hom}_R(R(s),Y)
   \simeq Y^{-s}.
\]
Here $Y^{-s}$ denotes evaluation in Raynaud-module degree $-s$.
Taking $Y=D^v$ and $v<s$ makes this zero, since $D^v$ is supported
in degrees $-v,-v+1$, both strictly larger than $-s$.

To pass from the free coefficient modules to the row module, use the
representative $P_s^\bullet$ in (1). Every finite stupid
truncation of $P_s^\bullet$ is obtained by finite cones from finite
sums of $\widehat R(s)$, and consequently has zero derived Hom into
$D^v$ when $v<s$. The bounded-above complex $P_s^\bullet$ is the
homotopy colimit of these finite stupid truncations. Derived Hom in
the first variable converts this homotopy colimit into a homotopy
limit, proving \eqref{embedding-backward-ext-vanishing}.
No finite-length row resolution is required.

We now use this vanishing in the bounded-complex hyper-Ext spectral sequence
\[
 E_1^{p,q}=\bigoplus_s
     \operatorname{Ext}_R^q(C^s,D^{s+p})
 \quad\Longrightarrow\quad
     \operatorname{Hom}_{D(R)}(C,D[p+q]).
\]
The terms with $q<0$ vanish because $C^s,D^{s+p}$ are modules.
Those with $p<0$ vanish by
\eqref{embedding-backward-ext-vanishing}. In total degree zero,
there is therefore only the term $(p,q)=(0,0)$, and
\[
 E_2^{0,0}
  =\ker\left(
    \bigoplus_s\operatorname{Hom}_R(C^s,D^s)
     \longrightarrow
    \bigoplus_s\operatorname{Hom}_R(C^s,D^{s+1})
   \right)
  =\operatorname{Hom}_{\Ch(R)}(C,D).
\]
There is no quotient by degree-minus-one Hom, since that group is
already zero. No higher differential enters this term, because its
source would have $p<0$, and none leaves it, because its target would
have $q<0$. The edge isomorphism is the map induced by localization,
which proves the asserted bijection.

Finally, a chain homotopy would have components
$C^s\to D^{s-1}$, and each such Hom group is zero by
\eqref{embedding-backward-ext-vanishing}.
\end{proof}

\begin{remark}
Interval independence allows us to meet the hypotheses of
\Cref{embedding-hom-comparison} by enlarging the upper endpoint beyond
the Raynaud-module support of the representatives. Every nonzero row
then has the form of an ordinary row of
the completed generator complex, with coefficient $\widehat R(s)$;
the exceptional top module does not occur in that row. If outer
cohomology is concentrated in degree zero, taking the corresponding
row of the projective gauge resolution supplies the representative
in (1), and \Cref{embedding-row-completeness} supplies (2).
\end{remark}

\paragraph*{The embedding.}

\begin{theorem}\label{embedding-diagonal-complexes}
\label{interval-embedding-goal}
There is an exact fully faithful functor
\[
 \mathcal S:\dc\longrightarrow\Ch(R)
\]
and a natural isomorphism \(q\mathcal S\simeq\mathrm{id}_{\dc}\), where
\(q:\Ch(R)\to D(R)\) is localization. Its values are supported in
bidegrees \((u,-u)\) and \((u,1-u)\).
For an object \(N\) of finite level \(\I\), this functor is naturally
identified with
\[
 \mathcal S_\I(N)=H^0_{\mathrm{out}}
       \bigl(\rmL\widehat{\rmt}_\I(\FG^{\Tot}(N))\bigr).
\]
These identifications commute with enlargement of the interval.
Moreover, every degree \(-1\) graded \(R\)-linear map between two such
representatives is zero; in particular, there are no nonzero chain
homotopies between them.
\end{theorem}

\begin{proof}
For a fixed interval, \Cref{embedding-natural-representative} gives
the representative property and exactness. To prove full faithfulness,
put \(C=\mathcal S_\I(N)\) and \(D=\mathcal S_\I(N')\).
By \Cref{interval-independence-derived}, we may enlarge \(\I\) to a
common interval \(\mathrm{J}\) whose upper endpoint lies strictly above the
\(R\)-degree supports of both \(C\) and \(D\).

Choose bounded-above resolutions of
\(\FG^{\Tot}(N)\) and \(\FG^{\Tot}(N')\) by finite direct sums of the
completed \(\mathrm{J}\)-projective generators. By
\Cref{completed rmt on diagonal heart}, the outer complex in each
inner \(R\)-complex degree \(s\) represents \(C^s\), respectively \(D^s\), as an
object of \(D(R)\). For every \(s\) with \(C^s\ne0\) or \(D^s\ne0\),
the corresponding row terms are finite direct sums of \(\widehat R(s)\).
By \Cref{embedding-row-completeness},
each \(D^s\) is complete. Thus \Cref{embedding-hom-comparison} gives
the natural bijection
\[
 \operatorname{Hom}_{\Ch(R)}(C,D)
 \xrightarrow{\ \sim\ }
 \operatorname{Hom}_{D(R)}(qC,qD)
 \cong \operatorname{Hom}_{\dc}(N,N').
\]
Under these identifications, the map induced by \(\mathcal S_\I\)
is the inverse bijection, by naturality of the representative
isomorphism. The same lemma gives the assertion about homotopies.

To assemble the functor across levels, use the comparison isomorphisms
for \(\I\subseteq \mathrm{J}\). They satisfy the cocycle identity
\eqref{interval-comparison-cocycle}, and their totalizations respect
the natural identification with \(N\).
Choose a level for each object, and define a morphism by passing to
any common larger interval and transporting back along the
comparison isomorphisms. The cocycle identity makes this independent
of the common interval and compatible with composition. This gives
the asserted functor on \(\dc\), together with its natural
representative isomorphism. A short exact sequence can be placed in
one common interval, so exactness also holds globally.
\end{proof}

\begin{corollary}\label{embedding-diagonal-cohomology}
Let \(M\in D_c^b(\FGauge\td \I)\), and put
\[
 N=\widehat{(\rmL\rmt)^{\Tot}}(M).
\]
For every \(a\in\Z\), there is a natural isomorphism of actual
\(R\)-complexes
\[
 \rmL^a\widehat{\rmt}_\I(M)
 \cong \mathcal S_\I\bigl(\widetilde{\HH}^{\,a}(N)\bigr).
\]
In particular, \(\rmL^a\widehat{\rmt}_\I(M)\) represents
\(\widetilde{\HH}^{\,a}(N)\), and is a coherent diagonal complex
after passage to \(D(R)\).
\end{corollary}

\begin{proof}
The completed coherent equivalence gives \(N\in D_c^b(R\td \I)\).
Since diagonal truncations preserve level \(\I\), the bounded diagonal
\(t\)-structure is available on this fixed-level category.
Consider the triangulated functor
\[
 \mathcal X_\I=\rmL\widehat{\rmt}_\I\circ\FG^{\Tot}
 :D_c^b(R\td \I)\longrightarrow D(\Ch(R)).
\]
By \Cref{completed rmt on diagonal heart}, it takes the diagonal
heart to the standard outer heart. The diagonal \(t\)-structure
on the source is bounded; induction on its truncation triangles
therefore shows that \(\mathcal X_\I\) is \(t\)-exact for these two
\(t\)-structures. Consequently it commutes with their cohomology
functors:
\[
 H^a_{\mathrm{out}}\mathcal X_\I(N)
 \cong
 \mathcal S_\I\bigl(\widetilde{\HH}^{\,a}(N)\bigr).
\]
The completed equivalence identifies \(\FG^{\Tot}(N)\) with \(M\).
Substitution gives the claimed isomorphism, and the representative
property follows from \Cref{embedding-diagonal-complexes}.
\end{proof}

\section{Diagonal complexes revisited}
\label{sec:diagonal-complexes-revisited}

\subsection{Acyclic diagonal complexes}
\label{subsec:acyclic-diagonal-complexes}

The completed equivalence allows us to return to diagonal complexes
with a concrete way to describe their torsion. We begin with the
acyclic objects, which totalization alone cannot distinguish from zero.
Under \(\FG^{\Tot}\), these objects correspond precisely to coherent
gauges with vanishing boundary value; the individual gauge components
retain the information that totalization loses.

\begin{proposition}[{\cite[Proposition~III.1.1]{Ek3}}]\label{acyclic diagonal complexes and boundary 0 F-gauge structures}
The functor \(N\mapsto\HH^1(\FG^{\Tot}(N))\) induces an equivalence between
the category of \(\Tot\)-acyclic diagonal complexes and the category of
coherent \(F\)-gauges \(M\) satisfying \(M^\infty=0\).
\end{proposition}

\begin{proof}
If \(N\) is \(\Tot\)-acyclic, it is \(\Tot\)-\(1\)-torsion, so
\(N[1]\in\dc_g\). By
\Cref{Ekedahl equivalence and the G t-structure},
\(\FG^{\Tot}(N)\) has cohomology only in
degree \(1\), and \(\HH^1(\FG^{\Tot}(N))\) is a coherent gauge with
vanishing boundary value.

Conversely, let \(M\) be a coherent gauge with \(M^\infty=0\). Then
\(\Tot(\rmL\widehat\rmt(M))=0\), and
\Cref{sub and quotient object of Tot-torsion objects}(4) makes each
diagonal cohomology object \(\Tot\)-acyclic. By
\Cref{Ekedahl equivalence and the G t-structure},
\(\rmL\widehat\rmt(M)\in\dc_g\). The description of this heart in the
proof of \Cref{criteria for G-amplitude [m n]} leaves only diagonal
degrees \(-1,0\); its degree-zero object has no nonzero
\(\Tot\)-\(1\)-torsion quotient. Since that object is itself
\(\Tot\)-acyclic, it vanishes. Thus
\(\rmL\widehat\rmt(M)[-1]\) is a \(\Tot\)-acyclic diagonal complex.
The completed equivalence and
\(\HH^1(\FG^{\Tot}(N))\cong\HH^0(\FG^{\Tot}(N[1]))\) show that
the two constructions are inverse.
\end{proof}

The following interval gauges give explicit models for acyclic
objects killed by \(p\). Their components are especially simple, but
the choices of arrows between them record nontrivial extensions.

\begin{definition}[{\cite[Definition~III.1.2]{Ek3}}]\label{M(I J)}
Let \(\I=[m,n]\) be a finite interval and \(J\subseteq[m,n-1]\).
Define \(M(\I,J)\) by
\[
M(\I,J)^i=\begin{cases}k,&i\in\I,\\0,&i\notin\I,\end{cases}
\]
with \(\tF:M(\I,J)^i\to M(\I,J)^{i+1}\) equal to the identity for
\(i\in J\) and zero otherwise, and \(\tV\) equal to the identity on
the remaining edges of \(\I\) and zero otherwise. We use the same notation for
the associated \(\Tot\)-acyclic diagonal complex.
\end{definition}

Under \(\FG^{\Tot}\), \(U_0[1]\) corresponds to \(k\) in gauge degree
\(1\). Consequently
\(\rmL\widehat\rmt(M(\I,J)[-1])\) is an iterated extension of the
objects \(U_0[i-1](1-i)\), \(i\in[m,n]\).

We encode \(M(\I,J)\) by a directed graph \(\Gamma\), so that the
extension data can be read from its arrows. Its vertices are the integers
in \(\I\); the edge between \(i\) and \(i+1\) points to the right when
\(\tF\) is the identity and to the left when \(\tV\) is the identity.

For example, \(M(\{0,1,2\},\{1\})\) is represented by
\(0\leftarrow1\rightarrow2\).

The \(\Tot\)-acyclic objects satisfy two-out-of-three for short exact
sequences in \(\dc\). They are also
closed under kernels and cokernels in \(\dc\). Indeed, for a map
between two such objects, let \(K\) and \(Q\) be its kernel and
cokernel. The two exact sequences through its image give
\(\Tot(Q)\simeq\Tot(K)[2]\). By \Cref{Basic Vanishing}(1), both
\(\Tot(K)\) and \(\Tot(Q)\) have amplitude in \([0,1]\); the shift by
two forces both to vanish.
Thus the acyclic objects form a full abelian subcategory of \(\dc\),
closed under pullbacks and pushouts. Its indecomposable objects killed
by \(p\) have the following description.

\begin{proposition}
\label{indecomposable Tot acyclic objects are M(I J)}
The \(M(\I,J)\) are precisely the indecomposable
\(\Tot\)-acyclic diagonal complexes killed by \(p\). 
\end{proposition}

Compare \cite[Proposition~III.1.3]{Ek3}; no group acts in our setting.

\begin{proof}
By the preceding proposition, we classify indecomposable coherent
gauges \(N\) with \(N^\infty=0\) and \(pN=0\). Such a gauge has finite
support: it has finite level and zero boundary, and coherence makes
each component finite-dimensional over \(k\). Let \(\I=[m,n]\) be the
smallest interval containing its support, and induct on \(n-m\). If
\(m=n\), indecomposability gives \(N^m\cong k\). Now assume \(m<n\), and set
\(N'^m=0\) and \(N'^i=N^i\) for \(i>m\), with the
adjacent maps zero. Since \(pN=0\), this remains a coherent gauge with
zero boundary value. Decompose \(N'\) into indecomposables; by
induction, each is an interval gauge. Any summand missing degree
\(m+1\) has no link to \(N^m\) and splits off \(N\).
Thus every summand meets degree \(m+1\).

Between any two interval summands meeting degree \(m+1\), there is a
morphism in one direction whose component there is the identity. To
see this, compare their directed graphs from \(m+1\). If one is an
initial segment of the other, the first extra arrow gives an inclusion
or projection. If their first differing arrows have opposite
orientations, the graph with a right-pointing arrow projects onto the
common initial segment, which embeds into the graph with a left-pointing
arrow. For equal graphs, use the identity.

Suppose \(\tF:N^m\to N^{m+1}\) is nonzero. Choose a one-dimensional
subspace \(k\subseteq N^m\) on which \(\tF\) is nonzero, and write
\(N'=\bigoplus_iN_i\). If \(\tF(k)\) has nonzero components in two
summands, the morphism just constructed gives a shear
\(\mathrm{id}+\varphi\) that eliminates one component. Repeating this,
we may assume
\(\tF(k)\subseteq N_{i_0}^{m+1}\) for one \(i_0\). Let
\(\pi:N^{m+1}\to N_{i_0}^{m+1}\) be the projection and put
\(T=\ker(\pi\tF)\). Since \(\tF(k)\) spans
\(N_{i_0}^{m+1}\cong k\), we have \(N^m=k\oplus T\). The map \(\tF\)
takes \(k\) into \(N_{i_0}^{m+1}\) and \(T\) into the other summands.
The relations \(\tV\tF=\tF\tV=p=0\) give
\(\tV(N_{i_0}^{m+1})=0\) and
\(\tV(\bigoplus_{i\ne i_0}N_i^{m+1})\subseteq T\).
All other maps respect the decomposition of \(N'\). Hence
\[
N''=(\cdots0\leftrightarrows k\leftrightarrows
N_{i_0}^{m+1}\leftrightarrows N_{i_0}^{m+2}
\leftrightarrows\cdots)
\]
is a nonzero direct summand of \(N\). By indecomposability,
\(N=N''\), an interval gauge. If \(\tF=0\) but \(\tV\ne0\) at the
boundary, apply the same argument to the finite-dimensional graded
\(k\)-dual, which exchanges \(\tF\) and \(\tV\). If both maps vanish,
\(N=N^m\oplus N'\), reducing to the one-vertex case.

Conversely, each \(M(\I,J)\) is killed by \(p\) and has zero boundary
value. Thus it gives a \(\Tot\)-acyclic object by
\Cref{acyclic diagonal complexes and boundary 0 F-gauge structures}.
Every endomorphism acts by a scalar on each one-dimensional component,
and every identity edge forces adjacent scalars to agree. Hence
\(\operatorname{End}(M(\I,J))=k\), so \(M(\I,J)\) is indecomposable.
\end{proof}

\subsection{Diagonal Dominos}
\label{subsec:diagonal-dominos}

Diagonal dominos extend the acyclic class. They are diagonal complexes
whose totalization vanishes after inverting \(p\), but which contain no
finite torsion. The class includes both classical dominos and acyclic
objects.

\begin{definition}[{\cite[Definition~III.2.1]{Ek3}}]\label{definition of diagonal domino}
A diagonal domino is a \(\Tot\)-torsion diagonal complex without finite
torsion.
\end{definition}

By \Cref{sub and quotient object of Tot-torsion objects}(1),
\(\Tot\)-torsion is preserved under subobjects and extensions. The
same holds for the absence of finite torsion: if
\(0\to A\to B\to C\to0\) has this property at both ends, then
\(t^1(B)\to C\) vanishes by the orthogonality in
\Cref{without finite torsion implies artinian}(1). Thus \(t^1(B)\)
lies in \(A\) and vanishes. Every \(\Tot\)-\(1\)-torsion object is a
diagonal domino by \Cref{sub and quotient object of Tot-torsion objects}(3).
In particular, all \(\Tot\)-acyclic objects and classical dominos belong
to this class.

By \Cref{Hodge numbers in degree plus minus 1}, Hodge numbers in total
degrees \(-1\) and \(1\) measure maps from and to shifts of \(U_0\).
These maps detect \(\Tot\)-\(0\)-torsion subobjects and
\(\Tot\)-\(1\)-torsion quotients. The next lemma also tracks these
conditions under duality and Nygaard modification.

\begin{lemma}[{\cite[Lemma~III.2.2]{Ek3}}]
\label{Tot 0 or Tot 1 of diagonal domino}
The following statements hold.

{\rm (1)} Let \(M\in\dc\) be \(\Tot\)-torsion. Then \(M\) has no nonzero
\(\Tot\td0\)-torsion subobject if and only if
\[
\operatorname{Hom}_{\dc}(U_0[k](-k),M)=0
\]
for every \(k\). Dually, \(M\) has no \(\Tot\td1\)-torsion quotient
precisely when \(\operatorname{Hom}_{\dc}(M,U_0[k](-k))=0\) for all \(k\).

{\rm (2)} If \(M\) is a diagonal domino then so is
\(
\widetilde{D}^{\,0}(M),
\)
and
\(
M=\widetilde{D}^{\,0}\!\bigl(\widetilde{D}^{\,0}(M)\bigr).
\)
Moreover, \(M\) is \(\Tot\td0\)-torsion if and only if
\(\widetilde{D}^{\,0}(M)\) is \(\Tot\td1\)-torsion.

{\rm (3)} Let \(M\) be a \(\Tot\td0\)-torsion diagonal domino and let
\(\varphi:\Z\to\N\) have finite support. Then \(M(\varphi)\) is a
\(\Tot\td0\)-torsion diagonal domino. Dually,
if \(M\) is \(\Tot\td1\)-torsion, then \(M(-\varphi)\) is a
\(\Tot\td1\)-torsion diagonal domino.
\end{lemma}

\begin{proof}
For (1), \(U_0[k](-k)\) is \(\Tot\)-acyclic. By
\Cref{sub and quotient object of Tot-torsion objects}(2), the image of
any map \(U_0[k](-k)\to M\) is \(\Tot\)-\(0\)-torsion. This proves one
direction. Conversely, suppose the Hom groups vanish for all \(k\).
Since \(\HH^0(\Tot(M))\) is torsion, the universal coefficient theorem
reduces its vanishing to
\[
\HH^{-1}\bigl(k\otimes_W^{\rmL}\Tot(M)\bigr)=0.
\]
By \Cref{Hodge numbers in degree plus minus 1}, the terms with
\(i+j=-1\) vanish in the spectral sequence
\[
\HH^{i}\!\left(R_1\otimes_R^\rmL M\right)^j
\Longrightarrow
\HH^{i+j}\!\bigl(k\otimes_W^\rmL \Tot(M)\bigr)
\]
Thus \(\HH^0(\Tot(M))=0\). If \(N\hookrightarrow M\) is a
\(\Tot\)-\(0\)-torsion subobject, then
\(\HH^0(\Tot(N))\hookrightarrow\HH^0(\Tot(M))=0\), while
\(\HH^1(\Tot(N))=0\). Hence \(N\) is \(\Tot\)-acyclic. If nonzero,
it contains some \(U_0[i](-i)\) by
\Cref{simple Tot acyclic objects are U_0^i}, contradicting the
Hom vanishing. The quotient statement is dual.

For (2), \Cref{without finite torsion implies artinian}(1) identifies
\(D(M)\) with \(\widetilde D^{\,0}(M)\), shows that it has no finite
torsion, and gives biduality. The formula
\[
\Tot(D(M))=R\!\operatorname{Hom}_W(\Tot(M),W)
\]
shows that \(\widetilde D^{\,0}(M)\) is \(\Tot\)-torsion, and that it is
\(\Tot\)-\(0\)-torsion precisely when \(M\) is
\(\Tot\)-\(1\)-torsion.

For (3), duality and
\(D(M(\varphi))=D(M)(\td\varphi\td)\), where
\(\td\varphi\td(a)=-\varphi(-a)\), reduce the claims to the
\(\Tot\)-\(1\)-torsion case under negative modification. Since twists
add, it suffices to consider \(-\delta_i\). A \(\Tot\)-\(1\)-torsion domino \(N\)
has \(G^2(N)=0\), hence \(N[1]\in\dc_g\). Choose a level interval
containing \(i\) and the support of \(N\). By
\Cref{Ekedahl equivalence and the G t-structure},
\(\FG^{\Tot}(N)\) is concentrated in cohomological degree \(1\), and
\Cref{definition of FG}(3) gives
\[
\FG^{\Tot}(N)^i\simeq\Tot(N(-\delta_i)).
\]
Thus \(N(-\delta_i)\) has no degree-zero total cohomology. Nygaard
modification preserves diagonality and the absence of finite torsion;
its degree-one total cohomology is torsion by \Cref{Basic Vanishing}(1).
It is therefore a \(\Tot\)-\(1\)-torsion domino. Duality gives the
\(\Tot\)-\(0\)-torsion assertion.
\end{proof}

\begin{lemma}
Every nonzero \(\Tot\)-acyclic object \(M\in\dc\) has a nonzero Hodge
number in each of total degrees \(-1\) and \(1\).
\end{lemma}

\begin{proof}
Combine \Cref{simple Tot acyclic objects are U_0^i,Hodge numbers in degree plus minus 1}.
\end{proof}

Write \(m_i\) for the modification \((\td)(\varphi)\) with
\(\varphi(n)=-i\) for all sufficiently large \(n\). Thus \(m_i\)
carries \(U_i\) to \(U_0\).

The torsion conditions assemble into a canonical filtration. After
modification, its integral graded pieces are acyclic; the intervening
pieces satisfy the Hom vanishings above. The filtration thus records
how total cohomology changes under modification.

\begin{theorem}[{\cite[Theorem~III.2.3]{Ek3}}]\label{type filtration on diagonal domino}
Every diagonal domino \(M\) carries a canonical type filtration with the
following properties.

    {\rm (1)} It is a finite, exhaustive, and separated decreasing filtration:
    \[M \supseteq ...\supseteq W^{s-\frac{1}{2}}(M)\supseteq W^s(M)\supseteq W^{s+\frac{1}{2}}(M)\supseteq ...,\quad s\in \frac{1}{2}\Z. \]
    Its graded pieces satisfy the following conditions.

    {\rm (i)} If \(s\in \Z\), then \(m_s(W^s(M)/W^{s+\frac{1}{2}}(M))\) is \(\Tot\)-acyclic. 

    {\rm (ii)} If \(s\in \frac{1}{2}+\Z\), then \[\begin{aligned}
        \Hom_{\dc}(U_0[k](-k), m_{s+\frac{1}{2}}(W^s(M)/W^{s+\frac{1}{2}}(M)))=0\\
        \Hom_{\dc}(m_{s+\frac{1}{2}}(W^s(M)/W^{s+\frac{1}{2}}(M)),U_{-1}[k](-k))=0
    \end{aligned}\]

    {\rm (2)} Conditions (i) and (ii) characterize the filtration. Each
    \(W^s\) is an idempotent radical, and
    \[m_r(W^s(M))=W^{s-r}(m_r(M)).\]

    {\rm (3)} There is a natural isomorphism
    \[\widetilde{D}^0(W^s(M))\cong
    (\mathrm{Id}/W^{-s+\frac{1}{2}})\widetilde{D}^0(M).\]
\end{theorem}

\begin{proof}
Let \(F^1,F^2,G^1,G^2\) be the two fundamental filtrations of \(M\).
Their common quotient \(F^2/F^1=G^2/G^1\) is Mazur--Ogus. Its first
total cohomology vanishes by \Cref{Mazur-Ogus vanishing}, and its
zeroth vanishes because \(M\) is \(\Tot\)-torsion. A Mazur--Ogus object
that is \(\Tot\)-acyclic is zero, leaving
\[
0\subseteq G^1\subseteq F^1\subseteq M,
\]
where \(F^1\) is the largest \(\Tot\)-\(0\)-torsion subobject,
\(M/G^1\) is the largest \(\Tot\)-\(1\)-torsion quotient, and
\(F^1/G^1\) is \(\Tot\)-acyclic. Thus \(F^1\) and \(G^1\) are
admissible idempotent radicals.

We next need the inclusions
\[
m_r(F^1(M))\subseteq G^1(m_rM)\quad(r<0),
\qquad
F^1(m_rM)\subseteq m_r(G^1(M))\quad(r>0).
\]
The second is dual to the first. By
\Cref{Tot 0 or Tot 1 of diagonal domino}(3), \(m_r(F^1(M))\) is
\(\Tot\)-\(0\)-torsion. Its image in
\(F^1(m_rM)/G^1(m_rM)\) is also \(\Tot\)-acyclic. If this image were
nonzero, it would have a quotient \(U_0[i](-i)\); after applying \(m_{-r}\),
the \(\Tot\)-\(0\)-torsion object \(F^1(M)\) would have a quotient
\(U_r[i](-i)\) with nonzero \(\HH^1\), a contradiction. The image is
therefore zero.

For \(s\in\Z\), define
\[
W^sM:=m_{-s}F^1(m_sM),
\qquad
W^{s+\frac12}M:=m_{-s}G^1(m_sM).
\]
These admissible subobjects form a decreasing filtration by the
inclusions above. Property (i) follows from the \(\Tot\)-acyclicity of
\(F^1(m_sM)/G^1(m_sM)\). To prove (ii), put
\[
H:=m_1G^1(m_{-1}N)/F^1N
\]
for a diagonal domino \(N\). By construction, \(H\) has no nonzero
\(\Tot\)-\(0\)-torsion subobject, while \(m_{-1}H\) has no nonzero
\(\Tot\)-\(1\)-torsion quotient. The required Hom vanishings follow
from \Cref{Tot 0 or Tot 1 of diagonal domino}(1).

For uniqueness, modifications translate the indices, so we need only
recover \(W^0\) and \(W^{1/2}\). Conditions
(i) and (ii), together with
\Cref{Tot 0 or Tot 1 of diagonal domino}, give
\[
\operatorname{Hom}_{\dc}(W^{1/2},U_0[k](-k))=0,
\qquad
\operatorname{Hom}_{\dc}(U_0[k](-k),M/W^0)=0.
\]
Hence \(W^{1/2}\subseteq G^1\), while \(M/W^0\) is
\(\Tot\)-\(1\)-torsion. Since \(W^0/W^{1/2}\) is
\(\Tot\)-\(0\)-torsion, the quotient \(M/W^{1/2}\) is
\(\Tot\)-\(1\)-torsion, and therefore \(G^1\subseteq W^{1/2}\).
Thus \(W^{1/2}=G^1\); the equality \(W^0=F^1\) is dual.

Compatibility with \(m_r\) follows directly from the definition. Each
\(W^s\) is an idempotent radical because \(F^1\) and \(G^1\) are.
Part (3) follows from \Cref{Tot 0 or Tot 1 of diagonal domino}(2) and
the identities
\[
\widetilde D^0(F^1(-))=\widetilde D^0(-)/G^1\widetilde D^0(-),
\qquad
\widetilde D^0(G^1(-))=\widetilde D^0(-)/F^1\widetilde D^0(-).
\]

Finally, \Cref{without finite torsion implies artinian} yields
\(s_0\) such that \(W^s=W^{s_0}\) for all \(s\geqslant s_0\). For
this stable term, every sufficiently positive modification is
\(\Tot\)-\(0\)-torsion. By
\Cref{Tot 0 or Tot 1 of diagonal domino}(3), the same holds after every
modification \(\varphi:\Z\to\Z\). The finite-torsion criterion
\Cref{description of finite torsion object} then makes \(W^{s_0}\)
finite torsion, hence zero. This proves separation and finiteness in
one direction. Duality and (3) give exhaustiveness and finiteness in
the other.
\end{proof}

Since each \(W^s\) is an idempotent radical, the classes defined by
\(W^s(X)=X\) and by \(W^s(X)=0\) are both closed under extensions.
This does not make \(W^s\) exact: the two-vertex example below shows
that \(W^0\) is neither left nor right exact, and modification gives
the same failure for every integral index.

\begin{definition}[{\cite[Definition~III.2.4]{Ek3}}, with our type convention]\label{type function}
For a diagonal domino \(M\), define its type to be the function
\(\sigma:\frac12\Z\to\N^{\Z}\) given by
\[\sigma(s)(i)= T^{-i,i}\!\left( W_M^s/W_M^{s+1/2}
\right).\]
We say that \(M\) is of type \(s\) if
\(\sigma(t)\) is the zero function whenever \(t\neq s\).
\end{definition}

Type \(0\) diagonal dominos are precisely the \(\Tot\)-acyclic objects.
Since \(m_r\) translates the filtration, it reduces type \(r\) to
type \(0\) and type \(r+\frac12\) to type \(\frac12\). Thus it suffices
to understand type \(\frac12\). The next criterion identifies these
objects within the gauge heart by an orthogonality condition.

\begin{proposition}[{\cite[Proposition~III.2.5]{Ek3}}]\label{when is a torsion object in G a type 1/2 diagonal domino}
A diagonal domino of type \(\frac12\) lies in \(G\). Conversely, a
torsion object \(M\in G\) is a diagonal domino of type \(\frac12\) if
and only if
\[
\Hom_G(U_1[k](-k),M)=0
\]
for every \(k\).
\end{proposition}

\begin{proof}
An object of type \(\frac12\) has no nonzero
\(\Tot\)-\(1\)-torsion quotient, so it lies in \(G\). Its defining
orthogonality gives
\[
\Hom_G(U_1[k](-k),M)=\Hom_{\dc}(U_1[k](-k),M)=0.
\]
Conversely, let \(M\in G\) be torsion with these Hom groups zero.
Its diagonal truncation sequence in \(G\) is
\[0\longrightarrow \widetilde \HH^{-1}(M)[1]
\longrightarrow M \longrightarrow \widetilde \HH^0(M) \longrightarrow 0. \]
Hence
\(\Hom_G(U_1[k](-k),\widetilde\HH^{-1}(M)[1])=0\) for every \(k\).

\begin{lemma}[{\cite[Lemma~III.2.5.1]{Ek3}}]
Let \(N\in\dc\). If, for some \(i\),
\[
\Hom_{D(R)}(U_i[k](-k),N[1])=0
\]
for every \(k\), then \(N=0\).
\end{lemma}

\begin{proof}
Applying \(m_i\), we may assume \(i=0\). By
\Cref{Hodge numbers in degree plus minus 1}, the terms of total
degree zero vanish in the spectral sequence
\[\HH^i(R_1\otimes_R^{\rmL} N)^j \Longrightarrow
\HH^{i+j}\!\bigl(k\otimes_W^{\rmL}\Tot(N)\bigr)
\]
and hence \(\HH^0(k\otimes_W^{\rmL}\Tot(N))=0\). The universal
coefficient sequence and Nakayama's lemma give
\(\HH^0(\Tot(N))=0\) and \(\HH^1(\Tot(N))[p]=0\). Since
\(\HH^1(\Tot(N))\) is torsion, it too vanishes. Thus \(N\) is
\(\Tot\)-acyclic. The same
spectral sequence now forces every Hodge number to vanish, giving
\(R_1\otimes_R^{\rmL}N=0\). Conservativity yields \(N=0\).
\end{proof}

The lemma gives \(\widetilde\HH^{-1}(M)=0\), so \(M\in\dc\). We have
\(\Hom_{\dc}(U_1[k](-k),M)=0\), while
\Cref{Tot 0 or Tot 1 of diagonal domino}(1) gives
\(\Hom_{\dc}(M,U_0[k](-k))=0\). Together these are precisely the conditions
for type \(\frac12\).
\end{proof}

\begin{corollary}\label{type 1/2 F-gauge structures}
The category of type \(1/2\) diagonal complexes is equivalent to the
category of torsion coherent \(F\)-gauges \(M\) satisfying the following
condition:

\[
(*)\qquad
\left\{
\begin{array}{l}
\text{For any } i,\ \text{if }
x\in M^{i+1},\ y\in M^{i-1},
\\[2mm]
\tV x=0,\quad \tF y=0,\quad
\tau \text{ applied to the image of }x
\text{ in }M^\infty
\text{ equals the image of }y
\text{ in }M^{-\infty},
\\[2mm]
\text{then } x=y=0.
\end{array}
\right.
\]
\end{corollary}

This corresponds to \cite[Corollary~III.2.5.1]{Ek3}
under the same change of type convention.

\begin{proof}
The preceding proposition reduces the claim to computing maps from
\(\FG^{\Tot}(U_1[k](-k))\), as \(k\) varies. This gauge complex has
cohomology only in degree zero, and its cohomology gauge has components
\[
\HH^0\!\bigl(\FG^{\Tot}(U_1[k](-k))\bigr)^i
=
\begin{cases}
k,& i\neq k+1,\\
0,& i=k+1.
\end{cases}
\]
On this gauge, \(\tF\) is the identity above degree \(k+1\) and zero
otherwise, while \(\tV\) is the identity below degree \(k+1\) and
zero otherwise; \(\tau\) is the \(p\)-power map. Thus, for any
\(F\)-gauge \(M\),
\[
\Hom\!\bigl(\FG^{\Tot}(U_1[k](-k)),M\bigr)
=
\Bigl\{
(x,y)\in M^{k+2}\times M^k
:\tV x=0,\ \tF y=0,\ \tau(x_{\infty})=y_{-\infty}
\Bigr\}.
\]
Replacing \(k+1\) by \(i\) gives exactly condition \((*)\).
\end{proof}

We next ask which modifications of acyclic objects have type \(1/2\).
Decorating their graphs will also help classify the weakly simple
objects.
Let \(M\in\dc\) be \(\Tot\)-acyclic. Put
\[
I_1:=\{\,i\in\Z :
\Hom_{\dc}(M,U_0[i](-i))\neq 0
\,\},\qquad 
I_2:=\{\,i\in\Z :
\Hom_{\dc}(U_0[i](-i),M)\neq 0
\,\},
\]
and assume \(I_1\cap I_2=\varnothing\). Choose \(J\subseteq\Z\) so
that \(i+1\in J\) for \(i\in I_1\), whereas \(i+1\notin J\) for
\(i\in I_2\), and put \(N:=M(\varphi_J)\), where
\[
\varphi_J(n)=
\begin{cases}
0,& n\notin J,\\
1,& n\in J.
\end{cases}
\]

Then \(N\) has type \(1/2\). A nonzero map
\(U_1[i](-i)\to N\) would, after twisting back, give either a map
\(U_1[i](-i)\to M\), contradicting the type filtration, or a map
\(U_0[i](-i)\to M\), contradicting the choice of \(J\). Dually, a nonzero
map \(N\to U_0[i](-i)\) would give either \(M\to U_0[i](-i)\), again excluded
by the choice of \(J\), or \(M\to U_{-1}[i](-i)\), contradicting
\(W^{-1/2}M=M\) and \(W^{-1/2}U_{-1}[i](-i)=0\).

Conversely, suppose \(N\) has type \(1/2\) and
\(M:=N(-\varphi_J)\) is \(\Tot\)-acyclic. Reversing the argument gives
\(i+1\in J\) for \(i\in I_1\) and \(i+1\notin J\) for \(i\in I_2\).

A source is a vertex with no incoming arrows, and a sink is a vertex
with no outgoing arrows; at an endpoint, only its single edge matters.
For \(M(\I,J)\), the sets \(I_1\) and \(I_2\) comprise, respectively,
the source and sink indices shifted by \(-1\). If \(K\) satisfies the
preceding condition, the twist
\(M(\I,J)(\varphi_K)\) is therefore a diagonal domino of type
\(1/2\). We represent it by the directed graph of \(M(\I,J)\), marking
the vertices in \(K\) with a superscript \(*\). For example, if
\(\I=[0,4]\), \(J=\{0,2,3\}\), and \(K=\{2\}\), the graph is
\[
0 \rightarrow 1 \leftarrow 2^{*}
\rightarrow 3 \rightarrow 4.\]
The type \(1/2\) condition says exactly that the graph has neither an
undistinguished source nor a distinguished sink. The forbidden local
patterns are
\[
\cdots\!\leftarrow\!i\!\rightarrow\!\cdots, 
\qquad i\!\rightarrow\!\cdots, \qquad
\text{or} \qquad \cdots\!\leftarrow\!i,
\]
and
\[\cdots\!\rightarrow\!i^*\!\leftarrow\!\cdots,
\qquad i^*\!\leftarrow\!\cdots,\qquad\text{or}
\qquad \cdots\!\rightarrow\!i^*.
\]

The next calculation reads the vertex multiplicities of an acyclic
object from its cohomology gauge.

\begin{lemma}[{\cite[Lemma~III.2.6.1]{Ek3}}]\label{domino numbers of Tot acyclic diagonal dominos}
The following statements hold.

{\rm (1)} \(T^{\,i,-i}(-)\) is additive for short exact sequences in \(\dc\).

{\rm (2)} \(T^{\,i,-i}(-\varphi)=T^{\,i,-i}\) for any \(\varphi: \Z\to\Z\) of finite support.

{\rm (3)} If \(M\in\dc\) is \(\Tot\)-acyclic, then
\(T^{i,-i}(M)=\operatorname{length}_W\!\bigl(\HH^1(\FG^{\Tot}(M))^{i+1}\bigr).\)
\end{lemma}

\begin{proof}
For an exact sequence \(0\to M_1\to M_2\to M_3\to0\) in \(\dc\),
the connecting map \(\HH^{i-1}(M_3)\to\HH^i(M_1)\) can be nonzero
only in module degree \(-i+1\), where its target is killed by a power
of \(V\). Its image therefore has finite length. The cohomology sequence
is exact modulo finite length, proving (1). For (2),
\(T^{i,-i}\) depends only on the underlying \(R^0\)-complex, not on
\(d\). For (3), the equivalence of
\Cref{acyclic diagonal complexes and boundary 0 F-gauge structures}
takes the simple object \(U_0[i](-i)\) to \(k\) in gauge degree
\(i+1\). Apply (1) and dévissage in the acyclic category.
\end{proof}

\begin{proposition}[{\cite[Proposition~III.2.6]{Ek3}}]\label{swap direction isomorphic}
There is an isomorphism of diagonal dominos
\[
(\cdots i^*\!\longrightarrow i+1\cdots)
\cong
(\cdots i\longleftarrow(i+1)^*\cdots),
\]
with every other vertex and arrow fixed.
\end{proposition}

\begin{proof}
Apply the inverse Nygaard modification at every distinguished vertex
common to both graphs. We may thus assume that \(i\) and \(i+1\) are,
respectively, the only distinguished vertices on the two sides. Let
\(M(\I,J)\) be the acyclic interval object underlying the left graph,
with \(i\in J\), and put
\(\theta=\varphi_{\{i\}}-\varphi_{\{i+1\}}\). The claim becomes
\(M(\I,J)(\theta)\cong M(\I,J\setminus\{i\})\).

The object \(M(\I,J)\) is characterized by
\[
\dim_k M(\I,J)^j=
\begin{cases}
1,&j\in\I,\\
0,&j\notin\I,
\end{cases}
\]
and by the requirement that
\(M([j,j+1],\{j\}\cap J)\) occur as a subquotient for each
\(j\in[m,n-1]\). Assuming the two-vertex case, this characterization
settles the general case directly. Regard \([m,i-1]\) and
\([i+2,n]\) as two tails. An arrow pointing into a tail makes it a
subobject, and an arrow pointing out makes it a quotient. This gives a filtration with
\(M([i,i+1],\{i\})\) as its middle subquotient. The tails are fixed
by \(\theta\), since their \(R\)-degrees avoid \(i\); the two-vertex
calculation replaces the middle factor by the opposite acyclic
interval. Exactness of Nygaard modification in \(D(R)\) and extension
closure therefore show that \(M(\I,J)(\theta)\) is acyclic. It remains
indecomposable and killed by \(p\), so it is an interval object by
\Cref{indecomposable Tot acyclic objects are M(I J)}. By
\Cref{domino numbers of Tot acyclic diagonal dominos}(2)--(3), its
cohomology gauge has the same vertex dimensions as \(M(\I,J)\).
The tails retain their subobject or quotient positions, and their
boundary extensions remain nonsplit. Thus their arrows are unchanged,
while the central arrow reverses. These are exactly the required
two-vertex subquotients, so the characterization gives
\(M(\I,J)(\theta)\cong M(\I,J\setminus\{i\})\).

It remains to prove the two-vertex case. After shifting, assume
\(i=0\) and put
\(Q=(0^*\!\rightarrow1)(-\varphi_{\{1\}})\).
The exact sequence \(0\to(1)\to(0\to1)\to(0)\to0\) gives
\[
0\longrightarrow U_{-1}\longrightarrow Q\longrightarrow
U_1[-1](1)\longrightarrow0,
\]
and hence
\(0\to\HH^0(\Tot Q)\to k\to k\to\HH^1(\Tot Q)\to0\).
If \(Q\) were not \(\Tot\)-acyclic, both end groups would be \(k\).
Then \(0\neq W^0Q\neq Q\), and
\Cref{two fundamental filtrations} shows that
\(W^0Q\to U_1[-1](1)\) is an isomorphism after totalization.
By \Cref{Totsubquo}, its kernel and cokernel are \(\Tot\)-acyclic.
The target has no nonzero acyclic quotient, and the kernel lies in
\(U_{-1}\), which has no nonzero acyclic subobject. Thus the map is an
isomorphism and the sequence splits. Twisting back would decompose
\((0\to1)\), contrary to
\Cref{indecomposable Tot acyclic objects are M(I J)}.
Therefore \(Q\) is \(\Tot\)-acyclic.

By \Cref{domino numbers of Tot acyclic diagonal dominos}(2)--(3),
\(\HH^1(\FG^{\Tot}(Q))\) has one copy of \(k\) in gauge degrees
\(0,1\) and is zero elsewhere. Modification also preserves
indecomposability, so \(Q\) is either \((0\to1)\)
or \((0\leftarrow1)\). The first is impossible: twisting it by
\(\varphi_{\{1\}}\) would give an undistinguished source and a
distinguished sink, whereas
\(Q(\varphi_{\{1\}})=(0^*\!\rightarrow1)\) has type \(1/2\).
Hence \(Q\cong(0\leftarrow1)\). The preceding reduction proves the
general case; applying \(\varphi_{\{i+1\}}\) and restoring the common
distinguished vertices gives the stated isomorphism.
\end{proof}

The type filtration has useful functoriality, but its terms need not
be exact functors. The sequence
\(0\to U_{-1}[1](-1)\to Q\to U_1\to0\) shows that \(W^0\) is neither left
nor right exact. Diagonal dominos of type \(1/2\) also fail the
two-out-of-three property, as witnessed by
\(0\to N_{m,m+1}\to N'_{m,m+2}\to U_1\to0\). To work with their extensions,
we use the exact structure furnished by the next lemma.

    \begin{lemma}\label{type 1/2 pullback stability}
Pulling back or pushing out a short exact sequence of diagonal dominos
of type \(1/2\) in \(\dc\) along a morphism between such objects again
produces a sequence of the same type.
    \end{lemma}
    \begin{proof}
Let \(0\to N_1\to N_2\xrightarrow{j}N_3\to0\) be exact and let
\(p:M_3\to N_3\) be a morphism of diagonal dominos of type \(1/2\). Put
\(M_2=N_2\times_{N_3}M_3\). This is again a diagonal domino, since it
fits into \(0\to N_1\to M_2\to M_3\to0\). Functoriality gives
\(W^1(M_2)\subseteq W^1(N_2\oplus M_3)=0\). It also places
\(N_1\) and the quotient \(M_3\) inside the type \(1/2\) part of
\(M_2\), so \(W^{1/2}(M_2)=M_2\). Thus \(M_2\) has type
\(1/2\). The pushout argument is dual.
    \end{proof}

For \(m<n\), put
\[
\begin{aligned}
N_{m,n}&=(m\leftarrow m+1\leftarrow\cdots\leftarrow n^*),\\
N'_{m,n}&=(m\leftarrow(m+1)^*\leftarrow\cdots\leftarrow n^*).
\end{aligned}
\]
Only the last vertex of \(N_{m,n}\) is distinguished, whereas
every vertex except the first is distinguished in \(N'_{m,n}\).
Both are of type \(1/2\), by the source-and-sink criterion.
In \cite{Ek3}, \(N'_{m,m+2}\) is denoted \(N_m\).

To classify the weakly simple objects, we first control the kernel of
a quotient by \(U_1[i](-i)\).

\begin{lemma}\label{kernel reduction for weakly simple dominos}
Let \(M\) be a weakly simple type \(1/2\) diagonal domino and suppose
that in \(\dc\) there is an exact sequence
\[
0\longrightarrow N\longrightarrow M\longrightarrow U_1[i](-i)
\longrightarrow0.
\]
If \(W^{1/2}N=0\), then \(N\) is \(\Tot\)-acyclic.
Otherwise \(N\) itself is a weakly simple type \(1/2\) diagonal domino.
\end{lemma}

\begin{proof}
Since \(N\subseteq M\), we have \(W^1N=0\), so
\(W^{1/2}N\) is of type \(1/2\).
Suppose \(W^{1/2}N\neq0\), and set
\(B=N/W^{1/2}N\), \(Q=M/W^{1/2}N\).
The radical property gives \(W^{1/2}B=0\), while
\[
0\longrightarrow B\longrightarrow Q\longrightarrow U_1[i](-i)
\longrightarrow0
\]
is exact. The type filtration makes \(B\) a diagonal domino; hence
\(Q\) is one as well.
As a quotient of \(M\), it has no nonzero \(\Tot\)-\(1\)-torsion
quotient, so it lies in \(G\).
Since \(0\neq W^{1/2}N\neq M\), weak simplicity implies that \(Q\) is not
of type \(1/2\). By
\Cref{when is a torsion object in G a type 1/2 diagonal domino},
there is a nonzero map \(U_1[j](-j)\to Q\).
This map cannot factor through \(B\), since
\(W^{1/2}U_1[j](-j)=U_1[j](-j)\) and \(W^{1/2}B=0\).
Its composite with \(Q\to U_1[i](-i)\) is nonzero. Under
\(\FG^{\Tot}\), the gauge of \(U_1[a](-a)\) vanishes exactly in degree
\(a+1\); compatibility with \(\tF,\tV,\tau\) at this gap shows that
maps between gauges with different gaps vanish, while every nonzero
endomorphism is an isomorphism. Thus \(j=i\), and the composite is an
isomorphism.
Thus \(Q\simeq B\oplus U_1[i](-i)\). But \(B\) is a quotient of \(M\), so
functoriality and
\(W^{1/2}M=M\) give \(W^{1/2}B=B\).
Hence \(B=0\), proving \(N=W^{1/2}N\).

Now suppose that \(0\to A\to N\to B\to0\) is admissible in the type
\(1/2\) category, with \(A,B\neq0\).
Again \(Q=M/A\) lies in \(G\) but is not of type \(1/2\).
The same argument, using \(\Hom(U_1[j](-j),B)=0\), splits
\(0\to B\to Q\to U_1[i](-i)\to0\).
Let \(E\subseteq M\) be the preimage of the \(U_1[i](-i)\)-factor.
Then \(0\to A\to E\to U_1[i](-i)\to0\) and \(M/E=B\).
The class \(W^{1/2}X=X\) is extension-closed, so \(W^{1/2}E=E\);
also \(W^1E=0\), since \(E\subseteq M\).
Thus \(E\) and \(B\) have type \(1/2\), contradicting the weak
simplicity of \(M\). Hence \(N\) is weakly simple.

If \(W^{1/2}N=0\), then \(N\) is \(\Tot\)-\(1\)-torsion and
\[
0\longrightarrow\HH^0(\Tot M)\longrightarrow k
 \longrightarrow\HH^1(\Tot N)\longrightarrow0.
\]
The first term is nonzero because \(M\) is a nonzero object of pure
type \(1/2\). Its injection into \(k\) is therefore an isomorphism.
Hence \(N\) is \(\Tot\)-acyclic.
\end{proof}

\begin{theorem}\label{weakly simple type 1/2 diagonal dominos}
The additive category of type \(1/2\) diagonal dominos, with the
exact sequences induced from \(\dc\), is exact, weakly noetherian
and weakly artinian. Its weakly simple objects are precisely
\[
N_{m,n}\quad\text{and}\quad N'_{m,n},\qquad m<n.
\]
\end{theorem}

The classification above is also established by Zhang in
\cite[Theorem~4.22]{ZhangHN}, correcting
\cite[Theorem~III.2.7]{Ek3} over any perfect field.\footnote{The author
thanks Yuanning Zhang for pointing out the error in Ekedahl's
statement; see \cite[Remark~4.23]{ZhangHN} for an explanation of the
gap in the original proof.}

\begin{proof}
Extension closure and \Cref{type 1/2 pullback stability} give the exact
structure. Ascending admissible chains stabilize because \(\dc\) is
Noetherian; descending chains stabilize by
\Cref{without finite torsion implies artinian}.

On the type \(1/2\) category the functors
\[
E_0(X)=\HH^0(\Tot X),\qquad E_1(X)=\HH^1(\Tot(m_1X))
\]
are exact and detect zero objects. Indeed, \(X\) is
\(\Tot\)-\(0\)-torsion, whereas \(m_1X\) is pure type \(-1/2\) and
therefore \(\Tot\)-\(1\)-torsion. If either indicated cohomology
vanishes, the corresponding object is acyclic and thus zero by purity.
The vertex filtration computes
\[
\begin{array}{c|cc}
 &\operatorname{length}_W E_0&\operatorname{length}_W E_1\\ \hline
N_{m,n}&1&n-m\\
N'_{m,n}&n-m&1 .
\end{array}
\]
A distinguished vertex contributes \(k\) to \(E_0\) and becomes
acyclic after \(m_1\); an undistinguished vertex is initially
acyclic and contributes \(k\) to \(E_1\) after \(m_1\).
If an exact functor detects zero objects, an object whose image has
length one admits no nontrivial admissible short exact sequence.
Hence both families are weakly simple.

For exhaustiveness, let \(M\) be weakly simple.
Kernels and images of multiplication by \(p\), taken in \(G\),
are torsion subobjects of \(M\) with no maps from any \(U_1[j](-j)\),
hence are of type \(1/2\).
The resulting sequence is exact in \(\dc\) as well, since all its
objects belong to both hearts.
Multiplication by \(p\) cannot be injective on the nonzero finite
torsion module \(\HH^0(\Tot M)\); its kernel is therefore nonzero.
Weak simplicity gives \(pM=0\).

The type filtration supplies a nonzero map \(M\to U_1[i](-i)\) for
some \(i\). It is surjective: a nonzero proper image would be a
\(U_j[i](-i)\) with \(j\leqslant0\), excluded as a quotient of a pure
type \(1/2\) object by the type filtration.
Write the kernel as \(N\).
We induct on the total domino number \(\sum_a T^{a,-a}(M)\). It is
finite and additive, and decreases on passing to \(N\).

If \(N\) is acyclic, the cohomology sequence in gauge degree \(i+1\)
gives \(\HH^0(\FG^{\Tot}(M)^{i+1})=0\): the corresponding groups for
\(N\) and \(U_1[i](-i)\) vanish, and \(M\in G\).
Consequently \(M(-\varphi_{\{i+1\}})\) is acyclic.
It is still indecomposable and killed by \(p\).
By \Cref{indecomposable Tot acyclic objects are M(I J)} and the
source-and-sink criterion, its decorated graph has a single
distinguished source and arrows pointing away from that source.
By \Cref{swap direction isomorphic}, we may move the distinguished
vertex to the right endpoint, giving \(M\simeq N_{m,n}\).

Otherwise \Cref{kernel reduction for weakly simple dominos} makes
\(N\) weakly simple of type \(1/2\); by induction it belongs to one of
the two families.
Put \(t=i+1\), the new vertex contributed by \(U_1[i](-i)\).

First suppose that \(N=N_{m,n}\).
If \(t\in[m,n]\), move the unique distinguished vertex to \(t\).
Modification by \(-\varphi_{\{t\}}\) then makes the kernel and
quotient acyclic, hence also \(M\).
The resulting indecomposable acyclic object has domino number two
at \(t\), contradicting the classification of indecomposable
acyclic objects killed by \(p\).
Thus \(t\notin[m,n]\).
Modification by \(-\varphi_{\{n,t\}}\) gives an indecomposable acyclic
object with vertex set \([m,n]\cup\{t\}\).
Its support must be an interval, so \(t=m-1\) or \(t=n+1\).

The source-and-sink criterion and local swaps put the decorated
graph in the form
\[
(a^*\longrightarrow\cdots\longrightarrow c
 \longleftarrow\cdots\longleftarrow b^*),\qquad a<c<b.
\]
For a right-end extension, move the old distinguished vertex
leftwards to obtain this form; a left-end extension already has
this form.
If \(c-a\geqslant2\), move the left distinguished vertex to \(c-1\)
and the right one to \(c\). Cut the arrow \((c-1)^*\to c^*\).
The two interval pieces both have a unique distinguished source
and undistinguished sinks, and each has at least two vertices.
They are members of the first family up to local swaps.
Cutting this arrow gives a nontrivial admissible short exact sequence,
contradicting weak simplicity.
If \(b-c\geqslant2\), use the symmetric construction, moving the
stars to \(c,c+1\) and cutting \((c+1)^*\to c^*\).
Thus \(c-a=b-c=1\), and
\[
M\simeq(a^*\to a+1\leftarrow(a+2)^*)\simeq N'_{a,a+2}.
\]

Now suppose that \(N=N'_{m,n}\).
Local swaps move the unique undistinguished vertex to any
\(q\in[m,n]\).
If \(t\in[m,n]\), choose \(q\neq t\) and set
\(K=[m,n]\setminus\{q\}\).
Modification by \(-\varphi_K\) makes both \(N\) and \(U_1[i](-i)\)
acyclic. Again the resulting indecomposable acyclic object has
domino number two at \(t\), a contradiction.

For \(t\notin[m,n]\), take
\(K=\{m+1,\ldots,n\}\cup\{t\}\).
The same acyclic classification forces \(t=m-1\) or \(t=n+1\).
The new distinguished vertex cannot be a sink, so the two
possibilities are
\[
\begin{aligned}
M&=(m\leftarrow(m+1)^*\leftarrow\cdots\leftarrow n^*
 \leftarrow(n+1)^*)=N'_{m,n+1},\\
M&=((m-1)^*\to m\leftarrow(m+1)^*\leftarrow\cdots\leftarrow n^*)
 \simeq N'_{m-1,n}.
\end{aligned}
\]
The last isomorphism is one local swap.
This completes the induction.
\end{proof}

\begin{remark}
Zhang constructs a Harder--Narasimhan filtration on torsion diagonal
complexes whose restriction to diagonal dominos refines Ekedahl's
type filtration \cite[Theorems~3.9 and~3.12]{ZhangHN}.
For his diagonal slope \(\mu_\Delta\), the objects \(N_{m,n}\) and
\(N'_{m,n}\) are stable of slopes \(\frac{1}{n-m+1}\) and
\(\frac{n-m}{n-m+1}\), respectively
\cite[Theorem~4.22]{ZhangHN}. Equivalently, they are simple in their
respective abelian categories of semistable objects of fixed slope
\cite[Theorem~3.13]{ZhangHN}.
\end{remark}

\subsection{Virtual F-crystal}
\label{subsec:virtual-f-crystal}

For a torsion-free gauge, we have the following equivalent linear-algebraic interpretation.
After inverting \(p\), its components are canonically identified, while
their integral images form a filtration by \(W\)-submodules. A virtual
\(F\)-crystal records the ambient space, its Frobenius, and a
\(W\)-submodule; the filtration recovers the gauge.

\begin{example}[{\cite[Chapter~II, Section~2, Example~(i)]{Ek3}}] \label{virtual F crystal}
A virtual \(F\)-crystal is a triple \((U,F,N)\), where
\(K=\operatorname{Frac}(W)\), \(U\) is a finite-dimensional
\(K\)-vector space, \(F\) is a \(\sigma\)-linear automorphism of \(U\),
and \(N\subseteq U\) is a \(W\)-submodule satisfying \(KN=U\).
It is of finite type if \(N\) is finitely generated over
\(W\).

Let \(M\) be a torsion-free \(F\)-gauge such that
\(M^{-\infty}\otimes_WK\) is finite-dimensional. The relation
\(\tF\tV=p\) identifies every component after inverting \(p\); in
particular,
\[M^0\otimes_W K
\xrightarrow{\cong}
M^{\infty}\otimes_W K\]
and likewise at the negative end. Thus \(M\) determines a virtual
\(F\)-crystal with \(U=M^{-\infty}\otimes_WK\),
\(N=M^{-\infty}\), and Frobenius the composite
\[
M^{-\infty}\otimes_W K
\leftarrow 
M^{0}\otimes_W K
\rightarrow
M^{\infty}\otimes_W K\xrightarrow{\tau}M^{-\infty}\otimes_W K.
\]
Taking the colimit along \(\tV\) gives canonical maps
\(M^i\to M^{-\infty}=\varinjlim_{j\to-\infty}M^j\hookrightarrow U\).
The first two arrows above identify \(M^{-\infty}\otimes_W K\) and
\(M^\infty\otimes_W K\) through \(M^0\otimes_W K\). Under this
identification, the image of \(M^i\) in \(M^\infty\otimes_W K\)
corresponds to \(p^{-i}M^i\subseteq U\).
The submodules \(M^i\subseteq U\) form a decreasing filtration satisfying:

{\rm (1)}
\(
pM^i\subseteq M^{i+1}.
\)

{\rm (2)}
\(
N=\bigcup_i M^i .
\)

{\rm (3)}
$F$ maps \(\bigcup_i p^{-i}M^i\)
bijectively to 
\(\bigcup_i M^i .\)

Conversely, a virtual \(F\)-crystal with a decreasing filtration
\(\{M^i\}\) satisfying (1)--(3) defines an \(F\)-gauge: \(\tV\) is
inclusion, \(\tF\) is multiplication by \(p\), and (3) supplies
\(\tau\). The virtual crystal is of finite type exactly when
\(M^{-\infty}=\bigcup_iM^i\) is finitely generated. Then
\(\bigcup_ip^{-i}M^i\) is also finitely generated, and the filtration
has finite level. Conversely, finite level and finite generation of one
component imply finite generation of every component and of the
virtual crystal.
\end{example}

The forgetful functor to virtual \(F\)-crystals has a right adjoint,
obtained by choosing the largest compatible filtration. For
\((U,F,N)\), set
\[M^i=\bigl(F^{-1}(p^iN)\bigr)\cap N=\{x\in N| Fx/p^i\in N\}.\]
This filtration is decreasing. Moreover,
\(\bigcup_i M^i=\bigl(F^{-1}(\bigcup_i p^iN)\bigr)\cap N=F^{-1}U\cap N=N\),
while
\(F\!\left(\bigcup_i p^{-i}M^i\right)=F\!\left(F^{-1}N\cap\bigcup_i p^{-i}N\right)=F(F^{-1}N\cap U)=N\).
This is the largest filtration satisfying (1)--(3), so it defines the
right adjoint:
\[
\Hom_{\FGauge}(M,\Nyg(N))
=\Hom_{\mathrm{virtual}}(M^{-\infty},N).
\]
We call \(\Nyg(N)\) the Nygaard \(F\)-gauge, or simply Nygaard gauge,
associated with the virtual \(F\)-crystal. In \cite{Ek3}, it is
denoted by \(\mathrm{Hodge}(N)\).

The Nygaard gauge has level \(\I\) precisely when
\(p^{-n-1}FN\subseteq N\) below \(\I\) and
\(p^{-i+1}FN\supseteq N\) above \(\I\). In particular, a finite-type
virtual \(F\)-crystal with effective Nygaard gauge is an \(F\)-crystal in
the usual sense.

\subsection{Mazur-Ogus objects}
\label{subsec:mazur-ogus-objects}

Among diagonal complexes, Mazur--Ogus objects are distinguished by
their orthogonality to the \(\Tot\)-acyclic objects. The next criterion
expresses this condition through the Hodge-to-total spectral sequence
and the diagonal sum of Hodge numbers.

\begin{definition}[{\cite[Definition~I.4.3]{Ek3}}]
A diagonal complex \(M\) is Mazur--Ogus if
\[\Hom_{\dc}(M,N)=\Hom_{\dc}(N,M)=0\]
for every \(\Tot\)-acyclic \(N\).
\end{definition}

\begin{proposition}[{\cite[Proposition~III.4.1]{Ek3}}]\label{Mazur--Ogus criterion}
For \(M\in\dc\), the following conditions are equivalent:

{\rm (1)} $M$ is Mazur--Ogus.

{\rm (2)} The spectral sequence
\[
E_1^{i,j}=\HH^i(R_1\otimes^{\rmL}_{R}M)^j
\Longrightarrow\HH^{i+j}(\Tot(M))
\]
degenerates at the \(E_1\)-page, \(\HH^1(\Tot(M))=0\), and
\(\HH^0(\Tot(M))\) is torsion-free.

{\rm (3)} \(b_0(M)=\sum_i h^{\,i,-i}(M)\). 
\end{proposition}

\begin{proof}
(1) implies (2) by \Cref{Mazur-Ogus vanishing}. To compare (2) and
(3), we have
\[
\sum_i h^{i,-i}(M)
\leqslant\dim_k\HH^0(k\otimes_W^{\rmL}\Tot(M)),
\]
The first inequality is an equality precisely when
\[\HH^j\!(R_1\otimes^{\rmL}_{R}M)^i=E^{i,j}_1=
E^{i,j}_{\infty}.
\]
On the other hand,
\[
b_0(M)\leqslant \dim_{k}(\HH^0(k\otimes _W^{\rmL}\Tot(M))),
\]
with equality precisely when \(\HH^0(\Tot(M))\) and
\(\HH^1(\Tot(M))\) are torsion-free. By \Cref{Basic Vanishing}, these
comparisons give the equivalence of (2) and (3).

Now assume (2). By \Cref{Hodge numbers in degree plus minus 1},
\[
\Hom_{\dc}(U_0[i](-i),M)=\Hom_{\dc}(M,U_0[i](-i))=0
\]
for every \(i\). By \Cref{Tot 0 or Tot 1 of diagonal domino}(1),
\(M\) has no nonzero \(\Tot\)-\(0\)-torsion subobject or
\(\Tot\)-\(1\)-torsion quotient. Let \(N\) be \(\Tot\)-acyclic.
By \Cref{sub and quotient object of Tot-torsion objects}(2),
every quotient of \(N\) is \(\Tot\)-\(0\)-torsion, and every
subobject of \(N\) is \(\Tot\)-\(1\)-torsion. Thus the image of
any map \(N\to M\) or \(M\to N\) is zero. Both Hom groups in
the definition therefore vanish, proving (1).
\end{proof}

The next theorem identifies Mazur--Ogus objects with Nygaard gauges. A
general torsion-free gauge can still carry an acyclic part in the
diagonal heart; the Nygaard gauge removes it.

\begin{theorem}[{\cite[Theorem~III.4.2]{Ek3}}]\label{Mazur-Ogus object and the Hodge F-gauge structure of virtual F-crystal}
The following statements hold.

{\rm (1)} An object \(M\in\dc_g\) is torsion-free in \(\dc_g\) if and only if
\(M\in\dc\), \(G^1(M)\) is \(\Tot\)-acyclic, and \(G^2(M)=M\).

{\rm (2)} If \((U,F,N)\) is a virtual \(F\)-crystal of finite type, then
\(\rmL\widehat\rmt(\Nyg(N))\) is Mazur--Ogus.

{\rm (3)} If \(M\in\dc\) is Mazur--Ogus, then \(\FG^{\Tot}(M)\) is
torsion-free and
\[
\FG^{\Tot}(M)
=\Nyg\bigl(\FG^{\Tot}(M)^{-\infty}\bigr),
\]
where \((\FG^{\Tot}(M)^{-\infty}\otimes_WK,F,
\FG^{\Tot}(M)^{-\infty})\) is the virtual \(F\)-crystal of
\Cref{virtual F crystal}.
\end{theorem}

\begin{proof}
Every Mazur--Ogus object \(M\) lies in \(\dc_g\) and is torsion-free there.
Indeed, \Cref{two fundamental filtrations} gives \(G_M^2=M\). The
torsion subobject \(T\subseteq M\) in \(\dc_g\) lies in \(\dc\) by
diagonality. By \Cref{Mazur-Ogus vanishing}, it is
\(\Tot\)-\(1\)-torsion, hence zero since it lies in \(\dc_g\).

To prove (1), suppose \(M\in\dc\), \(G_M^2=M\), and
\(G_M^1\) is \(\Tot\)-acyclic. The quotient \(M/G_M^1\) is
Mazur--Ogus. Applying \(\HH_g^\bullet\) to the defining triangle gives
\[
0\longrightarrow M\longrightarrow M/G_M^1
\longrightarrow G_M^1[1]\longrightarrow0
\]
in \(\dc_g\). Hence \(M\) is torsion-free as a subobject of the
torsion-free object \(M/G_M^1\). Conversely, if \(M\in\dc_g\) is
torsion-free, its torsion subobject
\(\widetilde\tau_{\leqslant-1}M\) vanishes, so \(M\in\dc\).
The torsion subobject \(W^{1/2}(F^1M)\) of \(M\) also vanishes. Therefore
\(F^1M\) is \(\Tot\)-acyclic, \(G_M^2=M\), and \(G_M^1=F^1M\).

For (2), take a finite-type virtual \(F\)-crystal \(N\) and set
\(M=\rmL\widehat\rmt(\Nyg(N))\). The Nygaard gauge is coherent and
torsion-free, so \(M\) is torsion-free in \(\dc_g\). By (1), \(G^1M\) is
\(\Tot\)-acyclic, and
\[
\FG^{\Tot}(M)^{-\infty}
=\FG^{\Tot}(M/G^1M)^{-\infty}=N.
\]
By the adjunction of \Cref{virtual F crystal}, the identity on the
virtual crystal \(N\) gives a retraction
\(\FG^{\Tot}(M/G^1M)\to\Nyg(N)=\FG^{\Tot}(M)\). Hence
\(0\to G^1M\to M\to M/G^1M\to0\) splits. But a
\(\Tot\)-acyclic object lying in \(\dc_g\) is zero, so \(G^1M=0\).
Therefore \(M\) is Mazur--Ogus.

For (3), let \(M\) be Mazur--Ogus and put
\(N=\FG^{\Tot}(M)^{-\infty}\). The adjunction gives
\(\FG^{\Tot}(M)\to\Nyg(N)\). After applying
\(\rmL\widehat\rmt\), we obtain a morphism
\[
M\longrightarrow\rmL\widehat\rmt(\Nyg(N))
\]
whose totalization is an isomorphism. Thus its kernel and cokernel are
\(\Tot\)-acyclic. Source and target are Mazur--Ogus, so orthogonality
forces both to vanish. Applying \(\FG^{\Tot}\) proves (3).
\end{proof}

Torsion-freeness in \(\dc_g\) alone does not imply the Mazur--Ogus
condition. For example, modify the gauge \(M\) of \Cref{The first example}
by replacing its degree-two component with
\(p^2We_1\oplus pWe_2\). The resulting subgauge is torsion-free in
\(\dc_g\), but its quotient contributes a copy of \(U_0[1](-1)\), so it is not
Mazur--Ogus.

\begin{corollary}\label{Hodge filtration from the Nygaard gauge}
If \(M\in\dc\) is Mazur--Ogus, then the Hodge filtration on
\(\Tot(M)/p\) is induced by the images of the \(\HH^0(\FG^{\Tot}(M))^i\).
\end{corollary}

\begin{proof}
Put \(N=\HH^0(\FG^{\Tot}(M))\). By
\Cref{Mazur-Ogus object and the Hodge F-gauge structure of virtual F-crystal},
\(\FG^{\Tot}(M)\simeq N[0]\), where \(N\) is torsion-free and
\(N^{-\infty}\cong\HH^0(\Tot(M))\).
By \Cref{distinguished triangle of M(F V)}(3),
\[
t_{\geqslant i}\Hod(M)
\simeq \operatorname{Cone}(\tF:N^{i-1}\to N^i).
\]
Since \(\tV\tF=p\), the map \(\tF\) is injective, so the zeroth
cohomology of this cone is \(N^i/\tF N^{i-1}\).

At sufficiently negative gauge degrees, \(N^i=N^{-\infty}\)
and \(\tF=p\); hence the same comparison identifies
\(\Hod(M)\) with \(N^{-\infty}/p\).
Naturality with respect to \(\tV\) identifies the Hodge filtration
map with the map of cones induced by
\[
\begin{tikzcd}
N^{i-1} \arrow[r,"\tF"] \arrow[d] &
N^i \arrow[d] \\
N^{-\infty} \arrow[r,"p"] &
N^{-\infty},
\end{tikzcd}
\]
where the vertical arrows are the canonical maps to the colimit
along \(\tV\). Thus the induced map
\(N^i/\tF N^{i-1}\to N^{-\infty}/p\) sends the class of \(x\)
to its canonical image modulo \(p\).
Its image is exactly the image of \(N^i\), as claimed.
\end{proof}

\begin{theorem}[{\cite[Theorem~III.4.3]{Ek3}}]\label{equivalence between Mazur-Ogus objects and virtual F-gauge structures.}
The functors \(\FG^{\Tot}(-)^{-\infty}\) and
\(\widehat\rmt(\Nyg(-))\) are inverse equivalences between the
category of Mazur--Ogus diagonal complexes and the category of
finite-type virtual \(F\)-crystals.
\end{theorem}

\begin{proof}
Combine \Cref{Mazur-Ogus object and the Hodge F-gauge structure of virtual F-crystal} with the degree-zero case of \Cref{completed rmt and diagonal cohomology}.
\end{proof}

The equivalence also identifies the filtration arising from ordinary
cohomology. Using the model \(\widehat\rmt(\Nyg(-))\), we study the
spectral sequence
\[
\HH^i(M)^j\Longrightarrow\HH^{i+j}(\Tot(M)).
\]
Its steps will correspond to maximal divisible subcrystals. We first
check the completeness needed to pass to quotients in this model.

\begin{lemma}[{\cite[Lemma~III.4.4]{Ek3}}]\label{completeness for t'(M)^i}
Let \(M\) be a coherent \(F\)-gauge of level \(\I\). Then
\(\widehat\rmt(M)^i\) is complete, and hence separated, for the standard
topology for every \(i\).
\end{lemma}

\begin{proof}
By \Cref{resolution of coherent complexes}, choose a presentation
\(P^1\xrightarrow{d_1}P^0\to M\to0\) by finite sums of the
\(G_r\). By \Cref{Exactness after completion} and right exactness of
\(\widehat\rmt\), for every \(i\) we obtain an exact sequence
\[
\widehat\rmt(\widehat P^{\,1})^i
\xrightarrow{\;\widehat\rmt(\widehat d_1)^i\;}
\widehat\rmt(\widehat P^{\,0})^i
\longrightarrow\widehat\rmt(M)^i\longrightarrow0.
\]
By \Cref{description of t(G_r)} and
\(\widehat\rmt(\widehat G_r)\cong\widehat{\rmt(G_r)}\), the first
two terms are complete. Since the standard topology on the quotient
is the quotient topology, it remains to
prove that the first map has closed image.

In \(R\)-module degree \(-i\), this follows from Artin--Rees for
finite free \(\widehat R^{\,0}\)-modules with their \(V\)-adic
topology, or for finite free \(\widehat R^{\mathrm{top}}\)-modules
with their \(p\)-adic topology at the top. The \(V\)-adic case uses
the same Rees-module argument as \Cref{Exactness after completion}:
since \(V\) is regular and normal, its Rees ring is a skew polynomial
ring over the Noetherian ring \(\widehat R^{\,0}\).

The remaining component, in \(R\)-module degree \(-i+1\), is a map
\((\widehat R^{\,1})^{\oplus a}\to(\widehat R^{\,1})^{\oplus b}\)
given by right multiplication by a matrix over \(R^0\), since it
comes from the uncompleted presentation. Recall the normal form
\[
\widehat R^{\,1}
=\Bigl(\prod_{q>0}WdV^q\Bigr)
 \oplus\Bigl(\bigoplus_{q\geqslant0}WF^qd\Bigr)^{\wedge}_p.
\]
Let \(M^+\) be the sum of the \(a\) copies of the second summand,
and choose \(j\geqslant0\) bounding the powers of \(V\) in the matrix.
Then \(M^+\) maps into
\(\bigl((\bigoplus_{q\geqslant-j}WF^qd)^{\wedge}_p\bigr)^{\oplus b}\),
where \(F^{-q}d=dV^q\) for \(q>0\).
Both are finite free modules over \(W_\sigma\{F\}\), which is
Noetherian by \Cref{flatness of central completion}; their induced
standard topologies agree with the \(p\)-adic topologies. The
restriction is linear over this ring, so Artin--Rees shows that its
image is closed. The displayed normal form also shows that
the target of this restriction is closed in
\((\widehat R^{\,1})^{\oplus b}\), so the restricted image is closed
there as well.

After quotienting the source by \(M^+\), we obtain
\(\bigl(\prod_{q>0}WdV^q\bigr)^{\oplus a}\) with the product
\(p\)-adic topology. This is linearly compact, being an inverse
limit of discrete finite-length \(W\)-modules. The target modulo
the restricted image is Hausdorff, so the induced map has closed
image. Pulling back shows that the original image is closed,
proving the claim.
\end{proof}

Let \(B=(U,F,N)\) be a virtual \(F\)-crystal. We say that \(B\) is
divisible by \(p^i\) if \(p^{-i}F(N)\subseteq N\). There is a unique
largest subcrystal with this property; denote it by \(d_iB\). The
subcrystals \(d_iB\) form a decreasing filtration.

To compare this divisibility filtration with ordinary cohomology,
let \(N\in\dc_g\), and equip
\(U=\FG^{\Tot}(N)^{-\infty}\otimes_WK\) with its associated virtual
\(F\)-crystal structure. The spectral sequence
\[
\HH^a(N)^b\Longrightarrow
\HH^0(\Tot(N))=\FG^{\Tot}(N)^{-\infty}
\]
induces a filtration; write \(F_iN\) for the image of its \(i\)-th
step in \(U\). By construction, \((F_iN\otimes_WK,F,F_iN)\) is
divisible by \(p^i\), and hence
\[
F_iN\subseteq d_iB,
\qquad
B=\bigl(U,F,\FG^{\Tot}(N)^{-\infty}/\mathrm{tors}\bigr).
\]

\begin{theorem}[{\cite[Theorem~III.4.5]{Ek3}}]\label{F_iN=d_iN}
Let \(N\in\dc_g\) be Mazur--Ogus. Then \(F_iN=d_iB\) for every \(i\).
\end{theorem}

\begin{proof}
We already have \(F_iN\subseteq d_iB\). For the reverse inclusion, it
suffices to show that
\(d_iB\to N/F_iN\) is zero. Put \(M=\FG^{\Tot}(N)\); by
\Cref{completed rmt and diagonal cohomology}, identify \(N\) with
\(\widehat\rmt(M)\). For \(m_{-\infty}\in d_iB\), we must show that
\(\beta_j\otimes m_{-\infty}=0\) for \(j<i\). Since the filtration
\(d_\bullet B\) is decreasing, it suffices to take \(j=i-1\).

Choose \(m_{i-1}\in M^{i-1}\) and \(m_i\in M^i\) mapping to
\(m_{-\infty}\). The defining relations give
\[
\beta_{i-1}\otimes m_{-\infty}
=\beta_{i-1}\otimes m_{i-1}
=V(\beta_{i-1}\otimes m_i).
\]
Divisibility gives \(m'_{-\infty}\in d_iB\) with
\(Fm_{-\infty}=p^im'_{-\infty}\). Unwinding the Frobenius on the
associated virtual crystal gives
\(\tau\gamma_i^+(m_i)=m'_{-\infty}\), and the relations of
\(\widehat\rmt(M)\) therefore identify
\(\beta_{i-1}\otimes m_i\) with
\(\beta_{i-1}\otimes m'_{-\infty}\). Repeating the argument gives
\[
\beta_{i-1}\otimes m_{-\infty}
\in\bigcap_{\ell\geqslant0}V^\ell
\widehat\rmt(M)^{i,-i}=0,
\]
where the last equality follows from
\Cref{completeness for t'(M)^i}. Hence \(d_iB\subseteq F_iN\).
\end{proof}

The virtual crystal describes a Mazur--Ogus object through
totalization. The next theorem describes its structure within the
diagonal category: Hodge--Witt constituents and diagonal dominos meet
in a common finite-torsion quotient.

\begin{theorem}[{\cite[Theorem~III.4.6]{Ek3}}]\label{Structure of Mazur-Ogus object}
The category of Mazur--Ogus diagonal complexes is equivalent to the
category of tuples
\[
(M_1,M_2,N_1,N_2,\psi,\varphi,\rho)
\]
where \(M_1,M_2\) are torsion-free Hodge--Witt diagonal complexes;
\(N_1,N_2\) are diagonal dominos satisfying \(W^0N_1=0\) and
\(W^{1/2}N_2=N_2\); the maps
\(\psi:M_1\hookrightarrow M_2\) and
\(\varphi:N_1\hookrightarrow N_2\) have finite-torsion cokernels; and
\(
\rho:\operatorname{Coker}\psi
\longrightarrow
\operatorname{Coker}\varphi
\)
is an isomorphism between these cokernels.
\end{theorem}

\begin{proof}
Let \(M\in\dc_{\MO}\). In the notation of \Cref{definition of HW}, set
\[
M_1:={}_{\HW}(M),\qquad
M_2:=D({}_{\HW}(D(M))),\qquad
N_1:=\ker(M\to M_2),\qquad N_2:=M/M_1.
\]
Here \(M_2\) is the largest torsion-free Hodge--Witt quotient of
\(M\).

Let \(\psi:M_1\to M_2\) and \(\varphi:N_1\to N_2\) be the composites
through \(M\). Their cokernels are canonically identified:
\[
\operatorname{Coker}\psi
=
(M/N_1)/\operatorname{Im}M_1
=
M/(N_1+M_1)
=
(M/M_1)/\operatorname{Im}N_1
=
\operatorname{Coker}\varphi.
\]
Take this identification as \(\rho\).

Both \(N_1\) and \(N_2\) are torsion without finite-torsion
subobjects, hence diagonal dominos. Their intersection \(N_1\cap M_1\)
is both a diagonal domino and Hodge--Witt, so it vanishes. Since
\[
N_1\cap M_1
=
\operatorname{Ker}\psi
=
\operatorname{Ker}\varphi,
\]
both maps are monomorphisms. Their common cokernel is torsion as a
quotient of \(N_2\) and Hodge--Witt as a quotient of \(M_2\), hence
finite torsion.

Finally,
\[
W^0N_1=F^1N_1\subseteq F^1M=0,
\]
and
\[
0=M/G^2M
\twoheadrightarrow N_2/G^2N_2=
N_2/W^{1/2}N_2.
\]

Conversely, take a tuple as in the statement. Identify the two
cokernels via
\(\operatorname{Coker}\psi
\xrightarrow{\rho}
\operatorname{Coker}\varphi,\)
call the common object \(N'\), and let \(M\) be the pullback
\[\begin{tikzcd}
	& {N_2} \\
	{M_2} & {N'}
	\arrow[from=1-2, to=2-2]
	\arrow[from=2-1, to=2-2]
\end{tikzcd}\]
where the arrows are the quotient maps followed by the chosen
identifications
\(N_2\twoheadrightarrow
\operatorname{Coker}\varphi
\stackrel{\sim}{\longrightarrow}
N',\)
and
\(M_2\twoheadrightarrow
\operatorname{Coker}\psi
\stackrel{\sim}{\longrightarrow}
N'.\)

Let \(A\in\dc\) be \(\Tot\)-acyclic. To show that \(M\) is
Mazur--Ogus, we check both Hom groups using the pullback sequences
\[
0\to M_1\to M\to N_2\to0,
\qquad
0\to N_1\to M\to M_2\to0.
\]
We have \(\Hom(M_1,A)=0\) because \(M_1\) is Hodge--Witt, and
\(\Hom(N_2,A)=0\) because
\(W^{1/2}N_2=N_2\) while \(W^{1/2}A=0\). Hence
\(\Hom(M,A)=0\). Similarly, \(W^0A=A\), \(W^0N_1=0\), and the
Mazur--Ogus property of \(M_2\) give \(\Hom(A,M)=0\). Thus \(M\) is
Mazur--Ogus. The constructions are functorial and inverse.
\end{proof}

Duality exchanges the subobject and quotient in this description, along
with the two diagonal dominos. For a perfect pairing, it also relates
the discriminant of the Hodge--Witt part to the length of the torsion
total cohomology.

\begin{corollary}[{\cite[Corollary~III.4.6.1]{Ek3}}]\label{dual of Mazur-Ogus object}
Let \(M\in\dc_{\MO}\). Then \(D(M)=\widetilde{D}^{\,0}(M)\in\dc_{\MO}\),
and if
\[
M=(M_1,M_2,N_1,N_2,\psi,\varphi,\rho),
\]
then
\[
\widetilde{D}^{\,0}(M)
=
\bigl(
\widetilde{D}^{\,0}(M_2),
\widetilde{D}^{\,0}(M_1),
\widetilde{D}^{\,0}(N_2),
\widetilde{D}^{\,0}(N_1),
\widetilde{D}^{\,0}(\psi),
\widetilde{D}^{\,0}(\varphi),
\widetilde{D}^{\,1}(\rho)^{-1}
\bigr).
\]

In particular, if
\(M\xrightarrow{\sim}\widetilde{D}^{\,0}(M)\) is a perfect pairing then
\[v_p\!\left(\operatorname{disc}(\Tot(M_1))\right)=2\cdot\operatorname{length}
\HH^0(\Tot(N_2)),\]
where \(v_p(-)\) is the \(p\)-adic valuation on \(W\) with
\(v_p(p)=1\).
\end{corollary}

\begin{proof}
Since \(M\) has no nonzero finite-torsion subobject,
\Cref{finite torsion} gives \(D(M)=\widetilde D^{\,0}(M)\). By duality,
a nonzero map between \(D(M)\) and an acyclic object would yield a
nonzero map between \(M\) and the dual acyclic object, contrary to the
Mazur--Ogus condition. Thus
\(\widetilde D^{\,0}(M)\) is again Mazur--Ogus.

Under duality, \(\widetilde D^{\,0}(M_2)\) is the maximal Hodge--Witt
subobject of \(\widetilde D^{\,0}(M)\), and \(\widetilde D^{\,0}(M_1)\)
is its maximal Hodge--Witt quotient. This gives the stated tuple.

For the discriminant formula,
\[
v_p\!\left(\operatorname{disc}(\Tot(M_1))\right)
=
\operatorname{length}\,
\Tot\!\left(
\widetilde{D}^{\,0}(M_1)/M_1
\right),
\]
as
\[
\Tot\!\left(\widetilde{D}^{\,0}(M_1)\right)
=
\operatorname{Hom}(\Tot(M_1),W).
\]

Now
\[
\widetilde{D}^{\,0}(M_1)=M_2,
\]
and
\[
M_2/M_1
=
N_2/N_1
=
N_2/\widetilde{D}^{\,0}(N_2),
\]
and so
\[
\operatorname{length}\,
\Tot\!\left(
\widetilde{D}^{\,0}(M_1)/M_1
\right)
=
\operatorname{length}\,
\Tot(N_2)
-
\operatorname{length}\,
\Tot\!\left(
\widetilde{D}^{\,0}(N_2)
\right)
\]
\[
=
\operatorname{length}
\HH^0(\Tot(N_2))
+
\operatorname{length}
\HH^1\!\left(
\Tot(\widetilde{D}^{\,0}(N_2))
\right)
=
2\cdot
\operatorname{length}
\HH^0(\Tot(N_2)),
\]
where the final equality follows by duality.
\end{proof}

For the higher-level spectral sequences, we name the four constituents
of a Mazur--Ogus object.

\begin{definition}
Let \(M\in\dc_{\MO}\). Denote its \(p^\infty\)-torsion subobject by
\({}_pM\), and set
\[
M_{\HW}:=M/{}_pM,
\qquad
M_p:=M/{}_{\HW}(M).
\]
\end{definition}

In \Cref{Structure of Mazur-Ogus object}, the four terms are therefore
\(({}_{\HW}(M),M_{\HW},{}_pM,M_p)\).
For \(M\in\dc_{\MO}\), consider the level-\(n\) spectral sequence
\begin{equation}\label{level n Hodge-de Rham spectral sequence}
E_1^{i,j}
=
\HH^j\!\left(
R_n\otimes_R^{\rmL}M
\right)^i
\Longrightarrow
\HH^{i+j}\!\left(
W/p^n\otimes_W^{\rmL}\Tot(M)
\right).    
\end{equation}
The abutment \(W/p^n\otimes_W^{\rmL}\Tot(M)\) is concentrated in
degree \(0\), while \Cref{Basic Vanishing} confines the \(E_1\)-page
to total degrees \(-1,0,1\). Thus
\eqref{level n Hodge-de Rham spectral sequence} degenerates if and only
if \(\HH^j(R_n\otimes_R^{\rmL}M)^i=0\) whenever
\(i+j=\pm1\).

\begin{lemma}[{\cite[Lemma~III.4.9]{Ek3}}]\label{level n Hodge--de Rham degenerate: first criterion}
The following statements hold.

{\rm (1)} Let \(M\in\dc\). Then
\(\HH^{i+1}(R_n\otimes_R^{\rmL}M)^{-i}=0\)
for every \(i\) if and only if
\(\Hom_{\dc}(M,U_{n-1}[j](-j))=0\)
for every \(j\).

{\rm (2)} Let \(M\in\dc_{\MO}\). Then
\eqref{level n Hodge-de Rham spectral sequence} degenerates if and only if
\[\Hom_{\dc}(U_{-n+1}[j](-j),M)=\Hom_{\dc}(M,U_{n-1}[j](-j))=0\]
for every \(j\).
\end{lemma}

\begin{proof}
For (1), put \(\check{R}_n:=\Hom_W(R_n,K/W)(-1)\). Duality gives
\[\Hom_W\!\left(R_n\otimes_R^{\rmL}M,K/W\right)=
R\Hom_R\!\left(M,\Hom_W(R_n,K/W)\right)=
R\Hom_R(M,\check{R}_n)(1).\]
Hence \(\HH^{i+1}(R_n\otimes_R^{\rmL}M)^{-i}=0\) for every \(i\) if
and only if there is no nonzero morphism
\(M\rightarrow \check{R}_n(j)[-j]\)
for every \(j\). After shifting, it remains to prove that, for
\(N\in\dc\),
\[
\Hom_{\dc}(N,\check R_n)=0
\quad\Longleftrightarrow\quad
\Hom_{\dc}(N,U_{n-1})=0.
\]
The object \(\check R_n\) is a domino whose type vanishes above
\(n-1\) and has multiplicity one in type \(n-1\). The inclusion
\(U_{n-1}\hookrightarrow\check R_n\) gives the forward implication.

Conversely, if \(\Hom_{\dc}(N,U_{n-1})=0\), then the inclusions
\(U_i\hookrightarrow U_{n-1}\) give the same vanishing for every
\(i\leqslant n-1\). Dévissage then gives \(\Hom_{\dc}(N,T)=0\) for
every domino \(T\) whose type vanishes above \(n-1\), in particular for
\(T=\check R_n\).

To check the asserted type of \(\check R_n\), it suffices to show that
\(dV^{n-1}:\check{R}_n^0\longrightarrow\check{R}_n^1\)
is injective and \(\dim_k\operatorname{ker} dV^{n-2}=1\).
Equivalently, by duality, \(dV^{n-1}:R_n^0\longrightarrow R_n^1\)
is surjective and
\(\dim_k\operatorname{coker} dV^{n-2}=1\).
The normal form of \(R_n\) gives both claims. This proves (1), and (2)
follows from (1), the degeneration criterion above, and duality. In
fact, \(\check R_n\) has types
\(n-1,n-3,\ldots,-(n-3),-(n-1)\).
\end{proof}

The type filtration turns these Hom vanishings into conditions on the
two diagonal dominos in the structure theorem.

\begin{proposition}[{\cite[Proposition~III.4.10]{Ek3}}]\label{level n Hodge--de Rham degenerate: second criterion}
For \(M\in\dc_{\MO}\), the spectral sequence
\eqref{level n Hodge-de Rham spectral sequence} degenerates if and only if
\[
W^{1-n}({}_pM)=0,
\qquad
W^{n-1/2}(M_p)=M_p.
\]
\end{proposition}

\begin{proof}
The Hodge--Witt objects \({}_{\HW}(M)\) and \(M_{\HW}\) satisfy the
degeneration criterion at every level. Thus
\Cref{level n Hodge--de Rham degenerate: first criterion} reduces the
two Hom vanishings to those for \({}_pM\) and \(M_p\).
By \Cref{type filtration on diagonal domino}, for any diagonal domino
\(N\),
\[
W^{1-n}(N)=0
\Longleftrightarrow
\Hom_{\dc}(U_{-n+1}[j](-j),N)=0\quad\text{for all }j,
\]
and
\[
W^{n-1/2}(N)=N
\Longleftrightarrow
\Hom_{\dc}(N,U_{n-1}[j](-j))=0\quad\text{for all }j.
\]
The result follows.
\end{proof}

At higher levels, these two conditions become more restrictive. They
bound the level of degeneration when \(M\) is not Hodge--Witt.

\begin{corollary}[{\cite[Corollary~III.4.10.1]{Ek3}}]\label{Maximal n that Hodge--de Rham degenerate}
Let \(M\in\dc_{\MO}\).

{\rm (1)} If \eqref{level n Hodge-de Rham spectral sequence} degenerates
at level \(n\), then it degenerates at every level \(m\leqslant n\).

{\rm (2)} If it degenerates at level \(n\) and \(M\) is not
Hodge--Witt, then
\[
2n-1\leqslant\max_i m^{i,-i}(M).
\]

\end{corollary}

\begin{proof}
Part (1) follows because \(W^\bullet\) is decreasing. For (2),
\Cref{type filtration on diagonal domino} shows that a nonzero diagonal
domino \(N\) with \(W^{n-1/2}N=N\) contains some \(U_{n-1}[j](-j)\).
Since \(M_p\neq0\), a shift lets us assume
\(U_{n-1}\hookrightarrow M_p\). Since
\[\dim_k\Hom_{\dc}\!\left(U_{-n+1},U_{n-1}\right)=2n-1,\]
we obtain
\(\dim_k
\Hom_{\dc}\!\left(
U_{-n+1},M_p
\right)\geqslant 2n-1.
\)

By \Cref{level n Hodge--de Rham degenerate: first criterion},
\(\Hom_{\dc}(U_{-n+1},M)=0\), and hence
\[\Hom_{\dc}(U_{-n+1},M_p)
\hookrightarrow
\Ext^1_{\dc}(U_{-n+1},_{\HW}(M)).\]
Since \({}_{\HW}(M)\) is Hodge--Witt,
\Cref{Hodge numbers in degree plus minus 1} and the explicit Hodge
number calculation of \cite{IR} give
\[\dim_k\Ext^1_{\dc}(U_{-n+1},_{\HW}(M))=h^{1,-1}(_{\HW}M)=m^{1,-1}(_{\HW}(M))=m^{1,-1}(M).\]
Therefore
\(2n-1\leqslant m^{\,1,-1}(M)
\leqslant \max_i m^{\,i,-i}(M).\)
\end{proof}

The Mazur--Ogus condition is also preserved by the tensor and internal
Hom operations on \(D_c^b(R)\).

\begin{proposition}[{\cite[Proposition~III.4.11]{Ek3}}]\label{Mazur-Ogus stable under tensor and Hom}
\(\dc_{\MO}\) is stable under
\((-){\otimes}^{\rmL}_R(-)\)
and
\(R\Hom^{!}_{R}(-,-).\)
\end{proposition}

\begin{proof}
By \Cref{Mazur-Ogus vanishing}, an object \(M\in D_c^b(R)\) lies in
\(\dc_{\MO}\) if and only if
\(\HH^j\!(R_1\otimes^{\rmL}_R M)^i=0\) whenever
\(i+j\neq 0.\) The base-change identities
\[R_1\otimes^{\rmL}_R
\bigl((-)\otimes^{\rmL}_R(-)
\bigr)=\bigl(R_1\otimes^{\rmL}_R(-)\bigr)
\otimes^{\rmL}_k\bigl(R_1\otimes^{\rmL}_R(-)\bigr),\]
and
\[R_1\otimes^{\rmL}_R R\Hom^{!}_{R}(-,-)
=R\Hom_k\bigl(
R_1\otimes^{\rmL}_R(-),
R_1\otimes^{\rmL}_R(-)
\bigr)\]
preserve this concentration in total degree zero, proving the result.
\end{proof}

\subsection{Mazur--Ogus in degree \(n\)}
\label{subsec:mazur-ogus-in-degree-n}

For a general complex, the Mazur--Ogus equality can be imposed in a
single total degree. It then controls the diagonal cohomology object
there and the torsion of its two neighbors. This degreewise condition
also yields a splitting at the corresponding truncation.

\begin{definition}[{\cite[Definition~IV.1.1]{Ek3}}]\label{Mazur--Ogus in degree n}
An object \(M\in D_c^b(R)\) is Mazur--Ogus in degree \(n\) if
\[
b_n(M)=\sum_{i+j=n}h^{i,j}(M).
\]
It is Mazur--Ogus if this holds for every \(n\), equivalently if
\[
\sum_n b_n(M)=\sum_{i,j}h^{i,j}(M).
\]
Recall that the corresponding inequality holds degreewise.
\end{definition}

\begin{theorem}[{\cite[Theorem~IV.1.2]{Ek3}}]\label{properties of Mazur--Ogus in degree n}
The Mazur--Ogus condition in degree \(n\) has the following properties.

{\rm (1)} An object \(M\in D_c^b(R)\) is Mazur--Ogus in degree \(n\)
if and only if
\(\widetilde{\HH}^{n}(M)\) is Mazur--Ogus,
\(G^2\!\left(\widetilde{\HH}^{\,n-1}(M)\right)
=\widetilde{\HH}^{n-1}(M),\) and \(F^1\!\left(\widetilde{\HH}^{n+1}(M)\right)=0.\)

{\rm (2)} If \(M\) is Mazur--Ogus in degree \(n\), then the Hodge
filtration on \(\HH^n(\Tot(M))/p\) is the reduction modulo \(p\) of the
Nygaard filtration on
\((\HH^n(\Tot(M))\otimes_WK,F,\HH^n(\Tot(M)))\).

{\rm (3)} Under the same hypothesis, the filtrations on the virtual
\(F\)-crystals \(\HH^n(\Tot(M))\otimes_WK\) and
\(\HH^{n+1}(\Tot(M))\otimes_WK\) induced by
\[
\HH^i(M)^j\Longrightarrow\HH^{i+j}(\Tot(M))
\]
are the filtrations by the maximal subcrystals divisible by successive
powers of \(p\).

{\rm (4)} If \(M\in D_c^b(R)\) is Mazur--Ogus in degree \(n\), then
\[M \simeq \widetilde{\tau}_{\leqslant n-1}M \oplus \widetilde{\tau}_{>n-1}M.\]
In particular, if \(M\) is Mazur--Ogus, then
\[M\simeq \bigoplus_n\widehat{\rmt}\!\left(
\Nyg\bigl(\HH^n(\Tot(M))\otimes K \bigr)\right)[-n].\]

{\rm (5)} Suppose \(M\in D_c^b(R)\) is Mazur--Ogus in degree \(m\) but
not Hodge--Witt in degree \(m\). Suppose also that
\[
nb_m(M)
=
\sum_{i+j=m}
\operatorname{length}
\HH^j\!\left(
R_n\otimes_R^{\rmL}M
\right)^i
\]
for some \(n\). Then
\[
2n-1
\leqslant
\max_i
m^{m+i,-i}.
\]
\end{theorem}

\begin{proof}
For (1), \Cref{Totsubquo} gives
\(b_n(M)=b_0(\widetilde\HH^n(M))\) and
\(h^{i,n-i}(M)\geqslant h^{i,-i}(\widetilde\HH^n(M))\).
Together with the Hodge--Betti inequality, this yields
\[
b_n(M)
\leqslant\sum_i h^{i,-i}(\widetilde\HH^n(M))
\leqslant\sum_i h^{i,n-i}(M).
\]
If \(M\) is Mazur--Ogus in degree \(n\), equality holds throughout.
Thus \(\widetilde\HH^n(M)\) is Mazur--Ogus by
\Cref{Mazur--Ogus criterion}, and the Hodge-number inequalities are
equalities term by term.

To determine the neighboring diagonal cohomology objects, consider
the spectral sequence
\begin{equation}\label{Hodge and widetilde H spectral sequence}
    E_2^{i,j}=\HH^i\!\left(
    R_1\otimes_R^{\rmL}\widetilde{\HH}^{j}(M)
    \right)
    \Longrightarrow \HH^{i+j}\!\left(
    R_1\otimes_R^{\rmL}M
    \right).
\end{equation}
Fix an \(R\)-module degree \(a\). By \Cref{Basic Vanishing}(2),
the spectral sequence is supported in the three columns
\(i=-a-1,-a,1-a\). Its only possible differentials are therefore
\(d_2\)'s from the first column to the third. In total degree
\(n-a\), the middle term \((E_2^{-a,n})^a\) survives automatically.
The only differentials that could affect the other two terms are
\[
\begin{aligned}
(E_2^{-1-a,n})^a&\xrightarrow{\ d_2\ }(E_2^{1-a,n-1})^a,\\
(E_2^{-1-a,n+1})^a&\xrightarrow{\ d_2\ }(E_2^{1-a,n})^a.
\end{aligned}
\]
Since \(\widetilde\HH^n(M)\) is Mazur--Ogus,
the source of the first and the target of the second vanish by
\Cref{Mazur-Ogus vanishing}. Thus all three terms survive to
\(E_\infty\). Taking dimensions in the resulting filtration of
\(\HH^{n-a}(R_1\otimes_R^{\rmL}M)^a\) gives
\[
\begin{aligned}
h^{a,n-a}(M)
={}&h^{a,1-a}(\widetilde\HH^{n-1}(M))
  +h^{a,-a}(\widetilde\HH^n(M))\\
 &+h^{a,-1-a}(\widetilde\HH^{n+1}(M)).
\end{aligned}
\]
By the termwise equality above, the middle summand equals the
left-hand side. The other two summands therefore vanish for every
\(a\); equivalently,
\[
h^{i,j}(\widetilde\HH^{n-1}(M))=0\quad(i+j=1),
\qquad
h^{i,j}(\widetilde\HH^{n+1}(M))=0\quad(i+j=-1).
\]
By \Cref{Tot 0 or Tot 1 of diagonal domino,Hodge numbers in degree plus minus 1},
these vanishings are equivalent to
\(G^2(\widetilde\HH^{n-1}(M))=\widetilde\HH^{n-1}(M)\) and
\(F^1(\widetilde\HH^{n+1}(M))=0\), respectively.

Conversely, the stated conditions eliminate both neighboring
contributions. The spectral sequence therefore identifies
\(h^{i,n-i}(M)\) with \(h^{i,-i}(\widetilde\HH^n(M))\) for every \(i\).
Since \(\widetilde\HH^n(M)\) is Mazur--Ogus,
\[
b_n(M)=b_0(\widetilde\HH^n(M))
=\sum_i h^{i,-i}(\widetilde\HH^n(M))
=\sum_{i+j=n}h^{i,j}(M).
\]

For (2), \(\widetilde\HH^n(M)\) is Mazur--Ogus by (1). By
\Cref{equivalence between Mazur-Ogus objects and virtual F-gauge structures.,Hodge filtration from the Nygaard gauge},
its Hodge filtration is the reduction modulo \(p\) of its Nygaard
filtration. Part (1) identifies
\(\HH^n(\Tot(M))\otimes_WK\) with
\(\HH^0(\Tot(\widetilde\HH^n(M)))\otimes_WK\) as filtered virtual
\(F\)-crystals. The canonical maps
\(\widetilde{\tau}_{\leqslant n}(M)\longrightarrow M\)
and \(\widetilde{\tau}_{\leqslant n}(M)
\longrightarrow \widetilde{\HH}^{n}(M)[-n]\)
also induce isomorphisms on
\(\HH^i(R_1\otimes_R^{\rmL}(-))^j\) for \(i+j=n\).
These identifications give the same comparison of filtrations on
\(\HH^n(\Tot(M))/p\), proving (2).

For (3), the argument of \Cref{Totsubquo} reduces us to \(n=0\) and
\(M\in\dc\); by (1), \(F^1M=0\). The inclusion
\(F^2M\hookrightarrow M\) is an isomorphism on zeroth total
cohomology. Functoriality therefore reduces the claim to \(F^2M\),
which is Mazur--Ogus; now apply \Cref{F_iN=d_iN}.

For (4), the obstruction to the decomposition lies in
\[
\operatorname{Ext}^1_{D(R)}
\bigl(\widetilde{\tau}_{>n-1}M,\widetilde{\tau}_{\leqslant n-1}M\bigr).\]
We show that this group vanishes. Indeed,
\[
R\!\operatorname{Hom}_{R}(-,-)^0
=
R\!\operatorname{Hom}_{R}
\bigl(
W,
R\!\operatorname{Hom}^{!}_{R}(-,-)
\bigr)^0
=
\bigl(
R\!\operatorname{Hom}^{!}_{R}(-,-)^F
\bigr)^0.
\]

By dévissage, the functor \(((-)^F)^0\) has \(\dc\)-amplitude \([0,1]\).
It therefore suffices to show that
\[
R\!\operatorname{Hom}^{!}_{R}
\bigl(\widetilde{\tau}_{>n-1}M,
\widetilde{\tau}_{\leqslant n-1}M
\bigr)\]
has \(\widetilde D_g\)-amplitude \((-\infty,-1]\). By (1) and
\Cref{Basic Vanishing}(2),
\((R_1\otimes_R^{\rmL}\widetilde\tau_{>n-1}M)^j\) has amplitude
\([n+j,\infty)\), and its totalization has amplitude
\([n,\infty)\). Likewise,
\((R_1\otimes_R^{\rmL}\widetilde\tau_{\leqslant n-1}M)^j\) has
amplitude \((-\infty,n-1+j]\), and its totalization has amplitude
\((-\infty,n-1]\). The criterion
\Cref{criteria for G-amplitude [m n]} now gives the required vanishing.

For (5), the same reduction to \(\widetilde{\HH}^{m}(M)\) gives
\[
nb_0(\widetilde{\HH}^{m}(M))
=
\sum_i
\operatorname{length}\HH^i\!\left(
R_n\otimes_R^{\rmL}\widetilde{\HH}^{m}(M)
\right)^{-i},
\]
so the hypotheses of \Cref{Maximal n that Hodge--de Rham degenerate}
are satisfied. That corollary gives the desired bound.
\end{proof}





\section{Applying to supergeneral \(\HH^2_{\cris}(A/W)\)}
\label{sec:application-to-supergeneral-h2}
We conclude with three calculations for supergeneral abelian
threefolds. The defining relations differ by only a few terms, yet the
associated geometric dominos have different extension factors. The
equivalence developed above makes this dependence accessible to an
explicit calculation. We first isolate the maximal \(F/p\)-stable
lattice in \(\bigwedge^2\HH^1_{\cris}(A/W)\), then describe the remaining
torsion gauge as an extension of standard dominos. A common Frobenius
calculation treats the three cases together before their quotient
gauges are examined individually.

Zhang's classification of three-dimensional dominos with distinct
types includes all dominos of type \(\{2,3,4\}\)
\cite[Proposition~3.21]{ZhangDom}. Specializing his classification to
\((a,b,c)=(2,3,4)\) gives precisely two isomorphism classes:
\(U_2\oplus U_3\oplus U_4\) and \(U_{2,4}\oplus U_3\), where
\(U_{2,4}\) denotes the unique nontrivial extension of \(U_2\) by
\(U_4\), up to isomorphism. More generally, he computes the extension
groups between elementary dominos and gives a presentation of dominos
of fixed type in terms of extension data
\cite[Proposition~3.3 and Theorems~3.9 and~3.18]{ZhangDom}.
His geometric calculation for the cyclic relation \(F^g-V^g\)
identifies the \(g=3\) domino as \(U_{2,4}\oplus U_3\)
\cite[Theorem~5.13]{ZhangDom}; case~(1) below uses the relation
\(F^3+V^3\).
For extensions between stable diagonal dominos, see
\cite[Section~4.3, especially Theorem~4.31]{ZhangHN}.

\begin{example}\label{three supergeneral examples}
    Let \(A\) be an abelian threefold over \(k\), and let \(D_{0,2}\) denote the domino
    \[\HH^2(A,WO_A)/V^{-\infty}Z_2^{0,2}\longrightarrow F^{\infty}B_2^{1,2}(\HH^2(A,W\Omega^1_{A/k}))\]

    Suppose \(\HH^1_{\cris}(A/W)\simeq W_{\sigma}[F,V]/(P)\).
    In each case below, there is a short exact sequence
    \[
    0\longrightarrow U_3\oplus U_s\longrightarrow D_{0,2}
    \longrightarrow U_2\longrightarrow 0.
    \]
    The last column records the types of the elementary dominos in this extension:
    \[
    \begin{array}{c|l|c|c}
    \text{case} & P & s & \text{types in the extension}\\ \hline
    (1) & F^3+V^3 & 4 & \{2,3,4\}\\
    (2) & F^3+pF+V^3 & 7 & \{2,3,7\}\\
    (3) & F^3+pF+pV+V^3 & 6 & \{2,3,6\}.
    \end{array}
    \]
\end{example}

\begin{proof}
In each case put \(M=\HH^1_{\cris}(A/W)\), and let \(v\) be its cyclic generator.
We carry out the calculation in three stages: determine the maximal
stable lattice, identify the two layers killed by \(p\) in the quotient
gauge, and compare the resulting domino with the geometric one.

For the common Frobenius calculation, put
\[
 (\alpha,\beta)=(0,0),\ (1,0),\ (1,1),
 \qquad
 F^3v+p\alpha Fv+p\beta Vv+V^3v=0.
\]
Choose the basis
\[
 e_0=v,\quad e_1=Fv,\quad e_2=F^2v,\quad
 e_3=-V^3v,\quad e_4=-V^2v,\quad e_5=-Vv.
\]
The corresponding \(\sigma\)-linear Frobenius is
\[
 \begin{aligned}
 Fe_0&=e_1,& Fe_1&=e_2,&
 Fe_2&=e_3-p\alpha e_1+p\beta e_5,\\
 Fe_3&=pe_4,& Fe_4&=pe_5,& Fe_5&=-pe_0.
 \end{aligned}
\]
The negative sign in the last formula follows from
\(e_5=-Vv\) and \(FV=p\).
Set \(E=\bigwedge^2M\), \(b_{ij}=e_i\wedge e_j\), and
\(T=F_E/p\), where \(F_E=\bigwedge^2F\). Its action on the fifteen
basis vectors is
\begin{equation}\label{supergeneral-uniform-T}
 \begin{aligned}
 Tb_{01}&=p^{-1}b_{12},&
 Tb_{02}&=p^{-1}b_{13}+\beta b_{15},\\
 Tb_{03}&=b_{14},& Tb_{04}&=b_{15},\\
 Tb_{05}&=b_{01},&
 Tb_{12}&=p^{-1}b_{23}+\alpha b_{12}+\beta b_{25},\\
 Tb_{13}&=b_{24},& Tb_{14}&=b_{25},\\
 Tb_{15}&=b_{02},&
 Tb_{23}&=b_{34}-p\alpha b_{14}-p\beta b_{45},\\
 Tb_{24}&=b_{35}-p\alpha b_{15},&
 Tb_{25}&=b_{03}-p\alpha b_{01}+p\beta b_{05},\\
 Tb_{34}&=pb_{45},& Tb_{35}&=pb_{04},\\
 Tb_{45}&=pb_{05}.&&
 \end{aligned}
\end{equation}

The maximal \(T\)-stable lattice can be described by congruences on
the coefficients. These congruences arise by requiring successive
iterates to remain integral. For \(x=\sum_{i<j}x_{ij}b_{ij}\in E\),
define
\[
 \begin{aligned}
 \ell_4(x)&=x_{45}+\beta x_{25}-\alpha x_{14},&
 \ell_5(x)&=\beta x_{14}-\alpha x_{03},\\
 \ell_6(x)&=\beta x_{03}-\alpha x_{25},&
 \ell_7(x)&=\beta x_{25}-\alpha x_{14}.
 \end{aligned}
\]
We claim that the maximal \(W\)-submodule \(\Lambda\subseteq E\)
such that \(T\Lambda\subseteq\Lambda\) is given by
\begin{equation}\label{supergeneral-lattice-congruences}
 \begin{gathered}
 x_{01}\in p^2W,\qquad
 x_{02},x_{04},x_{12},x_{15}\in pW,\\
 x_{05}-p\alpha x_{25}\in p^2W,\qquad
 \ell_4(x),\ell_5(x),\ell_6(x),\ell_7(x)\in pW.
 \end{gathered}
\end{equation}
These conditions contain \(p^2E\), so they define a full lattice.

For maximality, let \(K\subseteq E\) be \(T\)-stable and take
\(x\in K\). Integrality of \(Tx\) first gives
\(x_{01},x_{02},x_{12}\in pW\). Applying these conditions to \(Tx\)
gives \(x_{01}\in p^2W\) and \(x_{05},x_{15}\in pW\).
Applying the latter conditions once more gives
\(x_{04}\in pW\) and \(x_{05}-p\alpha x_{25}\in p^2W\).
Successively applying the resulting conditions to \(Tx\) gives
\(\ell_4(x)\in pW\), then \(\ell_5(x)\in pW\), then
\(\ell_6(x)\in pW\), and finally \(\ell_7(x)\in pW\).
Thus \(K\subseteq\Lambda\).

The following coordinate identities verify these implications and
also prove stability. Let \(y\) be the coefficient
vector defined by \((Tx)_{ij}=\sigma(y_{ij})\); thus \(y\) is obtained
by applying the matrix in \eqref{supergeneral-uniform-T} to the
coefficient vector of \(x\). The coordinates needed for the first
conditions are
\[
 \begin{aligned}
 y_{01}&=x_{05}-p\alpha x_{25},&y_{02}&=x_{15},\\
 y_{12}&=p^{-1}x_{01}+\alpha x_{12},&y_{13}&=p^{-1}x_{02},\\
 y_{23}&=p^{-1}x_{12},&y_{04}&=px_{35},\\
 y_{15}&=\beta x_{02}+x_{04}-p\alpha x_{24}.&&
 \end{aligned}
\]
The remaining identities are
\begin{equation}\label{supergeneral-lattice-stability-identities}
 \begin{aligned}
 y_{05}-p\alpha y_{25}
   &=p\ell_4(x)-p\alpha\beta x_{12},\\
 \ell_4(y)
   &=\ell_5(x)+\beta^2x_{12}
       +p\bigl((\alpha^2-\beta)x_{23}+x_{34}\bigr),\\
 \ell_5(y)&=\ell_6(x)-p\alpha\beta x_{23},\\
 \ell_6(y)&=\ell_7(x)-\alpha\beta x_{12},\\
 \ell_7(y)&=\ell_5(x)+\beta^2x_{12}+p\alpha^2x_{23}.
 \end{aligned}
\end{equation}
Since \(\sigma\) fixes \(p,\alpha,\beta\) and preserves every ideal
\(p^aW\), the displayed divisibility conditions hold for \(Tx\)
if and only if they hold for \(y\). The identities show that
\eqref{supergeneral-lattice-congruences} is preserved by \(T\).
Consequently \(T\Lambda\subseteq\Lambda\), proving the claim.
This argument uses no division by \(2\), and applies also when
\(p=2\).

Specializing the congruences gives a basis in each case. This makes
the differences between the three lattices explicit and prepares the
quotient calculation. Angle brackets denote the \(W\)-span of the
listed vectors; each list has fifteen elements. In the first case,
\begin{equation}\label{supergeneral-lattice-first}
 \begin{aligned}
 \Lambda_{(1)}=\langle{}
 &p^2b_{01},p^2b_{05},\\
 &pb_{02},pb_{04},pb_{12},pb_{15},pb_{45},\\
 &b_{03},b_{13},b_{14},b_{23},b_{24},b_{25},b_{34},b_{35}
 \rangle_W.
 \end{aligned}
\end{equation}
In the second case, the additional congruences require three more
basis vectors to be divisible by \(p\):
\begin{equation}\label{supergeneral-lattice-second}
 \begin{aligned}
 \Lambda_{(2)}=\langle{}
 &p^2b_{01},p^2b_{05},\\
 &pb_{02},pb_{03},pb_{04},pb_{12},pb_{14},pb_{15},pb_{25},pb_{45},\\
 &b_{13},b_{23},b_{24},b_{34},b_{35}
 \rangle_W.
 \end{aligned}
\end{equation}
In the third case, the congruences couple three coefficients rather
than forcing them separately to vanish modulo \(p\). To express the
resulting basis, put
\begin{equation}\label{supergeneral-fixed-gstar}
 g_*:=b_{25}+b_{03}+b_{14}+pb_{05}.
\end{equation}
Then
\begin{equation}\label{supergeneral-lattice-third}
 \begin{aligned}
 \Lambda_{(3)}=\langle{}
 &p^2b_{01},p^2b_{05},\\
 &pb_{02},pb_{03},pb_{04},pb_{12},pb_{14},pb_{15},pb_{45},\\
 &b_{13},b_{23},b_{24},b_{34},b_{35},g_*
 \rangle_W.
 \end{aligned}
\end{equation}
Indeed, the last four congruences in this case reduce to
\(x_{03}\equiv x_{14}\equiv x_{25}\pmod p\) and
\(x_{45}\equiv0\pmod p\). Subtracting \(x_{25}g_*\) gives exactly
the remaining divisibilities in the displayed basis.
Moreover, \eqref{supergeneral-uniform-T} gives \(Tg_*=g_*\).

The matrix of \(F\) on \(M\) has determinant \(p^3\), and hence
the matrix of \(T\) on \(E\) has determinant
\((p^3)^5/p^{15}=1\). The integral matrix of \(T\) on each
\(\Lambda_{(i)}\) therefore has unit determinant. Thus \(T\) is
bijective on all three lattices. Over algebraically closed \(k\),
Frobenius descent for this slope-zero crystal gives
\[
 \Lambda_{(i)}=W\cdot\Lambda_{(i)}^{T=1}
             =W\cdot E^{F_E=p}.
\]
For the last equality, any element of \(E^{F_E=p}\) spans a
\(T\)-stable \(W\)-submodule of \(E\), and hence belongs to the
maximal lattice.

For the quotient gauge, we also need the first Nygaard step. Write
\(E^1=\Fil_N^1(E)=E\cap F_E^{-1}(pE)\). Its basis is
obtained from the \(b_{ij}\)'s by replacing
\(b_{01},b_{02},b_{12}\) with \(pb_{01},pb_{02},pb_{12}\).
The preceding bases give the following elementary divisors:
\[
 \begin{array}{c|c|c}
  i&E/\Lambda_{(i)}&E^1/\Lambda_{(i)}\\ \hline
  1&(W/p^2W)^2\oplus k^5&W/p^2W\oplus k^4\\
  2&(W/p^2W)^2\oplus k^8&W/p^2W\oplus k^7\\
  3&(W/p^2W)^2\oplus k^7&W/p^2W\oplus k^6
 \end{array}
\]

The elementary divisors determine the sizes of the quotient modules,
but the gauge arrows are needed to identify their domino factors.
We therefore describe the quotient \(F\)-gauge before separating its
two \(p\)-layers. Write $E^q=\Fil_N^qE$. The Frobenius formulas above give
\[
\begin{aligned}
E^0&=E,\\
E^1&=\langle pb_{01},pb_{02},pb_{12},\ b_{ij}\ (0\leq i<j\leq5,\ j\geq3)\rangle_W,\\
E^2&=\langle p^2b_{01},p^2b_{02},p^2b_{12},\
pb_{ij}\ (0\leq i\leq2<j\leq5),\ b_{34},b_{35},b_{45}\rangle_W.
\end{aligned}
\]
Let $\mathcal G$ be the gauge of level $[0,2]$ with pieces $E^q$,
with $\tF$ given by multiplication by $p$, $\tV$ by inclusion, and
$\tau=F_E/p^2:E^2\xrightarrow{\sim}E^0$.
For the corresponding maximal lattice $\Lambda$, the subgauge
\[
\mathcal H=(\Lambda,\Lambda,p\Lambda)\subset\mathcal G
\]
is the extension to level $[0,2]$ of a gauge of level $[1,2]$.
Indeed, it is the gauge of the Dieudonn\'e module
$D_\Lambda=(\Lambda,F_D=T,V_D=pT^{-1})$ placed in bidegree $(1,-1)$.
In particular, $V_D$ is topologically nilpotent.
Put $Q=\mathcal G/\mathcal H$, so that
\begin{equation}\label{supergeneral-quotient-gauge}
Q^0=E/\Lambda,\qquad Q^1=E^1/\Lambda,\qquad Q^2=E^2/p\Lambda.
\end{equation}
The map $\tau$ identifies the last piece with the first semilinearly.
Recording degree $2$ through this identification, the four arrows are
\[
\begin{aligned}
\tF_0:Q^0\longrightarrow Q^1,&\qquad x\longmapsto px,\\
\tV_1:Q^1\longrightarrow Q^0,&\qquad x\longmapsto x,\\
\tau\circ\tF_1:Q^1\longrightarrow Q^0,&\qquad x\longmapsto Tx,\\
\tV_2\circ\tau^{-1}:Q^0\longrightarrow Q^1,&\qquad x\longmapsto pT^{-1}x.
\end{aligned}
\]
In particular, $Q^1$ is a submodule of $Q^0$. In this description,
$T$ is a partial operator on $Q^0$ with domain $Q^1$. Keeping track
of this domain will determine where the chains in the calculation
below terminate.

The displayed lattice bases show that $p^2Q=0$. Define
$K=Q[p]=\ker(p:Q\to Q)$ and $P=pQ$. There is a short exact sequence
of \(F\)-gauges
\begin{equation}\label{supergeneral-p-layers}
0\longrightarrow K\longrightarrow Q\xrightarrow{p}P\longrightarrow0.
\end{equation}
On both $K$ and $P$, the arrows $x\mapsto px$ and
$x\mapsto pT^{-1}x$ from degree $0$ to degree $1$ vanish.
For the second assertion, if $px\in\Lambda$, then
$pT^{-1}x=T^{-1}(px)\in\Lambda$, since $T^{-1}\Lambda=\Lambda$.
This applies to $K$, and also to $P$ because $p^2Q=0$.
Thus each layer is described by two $k$-vector spaces $C\subset B$
and the two maps
\[
\iota:C\hookrightarrow B,\qquad T:C\longrightarrow B,
\]
where $\iota$ is linear and $T$ is $\sigma$-linear. Degree $2$ is
recovered from $B$ using $\tau$.

For $s\geq1$, a direct application of \Cref{definition of FG} to $U_s$
shows that $\FG^{\Tot}(U_s)$ is quasi-isomorphic to the gauge
$U_s:=\HH^0(\FG^{\Tot}(U_s))$ placed in complex degree $0$.
This gauge has the following model:
\[
B_s=\langle u_0,\ldots,u_{s-1}\rangle_k,\qquad
C_s=\langle u_1,\ldots,u_{s-1}\rangle_k\subset B_s,\qquad
Tu_\ell=u_{\ell-1}\quad(1\leq\ell<s).
\]
Thus a chain of $s$ vertices gives the gauge of $U_s$ when its
nonterminal vertices form a basis of the middle space. The terminal
vertex lies outside the domain of the partial operator. In particular,
the index of the domino counts all vertices, including the terminal
one.

For the calculations below, introduce the following classes in $Q^0$:
\begin{equation}\label{supergeneral-quotient-basis-notation}
\begin{aligned}
a&=\overline{pb_{01}},&b&=\overline{b_{02}},&c&=\overline{b_{03}},&
d&=\overline{b_{04}},&e&=\overline{pb_{05}},\\
f&=\overline{b_{12}},&g&=\overline{b_{14}},&h&=\overline{b_{15}},&
i&=\overline{b_{25}},&j&=\overline{b_{45}}.
\end{aligned}
\end{equation}
The letter $g$ here denotes a quotient class, not the lattice
generator $g_*$ used in the third case above.
In every case,
\[
P^0=\langle a,e\rangle_k,\qquad
P^1=\langle e\rangle_k\subset P^0,\qquad Te=a.
\]
Consequently the quotient in \eqref{supergeneral-p-layers} is
\begin{equation}\label{supergeneral-p-image-U2}
P\cong U_2.
\end{equation}

The image of multiplication by \(p\) is therefore the same in all
three cases. The differences lie in the \(p\)-kernels, which we
describe by choosing bases adapted to the partial operator.
In each case put $B=K^0$ and regard $C=K^1$ as a subspace of $B$.
All spans in the remainder of this step are over $k$.

In case (1), \(c=g=i=0\), and
\[
B=\langle a,b,d,e,f,h,j\rangle_k,\qquad
C=\langle a,d,e,h,j\rangle_k.
\]
The complete table on $C$ is
\[
Ta=f,\qquad Td=h,\qquad Te=a,\qquad Th=b,\qquad Tj=e.
\]
It gives the two chains
\[
j\xrightarrow{T}e\xrightarrow{T}a\xrightarrow{T}f,
\qquad
d\xrightarrow{T}h\xrightarrow{T}b.
\]
The two chains already supply the required bases: their vertices
form a basis of $B$, and their nonterminal vertices form a basis of
$C$. Hence
\[
K\cong U_4\oplus U_3.
\]

In case (2), all ten classes in
\eqref{supergeneral-quotient-basis-notation} form
a basis of $B$, and
\[
C=\langle a,c,d,e,g,h,i,j\rangle_k.
\]
The complete table is
\[
\begin{aligned}
Ta&=f,&Tc&=g,&Td&=h,&Te&=a,\\
Tg&=i,&Th&=b,&Ti&=c-a,&Tj&=e.
\end{aligned}
\]
The chain $d\mapsto h\mapsto b$ again gives a summand $U_3$.
On the complementary seven-dimensional space, the displayed table
does not yet have chain form. The following change of basis puts it
in that form:
\[
\begin{array}{c|ccccccc}
\ell&6&5&4&3&2&1&0\\ \hline
x_\ell&g+j&i+e&c&g&i&c-a&g-f.
\end{array}
\]
These satisfy $Tx_\ell=x_{\ell-1}$ for $1\leq\ell\leq6$;
in particular, $T(i+e)=(c-a)+a=c$.
The change of basis is invertible, with inverse
\[
\begin{aligned}
g&=x_3,&i&=x_2,&c&=x_4,\\
a&=x_4-x_1,&e&=x_5-x_2,&j&=x_6-x_3,&f&=x_3-x_0.
\end{aligned}
\]
Thus $x_1,\ldots,x_6,d,h$ form a basis of $C$, while adjoining
$x_0,b$ gives a basis of $B$. The two terminal classes modulo $C$
are $-f,b$. Therefore
\[
K\cong U_7\oplus U_3.
\]

In case (3), the lattice generator \(g_*\) gives the relation
\(i=-c-g-e\).
Consequently
\[
B=\langle a,b,c,d,e,f,g,h,j\rangle_k,\qquad
C=\langle a,c,d,e,g,h,j\rangle_k.
\]
The table on $C$ is
\[
\begin{aligned}
Ta&=f,&Tc&=g,&Td&=h,&Te&=a,\\
Tg&=-c-g-e,&Th&=b,&Tj&=e.
\end{aligned}
\]
The chain $d\mapsto h\mapsto b$ still gives $U_3$. The relation
from the lattice leaves a six-dimensional complement, for which we
choose the basis
\[
\begin{array}{c|cccccc}
\ell&5&4&3&2&1&0\\ \hline
y_\ell&c+g+j&-c&-g&c+g+e&-c-e+a&-g-a+f.
\end{array}
\]
Then $Ty_\ell=y_{\ell-1}$ for $1\leq\ell\leq5$.
For example, $T(c+g+j)=g+(-c-g-e)+e=-c$.
The inverse change of basis is
\[
\begin{aligned}
c&=-y_4,&g&=-y_3,&e&=y_2+y_4+y_3,\\
a&=y_1+y_2+y_3,&j&=y_5+y_4+y_3,&f&=y_0+y_1+y_2.
\end{aligned}
\]
In particular, $y_1,\ldots,y_5,d,h$ form a basis of $C$, and the
two terminal classes modulo $C$ are $f,b$. Hence
\[
K\cong U_6\oplus U_3.
\]
Each change of basis, as well as its inverse, has integral
coefficients fixed by $\sigma$. The chain descriptions therefore
respect the semilinear maps and remain valid in every characteristic,
including $p=2$.

The two layers are now identified, so we can pass from the gauge
calculation to an extension of dominos.
Write \(s=4,7,6\) in cases (1), (2), (3), respectively. By
\Cref{Ekedahl's equivalence theorem}, the preceding calculations give
\[
\begin{aligned}
\widehat{(\rmL\rmt)^{\Tot}}(K)&\simeq U_3\oplus U_s,\\
\widehat{(\rmL\rmt)^{\Tot}}(P)&\simeq U_2.
\end{aligned}
\]
Set \(D=\widehat{(\rmL\rmt)^{\Tot}}(Q)\). Applying the same functor
to \eqref{supergeneral-p-layers} gives a distinguished triangle
\[
U_3\oplus U_s\longrightarrow D\longrightarrow U_2
\longrightarrow (U_3\oplus U_s)[1].
\]
Both end terms are graded $R$-modules concentrated in complex
degree $0$. Taking cohomology for the usual $t$-structure on $D(R)$
therefore shows that \(D\) is represented by a graded
$R$-module, with an exact sequence
\begin{equation}\label{supergeneral-domino-extension}
0\longrightarrow U_3\oplus U_s\longrightarrow D
\longrightarrow U_2\longrightarrow0.
\end{equation}
Dominos are closed under extensions, so $D$ is a domino. The passage
to this exact sequence uses the distinguished triangle and the
concentration of its two computed end terms in degree zero. No
general preservation of $p$-kernels by the completed inverse functor
is required.

It remains to identify $D$ with the geometric domino $D_{0,2}$.
This comparison is where the Mazur--Ogus results connect the lattice
calculation to the de Rham--Witt complex.
Let $C_A=R\Gamma(A,W\Omega^\bullet_{A/k})$ and
$X=\widetilde{\HH}^{2}(C_A)\in\dc$; this notation places the diagonal
cohomology object in the diagonal heart, with no additional shift.
The usual cohomology formulas for an abelian variety give torsion-free
crystalline cohomology and equality of the Betti number with the sum
of Hodge numbers in each total degree. In degree $2$, this is
\[
\operatorname{rank}_W\HH^2_{\cris}(A/W)=15
=h^{0,2}+h^{1,1}+h^{2,0}=3+9+3,
\qquad \HH^2_{\cris}(A/W)\cong\wedge^2M=E.
\]
Thus $C_A$ is Mazur--Ogus in every degree. By
\Cref{properties of Mazur--Ogus in degree n,Totsubquo,Mazur-Ogus object and the Hodge F-gauge structure of virtual F-crystal},
$X$ is Mazur--Ogus and its \(F\)-gauge is the Nygaard gauge
$\FG^{\Tot}(X)\cong\mathcal G$ of $E$.
Let $H$ be $D_\Lambda$ placed in bidegree $(1,-1)$, so that
$\FG^{\Tot}(H)\cong\mathcal H$. Applying the completed inverse functor to
$0\to\mathcal H\to\mathcal G\to Q\to0$ gives
\[
H\longrightarrow X\longrightarrow D\longrightarrow H[1].
\]
Since $H$ has usual cohomology only in degree $-1$, this triangle
identifies $\HH^0(X)$ with $D$.

On the other hand, the commuting $t$-structures of
\Cref{diagonal t-structure} identify
\[
\HH^0(X)\cong
\frac{\dt_{\leqslant0}\HH^2(C_A)}{\dt_{\leqslant-1}\HH^2(C_A)}.
\]
The denominator is zero because $\HH^2(C_A)$ is concentrated in
nonnegative $R$-degrees. The numerator is the graded $R$-module
\[
\HH^2(A,WO_A)\longrightarrow
F^\infty B_2^{1,2}\bigl(\HH^2(A,W\Omega^1_{A/k})\bigr).
\]
The identification with $D$ shows that this graded module is a
domino. Its degree-zero term therefore has zero
$V^{-\infty}Z_2^{0,2}$ part, so forming its associated domino leaves it
unchanged. It is consequently $D_{0,2}$ as defined in the example.
Substituting $s=4,7,6$ in \eqref{supergeneral-domino-extension}
proves the three asserted exact sequences.

\end{proof}





\appendix
\section{Supplementary proofs}
\label{sec:appendix}

\subsection{From \(R_1\) to \(R_n\)}
\label{subsec:complete-vanishing}

The following formulation adjusts the ranges in
\cite[Proposition I.1.1]{Ek3}; throughout, the first index is the
internal degree and the second is the cohomological degree.

\begin{proposition}\label{Ek2I.1.1}
    Let \(A\) be a thick subcategory of ungraded \(W\)-modules,
    stable under \(\sigma_*\), and let \(M\in D(R)\).

    {\rm (1)} Suppose that \(M\) is bounded to the left or bounded above.
    If, for some \(r,s\in\Z\), one has
    \(\HH^j(R_1\otimes_R^{\rmL}M)^i\in A\) for every
    \[\begin{aligned}
        (i,j)\in \;\;&\bigl\{(i,j)\in \Z^2\big| i+j=r,\;\text{and}\;j\geqslant s\bigr\}\\
        \cup &\bigl\{(i,j)\in \Z^2\big|i+j=r+1,\;\text{and}\;j\geqslant s+1\bigr\}\\
        \cup &\bigl\{(i,j)\in \Z^2\big|i+j=r-1,\;\text{and}\;j\geqslant s\bigr\},
    \end{aligned}\]
    then \(\HH^j(R_n\otimes_R^{\rmL}M)^i\in A\) for every
    \(n\geqslant1\) whenever \(i+j=r\) and \(j\geqslant s\).

    {\rm (2)} Suppose that \(M\) is bounded to the right or bounded below.
    If, for some \(r,s\in\Z\), one has
    \(\HH^j(R_1\otimes_R^{\rmL}M)^i\in A\) for every
    \[\begin{aligned}
        (i,j)\in\;\; &\bigl\{(i,j)\in \Z^2\big| i+j=r,\;\text{and}\; j\leqslant s\bigr\}\\
        \cup&\bigl\{(i,j)\in \Z^2\big|i+j=r+1,\;\text{and}\;j\leqslant s\bigr\}\\
        \cup&\bigl\{(i,j)\in \Z^2\big|i+j=r-1,\;\text{and}\;j\leqslant s-1\bigr\},
    \end{aligned}\]
    then \(\HH^j(R_n\otimes_R^{\rmL}M)^i\in A\) for every
    \(n\geqslant1\) whenever \(i+j=r\) and \(j\leqslant s\).
\end{proposition}

\begin{center}
    \begin{tikzpicture}
        \draw [step=1cm, gray, very thin] (0,0) grid (9,9);

        \draw [thick,->] (0, 1) -- (9, 1);
        \draw [thick,->] (1, 0) -- (1, 9);
        
        \draw [color=brown, fill=brown] (4, 5) circle (0.15);
        \draw [color=brown, fill=brown] (4, 6) circle (0.15);
        \draw [color=brown, fill=brown] (3, 5) circle (0.15);
        \draw [color=brown, fill=brown] (3, 6) circle (0.15);
        \draw [color=brown, fill=brown] (3, 7) circle (0.15);
        \draw [color=brown, fill=brown] (2, 6) circle (0.15);
        \draw [color=brown, fill=brown] (2, 7) circle (0.15);
        \draw [color=brown, fill=brown] (2, 8) circle (0.15);
        \draw [color=brown, fill=brown] (1, 7) circle (0.15);
        \draw [color=brown, fill=brown] (1, 8) circle (0.15);
        \draw [color=brown, fill=brown] (1, 9) circle (0.15);
        \draw [color=brown, fill=brown] (0, 8) circle (0.15);
        \draw [color=brown, fill=brown] (0, 9) circle (0.15);
        
        \draw [color=teal, fill=teal] (4, 5) circle (0.07);
        \draw [color=teal, fill=teal] (3, 6) circle (0.07);
        \draw [color=teal, fill=teal] (2, 7) circle (0.07);
        \draw [color=teal, fill=teal] (1, 8) circle (0.07);
        \draw [color=teal, fill=teal] (0, 9) circle (0.07);

        \draw [color=brown, fill=brown] (5, 3) circle (0.15);
        \draw [color=brown, fill=brown] (6, 3) circle (0.15);
        \draw [color=brown, fill=brown] (5, 2) circle (0.15);
        \draw [color=brown, fill=brown] (6, 2) circle (0.15);
        \draw [color=brown, fill=brown] (7, 2) circle (0.15);
        \draw [color=brown, fill=brown] (6, 1) circle (0.15);
        \draw [color=brown, fill=brown] (7, 1) circle (0.15);
        \draw [color=brown, fill=brown] (8, 1) circle (0.15);
        \draw [color=brown, fill=brown] (7, 0) circle (0.15);
        \draw [color=brown, fill=brown] (8, 0) circle (0.15);
        \draw [color=brown, fill=brown] (9, 0) circle (0.15);

        \draw [color=teal, fill=teal] (5, 3) circle (0.07);
        \draw [color=teal, fill=teal] (6, 2) circle (0.07);
        \draw [color=teal, fill=teal] (7, 1) circle (0.07);
        \draw [color=teal, fill=teal] (8, 0) circle (0.07);
    \end{tikzpicture}
\end{center}

We prove the proposition after two preliminary lemmas. Write
\(B\leqq C\) when \(B\) is a subquotient of \(C\) in the ambient
abelian category. This relation is transitive, and \(B\leqq0\)
implies \(B=0\). For a fixed \(M\), abbreviate
\[
    D^{i,j}:=\HH^j(D\otimes_R^{\rmL}M)^i
    \qquad(D=R_n,Z_n,B_n).
\]

\begin{lemma}\label{Ek2I.1.1.1}
    Let \(M\in D(R)\). If \(R_1^{i,j}=0\), then
    \[B_1^{i+1,j}\leqq Z_1^{i,j+1}, \quad B_1^{i,j+1}\leqq Z_1^{i,j+1},\quad Z_1^{i,j}\leqq B_1^{i,j},\quad Z_1^{i,j} \leqq B_1^{i+1,j-1}.\]
\end{lemma}

\begin{proof}
    The exact sequences of graded right \(R\)-modules
    \[\begin{aligned}
        &0\rightarrow Z_1\rightarrow R_1\rightarrow B_1(1)\rightarrow 0,\\
        0&\rightarrow \sigma_*(B_1)\rightarrow \sigma_*(Z_1)\rightarrow R_1\rightarrow 0,
    \end{aligned}\]
    induce the long exact sequences \[\begin{aligned}...\rightarrow B_1^{i+1,j-1}\rightarrow Z_1^{i,j}\rightarrow &R_1^{i,j}\rightarrow B_1^{i+1,j}\rightarrow Z_1^{i,j+1}\rightarrow...\\
    ...\rightarrow \sigma_*(B_1^{i,j})\rightarrow \sigma_*(Z_1^{i,j})\rightarrow &R_1^{i,j} \rightarrow \sigma_*(B_1^{i,j+1})\rightarrow \sigma_*(Z_1^{i,j+1}) \rightarrow  ...\end{aligned}\]
    The assertions follow by setting \(R_1^{i,j}=0\) and removing
    the Frobenius twists in the second sequence.
\end{proof}

\begin{lemma}\label{Ek2I.1.1.2}
    Let \(M\in D(R)\). 

    {\rm (1)} If \(M\) is bounded to the left or bounded above and \(R_1^{i,j}=0\) for all \[(i,j)\in \Bigl\{(i,j)\in \Z^2\Big|i+j=r\; \text{or}\;\; r-1, \;\text{and}\; j\geqslant  s\Bigr\},\] then \(Z_1^{i,j}=B_1^{i,j}=0\) for all \((i,j)\in \Z^2\) such that \(i+j=r\) and \(j\geqslant s\). 

    {\rm (2)}  If \(M\) is bounded to the right or bounded below and \(R_1^{i,j}=0\) for all \[(i,j)\in \Bigl\{(i,j)\in \Z^2\Big|i+j=r\; \text{or}\;\; r-1, \;\text{and}\; i\geqslant  r-s\Bigr\},\] then \(Z_1^{i,j}=B_1^{i,j}=0\) for all \((i,j)\in \Z^2\) such that \(i+j=r\) and \(j\leqslant s\).
\end{lemma}

\begin{proof}
    We prove (1); the proof of (2) is analogous. The finite free
    resolutions \[\begin{aligned}
        &0 \rightarrow R\xrightarrow{(V,-F)} R\oplus R\xrightarrow{F+V} R\rightarrow B_1(1)\rightarrow 0\\
        0\rightarrow& R(-1)\xrightarrow{(F,-d)} R(-1)\oplus R\xrightarrow{dV^2+V} R\rightarrow Z_1\rightarrow 0, 
    \end{aligned}\]
    and the boundedness hypothesis give
    \(Z_1^{i,r-i}=B_1^{i,r-i}=0\) for \(i\ll0\).
    When \(r-i\geqslant s\), repeated application of
    \Cref{Ek2I.1.1.1} yields
    \[Z_1^{i,r-i}\leqq B_1^{i,r-i}\leqq Z_1^{i-1,r-i+1}
    \leqq B_1^{i-1,r-i+1}\leqq\cdots\leqq0.\]
    Hence \(Z_1^{i,r-i}=B_1^{i,r-i}=0\), as required.
\end{proof}

\begin{proof}[Proof of \Cref{Ek2I.1.1}]
We prove (1); the proof of (2) is analogous. We reduce to the case
\(A=0\) and argue by induction on \(n\), using the exact sequences
of graded right \(R\)-modules
\[\begin{aligned}
    0\rightarrow R_n\xrightarrow{\rho} R_{n+1} \rightarrow \mathrm{gr}^n_1\rightarrow 0\\
    0\rightarrow B_n(1)\rightarrow \mathrm{gr}^n_1\rightarrow Z_n\rightarrow 0\\
    0\rightarrow B_1\rightarrow B_{n+1}\rightarrow B_n\rightarrow 0\\
    0\rightarrow B_1\rightarrow Z_{n+1} \rightarrow Z_n\rightarrow  0. 
\end{aligned}\]
It suffices to prove \(\HH^j(D\otimes_R^{\rmL}M)^i=0\) for
\(D=B_1,Z_1,B_1(1)\) whenever \(i+j=r\) and \(j\geqslant s\).
For \(D=B_1,Z_1\), this follows from \Cref{Ek2I.1.1.2}.
For \(D=B_1(1)\), the required vanishing is \(B_1^{i+1,j}=0\).
At the endpoint \(j=s\), \Cref{Ek2I.1.1.1} gives
\(B_1^{r+1-s,s}\leqq Z_1^{r-s,s+1}\), and
\Cref{Ek2I.1.1.2} gives
\(Z_1^{r-s,s+1}\leqq B_1^{r-s,s+1}\leqq\cdots\leqq0\).
For \(j\geqslant s+1\), apply \Cref{Ek2I.1.1.2} on the diagonal
of total degree \(r+1\).
\end{proof}

\subsection{Hodge and conjugate filtration}
\label{subsec:hodge-and-conjugate-filtration}

The reduction by \(R_1\) retains two familiar filtrations from
de Rham cohomology. The Hodge filtration keeps a tail of the
horizontal complex; the conjugate filtration keeps its lower
degrees and ends with cycles. We use the cycle bimodule introduced
above to express these filtrations in terms of derived tensor products.

\begin{definition}\label{Hodge complex and its filtrations}
For \(M\in D(R)\), choose a semi-free resolution \(P\to M\). Here
semi-free means that \(P\) has an exhaustive increasing filtration by
subcomplexes, split on the underlying graded modules, whose successive
quotients are direct sums of the shifts \(R(i)[j]\). Set
\[
A=R_1\otimes_R P,\qquad Z=Z_1\otimes_R P.
\]
These represent the corresponding derived tensor products. Write
\(A^a,Z^a\) for their horizontal-degree-\(a\) columns. Left
multiplication by \(d\) factors through \(Z_1\), giving
\[
\partial:A^a\longrightarrow Z^{a+1},\qquad
j:Z^a\longrightarrow A^a,\qquad
j\partial=d,\quad \partial j=0.
\]
Define
\[
\begin{aligned}
\Hod(M)&:=\Tot(A),\\
t_{\geqslant i}\Hod(M)&:=
 \Tot\bigl(0\longrightarrow A^i\xrightarrow{d}A^{i+1}
                   \xrightarrow{d}\cdots\bigr),\\
\tau_{\leqslant i}\Hod(M)&:=
 \Tot\bigl(\cdots\xrightarrow{d}A^{i-2}
     \xrightarrow{d}A^{i-1}\xrightarrow{\partial}Z^i
     \longrightarrow0\bigr).
\end{aligned}
\]
The columns retain their horizontal degrees, and \(\Tot\) uses direct
sums with differential
\(d+(-1)^a\delta\) on horizontal degree \(a\), where \(\delta\)
is the vertical differential. Thus a single column \(A^i\) contributes
\(A^i[-i]\) to the total complex.
\end{definition}

Compare \cite[Equation~0.3.4]{Ek3}; the semi-free construction also applies
to unbounded complexes.

These constructions recover the usual Hodge and conjugate filtrations.
Indeed, if \(X/k\) is smooth and proper and
\(M=R\Gamma(X,W\Omega_X^\bullet)\), then
\(\Hod(M)\simeq R\Gamma(X,\Omega_{X/k}^\bullet)\).
Under this identification, \(t_{\geqslant i}\Hod(M)\) corresponds to
\(R\Gamma(X,\Omega_{X/k}^{\geqslant i})\), the Hodge filtration,
whereas \(\tau_{\leqslant i}\Hod(M)\) corresponds to
\(R\Gamma(X,\tau_{\leqslant i}\Omega_{X/k}^\bullet)\), the conjugate
filtration. Here the first truncation is the stupid truncation, while
the second is the canonical truncation of the de Rham complex of
sheaves, taken before derived global sections. Their images on
cohomology give the Hodge and conjugate filtrations on
\(\HH^*_{\mathrm{dR}}(X/k)\).

The last arrow in \(\tau_{\leqslant i}\Hod(M)\) is induced by
\(d:R_1\to Z_1\) and satisfies \(\partial d=0\) strictly; it is not
obtained by taking a kernel after derived tensoring.

Since semi-free resolutions are \(K\)-projective and \(K\)-flat, these
constructions descend functorially to \(D(R)\), and likewise to the
derived \(\infty\)-category. This formulation avoids any claim that
direct-sum totalization preserves arbitrary columnwise
quasi-isomorphisms of unbounded double complexes.

The successive steps of these filtrations are related by the
following distinguished triangles in \(D(k)\). They show where
the same column of the reduction enters each filtration, and where
the Frobenius twist appears:
\begin{equation}\label{Hodge filtration triangles}
\begin{aligned}
t_{\geqslant i+1}\Hod(M)&\longrightarrow
t_{\geqslant i}\Hod(M)\longrightarrow
(R_1\otimes_R^{\rmL}M)^i[-i]\longrightarrow
(t_{\geqslant i+1}\Hod(M))[1],\\
\tau_{\leqslant i-1}\Hod(M)&\longrightarrow
\tau_{\leqslant i}\Hod(M)\longrightarrow
\sigma_*^{-1}\bigl((R_1\otimes_R^{\rmL}M)^i\bigr)[-i]
\longrightarrow(\tau_{\leqslant i-1}\Hod(M))[1].
\end{aligned}
\end{equation}
These triangles are the derived form of \cite[Equation~0.3.4]{Ek3}.
The second triangle carries a Frobenius twist because
\(Z_1/B_1\simeq\sigma_*^{-1}R_1\); applying \(\sigma_*\) removes it
from the third term.

The first triangle is the quotient of two consecutive horizontal
tails. For the second, compare the truncated double complexes by the
identity below degree \(i-1\) and by \(j:Z^{i-1}\to A^{i-1}\) in
degree \(i-1\). This is termwise injective because \(P\) is semi-free,
and its quotient has two horizontal columns:
\[
\bigl((R_1/Z_1)\otimes_R P\bigr)^{i-1}
 \xrightarrow{\ d\ }(Z_1\otimes_R P)^i.
\]
Since \(d:R_1/Z_1\to B_1\) is an isomorphism of horizontal degree
\(1\), the totalized quotient is quasi-isomorphic to
\[
\bigl((Z_1/B_1)\otimes_R P\bigr)^i[-i].
\]
Cartier identifies this complex with the third term of the second
triangle.

\subsection{Ringoids and their modules}
\label{subsec:ringoids-and-modules}

The fixed-level construction raises a categorical question: how does
the derived category of modules of a given level sit inside the
derived category of all graded modules? Ringoids give a convenient
language for this comparison. A graded module has one group in each
degree, with operators mapping between them; a ringoid records these
degrees and operators while keeping the individual degrees visible.

\begin{definition}\label{definition of ringoid}
A ringoid is a small preadditive category \(\mathcal C\): its objects
form a set, each \(\mathcal C(s,t)\) is an abelian group, and
composition is bilinear. Thus one can add arrows with the same source
and target, and multiply composable arrows. Each object has an
identity, but direct sums of objects are not required.
\end{definition}

A unital ring is a ringoid with one object, whose endomorphism ring
is the given ring. For a general ringoid, each \(\mathcal C(s,s)\)
is a ring, while \(\mathcal C(s,t)\) records the operators from
\(s\) to \(t\). No grading or commutativity of these rings is assumed.

\begin{definition}\label{modules over a ringoid}
A left \(\mathcal C\)-module is an additive covariant functor
\(M:\mathcal C\to\mathrm{Ab}\). Explicitly, it consists of abelian
groups \(M(s)\) and bilinear actions
\[
\mathcal C(s,t)\times M(s)\longrightarrow M(t),\qquad
(u,x)\longmapsto ux,
\]
with \(1_sx=x\) and \((vu)x=v(ux)\). A morphism \(M\to N\) is a
family of group homomorphisms \(f_s:M(s)\to N(s)\) commuting with
all these actions. We denote this category by
\(\mathcal C\text{-}\mathrm{Mod}\).
A right \(\mathcal C\)-module is an additive functor
\(\mathcal C^{\mathrm{op}}\to\mathrm{Ab}\); an arrow \(u:s\to t\)
then acts from \(N(t)\) to \(N(s)\).
\end{definition}

For a ringoid with one object, these are the usual left and right
modules over its endomorphism ring. More generally, the coefficient
ringoid of graded \(R\)-modules has objects \(\Z\) and morphisms
\(\mathcal R(i,j)=R^{j-i}\), with composition given by multiplication
in \(R\). A left \(\mathcal R\)-module is exactly a graded left
\(R\)-module: its value at \(i\) is \(M^i\), and a homogeneous
operator \(r\in R^{j-i}\) acts from \(M^i\) to \(M^j\).
In particular, \(F,V\) act at each vertex and \(d\) maps vertex
\(i\) to vertex \(i+1\). The semilinearity of \(F,V\) is encoded by
the relations with the scalar operators in \(W\); the ringoid is
preadditive, not required to be \(W\)-linear.

\begin{proposition}\label{basic properties of ringoid modules}
Let \(\mathcal C\) be a ringoid.

{\rm (1)} The category \(\mathcal C\text{-}\mathrm{Mod}\) is abelian,
and a sequence is exact if and only if it is exact at every object
of \(\mathcal C\). All small limits and colimits exist and are
computed objectwise. In particular,
\[
\left(\bigoplus_\lambda M_\lambda\right)(s)
 =\bigoplus_\lambda M_\lambda(s),\qquad
\left(\prod_\lambda M_\lambda\right)(s)
 =\prod_\lambda M_\lambda(s).
\]
Direct sums, products, and filtered colimits are exact.

{\rm (2)} The representable modules
\(P_s=\mathcal C(s,-)\) form a set of projective generators, and
\(\Hom_{\mathcal C}(P_s,M)\cong M(s)\).
In particular, every module has a resolution by direct sums of the
\(P_s\).

{\rm (3)} The category \(\mathcal C\text{-}\mathrm{Mod}\) is a
Grothendieck abelian category and has enough injectives.
\end{proposition}

\begin{proof}
Kernels, cokernels, limits, and colimits of abelian groups inherit
the operator actions objectwise. The abelian category axioms and the
exactness assertions therefore follow from those for abelian groups.
An additive natural transformation \(P_s\to M\) is determined by
the image of \(1_s\), giving the stated Hom formula. Since evaluation
at \(s\) is exact, \(P_s\) is projective. The maps corresponding to
all elements \(x\in M(s)\) give an epimorphism
\[
\bigoplus_{s\in\operatorname{Ob}(\mathcal C)}
\ \bigoplus_{x\in M(s)}P_s\longrightarrow M.
\]
Repeating this construction on kernels gives a projective resolution.
Finally, \(\bigoplus_sP_s\) is a generator and filtered colimits are
exact, so the category is Grothendieck and hence has enough
injectives; see the Stacks Project,
\href{https://stacks.math.columbia.edu/tag/079H}{Tag 079H}.
\end{proof}

These statements apply to arbitrary modules. Finite-generation and
completeness conditions require separate consideration, since they
need not be preserved by arbitrary sums or products.

To compare module categories, we also need tensor products over a
ringoid. They are formed at each object and then made compatible
with the arrows. The balancing relation is the usual one for modules
over a ring, with the additional bookkeeping that an arrow can connect
different objects.

\begin{definition}\label{tensor products over a ringoid}
Let \(N\) be a right \(\mathcal C\)-module and \(M\) a left
\(\mathcal C\)-module. Define
\[
N\otimes_{\mathcal C}M
:=\left(\bigoplus_{c\in\operatorname{Ob}(\mathcal C)}
N(c)\otimes_{\Z}M(c)\right)
\Big/\left\langle (nu)\otimes m-n\otimes(um)\right\rangle,
\]
where \(u:c\to d\), \(n\in N(d)\), and \(m\in M(c)\).
The two tensors in this relation belong to the summands indexed by
\(c\) and \(d\), respectively.
\end{definition}

Equivalently, a homomorphism \(N\otimes_{\mathcal C}M\to H\) is a
family of bilinear maps \(N(c)\times M(c)\to H\) satisfying
\(\langle nu,m\rangle=\langle n,um\rangle\).
For a ringoid with one object, this recovers the usual tensor product
over a ring. The representables satisfy
\[
N\otimes_{\mathcal C}\mathcal C(c,-)\cong N(c),\qquad
\mathcal C(-,c)\otimes_{\mathcal C}M\cong M(c),
\]
via \(n\otimes u\mapsto nu\) and \(u\otimes m\mapsto um\).
Tensor product is right exact and preserves colimits in each variable.
In particular, representable modules, and hence projective modules,
are flat.

The same construction applies when additional commuting actions are
present. A \((\mathcal B,\mathcal C)\)-bimodule is a family of groups
\(T(b,c)\), covariant and additive in \(b\), contravariant and
additive in \(c\), with commuting actions. For a left
\(\mathcal C\)-module \(M\), the groups
\((T\otimes_{\mathcal C}M)(b)=T(b,-)\otimes_{\mathcal C}M\)
form a left \(\mathcal B\)-module.

In particular, let \(f:\mathcal C\to\mathcal B\) be an additive
functor. Composition makes \(\mathcal B\) a bimodule over itself,
and restriction along \(f\) supplies the \(\mathcal C\)-actions.
The notation \(\mathcal B\otimes_{\mathcal C}\mathcal B\) denotes
a bimodule, with component from \(x\) to \(y\) given by
\[
\left(\bigoplus_{c\in\operatorname{Ob}(\mathcal C)}
\mathcal B(f(c),y)\otimes_{\Z}\mathcal B(x,f(c))\right)
\Big/\left\langle b f(u)\otimes a-b\otimes f(u)a\right\rangle.
\]
Here \(u:c\to d\), \(a:x\to f(c)\), and \(b:f(d)\to y\).
Thus an elementary tensor describes two composable arrows through an
object \(f(c)\). Composition induces the multiplication map
\(\mathcal B\otimes_{\mathcal C}\mathcal B\to\mathcal B\),
\(b\otimes a\mapsto ba\). This is a map of bimodules; the tensor
product is not being asserted to be a ringoid.

For modules, replacing either factor by a projective resolution gives
\(N\otimes_{\mathcal C}^{\rmL}M\), with
\(\operatorname{Tor}^{\mathcal C}_q(N,M)
=\HH^{-q}(N\otimes_{\mathcal C}^{\rmL}M)\).
The derived bimodule tensor product is computed componentwise as
\[
(\mathcal B\otimes_{\mathcal C}^{\rmL}\mathcal B)(x,y)
=\mathcal B(f(-),y)\otimes_{\mathcal C}^{\rmL}
 \mathcal B(x,f(-)).
\]
Consequently, saying that its multiplication map to \(\mathcal B\)
is a quasi-isomorphism requires both that ordinary multiplication
be an isomorphism and that all higher Tor groups vanish.
Neither assertion is automatic for an additive functor \(f\).

We can now apply this language to modules of fixed level.

\begin{definition}
Let \(\I=[m,n]\). An \(R\)-module of level \(\I\) is a graded
left \(R\)-module of the form
\[
M^m\xrightarrow{d}M^{m+1}\xrightarrow{d}\cdots\xrightarrow{d}M^n
\]
with \(F\) bijective on \(M^n\). Equivalently, the top term carries
an action of \(R^{\mathrm{top}}=W_{\sigma}[F,F^{-1}]\) extending its
\(R^0\)-action. Morphisms are grading-preserving \(R\)-linear maps.
We denote this abelian category by \(R\td\mathrm{Mod}\td\I\).
\end{definition}

\begin{proposition}\label{derived category of R-modules of level I}
The exact inclusion of \(R\)-modules of level \(\I\) into graded
\(R\)-modules induces a fully faithful functor
\[
    D(R\td\mathrm{Mod}\td\I)\longrightarrow D(R).
\]
Its essential image consists precisely of those \(K\in D(R)\) for
which \(\HH^q(K)\) is an \(R\)-module of level \(\I\) for every
\(q\in\Z\).
\end{proposition}

This extends \cite[Proposition~0.5.2]{Ek3} from bounded coherent
complexes to the unbounded derived category.

\begin{proof}
Let \(\mathcal R\) be the coefficient ringoid of graded \(R\)-modules:
its objects are the integers and
\(\mathcal R(i,j)=R^{j-i}\). Let \(\mathcal R_\I\) be the
quotient--localization ringoid whose modules are the \(R\)-modules of
level \(\I\). It is obtained by retaining the consecutive vertices
\(m,\ldots,n\) and, at the terminal vertex, replacing
\(R^0\) by
\[
    R^0[F^{-1}]\cong W_\sigma[F,F^{-1}].
\]
The interval is convex: a composable string of homogeneous operators
whose endpoints lie in \(\I\) cannot leave \(\I\) and then return.
Moreover, \(R^0\to R^0[F^{-1}]\) is a flat Ore localization. Hence the
canonical multiplication map is a quasi-isomorphism
\begin{equation}\label{level-I-homological-epimorphism}
    \mathcal R_\I\otimes_{\mathcal R}^{\rmL}\mathcal R_\I
       \xrightarrow{\ \sim\ }\mathcal R_\I.
\end{equation}
Thus \(\mathcal R\to\mathcal R_\I\) is a homological epimorphism.

Let \(\iota_*:D(\mathcal R_\I)\to D(\mathcal R)\) denote restriction
of scalars. Its left adjoint is
\(\iota^*=\mathcal R_\I\otimes_{\mathcal R}^{\rmL}-\). For
\(X\in D(\mathcal R_\I)\), \eqref{level-I-homological-epimorphism}
gives
\[
\iota^*\iota_*X
 \simeq
 \bigl(\mathcal R_\I\otimes_{\mathcal R}^{\rmL}\mathcal R_\I\bigr)
 \otimes_{\mathcal R_\I}^{\rmL}X
 \simeq X.
\]
The counit of the adjunction is therefore an isomorphism, so
\(\iota_*\) is fully faithful.

Since restriction of scalars is exact, every object in its image has
cohomology of level \(\I\). Conversely, suppose that every
\(\HH^q(K)\) is of level \(\I\). For any such module \(N\), the same
calculation gives
\[
    \mathcal R_\I\otimes_{\mathcal R}^{\rmL}N\simeq N.
\]
For bounded \(K\), the hyper-Tor spectral sequence therefore shows
that the adjunction unit
\[
    K\longrightarrow\iota_*\iota^*K
\]
is a quasi-isomorphism. The general case follows by expressing \(K\)
as the filtered colimit of its bounded truncations; derived tensor
product and restriction of scalars both preserve these colimits.
Hence \(K\) belongs to the essential image of \(\iota_*\).
\end{proof}

\end{document}